\documentclass[10pt]{article}

\usepackage[utf8]{inputenc}

\usepackage{epsf}
\usepackage{amsmath}

\allowdisplaybreaks

\usepackage[showframe=false]{geometry}
\usepackage{changepage}

\usepackage{epsfig}
\usepackage{amssymb}

\usepackage{amsthm}
\usepackage{setspace}
\usepackage{cite}
\usepackage{mcite}

\usepackage{algorithmic}  
\usepackage{algorithm}

\usepackage{shadow}
\usepackage{fancybox}
\usepackage{fancyhdr}

\usepackage{color}
\usepackage[usenames,dvipsnames,svgnames,table]{xcolor}

\definecolor{mypurple}{rgb}{.4,.0,.5}

\usepackage[hyphens]{url}

\usepackage[colorlinks=true,
            linkcolor=black,
            urlcolor=blue,
            citecolor=purple]{hyperref}

\usepackage{breakurl}

\def\y{{\bf y}}

\def\x{{\bf x}}

\def\x{{\mathbf x}}

\def\u{{\bf u}}

\def\x{{\bf x}}
\def\y{{\bf y}}

\def\q{{\bf q}}
\def\m{{\bf m}}

\def\h{{\bf h}}

\def\be{\begin{equation}}
\def\ee{\end{equation}}
\def\ba{\left[\begin{array}}
\def\ea{\end{array}\right]}

\def\u{{\bf u}}

\def\x{{\bf x}}
\def\y{{\bf y}}

\def\q{{\bf q}}

\def\p{{\bf p}}

\def\1{{\bf 1}}

\def\0{{\bf 0}}

\def\calX{{\mathcal X}}
\def\barcalX{ \bar{ {\mathcal X} } }
\def\calY{{\mathcal Y}}

\def\mR{{\mathbb R}}
\def\mN{{\mathbb N}}
\def\mE{{\mathbb E}}
\def\mS{{\mathbb S}}

\def\lp{\left (}
\def\rp{\right )}

\newtheorem{theorem}{Theorem}
\newtheorem{proposition}{Proposition}

\begin{document}

\begin{singlespace}

\title {Fully lifted \emph{blirp} interpolation -- a large deviation view  
}
\author{
\textsc{Mihailo Stojnic
\footnote{e-mail: {\tt flatoyer@gmail.com}} }}
\date{}
\maketitle

\centerline{{\bf Abstract}} \vspace*{0.1in}

\cite{Stojnicnflgscompyx23} introduced a powerful \emph{fully lifted} (fl) statistical interpolating mechanism. It established a nested connection between blirps  (bilinearly indexed random processes) and their decoupled (linearly indexed) comparative counterparts. We here revisit the comparison from \cite{Stojnicnflgscompyx23} and introduce its a \emph{large deviation} upgrade. The new machinery allows to substantially widen the \cite{Stojnicnflgscompyx23}'s range of applicability. In addition to \emph{typical}, studying analytically much harder \emph{atypical} random structures features is now possible as well. To give a bit of a practical flavor, we show how the obtained results connect to the so-called \emph{local entropies} (LE)  and their predicated role in understanding solutions clustering  and associated \emph{computational gaps} in  hard random optimization problems. As was the case in \cite{Stojnicnflgscompyx23}, even though the technical considerations often appear as fairly involved, the final interpolating forms admit elegant expressions thereby providing a relatively easy to use tool readily available for further studies. Moreover, as the considered models encompass all well known random structures discussed in \cite{Stojnicnflgscompyx23}, the obtained results automatically apply to them as well.

\vspace*{0.25in} \noindent {\bf Index Terms: Random processes; comparisons; lifting; large deviations; local entropy}.

\end{singlespace}

\section{Introduction}
\label{sec:back}

In \cite{Stojnicnflgscompyx23} a powerful \emph{fully lifted} (fl)  probabilistic blirp interpolating  mechanism was introduced.  It arrived as a strong upgrade on partially lifted concepts from \cite{Stojnicgscomp16,Stojnicgscompyx16} and the basic ones from \cite{Slep62,Gordon85} (see also, e.g.,  \cite{Sudakov71,Fernique74,Fernique75,Kahane86} for early considerations as well as \cite{Stojnicgscomp16,Adler90,Lifshits85,LedTal91,Tal05} for a brief history, relevance,  and development overview). While the range of applicability in a variety of scientific fields is rather wide, applications in random optimizations are of our prevalent interest. They became particularly  fruitful over the last two decades (some of the most prominent examples include, compressed sensing, machine learning, and neural network  statistical  studies; see,  e.g., \cite{Guerra03,Tal06,Pan10,Pan10a,Pan13,Pan13a,StojnicISIT2010binary,StojnicCSetam09,StojnicUpper10,StojnicICASSP10knownsupp,StojnicICASSP10block,StojnicICASSP10var}). Characterizing typical behavior of their various features ranging from standard optimization metrics (objective values, optimal solutions, relations between optimizing variables) to associated algorithmic ones (accuracy, speed, convergence) became possible in large part due to a strong progress made in understanding and developing powerful comparison mechanisms. For example, many of the above performance metrics often exhibit  the so-called \emph{phase-transition} (PT) phenomenon where they undergo an abrupt change as one moves from one region of system parameters to another. In that light, \cite{StojnicISIT2010binary,StojnicCSetam09,StojnicUpper10,StojnicICASSP10block,StojnicICASSP10knownsupp} present concrete comparative techniques that \emph{precisely} demarcate the succeeding/failing parameters regions and the separating borderline between them -- the so-called phase transition (PT) curve.

While sometimes successful, simple comparison forms often need a major upgrade as one faces intrinsically  more complex structures. An easy to state (but notoriously difficult to solve)  prototype example is the famous SK quadratic (or particularly polynomial/tensorial) model where classical Slepian/Gordon max comparison forms \cite{Slep62,Gordon85,Sudakov71,Fernique74,Fernique75,Kahane86} come up short of enabling a precise characterization (see, e.g., \cite{Guerra03,Tal06,Pan10,Pan10a,Pan13,Pan13a}). Similar intricacies are faced in a host of different models including quadratic, bilinear, or minmax forms appearing in a plethora of applications ranging from high-dimensional geometry and signal processing to  machine learning and neural nets (see, e.g., \cite{StojnicLiftStrSec13,StojnicMoreSophHopBnds10,StojnicRicBnds13,StojnicAsymmLittBnds11,StojnicGardSphNeg13,StojnicGardSphErr13} and references therein).  As observed in \cite{Guerra03,Tal06,Pan10,Pan10a,Pan13,Pan13a} for the quadratic/tensorial max forms, and in \cite{Stojnicgscomp16,Stojnicgscompyx16,Stojnicnflgscompyx23,Stojnicsflgscompyx23} for forms from \cite{StojnicLiftStrSec13,StojnicMoreSophHopBnds10,StojnicRicBnds13,StojnicAsymmLittBnds11,StojnicGardSphNeg13,StojnicGardSphErr13}, the needed foundational upgrades are tightly connected to the core improvements in associated random processes' comparisons.

\subsection{An algorithmic hardness connection}
\label{sec:back1}

The results of \cite{Guerra03,Tal06,Pan10,Pan10a,Pan13,Pan13a,Stojnicgscomp16,Stojnicgscompyx16,Stojnicnflgscompyx23,Stojnicsflgscompyx23} relate to typical behavior most prominently represented via the so-called ground state regime. In recent years, a need for studying large deviations counterparts emerged as well. To a degree replacing the standard NP theory, the \emph{computational gaps} became a golden standard concept in attempting to explain  (in a practically more relevant way) mismatch between the theoretically best possible (the one unrestricted by time constraints) and the practically best achievable algorithmic performance (see, e.g., \cite{MMZ05,GamarSud14,GamarSud17,GamarSud17a,AchlioptasR06,AchlioptasCR11,GamMZ22,AlaouiMS22} for examples of well known problems where such gaps are expected to exist). While resolving the computational gaps mysteries in general remains a grand challenge, a gigantic amount of extraordinary effort was put forth to at least demystify them in some concrete instances \cite{GamarSud14,GamarSud17,GamarSud17a,Bald15,Bald16,Bald20,Bald21,BaldassiBBZ07,BaldMPZ23,AbbLiSly21a,Barb24,BarbAKZ23,GamKizPerXu22}. We single out two key features present in a majority od the existing approaches: \textbf{\emph{(i)}} the Overlap gap property (OGP) \cite{GamarSud14,GamarSud17,GamarSud17a,GamKizPerXu22} and  \textbf{\emph{(ii)}} the Local entropy (LE) \cite{Bald15,Bald16,Bald20,Bald21,BaldassiBBZ07,BaldMPZ23,AbbLiSly21a,Barb24,BarbAKZ23} (other approachers have been developed as well  \cite{GamarnikJW20,Wein23,BandeiraAHSWZ22,HopkinsSS15,HopkinsSSS16,HopkinsKPRSS17,DiakonikolasKS17,FeldmanPV18}, but they are typically tailored towards more specific groups of algorithms such as low-degree functions/polynomials, sum of squares hierarchies, statistical queries, and so on). In addition to their potential role in understanding the mystery of computational gaps, they are of independent interest as they provide deep insights into a hierarchical/clustering organization of random optimization problems solutions. Both OGP and LE concepts are usually very difficult to analytically handle as they often rely on combinatorial or nonconvex geometric structures. Interestingly, a large deviation view of standard statistical mechanics approach to studying associated free energies (and their above mentioned ground states) allows for particularly elegant LEs formulations \cite{Bald16,Bald15}. We here revisit the \emph{fully lifted} (fl) interpolating comparison from \cite{Stojnicnflgscompyx23} and introduce its a large deviation counterpart.  Differently from\cite{Stojnicnflgscompyx23}, which is intended for studying \emph{typical} features of random structures, the large deviations nature of the machinery introduced here allows to substantially extend the range of applicability. In particular, studying analytically much harder \emph{atypical} features is now possible as well. After we present the key technical details in Sections   \ref{sec:gencon}-\ref{sec:rthlev}, we in Section \ref{sec:examples} reconnect back to LEs as atypical random optimizations features and recontextualize the role of the obtained comparison results within the realm of studying  the associated computational gaps.

\section{A large deviation bilinear comparison -- first level of lifting}
\label{sec:gencon}

Consider sets $\calX=\{\x^{(1)},\x^{(2)},\dots,\x^{(l)}\}$, $\barcalX=\{\bar{\x}^{(1)},\bar{\x}^{(2)},\dots,\bar{\x}^{(l)}\}$,  with $\x^{(i)},\bar{\x}^{(i)} \in \mR^n$ and $\calY=\{\y^{(1)},\y^{(2)},\dots,\y^{(l)}\}$ with $\y^{(i)}\in \mR^m$. For vectors $\p=[\p_0,\p_1,\p_2]$  and $\q=[\q_0,\q_1,\q_2]$ with ordered components ($\p_0\geq \p_1\geq \p_2= 0$  and $\q_0\geq \q_1\geq \q_2= 0$), real scalars $p,\beta>0$ and $s$, and functions $f_{\bar{\x}^{(i)}} (\cdot): \mR^n\rightarrow \mR$, we are interested in the following
\begin{equation}\label{eq:genanal1}
 f(G,u^{(4,1)},u^{(4,2)},\calX,\barcalX,\calY,\p,\q,\beta,s,p,f_{\bar{\x}^{(i)}} (\cdot))
 =
 \frac{ \log\lp \sum_{i_3=1}^{l}
\lp  \sum_{i_1=1}^{l}\lp\sum_{i_2=1}^{l}e^{\beta  D_{00}^{(i_1,i_2,i_3)}} \rp^{s}\rp^p\rp  }{p|s|\sqrt{n}},
\end{equation}
where
\begin{eqnarray}\label{eq:genanal1a}
 D_{00}^{(i_1,i_2,i_3)} & \triangleq &
  \lp (\y^{(i_2)})^T
 G\x^{(i_1)}+\|\x^{(i_1)}\|_2\|\y^{(i_2)}\|_2 (u^{(4,1)}+u^{(4,2)})  + f_{\bar{\x}^{(i_3)}} (\x^{(i_1)})    \rp.
 \end{eqnarray}
Our focus is on  random mediums and to that end we consider  random matrices  $G\in \mR^{m\times n}$ with i.i.d. standard normal components and random variables $u^{(4,1)}\sim {\mathcal N}(0,\p_0\q_0-\p_1\q_1)$ and $u^{(4,2)}\sim {\mathcal N}(0,\p_1\q_1)$ (all three, $u^{(4,1)}$, $u^{(4,2)}$, and $G$ are independent of each other). Following into the footsteps of \cite{Stojnicnflgscompyx23}, we take a scalar $\m=[\m_1]$ and observe that the following function might be of critically importance for studying $  f(G,u^{(4,1)},u^{(4,2)},\calX,\barcalX,\calY,\p,\q,\beta,s,p,f_{\bar{\x}^{(i)}} (\cdot) ) $
\begin{equation}\label{eq:genanal2}
\xi(\calX,\barcalX,\calY,\p,\q,\m,\beta,s,p,f_{\bar{\x}^{(i)}} (\cdot))  \triangleq   \frac{ \mE_{G,u^{(4,2)}} \log
\lp
\sum_{i_3=1}^{l} \lp \mE_{u^{(4,1)}}\lp \sum_{i_1=1}^{l}\lp\sum_{i_2=1}^{l}e^{\beta  D_{00}^{(i_1,i_2,i_3)}  } \rp^{s}  \rp^{\m_1} \rp^p \rp  }{p|s|\sqrt{n}\m_1}.
\end{equation}
We adopt the convention that $\mE$ without specified subscript denotes the expectation over any underlying randomness. On the other hand, if the subsciprt is specified then the expectation is only with respect to that specified randomness. Following the main ides of \cite{Stojnicgscomp16,Stojnicgscompyx16,Stojnicnflgscompyx23}, we study properties of $\xi(\calX,\barcalX,\calY,\p,\q,\m,\beta,s,p,f_{\bar{\x}^{(i)}} (\cdot))$  via the following interpolating function $\psi(\cdot)$
\begin{equation}\label{eq:genanal3}
\psi(t)  =
  \frac{ \mE_{G,u^{(4,2)},\u^{(2,2)},\h^{(2)}} \log
\lp
\sum_{i_3=1}^{l} \lp \mE_{u^{(4,1)},\u^{(2,1)},\h^{(1)}}\lp \sum_{i_1=1}^{l}\lp\sum_{i_2=1}^{l}e^{\beta  D_{0}^{(i_1,i_2,i_3)}  } \rp^{s}  \rp^{\m_1} \rp^p \rp  }{p|s|\sqrt{n}\m_1},
\end{equation}
where
\begin{eqnarray}\label{eq:genanal3a}
 D_0^{(i_1,i_2,i_3)} & \triangleq & \sqrt{t}(\y^{(i_2)})^T
 G\x^{(i_1)}+\sqrt{1-t}\|\x^{(i_1)}\|_2 (\y^{(i_2)})^T(\u^{(2,1)}+\u^{(2,2)})\nonumber \\
 & & +\sqrt{t}\|\x^{(i_1)}\|_2\|\y^{(i_2)}\|_2(u^{(4,1)}+u^{(4,2)}) +\sqrt{1-t}\|\y^{(i_2)}\|_2(\h^{(1)}+\h^{(2)})^T\x^{(i_1)}
 + f_{\bar{\x}^{(i_3)}} (\x^{(i_1)}).
 \end{eqnarray}
The $m$ dimensional vectors $\u^{(2,1)}$ and $\u^{(2,2)}$  are comprised of i.i.d. zero-mean Gaussians with variances $\p_0-\p_1$ and $\p_1$, respectively. Analogously, $n$ dimensional vectors $\h^{(1)}$ and $\h^{(2)}$ are comprised of i.i.d. zero-mean Gaussians with variances $\q_0-\q_1$ and $\q_1$, respectively. Also, $\u^{(2,1)}$, $\u^{(2,2)}$, $\h^{(1)}$, and $\h^{(2)}$  are independent among themselves and of $G$, $u^{(4,1)}$, and $u^{(4,2)}$ as well.

From (\ref{eq:genanal3a}) one observes that  $D_0^{(i_1,i_2,i_3)}$ is a Gaussian \emph{bilinearly} indexed  random process (blirp). Moreover, it is not that difficult to see that  $\xi(\calX,\barcalX,\calY,\p,\q,\m,\beta,s,p,f_{\bar{\x}^{(i)}} (\cdot)) = \psi(1)$ and since $\psi(0)$ is often easier to handle than $\psi(1)$, finding a way to connect $\psi(1)$ and $\psi(0)$ is rather desirable. Such a connection would then establish a connection between  $\xi(\calX,\barcalX,\calY,\p,\q,\m,\beta,s,p,f_{\bar{\x}^{(i)}} (\cdot)) $ and $\psi(0)$ and ultimately between the original, \emph{bilinearly }indexed process and two \emph{linearly} indexed decoupled  alternatives.

To ensure the easiness of the exposition, we set
\begin{eqnarray}\label{eq:genanal4}
\bar{\u}^{(i_1,1)} & =  & \frac{G\x^{(i_1)}}{\|\x^{(i_1)}\|_2} \nonumber \\
\u^{(i_1,3,1)} & =  & \frac{(\h^{(1)})^T\x^{(i_1)}}{\|\x^{(i_1)}\|_2} \nonumber \\
\u^{(i_1,3,2)} & =  & \frac{(\h^{(2)})^T\x^{(i_1)}}{\|\x^{(i_1)}\|_2}.
\end{eqnarray}
Also, we denote by $\bar{\u}_j^{(i_1,1)}$ the $j$-th component of $\bar{\u}^{(i_1,1)}$ and write
\begin{eqnarray}\label{eq:genanal5}
\bar{\u}_j^{(i_1,1)} & =  & \frac{G_{j,1:n}\x^{(i_1)}}{\|\x^{(i_1)}\|_2},1\leq j\leq m,
\end{eqnarray}
where $G_{j,1:n}$ is the $j$-th row of $G$. Clearly, for any fixed $i_1$, $\bar{\u}^{(i_1,1)}$ has i.i.d. standard normals elements. As stated above, the i.i.d. zero-mean Gaussians elements of $\u^{(2,1)}$ and $\u^{(2,2)}$  have respective variances $\p_0-\p_1$ and $\p_1-\p_2$ and the i.i.d. zero-mean Gaussians elements of $\u^{(i_1,3,1)}$ and $\u^{(i_1,3,2)}$ have respective variances $\q_0-\q_1$ and $\q_1-\q_2$. After setting ${\mathcal U}_r=\{u^{(4,r)},\u^{(2,r)},\h^{(r)}\},r\in\{1,2\}$, (\ref{eq:genanal3}) can be rewritten as
\begin{equation}\label{eq:genanal6}
\psi(t)  =   \frac{\mE_{G,{\mathcal U}_2}
\log
\lp
\sum_{i_3=1}^{l}
\lp \mE_{{\mathcal U}_1} \lp \sum_{i_1=1}^{l}\lp\sum_{i_2=1}^{l}A_{i_3}^{(i_1,i_2)} \rp^{s}\rp^{\m_1} \rp^p \rp  }{p|s|\sqrt{n}\m_1}
=
 \frac{\mE_{G,{\mathcal U}_2}
\log
\lp
\sum_{i_3=1}^{l}
\lp \mE_{{\mathcal U}_1} Z_{i_3}^{\m_1} \rp^p \rp  }{p|s|\sqrt{n}\m_1},
\end{equation}
where $\beta_{i_1}=\beta\|\x^{(i_1)}\|_2$ and
\begin{eqnarray}\label{eq:genanal7}
B^{(i_1,i_2)} & \triangleq &  \sqrt{t}(\y^{(i_2)})^T\bar{\u}^{(i_1,1)}+\sqrt{1-t} (\y^{(i_2)})^T(\u^{(2,1)}+\u^{(2,2)}) \nonumber \\
D^{(i_1,i_2,i_3)} & \triangleq &  (B^{(i_1,i_2)}+\sqrt{t}\|\y^{(i_2)}\|_2 (u^{(4,1)}+u^{(4,2)})+\sqrt{1-t}\|\y^{(i_2)}\|_2(\u^{(i_1,3,1)}+\u^{(i_1,3,2)})
+ f_{\bar{\x}^{(i_3)}} (\x^{(i_1)}) )
\nonumber \\
A_{i_3}^{(i_1,i_2)} & \triangleq &  e^{\beta_{i_1}D^{(i_1,i_2,i_3)}}\nonumber \\
C_{i_3}^{(i_1)} & \triangleq &  \sum_{i_2=1}^{l}A_{i_3}^{(i_1,i_2)}\nonumber \\
Z_{i_3} & \triangleq & \sum_{i_1=1}^{l} \lp \sum_{i_2=1}^{l} A_{i_3}^{(i_1,i_2)}\rp^s =\sum_{i_1=1}^{l}  (C_{i_3}^{(i_1)})^s.
\end{eqnarray}
As \cite{Stojnicgscomp16,Stojnicgscompyx16,Stojnicnflgscompyx23} demonstrated, to establish a comparative connection between $\psi(1)$ and $\psi(0)$, studying monotonicity of $\psi(t)$ is particularly useful. The following derivative is therefore the key object of our focus
\begin{eqnarray}\label{eq:genanal9}
\frac{d\psi(t)}{dt} & = &
\frac{d}{dt}\lp
\frac{\mE_{G,{\mathcal U}_2}
\log
\lp
\sum_{i_3=1}^{l}
\lp \mE_{{\mathcal U}_1} Z_{i_3}^{\m_1} \rp^p \rp  }{p|s|\sqrt{n}\m_1}
\rp
\nonumber \\
& = &  \mE_{G,{\mathcal U}_2}
\sum_{i_3=1}^{l} \frac{
\lp \mE_{{\mathcal U}_1} Z_{i_3}^{\m_1} \rp^{p-1}    }{|s|\sqrt{n}\m_1  \lp
\sum_{i_3=1}^{l}
\lp \mE_{{\mathcal U}_1} Z_{i_3}^{\m_1} \rp^p \rp   }
\frac{d \mE_{{\mathcal U}_1} Z_{i_3}^{\m_1} }{dt}\nonumber \\
& = &  \mE_{G,{\mathcal U}_2}
\sum_{i_3=1}^{l} \frac{\m_1
\lp \mE_{{\mathcal U}_1} Z_{i_3}^{\m_1} \rp^{p-1}    }{|s|\sqrt{n}\m_1  \lp
\sum_{i_3=1}^{l}
\lp \mE_{{\mathcal U}_1} Z_{i_3}^{\m_1} \rp^p \rp   }
\mE_{{\mathcal U}_1} \frac{1}{Z_{i_3}^{1-\m_1}}\frac{d Z_{i_3}}{dt}\nonumber \\
& = &  \mE_{G,{\mathcal U}_2}
 \sum_{i_3=1}^{l} \frac{\m_1
\lp \mE_{{\mathcal U}_1} Z_{i_3}^{\m_1} \rp^{p-1}    }{|s|\sqrt{n}\m_1  \lp
\sum_{i_3=1}^{l}
\lp \mE_{{\mathcal U}_1} Z_{i_3}^{\m_1} \rp^p \rp   }
\mE_{{\mathcal U}_1} \frac{1}{Z_{i_3}^{1-\m_1}} \frac{d\lp \sum_{i_1=1}^{l} \lp \sum_{i_2=1}^{l} A_{i_3}^{(i_1,i_2)}\rp^s \rp }{dt}\nonumber \\
& = &   \mE_{G,{\mathcal U}_2}
\sum_{i_3=1}^{l} \frac{s\m_1
\lp \mE_{{\mathcal U}_1} Z_{i_3}^{\m_1} \rp^{p-1}    }{|s|\sqrt{n}\m_1  \lp
\sum_{i_3=1}^{l}
\lp \mE_{{\mathcal U}_1} Z_{i_3}^{\m_1} \rp^p \rp   }
\mE_{{\mathcal U}_1} \frac{1}{Z_{i_3}^{1-\m_1}}  \sum_{i=1}^{l} (C_{i_3}^{(i_1)})^{s-1} \nonumber \\
& & \times \sum_{i_2=1}^{l}\beta_{i_1}A_{i_3}^{(i_1,i_2)}\frac{dD^{(i_1,i_2,i_3)}}{dt},
\end{eqnarray}
where
\begin{eqnarray}\label{eq:genanal9a}
\frac{dD^{(i_1,i_2,i_3)}}{dt}= \lp \frac{dB^{(i_1,i_2)}}{dt}+\frac{\|\y^{(i_2)}\|_2 (u^{(4,1)}+u^{(4,2)})}{2\sqrt{t}}-\frac{\|\y^{(i_2)}\|_2 (\u^{(i_1,3,1)}+\u^{(i_1,3,2)})}{2\sqrt{1-t}}\rp.
\end{eqnarray}
Recalling on (\ref{eq:genanal7}) we also have
\begin{eqnarray}\label{eq:genanal10}
\frac{dB^{(i_1,i_2)}}{dt} & = &   \frac{d\lp\sqrt{t}(\y^{(i_2)})^T\bar{\u}^{(i_1,1)}+\sqrt{1-t} (\y^{(i_2)})^T(\u^{(2,1)}+\u^{(2,2)})\rp}{dt} \nonumber \\
 & = &
\sum_{j=1}^{m}\lp \frac{\y_j^{(i_2)}\bar{\u}_j^{(i_1,1)}}{2\sqrt{t}}-\frac{\y_j^{(i_2)}\u_j^{(2,1)}}{2\sqrt{1-t}}-\frac{\y_j^{(i_2)}\u_j^{(2,2)}}{2\sqrt{1-t}}\rp.
\end{eqnarray}
Combining (\ref{eq:genanal9a}) and (\ref{eq:genanal10}) we find
\begin{eqnarray}\label{eq:genanal10a}
\frac{dD^{(i_1,i_2,i_3)}}{dt} & = & \sum_{j=1}^{m}\lp \frac{\y_j^{(i_2)}\bar{\u}_j^{(i_1,1)}}{2\sqrt{t}}-\frac{\y_j^{(i_2)}\u_j^{(2,1)}}{2\sqrt{1-t}}-\frac{\y_j^{(i_2)}\u_j^{(2,2)}}{2\sqrt{1-t}}\rp \nonumber \\
& & +\frac{\|\y^{(i_2)}\|_2 u^{(4,1)}}{2\sqrt{t}}+\frac{\|\y^{(i_2)}\|_2 u^{(4,2)}}{2\sqrt{t}}-\frac{\|\y^{(i_2)}\|_2 \u^{(i_1,3,1)}}{2\sqrt{1-t}}-\frac{\|\y^{(i_2)}\|_2 \u^{(i_1,3,2)}}{2\sqrt{1-t}}.
\end{eqnarray}
Following the strategy of \cite{Stojnicnflgscompyx23}, we rearrange the above terms into three groups that depend on $G$, ${\mathcal U}_2$, and ${\mathcal U}_1$ and write
\begin{eqnarray}\label{eq:genanal10b}
\frac{dD^{(i_1,i_2,i_3)}}{dt} & = & \bar{T}_G+\bar{T}_2+\bar{T}_1,
\end{eqnarray}
where
\begin{eqnarray}\label{eq:genanal10c}
\bar{T}_G & = & \sum_{j=1}^{m} \frac{\y_j^{(i_2)}\bar{\u}_j^{(i_1,1)}}{2\sqrt{t}}\nonumber \\
\bar{T}_2 & = & -\sum_{j=1}^{m}\frac{\y_j^{(i_2)}\u_j^{(2,2)}}{2\sqrt{1-t}}+\frac{\|\y^{(i_2)}\|_2 u^{(4,2)}}{2\sqrt{t}} -\frac{\|\y^{(i_2)}\|_2 \u^{(i_1,3,2)}}{2\sqrt{1-t}} \nonumber\\
\bar{T}_1 & = &  -\sum_{j=1}^{m} \frac{\y_j^{(i_2)}\u_j^{(2,1)}}{2\sqrt{1-t}}  +\frac{\|\y^{(i_2)}\|_2 u^{(4,1)}}{2\sqrt{t}}-\frac{\|\y^{(i_2)}\|_2 \u^{(i_1,3,1)}}{2\sqrt{1-t}}.
\end{eqnarray}
A combination of (\ref{eq:genanal9}) and (\ref{eq:genanal10b}) gives
\begin{equation}\label{eq:genanal10d}
\frac{d\psi(t)}{dt}  =     \mE_{G,{\mathcal U}_2}
\sum_{i_3=1}^{l} \frac{s\m_1
\lp \mE_{{\mathcal U}_1} Z_{i_3}^{\m_1} \rp^{p-1}    }{|s|\sqrt{n}\m_1  \lp
\sum_{i_3=1}^{l}
\lp \mE_{{\mathcal U}_1} Z_{i_3}^{\m_1} \rp^p \rp   }
\mE_{{\mathcal U}_1} \frac{1}{Z_{i_3}^{1-\m_1}}  \sum_{i_1=1}^{l} (C_{i_3}^{(i_1)})^{s-1} \sum_{i_2=1}^{l}\beta_{i_1}A_{i_3}^{(i_1,i_2)}
\lp \bar{T}_G + \bar{T}_2+ \bar{T}_1\rp.
\end{equation}
A bit of additional rearranging allows to write
\begin{equation}\label{eq:genanal10e}
\frac{d\psi(t)}{dt}  =       \frac{\mbox{sign}(s)}{2\sqrt{n}} \sum_{i_1=1}^{l}  \sum_{i_2=1}^{l}
\beta_{i_1}\lp T_G + T_2+ T_1\rp,
\end{equation}
where
\begin{eqnarray}\label{eq:genanal10f}
T_G & = & \sum_{j=1}^{m}\frac{T_{G,j}}{\sqrt{t}}  \nonumber\\
T_2 & = & -\sum_{j=1}^{m}\frac{T_{2,1,j}}{\sqrt{1-t}}-\|\y^{(i_2)}\|_2\frac{T_{2,2}}{\sqrt{1-t}}+\|\y^{(i_2)}\|_2\frac{T_{2,3}}{\sqrt{t}} \nonumber\\
T_1 & = & -\sum_{j=1}^{m}\frac{T_{1,1,j}}{\sqrt{1-t}}-\|\y^{(i_2)}\|_2\frac{T_{1,2}}{\sqrt{1-t}}+\|\y^{(i_2)}\|_2\frac{T_{1,3}}{\sqrt{t}},
\end{eqnarray}
and
\begin{eqnarray}\label{eq:genanal10g}
T_{G,j} & = &  \mE_{G,{\mathcal U}_2}\lp
\sum_{i_3=1}^{l} \frac{
\lp \mE_{{\mathcal U}_1} Z_{i_3}^{\m_1} \rp^{p-1}    }{ \lp
\sum_{i_3=1}^{l}
\lp \mE_{{\mathcal U}_1} Z_{i_3}^{\m_1} \rp^p \rp   }
\mE_{{\mathcal U}_1}\frac{(C_{i_3}^{(i_1)})^{s-1} A_{i_3}^{(i_1,i_2)} \y_j^{(i_2)}\bar{\u}_j^{(i_1,1)}}{Z_{i_3}^{1-\m_1}} \rp \nonumber \\
T_{2,1,j} & = &  \mE_{G,{\mathcal U}_2}\lp
\sum_{i_3=1}^{l} \frac{
\lp \mE_{{\mathcal U}_1} Z_{i_3}^{\m_1} \rp^{p-1}    }{ \lp
\sum_{i_3=1}^{l}
\lp \mE_{{\mathcal U}_1} Z_{i_3}^{\m_1} \rp^p \rp   }
\mE_{{\mathcal U}_1}\frac{(C_{i_3}^{(i_1)})^{s-1} A_{i_3}^{(i_1,i_2)} \y_j^{(i_2)}\u_j^{(2,2)}}{Z_{i_3}^{1-\m_1}} \rp \nonumber \\
T_{2,2} & = &  \mE_{G,{\mathcal U}_2} \lp
\sum_{i_3=1}^{l} \frac{
\lp \mE_{{\mathcal U}_1} Z_{i_3}^{\m_1} \rp^{p-1}    }{ \lp
\sum_{i_3=1}^{l}
\lp \mE_{{\mathcal U}_1} Z_{i_3}^{\m_1} \rp^p \rp   }
\mE_{{\mathcal U}_1}\frac{(C_{i_3}^{(i_1)})^{s-1} A_{i_3}^{(i_1,i_2)} \u^{(i_1,3,2)}}{Z_{i_3}^{1-\m_1}} \rp \nonumber \\
T_{2,3} & = &  \mE_{G,{\mathcal U}_2}\lp
\sum_{i_3=1}^{l} \frac{
\lp \mE_{{\mathcal U}_1} Z_{i_3}^{\m_1} \rp^{p-1}    }{ \lp
\sum_{i_3=1}^{l}
\lp \mE_{{\mathcal U}_1} Z_{i_3}^{\m_1} \rp^p \rp   }
\mE_{{\mathcal U}_1}\frac{(C_{i_3}^{(i_1)})^{s-1} A_{i_3}^{(i_1,i_2)} u^{(4,2)}}{Z_{i_3}^{1-\m_1}} \rp \nonumber \\
T_{1,1,j} & = &  \mE_{G,{\mathcal U}_2} \lp
\sum_{i_3=1}^{l} \frac{
\lp \mE_{{\mathcal U}_1} Z_{i_3}^{\m_1} \rp^{p-1}    }{ \lp
\sum_{i_3=1}^{l}
\lp \mE_{{\mathcal U}_1} Z_{i_3}^{\m_1} \rp^p \rp   }
\mE_{{\mathcal U}_1}\frac{(C_{i_3}^{(i_1)})^{s-1} A_{i_3}^{(i_1,i_2)} \y_j^{(i_2)}\u_j^{(2,1)}}{Z_{i_3}^{1-\m_1}}\rp \nonumber \\
T_{1,2} & = &  \mE_{G,{\mathcal U}_2} \lp
\sum_{i_3=1}^{l} \frac{
\lp \mE_{{\mathcal U}_1} Z_{i_3}^{\m_1} \rp^{p-1}    }{ \lp
\sum_{i_3=1}^{l}
\lp \mE_{{\mathcal U}_1} Z_{i_3}^{\m_1} \rp^p \rp   }
\mE_{{\mathcal U}_1}\frac{(C_{i_3}^{(i_1)})^{s-1} A_{i_3}^{(i_1,i_2)} \u^{(i_1,3,1)}}{Z_{i_3}^{1-\m_1}}\rp \nonumber \\
T_{1,3} & = &  \mE_{G,{\mathcal U}_2}\lp
\sum_{i_3=1}^{l} \frac{
\lp \mE_{{\mathcal U}_1} Z_{i_3}^{\m_1} \rp^{p-1}    }{ \lp
\sum_{i_3=1}^{l}
\lp \mE_{{\mathcal U}_1} Z_{i_3}^{\m_1} \rp^p \rp   }
\mE_{{\mathcal U}_1}\frac{(C_{i_3}^{(i_1)})^{s-1} A_{i_3}^{(i_1,i_2)} u^{(4,1)}}{Z_{i_3}^{1-\m_1}}\rp,
\end{eqnarray}
where we recognize the above seven objects placed in the three mentioned groups as key components in the ensuing computations.  Each of these terms will be handled separately. In the process of doing so, we heavily rely on  \cite{Stojnicnflgscompyx23} and follow the approach presented therein as closely as possible. However, we proceed in a much faster fashion. In particular, whenever possible, we  avoid unnecessary repetitions of the already introduced concepts  and instead focus on key differences.

\subsection{Computing $\frac{d\psi(t)}{dt}$}
\label{sec:compderivative}

 It turns out as convenient to choose a particular order in which we handle the above seven terms. We first handle the last group ($T_{1,1,j}$, $T_{1,2}$, and $T_{1,3}$) to which we refer as $T_1$--group, then we handle the middle group ($T_{2,1,j}$, $T_{2,2}$, and $T_{2,3}$) to which we refer as $T_2$--group, and at the end we handle $T_{G,j}$ to which we refer as $T_G$--group.

\subsubsection{$T_1$--group -- first level}
\label{sec:handlT1}

Each of the three terms from $T_1$--group is handled separately.

\vspace{.0in}
\underline{\textbf{\emph{Determining}} $T_{1,1,j}$}
\label{sec:hand1T11}

Utilizing Gaussian integration by parts we have

 \begin{eqnarray}\label{eq:liftgenAanal19}
T_{1,1,j} & = & \mE_{G,{\mathcal U}_2}\lp
\sum_{i_3=1}^{l} \frac{
\lp \mE_{{\mathcal U}_1} Z_{i_3}^{\m_1} \rp^{p-1}    }{ \lp
\sum_{i_3=1}^{l}
\lp \mE_{{\mathcal U}_1} Z_{i_3}^{\m_1} \rp^p \rp   }
\mE_{{\mathcal U}_1}  \frac{(C_{i_3}^{(i_1)})^{s-1} A_{i_3}^{(i_1,i_2)}\y_j^{(i_2)}}{Z_{i_3}^{1-\m_1}}\rp \nonumber \\
& = & \mE_{G,{\mathcal U}_2} \lp
\sum_{i_3=1}^{l} \frac{
\lp \mE_{{\mathcal U}_1} Z_{i_3}^{\m_1} \rp^{p-1}    }{ \lp
\sum_{i_3=1}^{l}
\lp \mE_{{\mathcal U}_1} Z_{i_3}^{\m_1} \rp^p \rp   }
\mE_{{\mathcal U}_1}\lp \mE_{{\mathcal U}_1}(\u_j^{(2,1)}\u_j^{(2,1)}) \frac{d}{du_j^{(2,1)}}\lp\frac{(C_{i_3}^{(i_1)})^{s-1} A_{i_3}^{(i_1,i_2)}\y_j^{(i_2)}}{Z_{i_3}^{1-\m_1}}\rp\rp\rp \nonumber \\
& = & \mE_{G,{\mathcal U}_2}\lp
\sum_{i_3=1}^{l} \frac{
\lp \mE_{{\mathcal U}_1} Z_{i_3}^{\m_1} \rp^{p-1}    }{ \lp
\sum_{i_3=1}^{l}
\lp \mE_{{\mathcal U}_1} Z_{i_3}^{\m_1} \rp^p \rp   }
\mE_{{\mathcal U}_1}(\u_j^{(2,1)}\u_j^{(2,1)})\mE_{{\mathcal U}_1}\lp  \frac{d}{du_j^{(2,1)}}\lp\frac{(C_{i_3}^{(i_1)})^{s-1} A_{i_3}^{(i_1,i_2)}\y_j^{(i_2)}}{Z_{i_3}^{1-\m_1}}\rp\rp\rp \nonumber \\
& = & (\p_0-\p_1)  \mE_{G,{\mathcal U}_2} \lp
\sum_{i_3=1}^{l} \frac{
\lp \mE_{{\mathcal U}_1} Z_{i_3}^{\m_1} \rp^{p-1}    }{ \lp
\sum_{i_3=1}^{l}
\lp \mE_{{\mathcal U}_1} Z_{i_3}^{\m_1} \rp^p \rp   }
\mE_{{\mathcal U}_1}\lp  \frac{d}{du_j^{(2,1)}}\lp\frac{(C_{i_3}^{(i_1)})^{s-1} A_{i_3}^{(i_1,i_2)}\y_j^{(i_2)}}{Z_{i_3}^{1-\m_1}}\rp\rp\rp.
\end{eqnarray}
Since the inner expectation on the right hand side of the last equality is structurally identical to the one in (19) in \cite{Stojnicnflgscompyx23}, we can write the following analogue to  \cite{Stojnicnflgscompyx23}'s (20)
\begin{eqnarray}\label{eq:liftgenAanal19a}
T_{1,1,j} & = &   (\p_0-\p_1)\mE_{G,{\mathcal U}_2} \lp
\sum_{i_3=1}^{l} \frac{
\lp \mE_{{\mathcal U}_1} Z_{i_3}^{\m_1} \rp^{p-1}    }{ \lp
\sum_{i_3=1}^{l}
\lp \mE_{{\mathcal U}_1} Z_{i_3}^{\m_1} \rp^p \rp   }
\lp \Theta_1+\Theta_2 \rp\rp,
\end{eqnarray}
 with
{\small\begin{eqnarray}\label{eq:liftgenAanal19c}
\Theta_1 &  = &  \mE_{{\mathcal U}_1} \Bigg( \Bigg. \frac{\y_j^{(i_2)} \lp (C_{i_3}^{(i_1)})^{s-1}\beta_{i_1}A_{i_3}^{(i_1,i_2)}\y_j^{(i_2)}\sqrt{1-t} +A_{i_3}^{(i_1,i_2)}(s-1)(C_{i_3}^{(i_1)})^{s-2}\beta_{i_1}\sum_{p_2=1}^{l}A_{i_3}^{(i_1,p_2)}\y_j^{(p_2)}\sqrt{1-t}\rp}{Z_{i_3}^{1-\m_1}}\Bigg. \Bigg)\Bigg. \Bigg) \nonumber \\
\Theta_2 & = & -(1-\m_1)\mE_{{\mathcal U}_1} \lp\sum_{p_1=1}^{l}
\frac{(C_{i_3}^{(i_1)})^{s-1} A_{i_3}^{(i_1,i_2)}\y_j^{(i_2)}}{Z_{i_3}^{2-\m_1}}
s  (C_{i_3}^{(p_1)})^{s-1}\sum_{p_2=1}^{l}\beta_{p_1}A^{(p_1,p_2)}\y_j^{(p_2)}\sqrt{1-t}\rp\Bigg.\Bigg).\nonumber \\
\end{eqnarray}}

\noindent Moreover, the following analogue to \cite{Stojnicnflgscompyx23}'s (23) immediately follows as well
\begin{eqnarray}\label{eq:liftgenAanal19d}
  & &   \hspace{-.3in}  \sum_{i_1=1}^{l} 
 \sum_{i_2=1}^{l}\sum_{j=1}^{m} 
\lp
\sum_{i_3=1}^{l} \frac{
\lp \mE_{{\mathcal U}_1} Z_{i_3}^{\m_1} \rp^{p-1}    }{ \lp
\sum_{i_3=1}^{l}
\lp \mE_{{\mathcal U}_1} Z_{i_3}^{\m_1} \rp^p \rp   }
 \frac{\beta_{i_1}\Theta_1}{\sqrt{1-t}}\rp
   = 
   \nonumber 
   \\
 & = & 
   \Bigg ( \Bigg.
 \sum_{i_3=1}^{l} \frac{
\lp \mE_{{\mathcal U}_1} Z_{i_3}^{\m_1} \rp^{p}    }{ \lp
\sum_{i_3=1}^{l}
\lp \mE_{{\mathcal U}_1} Z_{i_3}^{\m_1} \rp^p \rp   }
 \mE_{{\mathcal U}_1}\frac{Z_{i_3}^{\m_1}}{\mE_{{\mathcal U}_1} Z_{i_3}^{\m_1}}
  \sum_{i_1=1}^{l}\frac{(C_{i_3}^{(i_1)})^s}{Z_{i_3}}\sum_{i_2=1}^{l}\frac{A_{i_3}^{(i_1,i_2)}}{C_{i_3}^{(i_1)}}\beta_{i_1}^2\|\y^{(i_2)}\|_2^2
 \Bigg ) \Bigg.
  \nonumber\\
& & \hspace{-.12in} +  \Bigg ( \Bigg.
 \sum_{i_3=1}^{l} \frac{
\lp \mE_{{\mathcal U}_1} Z_{i_3}^{\m_1} \rp^{p}    }{ \lp
\sum_{i_3=1}^{l}
\lp \mE_{{\mathcal U}_1} Z_{i_3}^{\m_1} \rp^p \rp   }
 \mE_{{\mathcal U}_1}\frac{Z_{i_3}^{\m_1}}{\mE_{{\mathcal U}_1} Z_{i_3}^{\m_1}}
   \sum_{i_1=1}^{l}\frac{(s-1)(C_{i_3}^{(i_1)})^s}{Z_{i_3}}\sum_{i_2=1}^{l}\sum_{p_2=1}^{l}\frac{A_{i_3}^{(i_1,i_2)}A_{i_3}^{(i_1,p_2)}}{(C_{i_3}^{(i_1)})^2}\beta_{i_1}^2(\y^{(p_2)})^T\y^{(i_2)}
 \Bigg ) \Bigg. .\nonumber \\
 \end{eqnarray}
We introduce the operator
\begin{eqnarray}\label{eq:genAanal19e}
 \Phi_{{\mathcal U}_1}^{(i_3)} & \triangleq &  \mE_{{\mathcal U}_{1}} \frac{Z_{i_3}^{\m_1}}{\mE_{{\mathcal U}_{1}}Z_{i_3}^{\m_1}}  \triangleq  \mE_{{\mathcal U}_{1}} \lp \frac{Z_{i_3}^{\m_1}}{\mE_{{\mathcal U}_{1}}Z_{i_3}^{\m_1}}\lp \cdot \rp\rp,
 \end{eqnarray}
and set
\begin{eqnarray}\label{eq:genAanal19e1}
  \gamma_{00}(i_3) & = &
  \frac{
\lp \mE_{{\mathcal U}_1} Z_{i_3}^{\m_1} \rp^{p}    }{ \lp
\sum_{i_3=1}^{l}
\lp \mE_{{\mathcal U}_1} Z_{i_3}^{\m_1} \rp^p \rp   }
\nonumber \\
  \gamma_0(i_1,i_2;i_3) & = &
\frac{(C_{i_3}^{(i_1)})^{s}}{Z_{i_3}}  \frac{A_{i_3}^{(i_1,i_2)}}{C_{i_3}^{(i_1)}} \nonumber \\
\gamma_{01}^{(1)}  & = &  \gamma_{00}(i_3)\Phi_{{\mathcal U}_1}^{(i_3)} (\gamma_0(i_1,i_2;i_3)) \nonumber \\
\gamma_{02}^{(1)}  & = &  \gamma_{00}(i_3)\Phi_{{\mathcal U}_1}^{(i_3)} (\gamma_0(i_1,i_2;i_3)\times \gamma_0(i_1,p_2;i_3)) \nonumber \\
\gamma_{1}^{(1)}   & = &  \gamma_{00}(i_3)\Phi_{{\mathcal U}_1}^{(i_3)}  \lp \gamma_0(i_1,i_2;i_3)\times \gamma_0(p_1,p_2;i_3) \rp \nonumber \\.
\gamma_{21}^{(1)}   & = &  \gamma_{00}(i_3)\Phi_{{\mathcal U}_1}^{(i_3)}   \gamma_0(i_1,i_2;i_3)  \times  \gamma_{00}(p_3) \Phi_{{\mathcal U}_1}^{(p_3)}  \gamma_0(p_1,p_2;p_3)
\nonumber  \\
\gamma_{22}^{(1)}   & = &  \gamma_{00}(i_3)\lp \Phi_{{\mathcal U}_1}^{(i_3)}   \gamma_0(i_1,i_2;i_3)  \times  \Phi_{{\mathcal U}_1}^{(i_3)}  \gamma_0(p_1,p_2;i_3) \rp.
\end{eqnarray}
It is relatively easy to check that all the above $\gamma$'s are valid measures. For example, we have
\begin{eqnarray}\label{eq:genAanal19f}
 \sum_{i_3=1}^{l}  \sum_{i_1=1}^{l}  \sum_{i_2=1}^{l} \gamma_{01}^{(1)}
  & = &
   \sum_{i_3=1}^{l}  \sum_{i_1=1}^{l}  \sum_{i_2=1}^{l} \gamma_{00}(i_3)\Phi_{{\mathcal U}_1}^{(i_3)} (\gamma_0(i_1,i_2;i_3))
   \nonumber \\
  & = & \sum_{i_3=1}^{l}   \sum_{i_1=1}^{l}  \sum_{i_2=1}^{l} \gamma_{00}(i_3)  \mE_{{\mathcal U}_1}\frac{Z_{i_3}^{\m_1}}{\mE_{{\mathcal U}_1} Z_{i_3}^{\m_1}}\frac{(C_{i_3}^{(i_1)})^{s}}{Z_{i_3}}  \frac{A_{i_3}^{(i_1,i_2)}}{C_{i_3}^{(i_1)}} \nonumber \\
& = &
\sum_{i_3=1}^{l}  \gamma_{00}(i_3)  \mE_{{\mathcal U}_1}\frac{Z_{i_3}^{\m_1}}{\mE_{{\mathcal U}_1} Z_{i_3}^{\m_1}}   \sum_{i_1=1}^{l} \frac{(C_{i_3}^{(i_1)})^{s}}{Z_{i_3}}  \sum_{i_2=1}^{l} \frac{A_{i_3}^{(i_1,i_2)}}{C_{i_3}^{(i_1)}}\nonumber\\
 & = &
\sum_{i_3=1}^{l}  \gamma_{00}(i_3)  \mE_{{\mathcal U}_1}\frac{Z_{i_3}^{\m_1}}{\mE_{{\mathcal U}_1} Z_{i_3}^{\m_1}}  \nonumber \\
 & = &
\sum_{i_3=1}^{l}  \gamma_{00}(i_3)   \nonumber \\
 & = & 1,
 \end{eqnarray}
where the fourth equality follows by the definitions of $Z_{i_3}$ and $C_{i_3}^{(i_1)}$  from (\ref{eq:genanal7}). Keeping in mind that one trivially has $\gamma_{01}^{(1)}\geq 0$, (\ref{eq:genAanal19f}) ensures that $\gamma_{01}^{(1)}$ is indeed a measure. The proofs for other $\gamma$'s are identical and we skip them. As a convention, $\langle \cdot \rangle_a$ will denote the average with respect to measure $a$. Additionally, since each $\gamma$ measure (and $\Phi(\cdot)$ operator) defined above are also functions of $t$, all functions that depend on them will be functions of $t$ as well. We skip stating this explicitly to lighten the notation and facilitate writing.

Utilizing (\ref{eq:genAanal19e1}) we can rewrite (\ref{eq:liftgenAanal19d}) as
\begin{eqnarray}\label{eq:liftgenAanal19g}
\sum_{i_1=1}^{l}\sum_{i_2=1}^{l}\sum_{j=1}^{m} \lp
\sum_{i_3=1}^{l} \frac{
\lp \mE_{{\mathcal U}_1} Z_{i_3}^{\m_1} \rp^{p-1}    }{ \lp
\sum_{i_3=1}^{l}
\lp \mE_{{\mathcal U}_1} Z_{i_3}^{\m_1} \rp^p \rp   }
 \frac{\beta_{i_1}\Theta_1}{\sqrt{1-t}}\rp
&  = & \beta^2 \Bigg ( \Bigg. \langle \|\x^{(i_1)}\|_2^2\|\y^{(i_2)}\|_2^2\rangle_{\gamma_{01}^{(1)}}
\nonumber \\
& &
+  (s-1) \langle \|\x^{(i_1)}\|_2^2(\y^{(p_2)})^T\y^{(i_2)}\rangle_{\gamma_{02}^{(1)}} \Bigg ) \Bigg. .
 \end{eqnarray}
In an analogous fashion, from (\ref{eq:liftgenAanal19c}) we find
\begin{eqnarray}\label{eq:liftgenAanal19h}
\sum_{i_1=1}^{l}\sum_{i_2=1}^{l}\sum_{j=1}^{m} \lp
\sum_{i_3=1}^{l} \frac{
\lp \mE_{{\mathcal U}_1} Z_{i_3}^{\m_1} \rp^{p-1}    }{ \lp
\sum_{i_3=1}^{l}
\lp \mE_{{\mathcal U}_1} Z_{i_3}^{\m_1} \rp^p \rp   }
\frac{\beta_{i_1}\Theta_2}{\sqrt{1-t}}\rp
 \hspace{-.07in}
 & = &  \hspace{-.05in} -s(1-\m_1) \mE_{G,{\mathcal U}_2} \Bigg( \Bigg.
 \sum_{i_3=1}^{l} \frac{
\lp \mE_{{\mathcal U}_1} Z_{i_3}^{\m_1} \rp^{p}    }{ \lp
\sum_{i_3=1}^{l}
\lp \mE_{{\mathcal U}_1} Z_{i_3}^{\m_1} \rp^p \rp   }
\nonumber \\
& &
 \hspace{-.05in} \times
 \frac{Z_{i_3}^{\m_1}}{\mE_{{\mathcal U}_1} Z_{i_3}^{\m_1}} \sum_{i_1=1}^{l}\frac{(C_{i_3}^{(i_1)})^s}{Z_{i_3}}\sum_{i_2=1}^{l}
\frac{A_{i_3}^{(i_1,i_2)}}{C_{i_3}^{(i_1)}} \nonumber \\
& &
 \hspace{-.05in}
\times
 \sum_{p_1=1}^{l} \frac{(C_{i_3}^{(p_1)})^s}{Z_{i_3}}\sum_{p_2=1}^{l}\frac{A_{i_3}^{(p_1,p_2)}}{C_{i_3}^{(p_1)}} \beta_{i_1}\beta_{p_1}(\y^{(p_2)})^T\y^{(i_2)} \Bigg.\Bigg)\nonumber \\
& =&  \hspace{-.05in} -s\beta^2(1-\m_1) \mE_{G,{\mathcal U}_2} \langle \|\x^{(i_1)}\|_2\|\x^{(p_1)}\|_2(\y^{(p_2)})^T\y^{(i_2)} \rangle_{\gamma_{1}^{(1)}}.
\nonumber \\
\end{eqnarray}
A combination of  (\ref{eq:liftgenAanal19a}), (\ref{eq:liftgenAanal19g}), and (\ref{eq:liftgenAanal19h}) gives
\begin{eqnarray}\label{eq:liftgenAanal19i}
\sum_{i_1=1}^{l}\sum_{i_2=1}^{l}\sum_{j=1}^{m} \beta_{i_1}\frac{T_{1,1,j}}{\sqrt{1-t
}}
& = &  (\p_0-\p_1) \mE_{G,{\mathcal U}_2} \lp
\sum_{i_3=1}^{l} \frac{
\lp \mE_{{\mathcal U}_1} Z_{i_3}^{\m_1} \rp^{p-1}    }{ \lp
\sum_{i_3=1}^{l}
\lp \mE_{{\mathcal U}_1} Z_{i_3}^{\m_1} \rp^p \rp   }
\lp \frac{\beta_{i_1}\Theta_1}{\sqrt{1-t}}+\frac{\beta_{i_1}\Theta_2}{\sqrt{1-t}} \rp\rp\nonumber \\
& = & (\p_0-\p_1)\beta^2 \nonumber \\
 & &
 \times \lp \mE_{G,{\mathcal U}_2} \langle \|\x^{(i_1)}\|_2^2\|\y^{(i_2)}\|_2^2\rangle_{\gamma_{01}^{(1)}} +  (s-1)\mE_{G,{\mathcal U}_2}\langle \|\x^{(i_1)}\|_2^2(\y^{(p_2)})^T\y^{(i_2)}\rangle_{\gamma_{02}^{(1)}} \rp \nonumber \\
& & - (\p_0-\p_1)s\beta^2(1-\m_1)\langle \|\x^{(i_1)}\|_2\|\x^{(p_1)}\|_2(\y^{(p_2)})^T\y^{(i_2)} \rangle_{\gamma_{1}^{(1)}}.
\end{eqnarray}

\underline{\textbf{\emph{Determining}} $T_{1,2}$}
\label{sec:hand1T12}

Following further the Gaussian integration by parts path we have
\begin{eqnarray}\label{eq:liftgenBanal20}
T_{1,2} \hspace{-.0in} & = & \mE_{G,{\mathcal U}_2} \lp
\sum_{i_3=1}^{l} \frac{
\lp \mE_{{\mathcal U}_1} Z_{i_3}^{\m_1} \rp^{p-1}    }{ \lp
\sum_{i_3=1}^{l}
\lp \mE_{{\mathcal U}_1} Z_{i_3}^{\m_1} \rp^p \rp   }
\mE_{{\mathcal U}_1} \frac{(C_{i_3}^{(i_1)})^{s-1} A_{i_3}^{(i_1,i_2)}\u^{(i_1,3,1)}}{Z_{i_3}^{1-\m_1}}\rp \nonumber \\
& = & \mE_{G,{\mathcal U}_2} \lp
\sum_{i_3=1}^{l} \frac{
\lp \mE_{{\mathcal U}_1} Z_{i_3}^{\m_1} \rp^{p-1}    }{ \lp
\sum_{i_3=1}^{l}
\lp \mE_{{\mathcal U}_1} Z_{i_3}^{\m_1} \rp^p \rp   }
\mE_{{\mathcal U}_1} \sum_{p_1=1}^{l}\mE_{{\mathcal U}_1}(\u^{(i_1,3,1)}\u^{(p_1,3,1)}) \frac{d}{d\u^{(p_1,3,1)}}\lp\frac{(C_{i_3}^{(i_1)})^{s-1} A_{i_3}^{(i_1,i_2)}}{Z_{i_3}^{1-\m_1}}\rp\rp \nonumber \\
& = &  \mE_{G,{\mathcal U}_2} \lp
\sum_{i_3=1}^{l} \frac{(\q_0-\q_1)
\lp \mE_{{\mathcal U}_1} Z_{i_3}^{\m_1} \rp^{p-1}    }{ \lp
\sum_{i_3=1}^{l}
\lp \mE_{{\mathcal U}_1} Z_{i_3}^{\m_1} \rp^p \rp   }
\mE_{{\mathcal U}_1} \sum_{p_1=1}^{l}\frac{(\x^{(i_1)})^T\x^{(p_1)}}{\|\x^{(i_1)}\|_2\|\x^{(p_1)}\|_2} \frac{d}{d\u^{(p_1,3,1)}}\lp\frac{(C_{i_3}^{(i_1)})^{s-1} A_{i_3}^{(i_1,i_2)}}{Z_{i_3}^{1-\m_1}}\rp\rp.
\nonumber \\
\end{eqnarray}
Since the inner expectation on the most right hand side of the last equality is structurally identical to the corresponding one in   \cite{Stojnicnflgscompyx23}'s (30), one can immediately write the following analogue to \cite{Stojnicnflgscompyx23}'s (31)

$ $
{\footnotesize \begin{eqnarray}\label{eq:liftgenBanal20a}
T_{1,2}  & = & \mE_{G,{\mathcal U}_2} \lp
\sum_{i_3=1}^{l} \frac{(\q_0-\q_1)
\lp \mE_{{\mathcal U}_1} Z_{i_3}^{\m_1} \rp^{p-1}    }{ \lp
\sum_{i_3=1}^{l}
\lp \mE_{{\mathcal U}_1} Z_{i_3}^{\m_1} \rp^p \rp   }
\mE_{{\mathcal U}_1} \sum_{p_1=1}^{l}\frac{(\x^{(i_1)})^T\x^{(p_1)}}{\|\x^{(i_1)}\|_2\|\x^{(p_1)}\|_2} \frac{d}{d\u^{(p_1,3,1)}}\lp\frac{(C_{i_3}^{(i_1)})^{s-1} A_{i_3}^{(i_1,i_2)}}{Z_{i_3}^{1-\m_1}}\rp\rp \nonumber \\
& = &
\mE_{G,{\mathcal U}_2} \Bigg( \Bigg.
\sum_{i_3=1}^{l} \frac{(\q_0-\q_1)
\lp \mE_{{\mathcal U}_1} Z_{i_3}^{\m_1} \rp^{p-1}    }{ \lp
\sum_{i_3=1}^{l}
\lp \mE_{{\mathcal U}_1} Z_{i_3}^{\m_1} \rp^p \rp   }
\nonumber \\
& & \times \mE_{{\mathcal U}_1} \lp\frac{(C_{i_3}^{(i_1)})^{s-1}\beta_{i_1}A_{i_3}^{(i_1,i_2)}\|\y^{(i_2)}\|_2\sqrt{1-t}+A_{i_3}^{(i_1,i_2)}(s-1)(C_{i_3}^{(i_1)})^{s-2}\beta_{i_1}\sum_{p_2=1}^{l}A_{i_3}^{(i_1,p_2)}\|\y^{(p_2)}\|_2\sqrt{1-t} }{Z_{i_3}^{1-\m_1}}\rp\Bigg.\Bigg) \nonumber  \\
  & & -\mE_{G,{\mathcal U}_2}\Bigg( \Bigg.
  \sum_{i_3=1}^{l} \frac{(\q_0-\q_1)(1-\m_1)
\lp \mE_{{\mathcal U}_1} Z_{i_3}^{\m_1} \rp^{p-1}    }{ \lp
\sum_{i_3=1}^{l}
\lp \mE_{{\mathcal U}_1} Z_{i_3}^{\m_1} \rp^p \rp   }
  \nonumber \\
& & \times\mE_{{\mathcal U}_1} \lp\sum_{p_1=1}^{l}\frac{(\x^{(i_1)})^T\x^{(p_1)}}{\|\x^{(i_1)}\|_2\|\x^{(p_1)}\|_2}
\frac{(C_{i_3}^{(i_1)})^{s-1} A_{i_3}^{(i_1,i_2)}}{Z_{i_3}^{2-\m_1}}
s  (C_{i_3}^{(p_1)})^{s-1}\sum_{p_2=1}^{l}\beta_{p_1}A_{i_3}^{(p_1,p_2)}\|\y^{(p_2)}\|_2\sqrt{1-t}\rp\Bigg. \Bigg).
\end{eqnarray}}
Utilizing (\ref{eq:genAanal19e1})  we then also have
\begin{eqnarray}\label{eq:liftgenBanal20b}
\sum_{i_1=1}^{l}\sum_{i_2=1}^{l} \beta_{i_1}\|\y^{(i_2)}\|_2 \frac{T_{1,2}}{\sqrt{1-t}} & = & (\q_0-\q_1)\beta^2 \nonumber \\
& & \times
\lp\mE_{G,{\mathcal U}_2}\langle \|\x^{(i_1)}\|_2^2\|\y^{(i_2)}\|_2^2\rangle_{\gamma_{01}^{(1)}} +   (s-1)\mE_{G,{\mathcal U}_2}\langle \|\x^{(i_1)}\|_2^2 \|\y^{(i_2)}\|_2\|\y^{(p_2)}\|_2\rangle_{\gamma_{02}^{(1)}}\rp\nonumber \\
& & - (\q_0-\q_1)s\beta^2(1-\m_1)\mE_{G,{\mathcal U}_2}\langle (\x^{(p_1)})^T\x^{(i_1)}\|\y^{(i_2)}\|_2\|\y^{(p_2)}\|_2 \rangle_{\gamma_{1}^{(1)}}.
\end{eqnarray}

\underline{\textbf{\emph{Determining}} $T_{1,3}$}
\label{sec:hand1T13}

Proceeding again via Gaussian integration by parts we find
\begin{eqnarray}\label{eq:liftgenCanal21}
T_{1,3} & = & \mE_{G,{\mathcal U}_2} \lp
\sum_{i_3=1}^{l} \frac{
\lp \mE_{{\mathcal U}_1} Z_{i_3}^{\m_1} \rp^{p-1}    }{ \lp
\sum_{i_3=1}^{l}
\lp \mE_{{\mathcal U}_1} Z_{i_3}^{\m_1} \rp^p \rp   }
\mE_{{\mathcal U}_1}  \frac{(C^{(i_1)})^{s-1} A^{(i_1,i_2)}u^{(4,1)}}{Z^{1-\m_1}} \rp \nonumber \\
& = & \mE_{G,{\mathcal U}_2} \lp
\sum_{i_3=1}^{l} \frac{
\lp \mE_{{\mathcal U}_1} Z_{i_3}^{\m_1} \rp^{p-1}    }{ \lp
\sum_{i_3=1}^{l}
\lp \mE_{{\mathcal U}_1} Z_{i_3}^{\m_1} \rp^p \rp   }
\mE_{{\mathcal U}_1} \lp\mE_{{\mathcal U}_1} (u^{(4,1)}u^{(4,1)})\lp\frac{d}{du^{(4,1)}} \lp\frac{(C^{(i_1)})^{s-1} A^{(i_1,i_2)}u^{(4,1)}}{Z^{1-\m_1}}\rp \rp\rp\rp \nonumber \\
& = &  (\p_0\q_0-\p_1\q_1) \mE_{G,{\mathcal U}_2} \lp
\sum_{i_3=1}^{l} \frac{
\lp \mE_{{\mathcal U}_1} Z_{i_3}^{\m_1} \rp^{p-1}    }{ \lp
\sum_{i_3=1}^{l}
\lp \mE_{{\mathcal U}_1} Z_{i_3}^{\m_1} \rp^p \rp   }
\mE_{{\mathcal U}_1} \lp\frac{d}{du^{(4,1)}} \lp\frac{(C^{(i_1)})^{s-1} A^{(i_1,i_2)}u^{(4,1)}}{Z^{1-\m_1}}\rp\rp\rp.
\end{eqnarray}
Recognizing that the last inner expectation is structurally identical to the corresponding one  in  \cite{Stojnicnflgscompyx23}'s (33)  allows us to immediately write the following analogues to  \cite{Stojnicnflgscompyx23}'s (34) and (35)

\begin{eqnarray}\label{eq:liftgenCanal21a}
T_{1,3}  & = & \mE_{G,{\mathcal U}_2} \lp
\sum_{i_3=1}^{l} \frac{ (\p_0\q_0-\p_1\q_1)
\lp \mE_{{\mathcal U}_1} Z_{i_3}^{\m_1} \rp^{p-1}    }{ \lp
\sum_{i_3=1}^{l}
\lp \mE_{{\mathcal U}_1} Z_{i_3}^{\m_1} \rp^p \rp   }
\mE_{{\mathcal U}_1} \lp\frac{d}{du^{(4,1)}} \lp\frac{(C_{i_3}^{(i_1)})^{s-1} A_{i_3}^{(i_1,i_2)}u^{(4,1)}}{Z_{i_3}^{1-\m_1}}\rp\rp\rp \nonumber \\
 & = &\mE_{G,{\mathcal U}_2} \Bigg( \Bigg.
 \sum_{i_3=1}^{l} \frac{ (\p_0\q_0-\p_1\q_1)
\lp \mE_{{\mathcal U}_1} Z_{i_3}^{\m_1} \rp^{p-1}    }{ \lp
\sum_{i_3=1}^{l}
\lp \mE_{{\mathcal U}_1} Z_{i_3}^{\m_1} \rp^p \rp   }
 \nonumber \\
& & \times  \mE_{{\mathcal U}_1} \lp\frac{(C_{i_3}^{(i_1)})^{s-1}\beta_{i_1}A_{i_3}^{(i_1,i_2)}\|\y^{(i_2)}\|_2\sqrt{t}+A_{i_3}^{(i_1,i_2)}(s-1)(C_{i_3}^{(i_1)})^{s-2}\beta_{i_1}\sum_{p_2=1}^{l}A_{i_3}^{(i_1,p_2)}\|\y^{(p_2)}\|_2\sqrt{t}}{Z_{i_3}^{1-\m_1}} \rp \Bigg.\Bigg)\nonumber \\
& & -\mE_{G,{\mathcal U}_2} \Bigg( \Bigg.
\sum_{i_3=1}^{l} \frac{ (\p_0\q_0-\p_1\q_1)
\lp \mE_{{\mathcal U}_1} Z_{i_3}^{\m_1} \rp^{p-1}    }{ \lp
\sum_{i_3=1}^{l}
\lp \mE_{{\mathcal U}_1} Z_{i_3}^{\m_1} \rp^p \rp   }
\nonumber \\
& & \times \mE_{{\mathcal U}_1} \lp
\frac{(C_{i_3}^{(i_1)})^{s-1} A_{i_3}^{(i_1,i_2)}}{Z_{i_3}^{2-\m_1}}
s  \sum_{p_1=1}^{l} (C_{i_3}^{(p_1)})^{s-1}\sum_{p_2=1}^{l}\beta_{p_1}A_{i_3}^{(p_1,p_2)}\|\y^{(p_2)}\|_2\sqrt{t}\rp \Bigg. \Bigg),
\end{eqnarray}
and
\begin{eqnarray}\label{eq:liftgenCanal21b}
\sum_{i_1=1}^{l}\sum_{i_2=1}^{l} \beta_{i_1}\|\y^{(i_2)}\|_2 \frac{T_{1,3}}{\sqrt{t}} & = & (\p_0\q_0-\p_1\q_1)\beta^2 \nonumber \\
& & \times
\lp \mE_{G,{\mathcal U}_2}\langle \|\x^{(i_1)}\|_2^2\|\y^{(i_2)}\|_2^2\rangle_{\gamma_{01}^{(1)}} +   (s-1)\mE_{G,{\mathcal U}_2}\langle \|\x^{(i_1)}\|_2^2 \|\y^{(i_2)}\|_2\|\y^{(p_2)}\|_2\rangle_{\gamma_{02}^{(1)}}\rp\nonumber \\
& & - (\p_0\q_0-\p_1\q_1)s\beta^2(1-\m_1)\mE_{G,{\mathcal U}_2}\langle \|\x^{(i_1)}\|_2\|\x^{(p_`)}\|_2\|\y^{(i_2)}\|_2\|\y^{(p_2)}\|_2 \rangle_{\gamma_{1}^{(1)}}. \nonumber \\
\end{eqnarray}

\subsubsection{$T_2$--group  -- first level}
\label{sec:handlT2}

Following the practice established in the previous subsection, each of the three terms from $T_2$-group are handled separately.

\underline{\textbf{\emph{Determining}} $T_{2,1,j}$}
\label{sec:hand1T21}

After applying the Gaussian integration by parts we have
\begin{eqnarray}\label{eq:genDanal19}
T_{2,1,j}& = &  \mE_{G,{\mathcal U}_2}\lp
\sum_{i_3=1}^{l} \frac{
\lp \mE_{{\mathcal U}_1} Z_{i_3}^{\m_1} \rp^{p-1}    }{ \lp
\sum_{i_3=1}^{l}
\lp \mE_{{\mathcal U}_1} Z_{i_3}^{\m_1} \rp^p \rp   }
\mE_{{\mathcal U}_1}\frac{(C_{i_3}^{(i_1)})^{s-1} A_{i_3}^{(i_1,i_2)} \y_j^{(i_2)}\u_j^{(2,2)}}{Z_{i_3}^{1-\m_1}} \rp \nonumber \\
 &  = & \mE_{G,{\mathcal U}_2,{\mathcal U}_1} \sum_{i_3=1}^{l}
  \frac{(C_{i_3}^{(i_1)})^{s-1} A_{i_3}^{(i_1,i_2)}\y_j^{(i_2)}\u_j^{(2,2)}}{Z_{i_3}^{1-\m_1} \lp
\sum_{i_3=1}^{l}
\lp \mE_{{\mathcal U}_1} Z_{i_3}^{\m_1} \rp^p \rp  }
\lp \mE_{{\mathcal U}_1} Z_{i_3}^{\m_1} \rp^{p-1}
\nonumber \\
& = &
\mE_{G,{\mathcal U}_1}\lp
\sum_{i_3=1}^{l}
\mE_{{\mathcal U}_2}\lp\mE_{{\mathcal U}_2} (\u_j^{(2,2)}\u_j^{(2,2)})\frac{d}{d\u_j^{(2,2)}}\lp \frac{(C_{i_3}^{(i_1)})^{s-1} A_{i_3}^{(i_1,i_2)}\y_j^{(i_2)}}{Z_{i_3}^{1-\m_1}  \lp
\sum_{i_3=1}^{l}
\lp \mE_{{\mathcal U}_1} Z_{i_3}^{\m_1} \rp^p \rp   }
\lp \mE_{{\mathcal U}_1} Z_{i_3}^{\m_1} \rp^{p-1}
 \rp\rp\rp \nonumber \\
& = &
\mE_{G,{\mathcal U}_2,{\mathcal U}_1}\lp
\sum_{i_3=1}^{l} \frac{
\p_1 \lp \mE_{{\mathcal U}_1} Z_{i_3}^{\m_1} \rp^{p-1}    }{ \lp
\sum_{i_3=1}^{l}
\lp \mE_{{\mathcal U}_1} Z_{i_3}^{\m_1} \rp^p \rp   }
\frac{d}{d\u_j^{(2,2)}}\lp \frac{(C_{i_3}^{(i_1)})^{s-1} A_{i_3}^{(i_1,i_2)}\y_j^{(i_2)}}{Z_{i_3}^{1-\m_1}}\rp\rp \nonumber \\
& & + \mE_{G,{\mathcal U}_2,{\mathcal U}_1}\lp
\sum_{i_3=1}^{l}
\frac{\p_1\lp \mE_{{\mathcal U}_1} Z_{i_3}^{\m_1} \rp^{p-1} \lp(C_{i_3}^{(i_1)})^{s-1} A_{i_3}^{(i_1,i_2)}\y_j^{(i_2)} \rp}{Z_{i_3}^{1-\m_1}}\frac{d}{d\u_j^{(2,2)}}\lp \frac{1}{  \lp
\sum_{i_3=1}^{l}
\lp \mE_{{\mathcal U}_1} Z_{i_3}^{\m_1} \rp^p \rp  }\rp\rp  \nonumber \\
& & + \mE_{G,{\mathcal U}_2,{\mathcal U}_1}\lp
\sum_{i_3=1}^{l}
\frac{\p_1 \lp(C_{i_3}^{(i_1)})^{s-1} A_{i_3}^{(i_1,i_2)}\y_j^{(i_2)} \rp}{Z_{i_3}^{1-\m_1}  \lp
\sum_{i_3=1}^{l}
\lp \mE_{{\mathcal U}_1} Z_{i_3}^{\m_1} \rp^p \rp  }\frac{d}{d\u_j^{(2,2)}}\lp
\lp \mE_{{\mathcal U}_1} Z_{i_3}^{\m_1} \rp^{p-1}
\rp\rp.\nonumber \\
\end{eqnarray}
To make writing a bit more convenient we set
\begin{eqnarray}\label{eq:genDanal19a}
T_{2,1,j}   =   T_{2,1,j}^{c} +  T_{2,1,j}^{d}+  T_{2,1,j}^{e},
\end{eqnarray}
where
\begin{eqnarray}\label{eq:genDanal19b}
T_{2,1,j}^c &  = &
\mE_{G,{\mathcal U}_2,{\mathcal U}_1}\lp
\sum_{i_3=1}^{l} \frac{
\p_1 \lp \mE_{{\mathcal U}_1} Z_{i_3}^{\m_1} \rp^{p-1}    }{ \lp
\sum_{i_3=1}^{l}
\lp \mE_{{\mathcal U}_1} Z_{i_3}^{\m_1} \rp^p \rp   }
\frac{d}{d\u_j^{(2,2)}}\lp \frac{(C_{i_3}^{(i_1)})^{s-1} A_{i_3}^{(i_1,i_2)}\y_j^{(i_2)}}{Z_{i_3}^{1-\m_1}}\rp\rp  \nonumber \\
T_{2,1,j}^d &  = & \mE_{G,{\mathcal U}_2,{\mathcal U}_1}\lp
\sum_{i_3=1}^{l}
\frac{\p_1\lp \mE_{{\mathcal U}_1} Z_{i_3}^{\m_1} \rp^{p-1} \lp(C_{i_3}^{(i_1)})^{s-1} A_{i_3}^{(i_1,i_2)}\y_j^{(i_2)} \rp}{Z_{i_3}^{1-\m_1}}\frac{d}{d\u_j^{(2,2)}}\lp \frac{1}{  \lp
\sum_{i_3=1}^{l}
\lp \mE_{{\mathcal U}_1} Z_{i_3}^{\m_1} \rp^p \rp  }\rp\rp  \nonumber \\
T_{2,1,j}^e &  = &
\mE_{G,{\mathcal U}_2,{\mathcal U}_1}\lp
\sum_{i_3=1}^{l}
\frac{\p_1 \lp(C_{i_3}^{(i_1)})^{s-1} A_{i_3}^{(i_1,i_2)}\y_j^{(i_2)} \rp}{Z_{i_3}^{1-\m_1}  \lp
\sum_{i_3=1}^{l}
\lp \mE_{{\mathcal U}_1} Z_{i_3}^{\m_1} \rp^p \rp  }\frac{d}{d\u_j^{(2,2)}}\lp
\lp \mE_{{\mathcal U}_1} Z_{i_3}^{\m_1} \rp^{p-1}
\rp\rp.
\end{eqnarray}
 Moreover, we observe that scaling $T_{2,1,j}^c$  by $\p_1$ gives the term structurally identical to the one considered in the first part of Section \ref{sec:hand1T11} scaled by $(\p_0-\p_1)$. That further allows to write
\begin{eqnarray}\label{eq:genDanal19b1}
\sum_{i_1=1}^{l}\sum_{i_2=1}^{l}\sum_{j=1}^{m} \beta_{i_1}\frac{T_{2,1,j}^c}{\sqrt{1-t
}}
& = &  \sum_{i_1=1}^{l}\sum_{i_2=1}^{l}\sum_{j=1}^{m} \beta_{i_1}\frac{T_{1,1,j}}{\sqrt{1-t}}\frac{\p_1}{\p_0-\p_1}\nonumber\\
& = & \p_1\beta^2
\lp \mE_{G,{\mathcal U}_2}\langle \|\x^{(i_1)}\|_2^2\|\y^{(i_2)}\|_2^2\rangle_{\gamma_{01}^{(1)}} +   (s-1)\mE_{G,{\mathcal U}_2}\langle \|\x^{(i_1)}\|_2^2(\y^{(p_2)})^T\y^{(i_2)}\rangle_{\gamma_{02}^{(1)}} \rp\nonumber \\
& & - \p_1s\beta^2(1-\m_1)\mE_{G,{\mathcal U}_2}\langle \|\x^{(i_1)}\|_2\|\x^{(p_1)}\|_2(\y^{(p_2)})^T\y^{(i_2)} \rangle_{\gamma_{1}^{(1)}}.
\end{eqnarray}

To determine  $T_{2,1,j}^d$ we start with
\begin{eqnarray}\label{eq:genDanal20}
\frac{d}{d\u_j^{(2,2)}}\lp \frac{ 1 }{  \lp
\sum_{i_3=1}^{l}
\lp \mE_{{\mathcal U}_1} Z_{i_3}^{\m_1} \rp^p \rp  } \rp
=
- p \sum_{p_3=1}^{l} \frac{
\lp \mE_{{\mathcal U}_1} Z_{p_3}^{\m_1} \rp^{p-1} }{\lp
\sum_{i_3=1}^{l}
\lp \mE_{{\mathcal U}_1} Z_{i_3}^{\m_1} \rp^p  \rp^2}
\mE_{{\mathcal U}_1}\frac{d Z_{p_3}^{\m_1}}{d\u_j^{(2,2)}}.
\end{eqnarray}
From  \cite{Stojnicnflgscompyx23}'s (41) we have
\begin{eqnarray}\label{eq:genDanal21}
\frac{dZ_{p_3}^{\m_1}}{d\u_j^{(2,2)}}  & = & \frac{\m_1}{Z_{p_3}^{1-\m_1}}s  \sum_{p_1=1}^{l}  (C_{p_3}^{(p_1)})^{s-1}\sum_{p_2=1}^{l}
\beta_{p_1}A_{p_3}^{(p_1,p_2)}\y_j^{(p_2)}\sqrt{1-t}.
\end{eqnarray}
Combining (\ref{eq:genDanal20}) and (\ref{eq:genDanal21}) one then finds
\begin{eqnarray}\label{eq:genDanal22}
\frac{d}{d\u_j^{(2,2)}}\lp \frac{ 1 }{  \lp
\sum_{i_3=1}^{l}
\lp \mE_{{\mathcal U}_1} Z_{i_3}^{\m_1} \rp^p \rp  } \rp
 & = &
- p \sum_{p_3=1}^{l} \frac{
\lp \mE_{{\mathcal U}_1} Z_{p_3}^{\m_1} \rp^{p-1} }{\lp
\sum_{i_3=1}^{l}
\lp \mE_{{\mathcal U}_1} Z_{i_3}^{\m_1} \rp^p  \rp^2}
\nonumber  \\
& &
\times
\mE_{{\mathcal U}_1}  \frac{\m_1}{Z_{p_3}^{1-\m_1}}s \sum_{p_1=1}^{l}  (C_{p_3}^{(p_1)})^{s-1}\sum_{p_2=1}^{l}
\beta_{p_1}A_{p_3}^{(p_1,p_2)}\y_j^{(p_2)}\sqrt{1-t}.
\end{eqnarray}
Plugging (\ref{eq:genDanal22}) in  (\ref{eq:genDanal19b}) gives
\begin{eqnarray}\label{eq:genDanal23}
T_{2,1,j}^d &  = & \mE_{G,{\mathcal U}_2,{\mathcal U}_1}\lp
\sum_{i_3=1}^{l}
\frac{\p_1\lp \mE_{{\mathcal U}_1} Z_{i_3}^{\m_1} \rp^{p-1} \lp(C_{i_3}^{(i_1)})^{s-1} A_{i_3}^{(i_1,i_2)}\y_j^{(i_2)} \rp}{Z_{i_3}^{1-\m_1}}\frac{d}{d\u_j^{(2,2)}}\lp \frac{1}{  \lp
\sum_{i_3=1}^{l}
\lp \mE_{{\mathcal U}_1} Z_{i_3}^{\m_1} \rp^p \rp  }\rp\rp
 \nonumber \\
& = & -s\sqrt{1-t}\p_1\m_1 p \mE_{G,{\mathcal U}_2,{\mathcal U}_1}\Bigg( \Bigg.
\sum_{i_3=1}^{l}
\frac{\lp \mE_{{\mathcal U}_1} Z_{i_3}^{\m_1} \rp^{p-1}\lp(C_{i_3}^{(i_1)})^{s-1} A_{i_3}^{(i_1,i_2)}\y_j^{(i_2)} \rp}{Z_{i_3}^{1-\m_1}} \nonumber \\
& & \times
\lp
\sum_{p_3=1}^{l}
\frac{\lp \mE_{{\mathcal U}_1} Z_{p_3}^{\m_1} \rp^{p-1}}{\lp
\sum_{i_3=1}^{l}
\lp \mE_{{\mathcal U}_1} Z_{i_3}^{\m_1} \rp^p  \rp^2}
\mE_{{\mathcal U}_1}  \frac{1}{Z_{p_3}^{1-\m_1}} \sum_{p_1=1}^{l}  (C_{p_3}^{(p_1)})^{s-1}\sum_{p_2=1}^{l}
\beta_{p_1}A_{p_3}^{(p_1,p_2)}\y_j^{(p_2)} \rp\Bigg. \Bigg)\nonumber \\
& = & -s\sqrt{1-t}\p_1\m_1 p \mE_{G,{\mathcal U}_2}\Bigg( \Bigg.
\sum_{i_3=1}^{l}
\frac{\lp \mE_{{\mathcal U}_1} Z_{i_3}^{\m_1} \rp^{p}}{\lp
\sum_{i_3=1}^{l}
\lp \mE_{{\mathcal U}_1} Z_{i_3}^{\m_1} \rp^p  \rp}
\mE_{{\mathcal U}_1}\frac{Z_{i_3}^{\m_1}}{\mE_{{\mathcal U}_1} Z_{i_3}^{\m_1}}
 \frac{(C_{i_3}^{(i_1)})^{s}}{Z_{i_3}}  \frac{A_{i_3}^{(i_1,i_2)}}{C_{i_3}^{(i_1)}}\y_j^{(i_2)} \nonumber \\
& & \times
\lp
\sum_{p_3=1}^{l}
\frac{\lp \mE_{{\mathcal U}_1} Z_{p_3}^{\m_1} \rp^{p}}{\lp
\sum_{i_3=1}^{l}
\lp \mE_{{\mathcal U}_1} Z_{i_3}^{\m_1} \rp^p  \rp}
\mE_{{\mathcal U}_1}  \frac{Z_{p_3}^{\m_1}}{\mE_{{\mathcal U}_1} Z_{p_3}^{\m_1}} \sum_{p_1=1}^{l}  \frac{(C_{p_3}^{(p_1)})^s}{Z_{p_3}}\sum_{p_2=1}^{l}
\frac{A_{p_3}^{(p_1,p_2)}}{C_{p_3}^{(p,1)}}\beta_{p_1}\y_j^{(p_2)}  \rp\Bigg. \Bigg).
\end{eqnarray}
It will turn out as convenient to also note the following
\begin{eqnarray}\label{eq:genDanal24}
 \sum_{i_1=1}^{l}  \sum_{i_2=1}^{l} \sum_{j=1}^{m}  \beta_{i_1}\frac{T_{2,1,j}^d}{\sqrt{1-t}} &  = &  -s\p_1\m_1 p \mE_{G,{\mathcal U}_2}  \sum_{j=1}^{m} \Bigg( \Bigg.
 \sum_{i_3=1}^{l}
\frac{\lp \mE_{{\mathcal U}_1} Z_{i_3}^{\m_1} \rp^{p}}{\lp
\sum_{i_3=1}^{l}
\lp \mE_{{\mathcal U}_1} Z_{i_3}^{\m_1} \rp^p  \rp}
\nonumber \\
& & \times
  \mE_{{\mathcal U}_1}\frac{Z_{i_3}^{\m_1}}{\mE_{{\mathcal U}_1} Z_{i_3}^{\m_1}}
 \sum_{i_1=1}^{l} \frac{(C_{i_3}^{(i_1)})^{s}}{Z_{i_3}}  \sum_{i_2=1}^{l} \frac{A_{i_3}^{(i_1,i_2)}}{C_{i_3}^{(i_1)}} \beta_{i_1}\y_j^{(i_2)} \nonumber \\
& & \times
\Bigg ( \Bigg.
\sum_{p_3=1}^{l}
\frac{\lp \mE_{{\mathcal U}_1} Z_{p_3}^{\m_1} \rp^{p}}{\lp
\sum_{p_3=1}^{l}
\lp \mE_{{\mathcal U}_1} Z_{p_3}^{\m_1} \rp^p  \rp}
\nonumber \\
& & \times
\mE_{{\mathcal U}_1}  \frac{Z_{p_3}^{\m_1}}{\mE_{{\mathcal U}_1} Z_{p_3}^{\m_1}} \sum_{p_1=1}^{l}  \frac{(C_{p_3}^{(p_1)})^s}{Z_{p_3}}\sum_{p_2=1}^{l}
\frac{A_{p_3}^{(p_1,p_2)}}{C_{p_3}^{(p,1)}}\beta_{p_1}\y_j^{(p_2)}   \Bigg . \Bigg ) \Bigg. \Bigg)\nonumber\\
& = & -s\p_1\m_1 p \mE_{G,{\mathcal U}_2} \langle \beta_{p_1}\beta_{i_1}(\y^{(p_2)})^T\y^{(i_2)} \rangle_{\gamma_{21}^{(1)}} \nonumber \\
& = & -s\beta^2\p_1\m_1 p \mE_{G,{\mathcal U}_2} \langle \|\x^{(i_1)}\|_2\|\x^{(p_1)}\|_2(\y^{(p_2)})^T\y^{(i_2)} \rangle_{\gamma_{21}^{(1)}}.
\end{eqnarray}

To determine $T_{2,1,j}^e$ we start by observing
\begin{eqnarray}\label{eq:genDanal24aa0}
 \frac{d}{d\u_j^{(2,2)}}\lp
\lp \mE_{{\mathcal U}_1} Z_{i_3}^{\m_1} \rp^{p-1}
\rp
=
 (p-1) \lp  \mE_{{\mathcal U}_1} Z_{i_3}^{\m_1} \rp^{p-2}
\mE_{{\mathcal U}_1}\frac{d Z_{i_3}^{\m_1}}{d\u_j^{(2,2)}}.
\end{eqnarray}
Combining (\ref{eq:genDanal24aa0}) and (\ref{eq:genDanal21}) one then finds
\begin{eqnarray}\label{eq:genDanal24aa1}
 \frac{d}{d\u_j^{(2,2)}}\lp
\lp \mE_{{\mathcal U}_1} Z_{i_3}^{\m_1} \rp^{p-1}
\rp
=
(p-1)  \lp  \mE_{{\mathcal U}_1} Z_{i_3}^{\m_1} \rp^{p-2}
\mE_{{\mathcal U}_1}
\frac{\m_1}{Z_{i_3}^{1-\m_1}}s  \sum_{p_1=1}^{l}  (C_{i_3}^{(p_1)})^{s-1}\sum_{p_2=1}^{l}
\beta_{p_1}A_{i_3}^{(p_1,p_2)}\y_j^{(p_2)}\sqrt{1-t}.
\end{eqnarray}
Plugging (\ref{eq:genDanal24aa1}) in  (\ref{eq:genDanal19b})  first gives
\begin{eqnarray}\label{eq:genDanal24ab0}
T_{2,1,j}^e &  = &
\mE_{G,{\mathcal U}_2,{\mathcal U}_1}
\Bigg ( \Bigg.
\sum_{i_3=1}^{l}
\frac{\p_1 \lp(C_{i_3}^{(i_1)})^{s-1} A_{i_3}^{(i_1,i_2)}\y_j^{(i_2)} \rp}{Z_{i_3}^{1-\m_1}  \lp
\sum_{i_3=1}^{l}
\lp \mE_{{\mathcal U}_1} Z_{i_3}^{\m_1} \rp^p \rp  }
\nonumber \\
& & \times
(p-1)  \lp  \mE_{{\mathcal U}_1} Z_{i_3}^{\m_1} \rp^{p-2}
\mE_{{\mathcal U}_1}
\frac{\m_1}{Z_{i_3}^{1-\m_1}}s  \sum_{p_1=1}^{l}  (C_{i_3}^{(p_1)})^{s-1}\sum_{p_2=1}^{l}
\beta_{p_1}A_{i_3}^{(p_1,p_2)}\y_j^{(p_2)}\sqrt{1-t}
\Bigg . \Bigg )
\nonumber \\
&  = &
\p_1 s \m_1 (p-1)\mE_{G,{\mathcal U}_2 }
\Bigg ( \Bigg.
\sum_{i_3=1}^{l}
\frac{\lp  \mE_{{\mathcal U}_1} Z_{i_3}^{\m_1} \rp^p }
{ \lp
\sum_{i_3=1}^{l}
\lp \mE_{{\mathcal U}_1} Z_{i_3}^{\m_1} \rp^p \rp  }
\nonumber \\
& & \times
\mE_{{\mathcal U}_1}
\frac{Z_{i_3}^{\m_1}}{ \mE_{{\mathcal U}_1} Z_{i_3}^{\m_1} }
\frac{ \lp (C_{i_3}^{(i_1)})^{s} A_{i_3}^{(i_1,i_2)}\y_j^{(i_2)} \rp}{Z_{i_3} C_{i_3}^{(i_1)} }
\nonumber \\
& & \times
\mE_{{\mathcal U}_1}
\frac{Z_{i_3}^{\m_1}}{ \mE_{{\mathcal U}_1} Z_{i_3}^{\m_1} }   \sum_{p_1=1}^{l}  \frac{(C_{i_3}^{(p_1)})^{s} }{Z_{i_3}}\sum_{p_2=1}^{l}
\frac{ \beta_{p_1}A_{i_3}^{(p_1,p_2)}\y_j^{(p_2)}\sqrt{1-t} } { C_{i_3}^{(p_1)} }
\Bigg . \Bigg ),
\end{eqnarray}
and then
\begin{eqnarray}\label{eq:genDanal24ab1}
\sum_{i_1=1}^{l}\sum_{i_2=1}^{l} \beta_{i_1}\|\y^{(i_2)}\|_2\frac{T_{2,2}^e}{\sqrt{1-t}}
     &  = &     s\beta^2\p_1\m_1 (p-1) \mE_{G,{\mathcal U}_2} \langle \|\x^{(i_1)}\|_2\|\x^{(p_1)}\|_2(\y^{(p_2)})^T\y^{(i_2)} \rangle_{\gamma_{22}^{(1)}}.
\end{eqnarray}
Combining  (\ref{eq:genDanal19a}), (\ref{eq:genDanal19b1}), (\ref{eq:genDanal24}), and (\ref{eq:genDanal24ab1}) we obtain
\begin{eqnarray}\label{eq:genDanal25}
 \sum_{i_1=1}^{l}  \sum_{i_2=1}^{l} \sum_{j=1}^{m}  \beta_{i_1}\frac{T_{2,1,j}}{\sqrt{1-t}}& = & \p_1\beta^2
 \lp \mE_{G,{\mathcal U}_2}\langle \|\x^{(i_1)}\|_2^2\|\y^{(i_2)}\|_2^2\rangle_{\gamma_{01}^{(1)}} +  (s-1)\mE_{G,{\mathcal U}_2}\langle \|\x^{(i_1)}\|_2^2(\y^{(p_2)})^T\y^{(i_2)}\rangle_{\gamma_{02}^{(1)}} \rp\nonumber \\
& & - \p_1s\beta^2(1-\m_1)\mE_{G,{\mathcal U}_2}\langle \|\x^{(i_1)}\|_2\|\x^{(p_1)}\|_2(\y^{(p_2)})^T\y^{(i_2)} \rangle_{\gamma_{1}^{(1)}}\nonumber \\
 &   &
  -s\beta^2\p_1\m_1\mE_{G,{\mathcal U}_2} \langle \|\x^{(i_1)}\|_2\|\x^{(p_1)}\|_2(\y^{(p_2)})^T\y^{(i_2)} \rangle_{\gamma_{21}^{(1)}}
  \nonumber
  \\
   &   &
   +  s\beta^2\p_1\m_1 (p-1) \mE_{G,{\mathcal U}_2} \langle \|\x^{(i_1)}\|_2\|\x^{(p_1)}\|_2(\y^{(p_2)})^T\y^{(i_2)} \rangle_{\gamma_{22}^{(1)}}.
\end{eqnarray}

\underline{\textbf{\emph{Determining}} $T_{2,2}$}
\label{sec:hand1T22}

As earlier, we apply the Gaussian integration by parts to find
\begin{eqnarray}\label{eq:liftgenEanal20}
T_{2,2} \hspace{-.05in} & = &  \mE_{G,{\mathcal U}_2} \lp
\sum_{i_3=1}^{l} \frac{
\lp \mE_{{\mathcal U}_1} Z_{i_3}^{\m_1} \rp^{p-1}    }{ \lp
\sum_{i_3=1}^{l}
\lp \mE_{{\mathcal U}_1} Z_{i_3}^{\m_1} \rp^p \rp   }
\mE_{{\mathcal U}_1}\frac{(C_{i_3}^{(i_1)})^{s-1} A_{i_3}^{(i_1,i_2)} \u^{(i_1,3,2)}}{Z_{i_3}^{1-\m_1}} \rp \nonumber \\
& = &  \mE_{G,{\mathcal U}_2,{\mathcal U}_1} \lp
\sum_{i_3=1}^{l} \frac{
\lp \mE_{{\mathcal U}_1} Z_{i_3}^{\m_1} \rp^{p-1}    }{ \lp
\sum_{i_3=1}^{l}
\lp \mE_{{\mathcal U}_1} Z_{i_3}^{\m_1} \rp^p \rp   }
 \frac{(C_{i_3}^{(i_1)})^{s-1} A_{i_3}^{(i_1,i_2)} \u^{(i_1,3,2)}}{Z_{i_3}^{1-\m_1}}\rp\nonumber \\
& = & \mE_{G,{\mathcal U}_1} \lp
\sum_{i_3=1}^{l}
\mE_{{\mathcal U}_2} \lp \sum_{p_1=1}^{l}\mE_{{\mathcal U}_2}(\u^{(i_1,3,2)}\u^{(p_1,3,2)}) \frac{d}{d\u^{(p_1,3,2)}}\lp\frac{(C_{i_3}^{(i_1)})^{s-1} A_{i_3}^{(i_1,i_2)  } \lp \mE_{{\mathcal U}_1} Z_{i_3}^{\m_1} \rp^{p-1} }  {Z_{i_3}^{1-\m_1}   \lp
\sum_{i_3=1}^{l}
\lp \mE_{{\mathcal U}_1} Z_{i_3}^{\m_1} \rp^p \rp  }
   \rp\rp\rp \nonumber \\
& = & \mE_{G,{\mathcal U}_2,{\mathcal U}_1} \lp
\sum_{i_3=1}^{l} \frac{\q_1
\lp \mE_{{\mathcal U}_1} Z_{i_3}^{\m_1} \rp^{p-1}    }{ \lp
\sum_{i_3=1}^{l}
\lp \mE_{{\mathcal U}_1} Z_{i_3}^{\m_1} \rp^p \rp   }
  \sum_{p_1=1}^{l} \frac{\beta^2(\x^{(i_1)})^T\x^{(p_1)}}{\beta_{i_1}\beta_{p_1}}
  \frac{d}{d\u^{(p_1,3,2)}}\lp\frac{(C_{i_3}^{(i_1)})^{s-1} A_{i_3}^{(i_1,i_2)}}{Z_{i_3}^{1-\m_1}}\rp\rp \nonumber \\
& & + \mE_{G,{\mathcal U}_2,{\mathcal U}_1} \Bigg ( \Bigg.
\sum_{i_3=1}^{l}
\frac{ \q_1\lp \mE_{{\mathcal U}_1} Z_{i_3}^{\m_1} \rp^{p-1}   (C_{i_3}^{(i_1)})^{s-1} A_{i_3}^{(i_1,i_2)}}{Z_{i_3}^{1-\m_1}}
\nonumber \\
& & \times
  \sum_{p_1=1}^{l} \frac{\beta^2(\x^{(i_1)})^T\x^{(p_1)}}{\beta_{i_1}\beta_{p_1}}
     \frac{d}{d\u^{(p_1,3,2)}}\lp\frac{1}{ \lp
\sum_{i_3=1}^{l}
\lp \mE_{{\mathcal U}_1} Z_{i_3}^{\m_1} \rp^p \rp  }  \rp  \Bigg ) \Bigg. \nonumber \\
& & + \mE_{G,{\mathcal U}_2,{\mathcal U}_1} \lp
\sum_{i_3=1}^{l}
\frac{ \q_1 (C_{i_3}^{(i_1)})^{s-1} A_{i_3}^{(i_1,i_2)}}{\lp
\sum_{i_3=1}^{l}
\lp \mE_{{\mathcal U}_1} Z_{i_3}^{\m_1} \rp^p \rp  Z_{i_3}^{1-\m_1}}
  \sum_{p_1=1}^{l} \q_1\frac{\beta^2(\x^{(i_1)})^T\x^{(p_1)}}{\beta_{i_1}\beta_{p_1}}\frac{d}{d\u^{(p_1,3,2)}}\lp  \lp \mE_{{\mathcal U}_1} Z_{i_3}^{\m_1} \rp^{p-1}  \rp\rp, \nonumber \\
 \end{eqnarray}
where we recall on
\begin{eqnarray}\label{eq:genEanal19c}
\mE_{{\mathcal U}_2}(\u^{(i_1,3,2)}\u^{(p_1,3,2)}) & = & \q_1\frac{(\x^{(i_1)})^T\x^{(p_1)}}{\|\x^{(i_1)}\|_2\|\x^{(p_1)}\|_2}
=\q_1\frac{\beta^2(\x^{(i_1)})^T\x^{(p_1)}}{\beta_{i_1}\beta_{p_1}}, \nonumber \\
\mE_{{\mathcal U}_2}(\u^{(i_1,3,1)}\u^{(p_1,3,1)}) & = & (\q_0-\q_1)\frac{(\x^{(i_1)})^T\x^{(p_1)}}{\|\x^{(i_1)}\|_2\|\x^{(p_1)}\|_2}=(\q_0-\q_1)\frac{\beta^2(\x^{(i_1)})^T\x^{(p_1)}}{\beta_{i_1}\beta_{p_1}}.
\end{eqnarray}
One can then rewrite (\ref{eq:liftgenEanal20}) as
\begin{eqnarray}\label{eq:genEanal19a}
T_{2,2}   =   T_{2,2}^{c} +  T_{2,2}^{d} +  T_{2,2}^{e},
\end{eqnarray}
where
\begin{eqnarray}\label{eq:genEanal19b}
T_{2,2}^c &  = &
\mE_{G,{\mathcal U}_2,{\mathcal U}_1} \lp
\sum_{i_3=1}^{l} \frac{\q_1
\lp \mE_{{\mathcal U}_1} Z_{i_3}^{\m_1} \rp^{p-1}    }{ \lp
\sum_{i_3=1}^{l}
\lp \mE_{{\mathcal U}_1} Z_{i_3}^{\m_1} \rp^p \rp   }
  \sum_{p_1=1}^{l} \frac{\beta^2(\x^{(i_1)})^T\x^{(p_1)}}{\beta_{i_1}\beta_{p_1}}
  \frac{d}{d\u^{(p_1,3,2)}}\lp\frac{(C_{i_3}^{(i_1)})^{s-1} A_{i_3}^{(i_1,i_2)}}{Z_{i_3}^{1-\m_1}}\rp\rp \nonumber \\
T_{2,2}^d &  = & \mE_{G,{\mathcal U}_2,{\mathcal U}_1} \Bigg ( \Bigg.
\sum_{i_3=1}^{l}
\frac{ \q_1\lp \mE_{{\mathcal U}_1} Z_{i_3}^{\m_1} \rp^{p-1}   (C_{i_3}^{(i_1)})^{s-1} A_{i_3}^{(i_1,i_2)}}{Z_{i_3}^{1-\m_1}}
\nonumber \\
& & \times
  \sum_{p_1=1}^{l} \frac{\beta^2(\x^{(i_1)})^T\x^{(p_1)}}{\beta_{i_1}\beta_{p_1}}
     \frac{d}{d\u^{(p_1,3,2)}}\lp\frac{1}{ \lp
\sum_{i_3=1}^{l}
\lp \mE_{{\mathcal U}_1} Z_{i_3}^{\m_1} \rp^p \rp  }  \rp  \Bigg ) \Bigg. \nonumber \\
T_{2,2}^e &  = &
\mE_{G,{\mathcal U}_2,{\mathcal U}_1} \lp
\sum_{i_3=1}^{l}
\frac{ \q_1 (C_{i_3}^{(i_1)})^{s-1} A_{i_3}^{(i_1,i_2)}}{\lp
\sum_{i_3=1}^{l}
\lp \mE_{{\mathcal U}_1} Z_{i_3}^{\m_1} \rp^p \rp  Z_{i_3}^{1-\m_1}}
  \sum_{p_1=1}^{l} \frac{\beta^2(\x^{(i_1)})^T\x^{(p_1)}}{\beta_{i_1}\beta_{p_1}} \frac{d}{d\u^{(p_1,3,2)}}\lp  \lp \mE_{{\mathcal U}_1} Z_{i_3}^{\m_1} \rp^{p-1}  \rp\rp. \nonumber \\
\end{eqnarray}
Since scaling $T_{2,2}^c$ by $\q_1$ produces the term structurally identical to the one considered in the second part of Section \ref{sec:hand1T11} scaled by $(\q_0-\q_1)$, one can immediately utilize (\ref{eq:liftgenBanal20b}) and write
 \begin{eqnarray}\label{eq:genEanal19c1}
\sum_{i_1=1}^{l}\sum_{i_2=1}^{l} \beta_{i_1}\|\y^{(i_2)}\|_2 \frac{T_{2,2}^c}{\sqrt{1-t}} & = & \q_1 \beta^2 \Bigg( \Bigg. \mE_{G,{\mathcal U}_2}\langle \|\x^{(i_1)}\|_2^2\|\y^{(i_2)}\|_2^2\rangle_{\gamma_{01}^{(1)}} \nonumber \\
& & +   (s-1)\mE_{G,{\mathcal U}_2}\langle \|\x^{(i_1)}\|_2^2 \|\y^{(i_2)}\|_2\|\y^{(p_2)}\|_2\rangle_{\gamma_{02}^{(1)}}\Bigg.\Bigg)   \nonumber \\
& & - \q_1s\beta^2(1-\m_1)\mE_{G,{\mathcal U}_2}\langle (\x^{(p_1)})^T\x^{(i_1)}\|\y^{(i_2)}\|_2\|\y^{(p_2)}\|_2 \rangle_{\gamma_{1}^{(1)}}.
\end{eqnarray}

To determining  $T_{2,2}^d$  we first observe
\begin{eqnarray}\label{eq:genEanal20}
\frac{d}{d\u^{(p_1,3,2)}}\lp\frac{1}{ \lp
\sum_{i_3=1}^{l}
\lp \mE_{{\mathcal U}_1} Z_{i_3}^{\m_1} \rp^p \rp  } \rp   & = &
 -p\sum_{p_3=1}^{l} \frac{
\lp \mE_{{\mathcal U}_1} Z_{p_3}^{\m_1} \rp^{p-1}}{\lp
\sum_{i_3=1}^{l}
\lp \mE_{{\mathcal U}_1} Z_{i_3}^{\m_1} \rp^p  \rp^2}
\mE_{{\mathcal U}_1}\frac{dZ_{p_3}^{\m_1}}{d\u^{(p_1,3)}}.
\end{eqnarray}
From \cite{Stojnicnflgscompyx23}'s (52) we also have
\begin{eqnarray}\label{eq:genEanal21}
\frac{dZ_{p_3}^{\m_1}}{d\u^{(p_1,3,2)}} & = & \frac{\m_1}{Z_{p_3}^{1-\m_1}} s  (C_{p_3}^{(p_1)})^{s-1}\sum_{p_2=1}^{l}
\beta_{p_1}A_{p_3}^{(p_1,p_2)}\|\y^{(p_2)}\|_2\sqrt{1-t}.
\end{eqnarray}
A combination of (\ref{eq:genEanal20}) and (\ref{eq:genEanal21}) then gives
\begin{eqnarray}\label{eq:genEanal22}
\frac{d}{d\u^{(p_1,3,2)}}\lp\frac{1}{ \lp
\sum_{i_3=1}^{l}
\lp \mE_{{\mathcal U}_1} Z_{i_3}^{\m_1} \rp^p \rp  } \rp   & = &
 -p\sum_{p_3=1}^{l} \frac{
\lp \mE_{{\mathcal U}_1} Z_{p_3}^{\m_1} \rp^{p-1}}{\lp
\sum_{i_3=1}^{l}
\lp \mE_{{\mathcal U}_1} Z_{i_3}^{\m_1} \rp^p  \rp^2}
 \nonumber \\
 & & \times \frac{\m_1}{Z_{p_3}^{1-\m_1}} s  (C_{p_3}^{(p_1)})^{s-1}\sum_{p_2=1}^{l}
\beta_{p_1}A_{p_3}^{(p_1,p_2)}\|\y^{(p_2)}\|_2\sqrt{1-t}.
\end{eqnarray}
Plugging  (\ref{eq:genEanal22}) back into the expression for $T_{2,2}^d$ given in  (\ref{eq:genEanal19b}) we also have
\begin{eqnarray}\label{eq:genEanal23}
 T_{2,2}^d &  = & -\q_1 p\mE_{G,{\mathcal U}_2,{\mathcal U}_1} \Bigg( \Bigg.\sum_{i_3=1}^{l}
\frac{ \lp \mE_{{\mathcal U}_1} Z_{i_3}^{\m_1} \rp^{p-1}   (C_{i_3}^{(i_1)})^{s-1} A_{i_3}^{(i_1,i_2)}}{Z_{i_3}^{1-\m_1}}
  \sum_{p_1=1}^{l}  \frac{\beta^2(\x^{(i_1)})^T\x^{(p_1)}}{\beta_{i_1}\beta_{p_1}}
   \nonumber \\
 & & \times
 \sum_{p_3=1}^{l} \frac{
\lp \mE_{{\mathcal U}_1} Z_{p_3}^{\m_1} \rp^{p-1}}{\lp
\sum_{i_3=1}^{l}
\lp \mE_{{\mathcal U}_1} Z_{i_3}^{\m_1} \rp^p  \rp^2}
 \frac{\m_1}{Z_{p_3}^{1-\m_1}} s  (C_{p_3}^{(p_1)})^{s-1}\sum_{p_2=1}^{l}
\beta_{p_1}A_{p_3}^{(p_1,p_2)}\|\y^{(p_2)}\|_2\sqrt{1-t}  \Bigg. \Bigg) \nonumber \\
&  = & -s\sqrt{1-t}\beta^2\q_1\m_1 p \mE_{G,{\mathcal U}_2,{\mathcal U}_1} \Bigg( \Bigg.
 \sum_{i_3=1}^{l} \frac{
\lp \mE_{{\mathcal U}_1} Z_{i_3}^{\m_1} \rp^{p}}{\lp
\sum_{i_3=1}^{l}
\lp \mE_{{\mathcal U}_1} Z_{i_3}^{\m_1} \rp^p  \rp}
\frac{Z_{i_3}^{\m_1}}{\mE_{{\mathcal U}_1} Z_{i_3}^{\m_1}}
 \frac{(C_{i_3}^{(i_1)})^{s-1} A_{i_3}^{(i_1,i_2)}}{Z_{i_3}} \nonumber \\
 & & \times
 \sum_{p_3=1}^{l} \frac{
\lp \mE_{{\mathcal U}_1} Z_{p_3}^{\m_1} \rp^{p}}{\lp
\sum_{p_3=1}^{l}
\lp \mE_{{\mathcal U}_1} Z_{p_3}^{\m_1} \rp^p  \rp}
\frac{Z_{p_3}^{\m_1}}{\mE_{{\mathcal U}_1} Z_{p_3}^{\m_1}}
  \sum_{p_1=1}^{l}
 \frac{(C_{p_3}^{(p_1)})^{s}}{Z_{p_3}}  \sum_{p_2=1}^{l}
\frac{A_{p_3}^{(p_1,p_2)}}{(C_{p_3}^{(p_1)})}\|\y^{(p_2)}\|_2\frac{(\x^{(i_1)})^T\x^{(p_1)}}{\beta_{i_1}} \Bigg. \Bigg),
\end{eqnarray}
and consequently
\begin{eqnarray}\label{eq:genEanal24}
\sum_{i_1=1}^{l}\sum_{i_2=1}^{l} \beta_{i_1}\|\y^{(i_2)}\|_2\frac{T_{2,2}^d}{\sqrt{1-t}}
& = & -s\beta^2\q_1\m_1 p \mE_{G,{\mathcal U}_2} \langle \|\y^{(i_2)}\|_2\|\y^{(p_2)}\|_2(\x^{(i_1)})^T\x^{(p_1)}\rangle_{\gamma_{21}^{(1)}}.
\end{eqnarray}

To determine $T_{2,2}^e$ we first note
\begin{eqnarray}\label{eq:genDanal24aa0bb0}
\frac{d}{d\u^{(p_1,3,2)}}\lp  \lp \mE_{{\mathcal U}_1} Z_{i_3}^{\m_1} \rp^{p-1}  \rp=
 (p-1) \lp  \mE_{{\mathcal U}_1} Z_{i_3}^{\m_1} \rp^{p-2}
\mE_{{\mathcal U}_1}\frac{d Z_{i_3}^{\m_1}}{d\u^{(p_1,3,2)}}.
\end{eqnarray}
Combining (\ref{eq:genDanal24aa0}) and  (\ref{eq:genEanal21})  one finds
\begin{equation}\label{eq:genDanal24aa1bb1}
\frac{d}{d\u^{(p_1,3,2)}}\lp  \lp \mE_{{\mathcal U}_1} Z_{i_3}^{\m_1} \rp^{p-1}  \rp
=
(p-1)  \lp  \mE_{{\mathcal U}_1} Z_{i_3}^{\m_1} \rp^{p-2}
\mE_{{\mathcal U}_1}
\frac{\m_1}{Z_{i_3}^{1-\m_1}} s  (C_{i_3}^{(p_1)})^{s-1}\sum_{p_2=1}^{l}
\beta_{p_1}A_{i_3}^{(p_1,p_2)}\|\y^{(p_2)}\|_2\sqrt{1-t}.
\end{equation}
Plugging  (\ref{eq:genEanal22}) back into the expression for $T_{2,2}^e$ given in  (\ref{eq:genEanal19b}) we first obtain
\begin{eqnarray}\label{eq:genDanal24aa1bb2}
T_{2,2}^e &  = &
\mE_{G,{\mathcal U}_2,{\mathcal U}_1} \Bigg ( \Bigg.
\sum_{i_3=1}^{l}
\frac{ \q_1 (C_{i_3}^{(i_1)})^{s-1} A_{i_3}^{(i_1,i_2)}}{\lp
\sum_{i_3=1}^{l}
\lp \mE_{{\mathcal U}_1} Z_{i_3}^{\m_1} \rp^p \rp  Z_{i_3}^{1-\m_1}}
  \sum_{p_1=1}^{l} \frac{\beta^2(\x^{(i_1)})^T\x^{(p_1)}}{\beta_{i_1}\beta_{p_1}}
  \nonumber \\
  & & \times
  (p-1)  \lp  \mE_{{\mathcal U}_1} Z_{i_3}^{\m_1} \rp^{p-2}
\mE_{{\mathcal U}_1}
\frac{\m_1}{Z_{i_3}^{1-\m_1}} s  (C_{i_3}^{(p_1)})^{s-1}\sum_{p_2=1}^{l}
\beta_{p_1}A_{i_3}^{(p_1,p_2)}\|\y^{(p_2)}\|_2\sqrt{1-t}
    \Bigg ) \Bigg.   \nonumber \\
    &  = & s\beta^2\q_1\m_1 (p-1)
\mE_{G,{\mathcal U}_2} \Bigg ( \Bigg.
\sum_{i_3=1}^{l}
\frac{    \lp  \mE_{{\mathcal U}_1} Z_{i_3}^{\m_1} \rp^p   }{\lp
\sum_{i_3=1}^{l}
\lp \mE_{{\mathcal U}_1} Z_{i_3}^{\m_1} \rp^p \rp  }
\mE_{{\mathcal U}_1}
\frac{Z_{i_3}^{\m_1}} { \mE_{{\mathcal U}_1} Z_{i_3}^{\m_1}}
\frac{     (C_{i_3}^{(i_1)})^{s}  }  {   Z_{i_3}  }
\frac{     A_{i_3}^{(i_1,i_2)} }  {  C_{i_3}^{(i_1)}  }
  \nonumber \\
  & & \times
\mE_{{\mathcal U}_1}
\frac{Z_{i_3}^{\m_1}} { \mE_{{\mathcal U}_1} Z_{i_3}^{\m_1}}
  \sum_{p_1=1}^{l}
\frac{  (C_{i_3}^{(p_1)})^{s} }{Z_{i_3}}
  \sum_{p_2=1}^{l}
  \frac{ A_{i_3}^{(p_1,p_2)}\|\y^{(p_2)}\|_2
(\x^{(i_1)})^T\x^{(p_1)} \sqrt{1-t} }{   C_{i_3}^{(p_1)}
 \beta_{i_1} }
    \Bigg ) \Bigg. ,
\end{eqnarray}
and
\begin{eqnarray}\label{eq:genDanal24aa1bb3}
\sum_{i_1=1}^{l}\sum_{i_2=1}^{l} \beta_{i_1}\|\y^{(i_2)}\|_2\frac{T_{2,2}^e}{\sqrt{1-t}}
     &  = &     s\beta^2\q_1\m_1 (p-1) \mE_{G,{\mathcal U}_2} \langle \|\x^{(i_1)}\|_2\|\x^{(p_1)}\|_2(\y^{(p_2)})^T\y^{(i_2)} \rangle_{\gamma_{22}^{(1)}}.
\end{eqnarray}
A combination of (\ref{eq:genEanal19a}), (\ref{eq:genEanal19c1}),
(\ref{eq:genEanal24}), and (\ref{eq:genDanal24aa1bb3})  gives
\begin{eqnarray}\label{eq:genEanal25}
\sum_{i_1=1}^{l}\sum_{i_2=1}^{l} \beta_{i_1}\|\y^{(i_2)}\|_2\frac{T_{2,2}}{\sqrt{1-t}} &  = &
\q_1\beta^2 \Bigg( \Bigg. \mE_{G,{\mathcal U}_2}\langle \|\x^{(i_1)}\|_2^2\|\y^{(i_2)}\|_2^2\rangle_{\gamma_{01}^{(1)}} \nonumber \\
& & +  (s-1)\mE_{G,{\mathcal U}_2}\langle \|\x^{(i_1)}\|_2^2 \|\y^{(i_2)}\|_2\|\y^{(p_2)}\|_2\rangle_{\gamma_{02}^{(1)}}\Bigg.\Bigg) \nonumber \\
& & - \q_1s\beta^2(1-\m_1)\mE_{G,{\mathcal U}_2}\langle (\x^{(p_1)})^T\x^{(i_1)}\|\y^{(i_2)}\|_2\|\y^{(p_2)}\|_2 \rangle_{\gamma_{1}^{(1)}} \nonumber \\
&  & -s\beta^2\q_1\m_1 p \mE_{G,{\mathcal U}_2} \langle \|\y^{(i_2)}\|_2\|\y^{(p_2)}\|_2(\x^{(i_1)})^T\x^{(p_1)}\rangle_{\gamma_{21}^{(1)}}
\nonumber \\
& &
+ s\beta^2\q_1\m_1 (p-1) \mE_{G,{\mathcal U}_2} \langle \|\x^{(i_1)}\|_2\|\x^{(p_1)}\|_2(\y^{(p_2)})^T\y^{(i_2)} \rangle_{\gamma_{22}^{(1)}}.
\end{eqnarray}

\underline{\textbf{\emph{Determining}} $T_{2,3}$}
\label{sec:hand1T23}

After Gaussian integration by parts we find
\begin{eqnarray}\label{eq:genFanal21}
T_{2,3} & = &  \mE_{G,{\mathcal U}_2}\lp
\sum_{i_3=1}^{l} \frac{
\lp \mE_{{\mathcal U}_1} Z_{i_3}^{\m_1} \rp^{p-1}    }{ \lp
\sum_{i_3=1}^{l}
\lp \mE_{{\mathcal U}_1} Z_{i_3}^{\m_1} \rp^p \rp   }
\mE_{{\mathcal U}_1}\frac{(C_{i_3}^{(i_1)})^{s-1} A_{i_3}^{(i_1,i_2)} u^{(4,2)}}{Z_{i_3}^{1-\m_1}} \rp \nonumber \\
& = & \mE_{G,{\mathcal U}_1} \lp  \mE_{{\mathcal U}_2} \lp\mE_{{\mathcal U}_2} (u^{(4,2)}u^{(4,2)})\lp\frac{d}{du^{(4,2)}} \lp\frac{(C_{i_3}^{(i_1)})^{s-1} A_{i_3}^{(i_1,i_2)}}{Z_{i_3}^{1-\m_1}}
\frac{
\lp \mE_{{\mathcal U}_1} Z_{i_3}^{\m_1} \rp^{p-1}    }{ \lp
\sum_{i_3=1}^{l}
\lp \mE_{{\mathcal U}_1} Z_{i_3}^{\m_1} \rp^p \rp   }
\rp \rp\rp\rp \nonumber \\
& = & \p_1\q_1\mE_{G,{\mathcal U}_2,{\mathcal U}_1} \lp
\sum_{i_3=1}^{l}
\frac{
\lp \mE_{{\mathcal U}_1} Z_{i_3}^{\m_1} \rp^{p-1}    }{ \lp
\sum_{i_3=1}^{l}
\lp \mE_{{\mathcal U}_1} Z_{i_3}^{\m_1} \rp^p \rp   }
\lp\frac{d}{du^{(4,2)}} \lp\frac{(C_{i_3}^{(i_1)})^{s-1} A_{i_3}^{(i_1,i_2)}}{Z_{i_3}^{1-\m_1}}\rp\rp\rp \nonumber \\
& & + \p_1\q_1\mE_{G,{\mathcal U}_2,{\mathcal U}_1} \lp \sum_{i_3=1}^{l}
 \frac{ \lp \mE_{{\mathcal U}_1} Z_{i_3}^{\m_1} \rp^{p-1}     (C_{i_3}^{(i_1)})^{s-1} A_{i_3}^{(i_1,i_2)}}{Z_{i_3}^{1-\m_1}}\lp\frac{d}{du^{(4,2)}} \lp\frac{1}{ \lp
\sum_{i_3=1}^{l}
\lp \mE_{{\mathcal U}_1} Z_{i_3}^{\m_1} \rp^p \rp } \rp\rp\rp
\nonumber \\
& & + \p_1\q_1\mE_{G,{\mathcal U}_2,{\mathcal U}_1} \lp \sum_{i_3=1}^{l}
 \frac{     (C_{i_3}^{(i_1)})^{s-1} A_{i_3}^{(i_1,i_2)}}{  \lp
\sum_{i_3=1}^{l}
\lp \mE_{{\mathcal U}_1} Z_{i_3}^{\m_1} \rp^p \rp    Z_{i_3}^{1-\m_1}}\lp\frac{d}{du^{(4,2)}} \lp   \lp \mE_{{\mathcal U}_1} Z_{i_3}^{\m_1} \rp^{p-1}   \rp\rp\rp.
\end{eqnarray}
Following the methodology presented above, we find it useful to rewrite (\ref{eq:genFanal21}) as
\begin{eqnarray}\label{eq:genFanal22}
T_{2,3} & = & T_{2,3}^c + T_{2,3}^d +T_{2,3}^e,
\end{eqnarray}
where
\begin{eqnarray}\label{eq:genFanal23}
T_{2,3}^c & = &  \p_1\q_1\mE_{G,{\mathcal U}_2,{\mathcal U}_1} \lp
\sum_{i_3=1}^{l}
\frac{
\lp \mE_{{\mathcal U}_1} Z_{i_3}^{\m_1} \rp^{p-1}    }{ \lp
\sum_{i_3=1}^{l}
\lp \mE_{{\mathcal U}_1} Z_{i_3}^{\m_1} \rp^p \rp   }
\lp\frac{d}{du^{(4,2)}} \lp\frac{(C_{i_3}^{(i_1)})^{s-1} A_{i_3}^{(i_1,i_2)}}{Z_{i_3}^{1-\m_1}}\rp\rp\rp \nonumber \\
 T_{2,3}^d & = & \p_1\q_1\mE_{G,{\mathcal U}_2,{\mathcal U}_1} \lp \sum_{i_3=1}^{l}
 \frac{ \lp \mE_{{\mathcal U}_1} Z_{i_3}^{\m_1} \rp^{p-1}     (C_{i_3}^{(i_1)})^{s-1} A_{i_3}^{(i_1,i_2)}}{Z_{i_3}^{1-\m_1}}   \lp \frac{d}{du^{(4,2)}} \lp\frac{1}{ \lp
\sum_{i_3=1}^{l}
\lp \mE_{{\mathcal U}_1} Z_{i_3}^{\m_1} \rp^p \rp } \rp\rp\rp
\nonumber \\
 T_{2,3}^e & = &
\p_1\q_1\mE_{G,{\mathcal U}_2,{\mathcal U}_1} \lp \sum_{i_3=1}^{l}
 \frac{     (C_{i_3}^{(i_1)})^{s-1} A_{i_3}^{(i_1,i_2)}}{  \lp
\sum_{i_3=1}^{l}
\lp \mE_{{\mathcal U}_1} Z_{i_3}^{\m_1} \rp^p \rp    Z_{i_3}^{1-\m_1}}\lp\frac{d}{du^{(4,2)}} \lp   \lp \mE_{{\mathcal U}_1} Z_{i_3}^{\m_1} \rp^{p-1}   \rp\rp\rp.
\end{eqnarray}
After recognizing  that $\frac{T_{2,3}^c}{\p_1\q_1}=\frac{T_{1,3}}{\p_0\q_0-\p_1\q_1}$, one, based on (\ref{eq:liftgenCanal21b}), obtains
 \begin{eqnarray}\label{eq:genFanal23b}
\sum_{i_1=1}^{l}\sum_{i_2=1}^{l} \beta_{i_1}\|\y^{(i_2)}\|_2 \frac{T_{2,3}^c}{\sqrt{t}} & = &
\p_1\q_1 \beta^2 \Bigg( \Bigg. \mE_{G,{\mathcal U}_2}\langle \|\x^{(i_1)}\|_2^2\|\y^{(i_2)}\|_2^2\rangle_{\gamma_{01}^{(1)}} \nonumber \\
& & \times +   (s-1)\mE_{G,{\mathcal U}_2}\langle \|\x^{(i_1)}\|_2^2 \|\y^{(i_2)}\|_2\|\y^{(p_2)}\|_2\rangle_{\gamma_{02}^{(1)}}\Bigg. \Bigg) \nonumber \\
& & - \p_1\q_1 s\beta^2(1-\m_1)\mE_{G,{\mathcal U}_2}\langle \|\x^{(i_1)}\|_2\|\x^{(p_`)}\|_2\|\y^{(i_2)}\|_2\|\y^{(p_2)}\|_2 \rangle_{\gamma_{1}^{(1)}}.
\end{eqnarray}

To find $T_{2,3}^d$, we first write
\begin{eqnarray}\label{eq:genFanal24}
\frac{d}{du^{(4,2)}} \lp\frac{1}{ \lp
\sum_{i_3=1}^{l}
\lp \mE_{{\mathcal U}_1} Z_{i_3}^{\m_1} \rp^p \rp } \rp
  & = &
 - p \sum_{p_3=1}^{l} \frac{\lp \mE_{{\mathcal U}_1} Z_{p_3}^{\m_1} \rp^{p-1} } { \lp
\sum_{i_3=1}^{l}
\lp \mE_{{\mathcal U}_1} Z_{i_3}^{\m_1} \rp^p \rp^2} \mE_{{\mathcal U}_1} \frac{dZ_{p_3}^{\m_1}}{d u^{(4,2)}}
\nonumber \\
& = &  - p \sum_{p_3=1}^{l} \frac{\lp \mE_{{\mathcal U}_1} Z_{p_3}^{\m_1} \rp^{p-1} } { \lp
\sum_{i_3=1}^{l}
\lp \mE_{{\mathcal U}_1} Z_{i_3}^{\m_1} \rp^p \rp^2}\mE_{{\mathcal U}_1} \frac{\m_1}{Z_{p_3}^{1-\m_1}}\frac{dZ_{p_3}}{d u^{(4,2)}}.
\end{eqnarray}
From \cite{Stojnicnflgscompyx23}'s (62) we also have
\begin{equation}\label{eq:genFanal25}
\frac{dZ_{p_3}}{du^{(4,2)}}=\frac{d\sum_{p_1=1}^{l}  (C_{p_3}^{(p_1)})^s}{du^{(4,2)}}
=s \sum_{p_1=1}^{l} (C_{p_3}^{(p_1)})^{s-1}\sum_{p_2=1}^{l}\frac{d(A_{p_3}^{(p_1,p_2)})}{du^{(4,2)}}
=
s \sum_{p_1=1}^{l} (C_{p_3}^{(p_1)})^{s-1}\sum_{p_2=1}^{l}
\beta_{p_1}A_{p_3}^{(p_1,p_2)}\|\y^{(p_2)}\|_2\sqrt{t}.
\end{equation}
A combination of (\ref{eq:genFanal24}) and (\ref{eq:genFanal25}) gives
\begin{eqnarray}\label{eq:genFanal26}
\frac{d}{du^{(4,2)}} \lp\frac{1}{ \lp
\sum_{i_3=1}^{l}
\lp \mE_{{\mathcal U}_1} Z_{i_3}^{\m_1} \rp^p \rp } \rp
 & = &  - p \sum_{p_3=1}^{l} \frac{\lp \mE_{{\mathcal U}_1} Z_{p_3}^{\m_1} \rp^{p-1} } { \lp
\sum_{i_3=1}^{l}
\lp \mE_{{\mathcal U}_1} Z_{i_3}^{\m_1} \rp^p \rp^2}\mE_{{\mathcal U}_1} \frac{\m_1}{Z_{p_3}^{1-\m_1}}  \nonumber \\
 & & \times s \sum_{p_1=1}^{l} (C_{p_3}^{(p_1)})^{s-1}\sum_{p_2=1}^{l}
\beta_{p_1}A_{p_3}^{(p_1,p_2)}\|\y^{(p_2)}\|_2\sqrt{t}.
\end{eqnarray}
Plugging (\ref{eq:genFanal26}) into the expression for $T_{2,3}^d$ from (\ref{eq:genFanal23}) gives
\begin{eqnarray}\label{eq:genFanal27}
T_{2,3}^d   & = & -\p_1\q_1  \mE_{G,{\mathcal U}_2,{\mathcal U}_1} \Bigg( \Bigg.
\sum_{i_3=1}^{l}
 \frac{ \lp \mE_{{\mathcal U}_1} Z_{i_3}^{\m_1} \rp^{p-1}     (C_{i_3}^{(i_1)})^{s-1} A_{i_3}^{(i_1,i_2)}}{Z_{i_3}^{1-\m_1}} \nonumber \\
& & \times p \sum_{p_3=1}^{l} \frac{\lp \mE_{{\mathcal U}_1} Z_{p_3}^{\m_1} \rp^{p-1} } { \lp
\sum_{i_3=1}^{l}
\lp \mE_{{\mathcal U}_1} Z_{i_3}^{\m_1} \rp^p \rp^2 }
\mE_{{\mathcal U}_1} \frac{\m_1}{Z_{p_3}^{1-\m_1}}   s \sum_{p_1=1}^{l} (C_{p_3}^{(p_1)})^{s-1}\sum_{p_2=1}^{l}
\beta_{p_1}A_{p_3}^{(p_1,p_2)}\|\y^{(p_2)}\|_2\sqrt{t}\Bigg. \Bigg) \nonumber \\
& = & -s\sqrt{t}\p_1\q_1\m_1 p \mE_{G,{\mathcal U}_2} \Bigg( \Bigg.
\sum_{p_3=1}^{l} \frac{\lp \mE_{{\mathcal U}_1} Z_{p_3}^{\m_1} \rp^{p} } { \lp
\sum_{i_3=1}^{l}
\lp \mE_{{\mathcal U}_1} Z_{i_3}^{\m_1} \rp^p \rp }
\mE_{{\mathcal U}_1} \frac{Z^{\m_1}}{\mE_{{\mathcal U}_1} Z^{\m_1}} \frac{(C^{(i_1)})^s}{Z}\frac{A^{(i_1,i_2)}}{(C^{(i_1)})}\nonumber \\
& & \times
\sum_{p_3=1}^{l} \frac{\lp \mE_{{\mathcal U}_1} Z_{p_3}^{\m_1} \rp^{p} } { \lp
\sum_{i_3=1}^{l}
\lp \mE_{{\mathcal U}_1} Z_{i_3}^{\m_1} \rp^p \rp }
  \mE_{{\mathcal U}_1} \frac{Z^{\m_1}}{\mE_{{\mathcal U}_1} Z^{\m_1}} \sum_{p_1=1}^{l} \frac{(C^{(p_1)})^{s}}{Z}\sum_{p_2=1}^{l}
\frac{A^{(p_1,p_2)}}{(C^{(p_1)})}\beta_{p_1}\|\y^{(p_2)}\|_2\Bigg. \Bigg).
\end{eqnarray}
We then note
\begin{eqnarray}\label{eq:genFanal28}
\sum_{i_1=1}^{l}\sum_{i_2=1}^{l} \beta_{i_1}\|\y^{(i_2)}\|_2\frac{T_{2,3}^d}{\sqrt{t}}
 & = & -s\beta^2\p_1\q_1\m_1 p \mE_{G,{\mathcal U}_2} \langle\|\x^{(i_2)}\|_2\|\x^{(p_2)}\|_2\|\y^{(i_2)}\|_2\|\y^{(p_2)} \|_2 \rangle_{\gamma_{21}^{(1)}}.
\end{eqnarray}

To determine $T_{2,3}^e$ we first write
\begin{eqnarray}\label{eq:genDanal24aa0bb0cc0}
\frac{d}{du^{(4,2)}}\lp  \lp \mE_{{\mathcal U}_1} Z_{i_3}^{\m_1} \rp^{p-1}  \rp=
 (p-1) \lp  \mE_{{\mathcal U}_1} Z_{i_3}^{\m_1} \rp^{p-2}
\mE_{{\mathcal U}_1}\frac{d Z_{i_3}^{\m_1}}{d u^{(4,2)}}.
\end{eqnarray}
Combining (\ref{eq:genDanal24aa0bb0cc0}) and  (\ref{eq:genFanal25})  one obtains
\begin{equation}\label{eq:genDanal24aa1cc1}
\frac{d}{du^{(4,2)}}\lp  \lp \mE_{{\mathcal U}_1} Z_{i_3}^{\m_1} \rp^{p-1}  \rp
=
(p-1)  \lp  \mE_{{\mathcal U}_1} Z_{i_3}^{\m_1} \rp^{p-2}
\mE_{{\mathcal U}_1}
\frac{\m_1}{Z_{i_3}^{1-\m_1}}   s \sum_{p_1=1}^{l} (C_{i_3}^{(p_1)})^{s-1}\sum_{p_2=1}^{l}
\beta_{p_1}A_{i_3}^{(p_1,p_2)}\|\y^{(p_2)}\|_2\sqrt{t}.
\end{equation}
Plugging  (\ref{eq:genDanal24aa1cc1}) back into the expression for $T_{2,3}^e$ given in (\ref{eq:genFanal23}) we further find
\begin{eqnarray}\label{eq:genDanal24aa1cc2}
T_{2,2}^e &  = &
\p_1\q_1 \mE_{G,{\mathcal U}_2,{\mathcal U}_1} \Bigg ( \Bigg.
\sum_{i_3=1}^{l}
 \frac{     (C_{i_3}^{(i_1)})^{s-1} A_{i_3}^{(i_1,i_2)}}{  \lp
\sum_{i_3=1}^{l}
\lp \mE_{{\mathcal U}_1} Z_{i_3}^{\m_1} \rp^p \rp    Z_{i_3}^{1-\m_1}}
\nonumber \\
  & & \times
(p-1)  \lp  \mE_{{\mathcal U}_1} Z_{i_3}^{\m_1} \rp^{p-2}
\mE_{{\mathcal U}_1}
\frac{\m_1}{Z_{i_3}^{1-\m_1}}   s \sum_{p_1=1}^{l} (C_{i_3}^{(p_1)})^{s-1}\sum_{p_2=1}^{l}
\beta_{p_1}A_{i_3}^{(p_1,p_2)}\|\y^{(p_2)}\|_2\sqrt{t}
    \Bigg ) \Bigg.   \nonumber \\
    &  = & s\beta^2\p_1\q_1\m_1 (p-1)
\mE_{G,{\mathcal U}_2} \Bigg ( \Bigg.
\sum_{i_3=1}^{l}
\frac{    \lp  \mE_{{\mathcal U}_1} Z_{i_3}^{\m_1} \rp^p   }{\lp
\sum_{i_3=1}^{l}
\lp \mE_{{\mathcal U}_1} Z_{i_3}^{\m_1} \rp^p \rp  }
\mE_{{\mathcal U}_1}
\frac{Z_{i_3}^{\m_1}} { \mE_{{\mathcal U}_1} Z_{i_3}^{\m_1}}
\frac{     (C_{i_3}^{(i_1)})^{s}  }  {   Z_{i_3}  }
\frac{     A_{i_3}^{(i_1,i_2)} }  {  C_{i_3}^{(i_1)}  }
  \nonumber \\
  & & \times
\mE_{{\mathcal U}_1}
\frac{Z_{i_3}^{\m_1}} { \mE_{{\mathcal U}_1} Z_{i_3}^{\m_1}}
  \sum_{p_1=1}^{l}
\frac{  (C_{i_3}^{(p_1)})^{s} }{Z_{i_3}}
  \sum_{p_2=1}^{l}
  \frac{ A_{i_3}^{(p_1,p_2)}\|\y^{(p_2)}\|_2
\|\x^{(i_1)}\|_2   \|\x^{(p_1) }\|_2 \sqrt{t} }{   C_{i_3}^{(p_1)}
 \beta_{i_1} }
    \Bigg ) \Bigg. ,
\end{eqnarray}
and then
\begin{eqnarray}\label{eq:genDanal24aa1cc3}
\sum_{i_1=1}^{l}\sum_{i_2=1}^{l} \beta_{i_1}\|\y^{(i_2)}\|_2\frac{T_{2,3}^e}{\sqrt{t}}
     &  = &     s\beta^2 \p_1\q_1\m_1 (p-1) \mE_{G,{\mathcal U}_2} \langle \|\x^{(i_1)}\|_2\|\x^{(p_1)}\|_2  \|\y^{(p_2)}\|_2    \|\y^{(i_2)}\|_2 \rangle_{\gamma_{22}^{(1)}}.
\end{eqnarray}
A combination of (\ref{eq:genFanal22}), (\ref{eq:genFanal23b}), (\ref{eq:genFanal28}), and (\ref{eq:genDanal24aa1cc3}) gives
\begin{eqnarray}\label{eq:genFanal29}
\sum_{i_1=1}^{l}\sum_{i_2=1}^{l} \beta_{i_1}\|\y^{(i_2)}\|_2\frac{T_{2,3}}{\sqrt{t}}
& = &
\p_1\q_1\beta^2 \Bigg( \Bigg. \mE_{G,{\mathcal U}_2}\langle \|\x^{(i_1)}\|_2^2\|\y^{(i_2)}\|_2^2\rangle_{\gamma_{01}^{(1)}} \nonumber \\
& & +   (s-1)\mE_{G,{\mathcal U}_2}\langle \|\x^{(i_1)}\|_2^2 \|\y^{(i_2)}\|_2\|\y^{(p_2)}\|_2\rangle_{\gamma_{02}^{(1)}}\Bigg.\Bigg) \nonumber \\
& & - \p_1\q_1 s\beta^2(1-\m_1)\mE_{G,{\mathcal U}_2}\langle \|\x^{(i_1)}\|_2\|\x^{(p_`)}\|_2\|\y^{(i_2)}\|_2\|\y^{(p_2)}\|_2 \rangle_{\gamma_{1}^{(1)}} \nonumber \\
&  & -s\beta^2\p_1\q_1\m_1\mE_{G,{\mathcal U}_2} \langle\|\x^{(i_2)}\|_2\|\x^{(p_2)}\|_2\|\y^{(i_2)}\|_2\|\y^{(p_2)}\|_2   \rangle_{\gamma_{21}^{(1)}}
\nonumber \\
& & +
 s\beta^2 \p_1 \q_1\m_1 (p-1) \mE_{G,{\mathcal U}_2} \langle \|\x^{(i_1)}\|_2\|\x^{(p_1)}\|_2  \|\y^{(p_2)}\|_2    \|\y^{(i_2)}\|_2 \rangle_{\gamma_{22}^{(1)}}.
\end{eqnarray}

\subsubsection{$T_G$--group -- first level}
\label{sec:handlTG}

We first recall on
\begin{eqnarray}\label{eq:genGanal1}
T_{G,j} & = &  \mE_{G,{\mathcal U}_2}\lp
\sum_{i_3=1}^{l} \frac{
\lp \mE_{{\mathcal U}_1} Z_{i_3}^{\m_1} \rp^{p-1}    }{ \lp
\sum_{i_3=1}^{l}
\lp \mE_{{\mathcal U}_1} Z_{i_3}^{\m_1} \rp^p \rp   }
\mE_{{\mathcal U}_1}\frac{(C_{i_3}^{(i_1)})^{s-1} A_{i_3}^{(i_1,i_2)} \y_j^{(i_2)}\bar{\u}_j^{(i_1,1)}}{Z_{i_3}^{1-\m_1}} \rp \nonumber \\
 & = &  \mE_{{\mathcal U}_2,{\mathcal U}_1}\lp
\sum_{i_3=1}^{l}
\mE_{G}
\frac{
\lp \mE_{{\mathcal U}_1} Z_{i_3}^{\m_1} \rp^{p-1}    }{ \lp
\sum_{i_3=1}^{l}
\lp \mE_{{\mathcal U}_1} Z_{i_3}^{\m_1} \rp^p \rp   }
\frac{(C_{i_3}^{(i_1)})^{s-1} A_{i_3}^{(i_1,i_2)} \y_j^{(i_2)}\bar{\u}_j^{(i_1,1)}}{Z_{i_3}^{1-\m_1}} \rp.
 \end{eqnarray}
Gaussian integration by parts gives
\begin{eqnarray}\label{eq:genGanal2}
T_{G,j}  & = &
 \mE_{{\mathcal U}_2,{\mathcal U}_1} \Bigg( \Bigg. \sum_{i_3=1}^{l}\mE_{G} \lp
 \frac{
\lp \mE_{{\mathcal U}_1} Z_{i_3}^{\m_1} \rp^{p-1}    }{ \lp
\sum_{i_3=1}^{l}
\lp \mE_{{\mathcal U}_1} Z_{i_3}^{\m_1} \rp^p \rp   }
 \sum_{p_1=1}^{l} \mE (\u_j^{(i_1,1)}\u_j^{(p_1,1)})\frac{d}{d\u_j^{(p_1,1)}}\lp \frac{(C_{i_3}^{(i_1)})^{s-1} A_{i_3}^{(i_1,i_2)}\y_j^{(i_2)}}{Z_{i_3}^{1-\m_1}}\rp \rp \nonumber \\
 & & +  \sum_{p_1=1}^{l} \mE_{G} \Bigg.\Bigg(   \frac{  \lp \mE_{{\mathcal U}_1} Z_{i_3}^{\m_1} \rp^{p-1}   (C_{i_3}^{(i_1)})^{s-1} A_{i_3}^{(i_1,i_2)}\y_j^{(i_2)}}{Z_{i_3}^{1-\m_1}}
 \nonumber  \\
 & &
 \times
  \sum_{p_1=1}^{l}
 \mE (\u_j^{(i_1,1)}\u_j^{(p_1,1)})\frac{d}{d\u_j^{(p_1,1)}}  \Bigg.\Bigg(   \frac{1}{ \lp
\sum_{i_3=1}^{l}
\lp \mE_{{\mathcal U}_1} Z_{i_3}^{\m_1} \rp^p \rp  }\Bigg.\Bigg)  \Bigg.\Bigg)
\nonumber \\
 & & +  \sum_{p_1=1}^{l} \mE_{G} \Bigg.\Bigg(   \frac{   (C_{i_3}^{(i_1)})^{s-1} A_{i_3}^{(i_1,i_2)}\y_j^{(i_2)}}{ \lp
\sum_{i_3=1}^{l}
\lp \mE_{{\mathcal U}_1} Z_{i_3}^{\m_1} \rp^p \rp   Z_{i_3}^{1-\m_1}}
  \sum_{p_1=1}^{l}
 \mE (\u_j^{(i_1,1)}\u_j^{(p_1,1)})\frac{d}{d\u_j^{(p_1,1)}}  \Bigg.\Bigg(    \lp \mE_{{\mathcal U}_1} Z_{i_3}^{\m_1} \rp^{p-1}     \Bigg.\Bigg)  \Bigg.\Bigg) \Bigg.\Bigg).
 \end{eqnarray}
We first set
\begin{eqnarray}\label{eq:genGanal3}
\Theta_{G,1} & = & \sum_{p_1=1}^{l} \mE (\u_j^{(i_1,1)}\u_j^{(p_1,1)})\frac{d}{d\u_j^{(p_1,1)}}\lp \frac{(C_{i_3}^{(i_1)})^{s-1} A_{i_3}^{(i_1,i_2)}\y_j^{(i_2)}}{Z_{i_3}^{1-\m_1}}\rp \nonumber \\
\Theta_{G,2} & = &   \lp   \frac{(C_{i_3}^{(i_1)})^{s-1} A_{i_3}^{(i_1,i_2)}\y_j^{(i_2)}}{Z_{i_3}^{1-\m_1}} \sum_{p_1=1}^{l} \mE (\u_j^{(i_1,1)}\u_j^{(p_1,1)})\frac{d}{d\u_j^{(p_1,1)}}\lp   \frac{1}{  \lp
\sum_{i_3=1}^{l}
\lp \mE_{{\mathcal U}_1} Z_{i_3}^{\m_1} \rp^p \rp   } \rp \rp \nonumber \\
\Theta_{G,3} & = &
 \Bigg.\Bigg(   \frac{   (C_{i_3}^{(i_1)})^{s-1} A_{i_3}^{(i_1,i_2)}\y_j^{(i_2)}}{ \lp
\sum_{i_3=1}^{l}
\lp \mE_{{\mathcal U}_1} Z_{i_3}^{\m_1} \rp^p \rp   Z_{i_3}^{1-\m_1}}
  \sum_{p_1=1}^{l}
 \mE (\u_j^{(i_1,1)}\u_j^{(p_1,1)})\frac{d}{d\u_j^{(p_1,1)}}  \Bigg.\Bigg(    \lp \mE_{{\mathcal U}_1} Z_{i_3}^{\m_1} \rp^{p-1}     \Bigg.\Bigg)  \Bigg.\Bigg)
\nonumber \\
T_{G,j}^c & = &  \mE_{G,{\mathcal U}_2}\lp \mE_{{\mathcal U}_1}  \sum_{i_3=1}^{l}  \frac{
\lp \mE_{{\mathcal U}_1} Z_{i_3}^{\m_1} \rp^{p-1}    }{ \lp
\sum_{i_3=1}^{l}
\lp \mE_{{\mathcal U}_1} Z_{i_3}^{\m_1} \rp^p \rp   }  \Theta_{G,1}\rp \nonumber \\
T_{G,j}^d & = &  \mE_{G,{\mathcal U}_2}\lp \mE_{{\mathcal U}_1}  \sum_{i_3=1}^{l}
\lp \mE_{{\mathcal U}_1} Z_{i_3}^{\m_1} \rp^{p-1}        \Theta_{G,2}\rp
\nonumber \\
T_{G,j}^e & = &  \mE_{G,{\mathcal U}_2}\lp \mE_{{\mathcal U}_1}  \sum_{i_3=1}^{l}
     \Theta_{G,3}\rp ,
 \end{eqnarray}
and write
\begin{eqnarray}\label{eq:genGanal4}
T_{G,j}  & = &  T_{G,j}^c +T_{G,j}^d+T_{G,j}^e.
 \end{eqnarray}
From \cite{Stojnicnflgscompyx23}'s (71) we  have
\begin{eqnarray}\label{eq:genGanal5}
\Theta_{G,1} & = & \lp \frac{\y_j^{(i_2)}}{Z^{1-\m_1}}\lp(C^{(i_1)})^{s-1}\beta_{i_1}A^{(i_1,i_2)}\y_j^{(i_2)}\sqrt{t}+(s-1)(C^{(i_1)})^{s-2}\beta_{i_1}\sum_{p_2=1}^{l}A^{(i_1,p_2)}\y_j^{(p_2)}\sqrt{t}\rp \rp \nonumber \\
& &  -(1-\m_1)
\mE \lp\sum_{p_1=1}^{l} \frac{(\x^{(i_1)})^T\x^{(p_1)}}{\|\x^{(i_1)}\|_2\|\x^{(p_1)}\|_2}
\frac{(C^{(i_1)})^{s-1} A^{(i_1,i_2)}\y_j^{(i_2)}}{Z^{2-\m_1}}
s  (C^{(p_1)})^{s-1}\sum_{p_2=1}^{l}\beta_{p_1}A^{(p_1,p_2)}\y_j^{(p_2)}\sqrt{t}\rp,\nonumber \\
\end{eqnarray}
which together with (\ref{eq:liftgenAanal19g}) allows to immediately write
\begin{eqnarray}\label{eq:genGanal6}
 \sum_{i_1=1}^{l} \sum_{i_2=1}^{l}\sum_{j=1}^{m}\beta_{i_1} \frac{T_{G,j}^c}{\sqrt{t}}
 & = & \sum_{i_1=1}^{l} \sum_{i_2=1}^{l}\sum_{j=1}^{m}\beta_{i_1} \frac{ \mE_{G,{\mathcal U}_2}\lp \mE_{{\mathcal U}_1}  \sum_{i_3=1}^{l}  \frac{
\lp \mE_{{\mathcal U}_1} Z_{i_3}^{\m_1} \rp^{p-1}    }{ \lp
\sum_{i_3=1}^{l}
\lp \mE_{{\mathcal U}_1} Z_{i_3}^{\m_1} \rp^p \rp   }  \Theta_{G,1}\rp  }  {\sqrt{t}} \nonumber\\
 & = & \beta^2 \lp \mE_{G,{\mathcal U}_2} \langle \|\x^{(i_1)}\|_2^2\|\y^{(i_2)}\|_2^2\rangle_{\gamma_{01}^{(1)}} + \mE_{G,{\mathcal U}_2}\langle \|\x^{(i_1)}\|_2^2(\y^{(p_2)})^T\y^{(i_2)}\rangle_{\gamma_{02}^{(1)}}    \rp   \nonumber \\
 &  & -s\beta^2(1-\m_1) \mE_{G,{\mathcal U}_2}\langle (\x^{(p_1)})^T\x^{(i_1)}(\y^{(p_2)})^T\y^{(i_2)}\rangle_{\gamma_{1}^{(1)}}.
 \end{eqnarray}

To determine $T_{G,j}^d$, we first observe
\begin{eqnarray}\label{eq:genGanal7aa0}
\frac{d}{d\u_j^{(p_1,1)}}\lp   \frac{1}{ \lp
\sum_{i_3=1}^{l}
\lp \mE_{{\mathcal U}_1} Z_{i_3}^{\m_1} \rp^p \rp    } \rp
& = & -p \sum_{p_3=1}^{l} \frac{ \lp
 \mE_{{\mathcal U}_1} Z_{p_3}^{\m_1} \rp^{p-1} }{ \lp
\sum_{i_3=1}^{l}
\lp \mE_{{\mathcal U}_1} Z_{i_3}^{\m_1} \rp^p \rp  ^2}\mE_{{\mathcal U}_1}  \frac{dZ_{p_3}^{\m_1}}{d\u_j^{(p_1,1)}}
\nonumber \\
& = & -\m_1 p \sum_{p_3=1}^{l} \frac{ \lp
 \mE_{{\mathcal U}_1} Z_{p_3}^{\m_1} \rp^{p-1} }{ \lp
\sum_{i_3=1}^{l}
\lp \mE_{{\mathcal U}_1} Z_{i_3}^{\m_1} \rp^p \rp  ^2}\mE_{{\mathcal U}_1} Z_{p_3}^{\m_1-1} \frac{dZ_{p_3} }{d\u_j^{(p_1,1)}},
\end{eqnarray}
and then utilize \cite{Stojnicnflgscompyx23}'s (73) to obtain
\begin{eqnarray}\label{eq:genGanal7}
\frac{d}{d\u_j^{(p_1,1)}}\lp   \frac{1}{ \lp
\sum_{i_3=1}^{l}
\lp \mE_{{\mathcal U}_1} Z_{i_3}^{\m_1} \rp^p \rp    } \rp
 & = & -\m_1 p \sum_{p_3=1}^{l} \frac{ \lp
 \mE_{{\mathcal U}_1} Z_{p_3}^{\m_1} \rp^{p-1} }{ \lp
\sum_{i_3=1}^{l}
\lp \mE_{{\mathcal U}_1} Z_{i_3}^{\m_1} \rp^p \rp  ^2}
\nonumber \\
& &
\times
\mE_{{\mathcal U}_1} Z_{p_3}^{\m_1-1}   s (C_{p_3}^{(p_1)})^{s-1}\sum_{p_2=1}^{l}A_{p_3}^{(p_1,p_2)}\beta_{p_1}\y_j^{(p_2)}\sqrt{t}.
\end{eqnarray}
A combination of (\ref{eq:genGanal3}) with (\ref{eq:genGanal7}) then gives
\begin{eqnarray}\label{eq:genGanal8}
\Theta_{G,2} & = &   \lp   \frac{(C_{i_3}^{(i_1)})^{s-1} A_{i_3}^{(i_1,i_2)}\y_j^{(i_2)}}{Z_{i_3}^{1-\m_1}} \sum_{p_1=1}^{l} \mE (\u_j^{(i_1,1)}\u_j^{(p_1,1)})\frac{d}{d\u_j^{(p_1,1)}}\lp   \frac{1}{  \lp
\sum_{i_3=1}^{l}
\lp \mE_{{\mathcal U}_1} Z_{i_3}^{\m_1} \rp^p \rp   } \rp \rp \nonumber \\
& = &   -s\m_1 p  \frac{(C_{i_3}^{(i_1)})^{s-1} A_{i_3}^{(i_1,i_2)}\y_j^{(i_2)}}{Z_{i_3}^{1-\m_1}}
 \sum_{p_1=1}^{l}
\frac{(\x^{(p_1)})^T\x^{(i_1)}}{\|\x^{(i_1)}\|_2\|\x^{(p_1)}\|_2}
  \sum_{p_3=1}^{l} \frac{ \lp
 \mE_{{\mathcal U}_1} Z_{p_3}^{\m_1} \rp^{p-1} }{ \lp
\sum_{i_3=1}^{l}
\lp \mE_{{\mathcal U}_1} Z_{i_3}^{\m_1} \rp^p \rp  ^2}
\nonumber \\
& &
\times
\mE_{{\mathcal U}_1} Z_{p_3}^{\m_1-1}    (C_{p_3}^{(p_1)})^{s-1}\sum_{p_2=1}^{l}A_{p_3}^{(p_1,p_2)}\beta_{p_1}\y_j^{(p_2)}\sqrt{t}
  \nonumber \\
  & = &   -s\m_1 p
\frac{Z_{i_3}^{\m_1}}{ \lp
\sum_{i_3=1}^{l}
\lp \mE_{{\mathcal U}_1} Z_{i_3}^{\m_1} \rp^p \rp }
   \frac{(C_{i_3}^{(i_1)})^{s} A_{i_3}^{(i_1,i_2)}  }{Z_{i_3}  C_{i_3}^{(i_1)}  }
  \sum_{p_3=1}^{l} \frac{ \lp
 \mE_{{\mathcal U}_1} Z_{p_3}^{\m_1} \rp^{p} }{ \lp
\sum_{p_3=1}^{l}
\lp \mE_{{\mathcal U}_1} Z_{p_3}^{\m_1} \rp^p \rp }
\nonumber \\
& &
\times
\mE_{{\mathcal U}_1}  \frac{Z_{p_3}^{\m_1} }
{\mE_{{\mathcal U}_1} Z_{p_3}^{\m_1} }
 \sum_{p_1=1}^{l}
 \frac{  ( C_{p_3}^{(p_1)} )^{s} }{ Z_{p_3} }    \sum_{p_2=1}^{l}  \frac{  A_{p_3}^{(p_1,p_2)}  } {  C_{p_3}^{(p_1)} }   \beta_{p_1}\y_j^{(p_2)} \y_j^{(i_2)}\frac{(\x^{(p_1)})^T\x^{(i_1)}}{\|\x^{(i_1)}\|_2\|\x^{(p_1)}\|_2} \sqrt{t}.
  \end{eqnarray}
From (\ref{eq:genGanal3}) and (\ref{eq:genGanal8}), we then find
\begin{eqnarray}\label{eq:genGanal9}
 \sum_{i_1=1}^{l} \sum_{i_2=1}^{l}\sum_{j=1}^{m}\beta_{i_1} \frac{T_{G,j}^d}{\sqrt{t}}
  & = & -s\beta^2\m_1 p \mE_{G,{\mathcal U}_2} \langle(\x^{(p_1)})^T\x^{(i_1)} (\y^{(p_2)})^T\y^{(i_2)} \rangle_{\gamma_{21}^{(1)}}.
 \end{eqnarray}

To determine $T_{G,j}^e$, we first observe
\begin{eqnarray}\label{eq:genGanal7aa0bb0}
\frac{d}{d\u_j^{(p_1,1)}}\lp
\lp \mE_{{\mathcal U}_1} Z_{i_3}^{\m_1} \rp^{p-1}     \rp
& = & (p-1) \lp
 \mE_{{\mathcal U}_1} Z_{i_3}^{\m_1} \rp^{p-2}
\mE_{{\mathcal U}_1}  \frac{dZ_{i_3}^{\m_1}}{d\u_j^{(p_1,1)}}
\nonumber \\
& = & \m_1
(p-1) \lp
 \mE_{{\mathcal U}_1} Z_{i_3}^{\m_1} \rp^{p-2}
\mE_{{\mathcal U}_1} Z_{i_3}^{\m_1-1} \frac{dZ_{i_3} }{d\u_j^{(p_1,1)}},
\end{eqnarray}
and then utilize again \cite{Stojnicnflgscompyx23}'s (73) to obtain
\begin{equation}\label{eq:genGanal7bb1}
\frac{d}{d\u_j^{(p_1,1)}}\lp
\lp \mE_{{\mathcal U}_1} Z_{i_3}^{\m_1} \rp^{p-1}     \rp
  =  \m_1
(p-1) \lp
 \mE_{{\mathcal U}_1} Z_{i_3}^{\m_1} \rp^{p-2}
\mE_{{\mathcal U}_1} Z_{i_3}^{\m_1-1}   s (C_{i_3}^{(p_1)})^{s-1}\sum_{p_2=1}^{l}A_{i_3}^{(p_1,p_2)}\beta_{p_1}\y_j^{(p_2)}\sqrt{t}.
\end{equation}
Plugging (\ref{eq:genGanal7bb1}) into the expression for $ \Theta_{G,3} $   from  (\ref{eq:genGanal3}) gives
\begin{eqnarray}\label{eq:genGanal8bb2}
\Theta_{G,3} & = &
 \Bigg.\Bigg(   \frac{   (C_{i_3}^{(i_1)})^{s-1} A_{i_3}^{(i_1,i_2)}\y_j^{(i_2)}}{ \lp
\sum_{i_3=1}^{l}
\lp \mE_{{\mathcal U}_1} Z_{i_3}^{\m_1} \rp^p \rp   Z_{i_3}^{1-\m_1}}
  \sum_{p_1=1}^{l}
 \mE (\u_j^{(i_1,1)}\u_j^{(p_1,1)})\frac{d}{d\u_j^{(p_1,1)}}  \Bigg.\Bigg(    \lp \mE_{{\mathcal U}_1} Z_{i_3}^{\m_1} \rp^{p-1}     \Bigg.\Bigg)  \Bigg.\Bigg)
\nonumber \\
 & = &   s\m_1 (p-1)  \frac{(C_{i_3}^{(i_1)})^{s-1} A_{i_3}^{(i_1,i_2)}\y_j^{(i_2)}}{ \lp
\sum_{i_3=1}^{l}
\lp \mE_{{\mathcal U}_1} Z_{i_3}^{\m_1} \rp^p \rp    Z_{i_3}^{1-\m_1}}
 \sum_{p_1=1}^{l}
\frac{(\x^{(p_1)})^T\x^{(i_1)}}{\|\x^{(i_1)}\|_2\|\x^{(p_1)}\|_2}
\nonumber \\
& &  \times
\lp
 \mE_{{\mathcal U}_1} Z_{i_3}^{\m_1} \rp^{p-2}
\mE_{{\mathcal U}_1} Z_{i_3}^{\m_1-1}    (C_{i_3}^{(p_1)})^{s-1}\sum_{p_2=1}^{l}A_{i_3}^{(p_1,p_2)}\beta_{p_1}\y_j^{(p_2)}\sqrt{t}
  \nonumber \\
  & = &   s\m_1 (p-1)
   \sum_{i_3=1}^{l} \frac{ \lp
 \mE_{{\mathcal U}_1} Z_{i_3}^{\m_1} \rp^{p} }{ \lp
\sum_{p_3=1}^{l}
\lp \mE_{{\mathcal U}_1} Z_{p_3}^{\m_1} \rp^p \rp }
\frac{Z_{i_3}^{\m_1}}{
 \mE_{{\mathcal U}_1} Z_{i_3}^{\m_1}   }
   \frac{(C_{i_3}^{(i_1)})^{s} A_{i_3}^{(i_1,i_2)}  }{Z_{i_3}  C_{i_3}^{(i_1)}  }
\nonumber \\
& &
\times
\mE_{{\mathcal U}_1}  \frac{Z_{i_3}^{\m_1} }
{\mE_{{\mathcal U}_1} Z_{i_3}^{\m_1} }
 \sum_{p_1=1}^{l}
 \frac{  ( C_{i_3}^{(p_1)} )^{s} }{ Z_{i_3} }    \sum_{p_2=1}^{l}  \frac{  A_{i_3}^{(p_1,p_2)}  } {  C_{i_3}^{(p_1)} }   \beta_{p_1}\y_j^{(p_2)} \y_j^{(i_2)}\frac{(\x^{(p_1)})^T\x^{(i_1)}}{\|\x^{(i_1)}\|_2\|\x^{(p_1)}\|_2} \sqrt{t}.
  \end{eqnarray}
From (\ref{eq:genGanal3}) and (\ref{eq:genGanal8bb2}), we then find
\begin{eqnarray}\label{eq:genGanal9bb3}
 \sum_{i_1=1}^{l} \sum_{i_2=1}^{l}\sum_{j=1}^{m}\beta_{i_1} \frac{T_{G,j}^e}{\sqrt{t}}
   & = &   s\beta^2\m_1 (p-1) \mE_{G,{\mathcal U}_2} \langle(\x^{(p_1)})^T\x^{(i_1)} (\y^{(p_2)})^T\y^{(i_2)} \rangle_{\gamma_{222}^{(1)}}.
 \end{eqnarray}
A further combination of (\ref{eq:genGanal4}), (\ref{eq:genGanal6}),  (\ref{eq:genGanal9}), and  (\ref{eq:genGanal9bb3}) gives
\begin{eqnarray}\label{eq:genGanal10}
 \sum_{i_1=1}^{l} \sum_{i_2=1}^{l}\sum_{j=1}^{m}\beta_{i_1} \frac{T_{G,j}}{\sqrt{t}}
  & = & \beta^2 \lp \mE_{G,{\mathcal U}_2} \langle \|\x^{(i_1)}\|_2^2\|\y^{(i_2)}\|_2^2\rangle_{\gamma_{01}^{(1)}} +   (s-1) \mE_{G,{\mathcal U}_2}\langle \|\x^{(i_1)}\|_2^2(\y^{(p_2)})^T\y^{(i_2)}\rangle_{\gamma_{02}^{(1)}}   \rp    \nonumber \\
 &  & -s\beta^2(1-\m_1) \mE_{G,{\mathcal U}_2}\langle (\x^{(p_1)})^T\x^{(i_1)}(\y^{(p_2)})^T\y^{(i_2)}\rangle_{\gamma_{1}^{(1)}} \nonumber \\
 & & -s\beta^2\m_1 p \mE_{G,{\mathcal U}_2} \langle(\x^{(p_1)})^T\x^{(i_1)} (\y^{(p_2)})^T\y^{(i_2)} \rangle_{\gamma_{21}^{(1)}}
 \nonumber \\
   &  &   +s\beta^2\m_1 (p-1) \mE_{G,{\mathcal U}_2} \langle(\x^{(p_1)})^T\x^{(i_1)} (\y^{(p_2)})^T\y^{(i_2)} \rangle_{\gamma_{22}^{(1)}}.
 \end{eqnarray}

\subsubsection{Connecting everything together}
\label{sec:conalt}

We now connect together all the results obtained above. From (\ref{eq:genanal10e}) and (\ref{eq:genanal10f}) we first write
\begin{eqnarray}\label{eq:ctp1}
\frac{d\psi(t)}{dt}  & = &       \frac{\mbox{sign}(s)}{2\beta\sqrt{n}} \lp \Omega_G+\Omega_1+\Omega_2+\Omega_3\rp,
\end{eqnarray}
where
\begin{eqnarray}\label{eq:ctp2}
\Omega_G & = & \sum_{i_1=1}^{l}  \sum_{i_2=1}^{l}\sum_{j=1}^{m} \beta_{i_1}\frac{T_{G,j}}{\sqrt{t}}  \nonumber\\
\Omega_1 & = & -\sum_{i_1=1}^{l}  \sum_{i_2=1}^{l} \sum_{j=1}^{m}\beta_{i_1}\frac{T_{2,1,j}}{\sqrt{1-t}}-\sum_{i_1=1}^{l}  \sum_{i_2=1}^{l} \sum_{j=1}^{m}\beta_{i_1}\frac{T_{1,1,j}}{\sqrt{1-t}} \nonumber\\
\Omega_2 & = & -\sum_{i_1=1}^{l}  \sum_{i_2=1}^{l}\beta_{i_1}\|\y^{(i_2)}\|_2\frac{T_{2,2}}{\sqrt{1-t}}-\sum_{i_1=1}^{l}  \sum_{i_2=1}^{l}\beta_{i_1}\|\y^{(i_2)}\|_2\frac{T_{1,2}}{\sqrt{1-t}} \nonumber\\
\Omega_3 & = & \sum_{i_1=1}^{l}  \sum_{i_2=1}^{l}\beta_{i_1}\|\y^{(i_2)}\|_2\frac{T_{2,3}}{\sqrt{t}}+ \sum_{i_1=1}^{l}  \sum_{i_2=1}^{l}\beta_{i_1}\|\y^{(i_2)}\|_2\frac{T_{1,3}}{\sqrt{t}}.
\end{eqnarray}
From (\ref{eq:genGanal10}) we have
\begin{eqnarray}\label{eq:cpt3}
\Omega_G & = & \beta^2 \lp \mE_{G,{\mathcal U}_2} \langle \|\x^{(i_1)}\|_2^2\|\y^{(i_2)}\|_2^2\rangle_{\gamma_{01}^{(1)}} +  (s-1)\mE_{G,{\mathcal U}_2}\langle \|\x^{(i_1)}\|_2^2(\y^{(p_2)})^T\y^{(i_2)}\rangle_{\gamma_{02}^{(1)}}  \rp     \nonumber \\
 &  & -s\beta^2(1-\m_1) \mE_{G,{\mathcal U}_2}\langle (\x^{(p_1)})^T\x^{(i_1)}(\y^{(p_2)})^T\y^{(i_2)}\rangle_{\gamma_{1}^{(1)}} \nonumber \\
 & & -s\beta^2\m_1 p \mE_{G,{\mathcal U}_2} \langle(\x^{(p_1)})^T\x^{(i_1)} (\y^{(p_2)})^T\y^{(i_2)} \rangle_{\gamma_{21}^{(1)}} \nonumber \\
  & & +s\beta^2\m_1 (p-1) \mE_{G,{\mathcal U}_2} \langle(\x^{(p_1)})^T\x^{(i_1)} (\y^{(p_2)})^T\y^{(i_2)} \rangle_{\gamma_{22}^{(1)}}.
 \end{eqnarray}
From (\ref{eq:liftgenAanal19i}) and (\ref{eq:genDanal25}), we have
\begin{eqnarray}\label{eq:cpt4}
-\Omega_1& = & (\p_0-\p_1)\beta^2 \lp \mE_{G,{\mathcal U}_2}\langle \|\x^{(i_1)}\|_2^2\|\y^{(i_2)}\|_2^2\rangle_{\gamma_{01}^{(1)}} +   (s-1)\mE_{G,{\mathcal U}_2}\langle \|\x^{(i_1)}\|_2^2(\y^{(p_2)})^T\y^{(i_2)}\rangle_{\gamma_{02}^{(1)}} \rp \nonumber \\
& & - (\p_0-\p_1)s\beta^2(1-\m_1)\mE_{G,{\mathcal U}_2}\langle \|\x^{(i_1)}\|_2\|\x^{(p_1)}\|_2(\y^{(p_2)})^T\y^{(i_2)} \rangle_{\gamma_{1}^{(1)}} \nonumber\\
& & +\p_1\beta^2 \lp \mE_{G,{\mathcal U}_2}\langle \|\x^{(i_1)}\|_2^2\|\y^{(i_2)}\|_2^2\rangle_{\gamma_{01}^{(1)}} +   (s-1)\mE_{G,{\mathcal U}_2}\langle \|\x^{(i_1)}\|_2^2(\y^{(p_2)})^T\y^{(i_2)}\rangle_{\gamma_{02}^{(1)}} \rp \nonumber \\
& & - \p_1s\beta^2(1-\m_1)\mE_{G,{\mathcal U}_2}\langle \|\x^{(i_1)}\|_2\|\x^{(p_1)}\|_2(\y^{(p_2)})^T\y^{(i_2)} \rangle_{\gamma_{1}^{(1)}}\nonumber \\
 &   &
  -s\beta^2\p_1\m_1  p \mE_{G,{\mathcal U}_2} \langle \|\x^{(i_1)}\|_2\|\x^{(p_1)}\|_2(\y^{(p_2)})^T\y^{(i_2)} \rangle_{\gamma_{21}^{(1)}} \nonumber \\
 &   &
 + s\beta^2\p_1\m_1 (p-1) \mE_{G,{\mathcal U}_2} \langle \|\x^{(i_1)}\|_2\|\x^{(p_1)}\|_2(\y^{(p_2)})^T\y^{(i_2)} \rangle_{\gamma_{22}^{(1)}} \nonumber \\
& = & \p_0\beta^2 \lp \mE_{G,{\mathcal U}_2}\langle \|\x^{(i_1)}\|_2^2\|\y^{(i_2)}\|_2^2\rangle_{\gamma_{01}^{(1)}} +   (s-1)\mE_{G,{\mathcal U}_2}\langle \|\x^{(i_1)}\|_2^2(\y^{(p_2)})^T\y^{(i_2)}\rangle_{\gamma_{02}^{(1)}} \rp \nonumber \\
& & - \p_0s\beta^2(1-\m_1)\mE_{G,{\mathcal U}_2}\langle \|\x^{(i_1)}\|_2\|\x^{(p_1)}\|_2(\y^{(p_2)})^T\y^{(i_2)} \rangle_{\gamma_{1}^{(1)}}\nonumber \\
 &   &
  -s\beta^2\p_1\m_1 p \mE_{G,{\mathcal U}_2} \langle \|\x^{(i_1)}\|_2\|\x^{(p_1)}\|_2(\y^{(p_2)})^T\y^{(i_2)} \rangle_{\gamma_{21}^{(1)}}  \nonumber \\
   &   &
  +s\beta^2\p_1\m_1 (p-1) \mE_{G,{\mathcal U}_2} \langle \|\x^{(i_1)}\|_2\|\x^{(p_1)}\|_2(\y^{(p_2)})^T\y^{(i_2)} \rangle_{\gamma_{22}^{(1)}}.
\end{eqnarray}
From (\ref{eq:liftgenBanal20b}) and (\ref{eq:genEanal25}), we have
\begin{eqnarray}\label{eq:cpt5}
-\Omega_2 & = & (\q_0-\q_1)\beta^2 \lp\mE_{G,{\mathcal U}_2}\langle \|\x^{(i_1)}\|_2^2\|\y^{(i_2)}\|_2^2\rangle_{\gamma_{01}^{(1)}} +   (s-1)\mE_{G,{\mathcal U}_2}\langle \|\x^{(i_1)}\|_2^2 \|\y^{(i_2)}\|_2\|\y^{(p_2)}\|_2\rangle_{\gamma_{02}^{(1)}}\rp\nonumber \\
& & - (\q_0-\q_1)s\beta^2(1-\m_1)\mE_{G,{\mathcal U}_2}\langle (\x^{(p_1)})^T\x^{(i_1)}\|\y^{(i_2)}\|_2\|\y^{(p_2)}\|_2 \rangle_{\gamma_{1}^{(1)}}\nonumber \\
&   &+
\q_1\beta^2\lp\mE_{G,{\mathcal U}_2}\langle \|\x^{(i_1)}\|_2^2\|\y^{(i_2)}\|_2^2\rangle_{\gamma_{01}^{(1)}} +   (s-1)\mE_{G,{\mathcal U}_2}\langle \|\x^{(i_1)}\|_2^2 \|\y^{(i_2)}\|_2\|\y^{(p_2)}\|_2\rangle_{\gamma_{02}^{(1)}}\rp\nonumber \\
& & - \q_1s\beta^2(1-\m_1)\mE_{G,{\mathcal U}_2}\langle (\x^{(p_1)})^T\x^{(i_1)}\|\y^{(i_2)}\|_2\|\y^{(p_2)}\|_2 \rangle_{\gamma_{1}^{(1)}} \nonumber \\
&  & -s\beta^2\q_1\m_1 p\mE_{G,{\mathcal U}_2} \langle \|\y^{(i_2)}\|_2\|\y^{(p_2)}\|_2(\x^{(i_1)})^T\x^{(p_1)}\rangle_{\gamma_{21}^{(1)}} \nonumber \\
&  & +s\beta^2\q_1\m_1 (p-1) \mE_{G,{\mathcal U}_2} \langle \|\y^{(i_2)}\|_2\|\y^{(p_2)}\|_2(\x^{(i_1)})^T\x^{(p_1)}\rangle_{\gamma_{22}^{(1)}} \nonumber \\
&  = &
\q_0\beta^2\lp\mE_{G,{\mathcal U}_2}\langle \|\x^{(i_1)}\|_2^2\|\y^{(i_2)}\|_2^2\rangle_{\gamma_{01}^{(1)}} +   (s-1)\mE_{G,{\mathcal U}_2}\langle \|\x^{(i_1)}\|_2^2 \|\y^{(i_2)}\|_2\|\y^{(p_2)}\|_2\rangle_{\gamma_{02}^{(1)}}\rp\nonumber \\
& & - \q_0s\beta^2(1-\m_1)\mE_{G,{\mathcal U}_2}\langle (\x^{(p_1)})^T\x^{(i_1)}\|\y^{(i_2)}\|_2\|\y^{(p_2)}\|_2 \rangle_{\gamma_{1}^{(1)}} \nonumber \\
&  & -s\beta^2\q_1\m_1 p \mE_{G,{\mathcal U}_2} \langle \|\y^{(i_2)}\|_2\|\y^{(p_2)}\|_2(\x^{(i_1)})^T\x^{(p_1)}\rangle_{\gamma_{21}^{(1)}} \nonumber \\
&  & +s\beta^2\q_1\m_1 (p-1) \mE_{G,{\mathcal U}_2} \langle \|\y^{(i_2)}\|_2\|\y^{(p_2)}\|_2(\x^{(i_1)})^T\x^{(p_1)}\rangle_{\gamma_{22}^{(1)}}.
\end{eqnarray}
From (\ref{eq:liftgenCanal21b}) and (\ref{eq:genFanal29}), we have
  \begin{eqnarray}\label{eq:cpt6}
\Omega_3 & = & (\p_0\q_0-\p_1\q_1)\beta^2 \lp \mE_{G,{\mathcal U}_2}\langle \|\x^{(i_1)}\|_2^2\|\y^{(i_2)}\|_2^2\rangle_{\gamma_{01}^{(1)}} +   (s-1)\mE_{G,{\mathcal U}_2}\langle \|\x^{(i_1)}\|_2^2 \|\y^{(i_2)}\|_2\|\y^{(p_2)}\|_2\rangle_{\gamma_{02}^{(1)}}\rp\nonumber \\
& & - (\p_0\q_0-\p_1\q_1)s\beta^2(1-\m_1)\mE_{G,{\mathcal U}_2}\langle \|\x^{(i_1)}\|_2\|\x^{(p_`)}\|_2\|\y^{(i_2)}\|_2\|\y^{(p_2)}\|_2 \rangle_{\gamma_{1}^{(1)}}\nonumber \\
&  & +
\p_1\q_1\beta^2\lp\mE_{G,{\mathcal U}_2}\langle \|\x^{(i_1)}\|_2^2\|\y^{(i_2)}\|_2^2\rangle_{\gamma_{01}^{(1)}} +   (s-1)\mE_{G,{\mathcal U}_2}\langle \|\x^{(i_1)}\|_2^2 \|\y^{(i_2)}\|_2\|\y^{(p_2)}\|_2\rangle_{\gamma_{02}^{(1)}}\rp\nonumber \\
& & - \p_1\q_1 s\beta^2(1-\m_1)\mE_{G,{\mathcal U}_2}\langle \|\x^{(i_1)}\|_2\|\x^{(p_`)}\|_2\|\y^{(i_2)}\|_2\|\y^{(p_2)}\|_2 \rangle_{\gamma_{1}^{(1)}} \nonumber \\
&  & -s\beta^2\p_1\q_1\m_1 p \mE_{G,{\mathcal U}_2} \mE_{G,{\mathcal U}_2}\langle\|\x^{(i_2)}\|_2\|\x^{(p_2)}\|_2\|\y^{(i_2)}\|_2\|\y^{(p_2)}\rangle_{\gamma_{21}^{(1)}} \nonumber \\
&  & +s\beta^2\p_1\q_1\m_1 (p-1) \mE_{G,{\mathcal U}_2} \mE_{G,{\mathcal U}_2}\langle\|\x^{(i_2)}\|_2\|\x^{(p_2)}\|_2\|\y^{(i_2)}\|_2\|\y^{(p_2)}\rangle_{\gamma_{22}^{(1)}} \nonumber \\
& = &
\p_0\q_0\beta^2\lp\mE_{G,{\mathcal U}_2}\langle \|\x^{(i_1)}\|_2^2\|\y^{(i_2)}\|_2^2\rangle_{\gamma_{01}^{(1)}} +   (s-1)\mE_{G,{\mathcal U}_2}\langle \|\x^{(i_1)}\|_2^2 \|\y^{(i_2)}\|_2\|\y^{(p_2)}\|_2\rangle_{\gamma_{02}^{(1)}}\rp\nonumber \\
& & - \p_0\q_0 s\beta^2(1-\m_1)\mE_{G,{\mathcal U}_2}\langle \|\x^{(i_1)}\|_2\|\x^{(p_`)}\|_2\|\y^{(i_2)}\|_2\|\y^{(p_2)}\|_2 \rangle_{\gamma_{1}^{(1)}} \nonumber \\
&  & -s\beta^2\p_1\q_1
\m_1 p \mE_{G,{\mathcal U}_2} \langle\|\x^{(i_1)}\|_2\|\x^{(p_1)}\|_2\|\y^{(i_2)}\|_2\|\y^{(p_2)}\rangle_{\gamma_{21}^{(1)}}   \nonumber \\
&  & +s\beta^2\p_1\q_1
\m_1 (p-1) \mE_{G,{\mathcal U}_2} \langle\|\x^{(i_1)}\|_2\|\x^{(p_1)}\|_2\|\y^{(i_2)}\|_2\|\y^{(p_2)}\rangle_{\gamma_{12}^{(1)}}.
\end{eqnarray}
Finally, a combination of (\ref{eq:ctp1}) and (\ref{eq:cpt3})-(\ref{eq:cpt6}) gives
\begin{eqnarray}\label{eq:cpt7}
\frac{d\psi(t)}{dt}  & = &       \frac{\mbox{sign}(s)\beta^2}{2\sqrt{n}} \lp \phi_1^{(1)}+\phi_{21}^{(1)} +\phi_{22}^{(1)} +\phi_{01}^{(1)}+\phi_{02}^{(1)}\rp ,
 \end{eqnarray}
where
\begin{eqnarray}\label{eq:cpt8}
\phi_1^{(1)} & = &
-s(1-\m_1)\mE_{G,{\mathcal U}_2} \langle (\p_0\|\x^{(i_1)}\|_2\|\x^{(p_1)}\|_2 -(\x^{(p_1)})^T\x^{(i_1)})(\q_0\|\y^{(i_2)}\|_2\|\y^{(p_2)}\|_2 -(\y^{(p_2)})^T\y^{(i_2)})\rangle_{\gamma_{1}^{(1)}} \nonumber \\
\phi_{21}^{(1)} & = &
-s\m_1 p \mE_{G,{\mathcal U}_2} \langle (\p_1\|\x^{(i_1)}\|_2\|\x^{(p_1)}\|_2 -(\x^{(p_1)})^T\x^{(i_1)})(\q_1\|\y^{(i_2)}\|_2\|\y^{(p_2)}\|_2 -(\y^{(p_2)})^T\y^{(i_2)})\rangle_{\gamma_{21}^{(1)}} \nonumber \\
\phi_{22}^{(1)} & = &
s\m_1 (p-1) \mE_{G,{\mathcal U}_2} \langle (\p_1\|\x^{(i_1)}\|_2\|\x^{(p_1)}\|_2 -(\x^{(p_1)})^T\x^{(i_1)})(\q_1\|\y^{(i_2)}\|_2\|\y^{(p_2)}\|_2 -(\y^{(p_2)})^T\y^{(i_2)})\rangle_{\gamma_{22}^{(1)}} \nonumber \\
\phi_{01}^{(1)} & = & (1-\p_0)(1-\q_0)\mE_{G,{\mathcal U}_2}\langle \|\x^{(i_1)}\|_2^2\|\y^{(i_2)}\|_2^2\rangle_{\gamma_{01}^{(1)}} \nonumber\\
\phi_{02}^{(1)} & = & (s-1)(1-\p_0)\mE_{G,{\mathcal U}_2}\left\langle \|\x^{(i_1)}\|_2^2 \lp\q_0\|\y^{(i_2)}\|_2\|\y^{(p_2)}\|_2-(\y^{(p_2)})^T\y^{(i_2)}\rp\right\rangle_{\gamma_{02}^{(1)}}. \end{eqnarray}

The following proposition summarizes the above results.
\begin{proposition}
\label{thm:thm1} Consider vector $\p=[\p_0,\p_1,\p_2]$ with $\p_0\geq \p_1\geq \p_2=0$,  vector $\q=[\q_0,\q_1,\q_2]$ with $\q_0\geq \q_1\geq \q_2=0$, and a scalar $\m_1$. For $k\in\{1,2\}$ let $G\in\mR^{m \times n},u^{(4,k)}\in\mR^1,\u^{(2,k)}\in\mR^{m\times 1}$, and $\h^{(k)}\in\mR^{n\times 1}$ be independent Gaussian objects with  i.i.d. zero-mean components.  Let the variances of the elements of $G$, $u^{(4,1)}$, $u^{(4,2)}$, $\u^{(2,1)}$, $\u^{(2,2)}$, $\h^{(1)}$, and $\h^{(2)}$ be $1$, $\p_0\q_0-\p_1\q_1$, $\p_1\q_1$, $\p_0-\p_1$, $\p_1$, $\q_0-\q_1$, and $\q_1$, respectively. Let ${\mathcal U}_k=[u^{(4,k)},\u^{(2,k)},\h^{(k)}]$ and  assume that sets ${\mathcal X}=\{\x^{(1)},\x^{(2)},\dots,\x^{(l)}\}$, $\bar{{\mathcal X}}=\{\bar{\x}^{(1)},\bar{\x}^{(2)},\dots,\bar{\x}^{(l)}\}$,  and  ${\mathcal Y}=\{\y^{(1)},\y^{(2)},\dots,\y^{(l)}\}$ where $\x^{(i)}, \bar{\x}^{(i)}\in \mR^{n}$ and $\y^{(i)}\in \mR^{m}$, $1\leq i\leq l$ are given. Let $p,\beta\geq 0$ and $s$ be real numbers and let $f_{\bar{\x}^{(i)}}(\cdot) : \mR^n  \rightarrow \mR$ be a given function. Consider 
\begin{equation}\label{eq:prop1eq1}
\psi( t)  =  \mE_{G,{\mathcal U}_2} \frac{1}{p|s|\sqrt{n}\m_1}
\log \lp  \sum_{i_3=1}^{l}  \lp  \mE_{{\mathcal U}_1} \lp \sum_{i_1=1}^{l}\lp\sum_{i_2=1}^{l}e^{\beta D_0^{(i_1,i_2,i_3 )}} \rp^{s}\rp^{\m_1} \rp^p \rp,
\end{equation}
where
\begin{eqnarray}\label{eq:prop1eq2}
 D_0^{(i_1,i_2,i_3)} & \triangleq & \sqrt{t}(\y^{(i_2)})^T
 G\x^{(i_1)}+\sqrt{1-t}\|\x^{(i_1)}\|_2 (\y^{(i_2)})^T(\u^{(2,1)}+\u^{(2,2)})\nonumber \\
 & & +\sqrt{t}\|\x^{(i_1)}\|_2\|\y^{(i_2)}\|_2(u^{(4,1)}+u^{(4,2)}) +\sqrt{1-t}\|\y^{(i_2)}\|_2(\h^{(1)}+\h^{(2)})^T\x^{(i_1)}
 + f_{\bar{\x}^{(i_3)} } ( \x^{(i_1)}  ). \nonumber \\
 \end{eqnarray}
Then
\begin{eqnarray}\label{eq:prop1eq3}
\frac{d\psi(t)}{dt}  & = &   \frac{\mbox{sign}(s)\beta^2}{2\sqrt{n}} \lp \phi_1^{(1)}+\phi_{21}^{(1)} +\phi_{22}^{(1)} +\phi_{01}^{(1)}+\phi_{02}^{(1)}\rp,
 \end{eqnarray}
where $\phi$'s are as in (\ref{eq:cpt8}).
\end{proposition}
\begin{proof}
  Follows immediately as a consequence of the above presentation.
\end{proof}

\section{Moving from the first to the second level of full lifting}
\label{sec:seclev}

In this section, we discuss how the first level results presented above extend to the second level. All key concepts needed for this extension are already present in Section \ref{sec:gencon}. However, one has to be very careful as to how to use them. Also, since the move from the first to the second level is pretty much repeated successively later on as one moves to higher levels,  we proceed very carefully and formalize all general steps that can then readily be recalled on and utilized on higher levels. As usual, to facilitate the exposition, we try to parallel the analytical methodology of Section \ref{sec:gencon} as closely as possible. On occasion and when circumstances allow, we avoid unnecessary repetitions and proceed at a bit faster pace.

We again consider vectors $\p$ and $\q$ but this time $\p=[\p_0,\p_1,\p_2,\p_3]$ with $\p_0\geq \p_1\geq \p_2\geq \p_3=0$ and $\q=[\q_0,\q_1,\q_2,\q_3]$ with $\q_0\geq \q_1\geq \q_2\geq \q_3 = 0$. In addition to  $G\in \mR^{m\times n}$ with i.i.d. standard normal elements we now have random variables $u^{(4,1)}\sim {\mathcal N}(0,\p_0\q_0-\p_1\q_1)$, $u^{(4,2)}\sim {\mathcal N}(0,\p_1\q_1-\p_2\q_2)$, and $u^{(4,3)}\sim {\mathcal N}(0,\p_2\q_2)$ that are independent of $G$ and among themselves. We now have a vector $\m=[\m_1,\m_2]$ and the following interpolating function $\psi(\cdot)$
\begin{equation}\label{eq:lev2genanal3}
\psi(t)  =  \mE_{G,{\mathcal U}_3} \frac{1}{p|s|\sqrt{n}\m_2} \log \lp \mE_{{\mathcal U}_2}\lp
\lp
\sum_{i_3=1}^{l}
\lp
 \mE_{{\mathcal U}_1}  Z_{i_3}^{\m_1}
\rp^p
\rp^{\frac{\m_2}{\m_1}}\rp\rp,
\end{equation}
with
\begin{eqnarray}\label{eq:lev2genanal3a}
Z_{i_3} & \triangleq & \sum_{i_1=1}^{l}\lp\sum_{i_2=1}^{l}e^{\beta D_0^{(i_1,i_2,i_3)}} \rp^{s} \nonumber \\
 D_0^{(i_1,i_2,i_3)} & \triangleq & \sqrt{t}(\y^{(i_2)})^T
 G\x^{(i_1)}+\sqrt{1-t}\|\x^{(i_1)}\|_2 (\y^{(i_2)})^T(\u^{(2,1)}+\u^{(2,2)}+\u^{(2,3)}  )\nonumber \\
 & & +\sqrt{t}\|\x^{(i_1)}\|_2\|\y^{(i_2)}\|_2(u^{(4,1)}+u^{(4,2)}+u^{(4,3)}  ) +\sqrt{1-t}\|\y^{(i_2)}\|_2(\h^{(1)}+\h^{(2)}+\h^{(3)}  )^T\x^{(i_1)}
\nonumber \\
& &  + f_{\bar{\x}^{(i_3)}} ( \x^{(i_1)}  ),
 \end{eqnarray}
 and ${\mathcal U}_k=[u^{(4,k)},\u^{(2,k)},\h^{(2k)}]$, $k\in\{1,2,3\}$. In (\ref{eq:lev2genanal3})
 and (\ref{eq:lev2genanal3a}), $\u^{(2,1)}$, $\u^{(2,2)}$, and $\u^{(2,3)}$  are $m$ dimensional Gaussians with variances $\p_0-\p_1$, $\p_1-\p_2$, and , $\p_2$, respectively. Analogously,  $\h^{(1)}$, $\h^{(2)}$, and $\h^{(3)}$ are $n$ dimensional Gaussians with variances $\q_0-\q_1$, $\q_1-\q_2$, and $\q_2$, respectively. All random quantities are zero-mean and independent among themselves. Following the convention established in (\ref{eq:genanal4}), we set
\begin{eqnarray}\label{eq:lev2genanal4}
\bar{\u}^{(i_1,1)} & =  & \frac{G\x^{(i_1)}}{\|\x^{(i_1)}\|_2} \nonumber \\
\u^{(i_1,3,1)} & =  & \frac{(\h^{(1)})^T\x^{(i_1)}}{\|\x^{(i_1)}\|_2} \nonumber \\
\u^{(i_1,3,2)} & =  & \frac{(\h^{(2)})^T\x^{(i_1)}}{\|\x^{(i_1)}\|_2} \nonumber \\
\u^{(i_1,3,3)} & =  & \frac{(\h^{(3)})^T\x^{(i_1)}}{\|\x^{(i_1)}\|_2}.
\end{eqnarray}
and recall that
\begin{eqnarray}\label{eq:lev2genanal5}
\bar{\u}_j^{(i_1,1)} & =  & \frac{G_{j,1:n}\x^{(i_1)}}{\|\x^{(i_1)}\|_2},1\leq j\leq m.
\end{eqnarray}
It is not that difficult to see that for any $i_1$ the elements of $\u^{(i_1,1)}$ are i.i.d. standard normals  and the elements of $\u^{(i_1,3,1)}$, $\u^{(i_1,3,2)}$, and $\u^{(i_1,3,3)}$ are zero-mean Gaussians with  respective variances $\q_0-\q_1$, $\q_1-\q_2$, and $\q_2$. We then rewrite (\ref{eq:lev2genanal3}) as
\begin{equation}\label{eq:lev2genanal6}
\psi(t)  =  \mE_{G,{\mathcal U}_3} \frac{1}{p|s|\sqrt{n}\m_2} \log \lp\mE_{{\mathcal U}_2}  \lp
\sum_{i_3=1}^{l}
\lp
\mE_{{\mathcal U}_1} \lp \sum_{i_1=1}^{l}\lp\sum_{i_2=1}^{l}A_{i_3}^{(i_1,i_2)} \rp^{s}\rp^{\m_1}
\rp^p
 \rp^{\frac{\m_2}{\m_1}}\rp,
\end{equation}
where $\beta_{i_1}=\beta\|\x^{(i_1)}\|_2$ and now
\begin{eqnarray}\label{eq:lev2genanal7}
B^{(i_1,i_2)} & \triangleq &  \sqrt{t}(\y^{(i_2)})^T\bar{\u}^{(i_1,1)}+\sqrt{1-t} (\y^{(i_2)})^T(\u^{(2,1)}+\u^{(2,2)}+\u^{(2,3)} ) \nonumber \\
D^{(i_1,i_2,i_3)} & \triangleq &  B^{(i_1,i_2)}+\sqrt{t}\|\y^{(i_2)}\|_2 (u^{(4,1)}+u^{(4,2)}+u^{(4,3)})+\sqrt{1-t}\|\y^{(i_2)}\|_2(\u^{(i_1,3,1)}+\u^{(i_1,3,2)}+\u^{(i_1,3,3)})  \nonumber \\
& & + f_{\bar{\x}^{(i_3)} } (\x^{(i_1)}) \nonumber \\
A_{i_3}^{(i_1,i_2)} & \triangleq &  e^{\beta_{i_1}D^{(i_1,i_2,i_3)}}\nonumber \\
C_{i_3}^{(i_1)} & \triangleq &  \sum_{i_2=1}^{l}A^{(i_1,i_2)}\nonumber \\
Z_{i_3} & \triangleq & \sum_{i_1=1}^{l} \lp \sum_{i_2=1}^{l} A_{i_3}^{(i_1,i_2)}\rp^s =\sum_{i_1=1}^{l}  (C_{i_3}^{(i_1)})^s.
\end{eqnarray}
To study monotonicity of $\psi(t)$, we again start by considering its derivative
\begin{eqnarray}\label{eq:lev2genanal9}
\frac{d\psi(t)}{dt} & = &  \frac{d}{dt}\lp
\mE_{G,{\mathcal U}_3} \frac{1}{p|s|\sqrt{n}\m_2} \log \lp \mE_{{\mathcal U}_2}\lp
\lp
\sum_{i_3=1}^{l}
\lp
 \mE_{{\mathcal U}_1}  Z_{i_3}^{\m_1}
\rp^p
\rp^{\frac{\m_2}{\m_1}}\rp\rp
 \rp\nonumber \\
& = &  \mE_{G,{\mathcal U}_3} \frac{1}{p|s|\sqrt{n}\m_2\mE_{{\mathcal U}_2}\lp\lp
\sum_{i_3=1}^{l}
\lp
 \mE_{{\mathcal U}_1}  Z_{i_3}^{\m_1}
\rp^p
\rp^{\frac{\m_2}{\m_1}}\rp}
\frac{d\lp \mE_{{\mathcal U}_2}\lp\lp
\sum_{i_3=1}^{l}
\lp
 \mE_{{\mathcal U}_1}  Z_{i_3}^{\m_1}
\rp^p
\rp^{\frac{\m_2}{\m_1}}\rp\rp}{dt}\nonumber \\
& = &  \mE_{G,{\mathcal U}_3,{\mathcal U}_2}
\sum_{i_3=1}^{l}
\frac{\lp
\sum_{i_3=1}^{l}
\lp
 \mE_{{\mathcal U}_1}  Z_{i_3}^{\m_1}
\rp^p
\rp^{\frac{\m_2}{\m_1}-1} \lp
 \mE_{{\mathcal U}_1}  Z_{i_3}^{\m_1}
\rp^{p-1}  }   {|s|\sqrt{n}\m_1\mE_{{\mathcal U}_2}\lp\lp
\sum_{i_3=1}^{l}
\lp
 \mE_{{\mathcal U}_1}  Z_{i_3}^{\m_1}
\rp^p
\rp^{\frac{\m_2}{\m_1}}\rp}
\frac{d \mE_{{\mathcal U}_1} Z_{i_3}^{\m_1} }{dt}\nonumber \\
& = & \mE_{G,{\mathcal U}_3,{\mathcal U}_2}
\sum_{i_3=1}^{l}
\frac{\lp
\sum_{i_3=1}^{l}
\lp
 \mE_{{\mathcal U}_1}  Z_{i_3}^{\m_1}
\rp^p
\rp^{\frac{\m_2}{\m_1}-1} \lp
 \mE_{{\mathcal U}_1}  Z_{i_3}^{\m_1}
\rp^{p-1}  }   {|s|\sqrt{n}\mE_{{\mathcal U}_2}\lp\lp
\sum_{i_3=1}^{l}
\lp
 \mE_{{\mathcal U}_1}  Z_{i_3}^{\m_1}
\rp^p
\rp^{\frac{\m_2}{\m_1}}\rp}
\mE_{{\mathcal U}_1} \frac{1}{Z_{i_3}^{1-\m_1}}\frac{d Z_{i_3}}{dt}\nonumber \\
& = &  \mE_{G,{\mathcal U}_3,{\mathcal U}_2}
\sum_{i_3=1}^{l}
\frac{\lp
\sum_{i_3=1}^{l}
\lp
 \mE_{{\mathcal U}_1}  Z_{i_3}^{\m_1}
\rp^p
\rp^{\frac{\m_2}{\m_1}-1} \lp
 \mE_{{\mathcal U}_1}  Z_{i_3}^{\m_1}
\rp^{p-1}  }   {|s|\sqrt{n}\mE_{{\mathcal U}_2}\lp\lp
\sum_{i_3=1}^{l}
\lp
 \mE_{{\mathcal U}_1}  Z_{i_3}^{\m_1}
\rp^p
\rp^{\frac{\m_2}{\m_1}}\rp}
\mE_{{\mathcal U}_1} \frac{1}{Z_{i_3}^{1-\m_1}}
\frac{d\lp \sum_{i_1=1}^{l} \lp \sum_{i_2=1}^{l} A_{i_3}^{(i_1,i_2)}\rp^s \rp }{dt}\nonumber \\
& = &  \mE_{G,{\mathcal U}_3,{\mathcal U}_2}
\sum_{i_3=1}^{l}
\frac{\lp
\sum_{i_3=1}^{l}
\lp
 \mE_{{\mathcal U}_1}  Z_{i_3}^{\m_1}
\rp^p
\rp^{\frac{\m_2}{\m_1}-1} \lp
 \mE_{{\mathcal U}_1}  Z_{i_3}^{\m_1}
\rp^{p-1}  }   {|s|\sqrt{n}\mE_{{\mathcal U}_2}\lp\lp
\sum_{i_3=1}^{l}
\lp
 \mE_{{\mathcal U}_1}  Z_{i_3}^{\m_1}
\rp^p
\rp^{\frac{\m_2}{\m_1}}\rp}
s \mE_{{\mathcal U}_1} \frac{1}{Z_{i_3}^{1-\m_1}}
 \sum_{i=1}^{l} (C_{i_3}^{(i_1)})^{s-1} \nonumber \\
& & \times \sum_{i_2=1}^{l}\beta_{i_1}A_{i_3}^{(i_1,i_2)}\frac{dD^{(i_1,i_2,i_3)}}{dt},
\end{eqnarray}
where
\begin{equation}\label{eq:lev2genanal9a}
\frac{dD^{(i_1,i_2,i_3)}}{dt}= \lp \frac{dB^{(i_1,i_2)}}{dt}+\frac{\|\y^{(i_2)}\|_2 (u^{(4,1)}+u^{(4,2)}+u^{(4,3)})}{2\sqrt{t}}-\frac{\|\y^{(i_2)}\|_2 (\u^{(i_1,3,1)}+\u^{(i_1,3,2)}+\u^{(i_1,3,3)})}{2\sqrt{1-t}}\rp.
\end{equation}
From (\ref{eq:lev2genanal7}) we find
\begin{eqnarray}\label{eq:lev2genanal10}
\frac{dB^{(i_1,i_2)}}{dt} & = &   \frac{d\lp\sqrt{t}(\y^{(i_2)})^T\bar{\u}^{(i_1,1)}+\sqrt{1-t} (\y^{(i_2)})^T(\u^{(2,1)}+\u^{(2,2)}+\u^{(2,3)})\rp}{dt} \nonumber \\
 & = &
\sum_{j=1}^{m}\lp \frac{\y_j^{(i_2)}\bar{\u}_j^{(i_1,1)}}{2\sqrt{t}}-\frac{\y_j^{(i_2)}\u_j^{(2,1)}}{2\sqrt{1-t}}-\frac{\y_j^{(i_2)}\u_j^{(2,2)}}{2\sqrt{1-t}}
-\frac{\y_j^{(i_2)}\u_j^{(2,3)}}{2\sqrt{1-t}}\rp.
\end{eqnarray}
Analogously to (\ref{eq:genanal10e}) we then write
\begin{equation}\label{eq:lev2genanal10e}
\frac{d\psi(t)}{dt}  =       \frac{\mbox{sign}(s)}{2\beta\sqrt{n}} \sum_{i_1=1}^{l}  \sum_{i_2=1}^{l}
\beta_{i_1}\lp T_G + T_2+ T_1\rp,
\end{equation}
where
\begin{eqnarray}\label{eq:lev2genanal10f}
T_G & = & \sum_{j=1}^{m}\frac{T_{G,j}}{\sqrt{t}}  \nonumber\\
T_3 & = & -\sum_{j=1}^{m}\frac{T_{3,1,j}}{\sqrt{1-t}}-\|\y^{(i_2)}\|_2\frac{T_{3,2}}{\sqrt{1-t}}+\|\y^{(i_2)}\|_2\frac{T_{3,3}}{\sqrt{t}} \nonumber\\
T_2 & = & -\sum_{j=1}^{m}\frac{T_{2,1,j}}{\sqrt{1-t}}-\|\y^{(i_2)}\|_2\frac{T_{2,2}}{\sqrt{1-t}}+\|\y^{(i_2)}\|_2\frac{T_{2,3}}{\sqrt{t}} \nonumber\\
T_1 & = & -\sum_{j=1}^{m}\frac{T_{1,1,j}}{\sqrt{1-t}}-\|\y^{(i_2)}\|_2\frac{T_{1,2}}{\sqrt{1-t}}+\|\y^{(i_2)}\|_2\frac{T_{1,3}}{\sqrt{t}},
\end{eqnarray}
and for $k\in\{1,2,3\}$
\begin{eqnarray}\label{eq:lev2genanal10g}
T_{G,j} & = &  \mE_{G,{\mathcal U}_3,{\mathcal U}_2} \lp
\sum_{i_3=1}^{l}
\frac{\lp
\sum_{i_3=1}^{l}
\lp
 \mE_{{\mathcal U}_1}  Z_{i_3}^{\m_1}
\rp^p
\rp^{\frac{\m_2}{\m_1}-1} \lp
 \mE_{{\mathcal U}_1}  Z_{i_3}^{\m_1}
\rp^{p-1}  }   {\mE_{{\mathcal U}_2}\lp\lp
\sum_{i_3=1}^{l}
\lp
 \mE_{{\mathcal U}_1}  Z_{i_3}^{\m_1}
\rp^p
\rp^{\frac{\m_2}{\m_1}}\rp}
  \mE_{{\mathcal U}_1}\frac{(C_{i_3}^{(i_1)})^{s-1} A_{i_3}^{(i_1,i_2)} \y_j^{(i_2)}\bar{\u}_j^{(i_1,1)}}{Z_{i_3}^{1-\m_1}} \rp \nonumber \\
T_{k,1,j} & = &   \mE_{G,{\mathcal U}_3,{\mathcal U}_2} \lp
\sum_{i_3=1}^{l}
\frac{\lp
\sum_{i_3=1}^{l}
\lp
 \mE_{{\mathcal U}_1}  Z_{i_3}^{\m_1}
\rp^p
\rp^{\frac{\m_2}{\m_1}-1} \lp
 \mE_{{\mathcal U}_1}  Z_{i_3}^{\m_1}
\rp^{p-1}  }   {\mE_{{\mathcal U}_2}\lp\lp
\sum_{i_3=1}^{l}
\lp
 \mE_{{\mathcal U}_1}  Z_{i_3}^{\m_1}
\rp^p
\rp^{\frac{\m_2}{\m_1}}\rp}
  \mE_{{\mathcal U}_1}\frac{(C_{i_3}^{(i_1)})^{s-1} A_{i_3}^{(i_1,i_2)} \y_j^{(i_2)}\u_j^{(2,k)}}{Z_{i_3}^{1-\m_1}} \rp \nonumber \\
T_{k,2} & = &   \mE_{G,{\mathcal U}_3,{\mathcal U}_2} \lp
\sum_{i_3=1}^{l}
\frac{\lp
\sum_{i_3=1}^{l}
\lp
 \mE_{{\mathcal U}_1}  Z_{i_3}^{\m_1}
\rp^p
\rp^{\frac{\m_2}{\m_1}-1} \lp
 \mE_{{\mathcal U}_1}  Z_{i_3}^{\m_1}
\rp^{p-1}  }   {\mE_{{\mathcal U}_2}\lp\lp
\sum_{i_3=1}^{l}
\lp
 \mE_{{\mathcal U}_1}  Z_{i_3}^{\m_1}
\rp^p
\rp^{\frac{\m_2}{\m_1}}\rp}
  \mE_{{\mathcal U}_1}\frac{(C_{i_3}^{(i_1)})^{s-1} A_{i_3}^{(i_1,i_2)} \u^{(i_1,3,k)}}{Z_{i_3}^{1-\m_1}} \rp \nonumber \\
T_{k,3} & = &   \mE_{G,{\mathcal U}_3,{\mathcal U}_2} \lp
\sum_{i_3=1}^{l}
\frac{\lp
\sum_{i_3=1}^{l}
\lp
 \mE_{{\mathcal U}_1}  Z_{i_3}^{\m_1}
\rp^p
\rp^{\frac{\m_2}{\m_1}-1} \lp
 \mE_{{\mathcal U}_1}  Z_{i_3}^{\m_1}
\rp^{p-1}  }   {\mE_{{\mathcal U}_2}\lp\lp
\sum_{i_3=1}^{l}
\lp
 \mE_{{\mathcal U}_1}  Z_{i_3}^{\m_1}
\rp^p
\rp^{\frac{\m_2}{\m_1}}\rp}  \mE_{{\mathcal U}_1}\frac{(C_{i_3}^{(i_1)})^{s-1} A_{i_3}^{(i_1,i_2)} u^{(4,k)}}{Z_{i_3}^{1-\m_1}} \rp.
\end{eqnarray}
Following the practice established in Section \ref{sec:gencon},  each of the above ten terms will be handled separately.

\subsection{Computing $\frac{d\psi(t)}{dt}$  -- second level}
\label{sec:lev2compderivative}

The three terms indexed by $k$ are grouped into three groups, $T_k$, $k\in\{1,2,3\}$. As earlier, we carefully select the order in which we handle each of these groups and start with $T_1$, then move to $T_2$, and then to $T_3$. After all three $T_k$ groups are handled we switch to $T_{G,j}$.

\subsubsection{$T_1$--group --- second level}
\label{sec:lev2handlTk1}

Paralleling what was done in Section \ref{sec:handlT1},  each of the three terms from $T_1$--group is handled separately.

\underline{\textbf{\emph{Determining}} $T_{1,1,j}$}
\label{sec:lev2hand1T11}

As in (\ref{eq:liftgenAanal19}), the Gaussian integration by parts gives
 \begin{eqnarray}\label{eq:lev2liftgenAanal19}
T_{1,1,j} & = &   \mE_{G,{\mathcal U}_3,{\mathcal U}_2} \lp
\sum_{i_3=1}^{l}
\frac{\lp
\sum_{i_3=1}^{l}
\lp
 \mE_{{\mathcal U}_1}  Z_{i_3}^{\m_1}
\rp^p
\rp^{\frac{\m_2}{\m_1}-1} \lp
 \mE_{{\mathcal U}_1}  Z_{i_3}^{\m_1}
\rp^{p-1}  }   {\mE_{{\mathcal U}_2}\lp\lp
\sum_{i_3=1}^{l}
\lp
 \mE_{{\mathcal U}_1}  Z_{i_3}^{\m_1}
\rp^p
\rp^{\frac{\m_2}{\m_1}}\rp}
  \mE_{{\mathcal U}_1}\frac{(C_{i_3}^{(i_1)})^{s-1} A_{i_3}^{(i_1,i_2)} \y_j^{(i_2)}\u_j^{(2,1)}}{Z_{i_3}^{1-\m_1}} \rp \nonumber \\
& = &  (\p_0 -\p_1) \mE_{G,{\mathcal U}_3,{\mathcal U}_2} \Bigg ( \Bigg.
\sum_{i_3=1}^{l}
\frac{\lp
\sum_{i_3=1}^{l}
\lp
 \mE_{{\mathcal U}_1}  Z_{i_3}^{\m_1}
\rp^p
\rp^{\frac{\m_2}{\m_1}-1} \lp
 \mE_{{\mathcal U}_1}  Z_{i_3}^{\m_1}
\rp^{p-1}  }   {\mE_{{\mathcal U}_2}\lp\lp
\sum_{i_3=1}^{l}
\lp
 \mE_{{\mathcal U}_1}  Z_{i_3}^{\m_1}
\rp^p
\rp^{\frac{\m_2}{\m_1}}\rp}
\nonumber \\
& & \times
 \mE_{{\mathcal U}_1}\lp  \frac{d}{d\bar{\u}_j^{(2,1)}}\lp\frac{(C_{i_3}^{(i_1)})^{s-1} A_{i_3}^{(i_1,i_2)}\y_j^{(i_2)}}{Z_{i_3}^{1-\m_1}}\rp\rp
 \Bigg. \Bigg ) .
\end{eqnarray}
Moreover, similarly  to (\ref{eq:liftgenAanal19a}), we also have
\begin{eqnarray}\label{eq:lev2liftgenAanal19a}
T_{1,1,j}
& = &  (\p_0-\p_1)
\mE_{G,{\mathcal U}_3,{\mathcal U}_2} \lp
\sum_{i_3=1}^{l}
\frac{\lp
\sum_{i_3=1}^{l}
\lp
 \mE_{{\mathcal U}_1}  Z_{i_3}^{\m_1}
\rp^p
\rp^{\frac{\m_2}{\m_1}-1} \lp
 \mE_{{\mathcal U}_1}  Z_{i_3}^{\m_1}
\rp^{p-1}  }   {\mE_{{\mathcal U}_2}\lp\lp
\sum_{i_3=1}^{l}
\lp
 \mE_{{\mathcal U}_1}  Z_{i_3}^{\m_1}
\rp^p
\rp^{\frac{\m_2}{\m_1}}\rp}
 \lp \Theta_1+\Theta_2 \rp\rp,
\end{eqnarray}
where $\Theta_1$ and $\Theta_2$ are as in (\ref{eq:liftgenAanal19c}). Analogously to (\ref{eq:liftgenAanal19d}), we write
 \begin{eqnarray}\label{eq:lev2liftgenAanal19d}
\sum_{i_1=1}^{l}\sum_{i_2=1}^{l}\sum_{j=1}^{m} \mE_{{\mathcal U}_2} \lp
\sum_{i_3=1}^{l}
\frac{\lp
\sum_{i_3=1}^{l}
\lp
 \mE_{{\mathcal U}_1}  Z_{i_3}^{\m_1}
\rp^p
\rp^{\frac{\m_2}{\m_1}-1} \lp
 \mE_{{\mathcal U}_1}  Z_{i_3}^{\m_1}
\rp^{p-1}  }   {\mE_{{\mathcal U}_2}\lp\lp
\sum_{i_3=1}^{l}
\lp
 \mE_{{\mathcal U}_1}  Z_{i_3}^{\m_1}
\rp^p
\rp^{\frac{\m_2}{\m_1}}\rp}
 \frac{\beta_{i_1}\Theta_1}{\sqrt{1-t}}\rp
&  = &  L^{(2)}_1+L^{(2)}_2,
 \end{eqnarray}
where
\begin{eqnarray}\label{eq:lev2liftgenAanal19d1}
L_1^{(2)} &  = & \mE_{{\mathcal U}_2} \lp
\sum_{i_3=1}^{l}
\frac{\lp
\sum_{i_3=1}^{l}
\lp
 \mE_{{\mathcal U}_1}  Z_{i_3}^{\m_1}
\rp^p
\rp^{\frac{\m_2}{\m_1}-1} \lp
 \mE_{{\mathcal U}_1}  Z_{i_3}^{\m_1}
\rp^{p-1}   Z_{i_3}^{\m_1} }   {\mE_{{\mathcal U}_2}\lp\lp
\sum_{i_3=1}^{l}
\lp
 \mE_{{\mathcal U}_1}  Z_{i_3}^{\m_1}
\rp^p
\rp^{\frac{\m_2}{\m_1}}\rp}
 \sum_{i_1=1}^{l}\frac{(C_{i_3}^{(i_1)})^s}{Z_{i_3}}\sum_{i_2=1}^{l}\frac{A_{i_3}^{(i_1,i_2)}}{C_{i_3}^{(i_1)}}\beta_{i_1}^2\|\y^{(i_2)}\|_2^2\rp \nonumber\\
L_2^{(2)} &  = & \mE_{{\mathcal U}_2} \Bigg ( \Bigg.
\sum_{i_3=1}^{l}
\frac{\lp
\sum_{i_3=1}^{l}
\lp
 \mE_{{\mathcal U}_1}  Z_{i_3}^{\m_1}
\rp^p
\rp^{\frac{\m_2}{\m_1}-1} \lp
 \mE_{{\mathcal U}_1}  Z_{i_3}^{\m_1}
\rp^{p-1}   Z_{i_3}^{\m_1} }   {\mE_{{\mathcal U}_2}\lp\lp
\sum_{i_3=1}^{l}
\lp
 \mE_{{\mathcal U}_1}  Z_{i_3}^{\m_1}
\rp^p
\rp^{\frac{\m_2}{\m_1}}\rp}
\nonumber \\
& & \times
\sum_{i_1=1}^{l}\frac{(s-1)(C_{i_3}^{(i_1)})^s}{Z_{i_3}}\sum_{i_2=1}^{l}\sum_{p_2=1}^{l}\frac{A_{i_3}^{(i_1,i_2)}A_{i_3}^{(i_1,p_2)}}{(C_{i_3}^{(i_1)})^2}\beta_{i_1}^2(\y^{(p_2)})^T\y^{(i_2)}
\Bigg.\Bigg ).
 \end{eqnarray}
 We introduce the operators
\begin{eqnarray}\label{eq:lev2genAanal19d2}
 \Phi_{{\mathcal U}_1}^{(i_3)} & \triangleq &  \mE_{{\mathcal U}_1} \frac{Z_{i_3}^{\m_1}}{\mE_{{\mathcal U}_1}\lp Z_{i_3}^{\m_1}\rp} \nonumber \\
 \Phi_{{\mathcal U}_2} & \triangleq & \mE_{{\mathcal U}_2}   \frac{\lp
 \sum_{i_3=1}^{l} \lp
 \mE_{{\mathcal U}_1}  Z_{i_3}^{\m_1}
\rp^p
      \rp^{\frac{\m_2}{\m_1}}}
      {\mE_{{\mathcal U}_2}\lp
      \lp
      \sum_{i_3=1}^{l}
\lp
 \mE_{{\mathcal U}_1}  Z_{i_3}^{\m_1}
\rp^p
      \rp^{\frac{\m_2}{\m_1}}\rp},
  \end{eqnarray}
and consider the following weighted/product Gibbs measures
\begin{eqnarray}\label{eq:lev2genAanal19e}
  \gamma_{00}(i_3) & = &
  \frac{
\lp \mE_{{\mathcal U}_1} Z_{i_3}^{\m_1} \rp^{p}    }{ \lp
\sum_{i_3=1}^{l}
\lp \mE_{{\mathcal U}_1} Z_{i_3}^{\m_1} \rp^p \rp   }
\nonumber \\
  \gamma_0(i_1,i_2;i_3) & = &
\frac{(C_{i_3}^{(i_1)})^{s}}{Z_{i_3}}  \frac{A_{i_3}^{(i_1,i_2)}}{C_{i_3}^{(i_1)}} \nonumber \\
\gamma_{01}^{(2)}  & = &  \Phi_{{\mathcal U}_2 } \lp  \gamma_{00}(i_3) \Phi_{{\mathcal U}_1}^{(i_3)}  (\gamma_0(i_1,i_2;i_3)) \rp  \nonumber \\
\gamma_{02}^{(2)}  & = & \Phi_{{\mathcal U}_1} \lp \gamma_{00}(i_3)\Phi_{{\mathcal U}_1}^{(i_3)} (\gamma_0(i_1,i_2;i_3)\times \gamma_0(i_1,p_2;i_3)) \rp \nonumber \\
\gamma_1^{(2)}  & = & \Phi_{{\mathcal U}_1} \lp \gamma_{00}(i_3)\Phi_{{\mathcal U}_1}^{(i_3)} (\gamma_0(i_1,i_2;i_3)\times \gamma_0(p_1,p_2;i_3)) \rp \nonumber\\
\gamma_{21}^{(2)}  & = & \Phi_{{\mathcal U}_2} \lp \gamma_{00}(i_3)\Phi_{{\mathcal U}_1}^{(i_3)} \gamma_0(i_1,i_2;i_3)\times
\gamma_{00}(p_3)\Phi_{{\mathcal U}_1} ^{(p_3)} \gamma_0(p_1,p_2;p_3 )  \rp  \nonumber \\
\gamma_{22}^{(2)}  & = & \Phi_{{\mathcal U}_2} \lp \gamma_{00}(i_3)  \lp  \Phi_{{\mathcal U}_1}^{(i_3)} \gamma_0(i_1,i_2;i_3)  \times
\Phi_{{\mathcal U}_1} ^{(i_3)} \gamma_0(p_1,p_2;i_3)  \rp  \rp  \nonumber \\
\gamma_3^{(2)}  & = & (\Phi_{{\mathcal U}_2}  \gamma_{00}(i_3) \Phi_{{\mathcal U}_1}^{(i_3)} \gamma_0(i_1,i_2;i_3)\times  \Phi_{{\mathcal U}_2}
\gamma_{00}(p_3)\Phi_{{\mathcal U}_1}^{(p_3)} \gamma_0(p_1,p_2;p_3)).
\end{eqnarray}
As in Section \ref{sec:gencon}, it is relatively easy to check that all of the above $\gamma$'s are indeed valid measures. For example, for $\gamma_{01}^{(2)}$ we have
\begin{eqnarray}\label{eq:lev2genAanal19f}
 \sum_{i_3=1}^{l}  \sum_{i_1=1}^{l}  \sum_{i_2=1}^{l} \gamma_{01}^{(2)} & = &
 \sum_{i_3=1}^{l}  \sum_{i_1=1}^{l}  \sum_{i_2=1}^{l}
 \mE_{{\mathcal U}_2}   \frac{\lp
 \sum_{i_3=1}^{l} \lp
 \mE_{{\mathcal U}_1}  Z_{i_3}^{\m_1}
\rp^p
      \rp^{\frac{\m_2}{\m_1}}}
      {\mE_{{\mathcal U}_2}\lp
      \lp
      \sum_{i_3=1}^{l}
\lp
 \mE_{{\mathcal U}_1}  Z_{i_3}^{\m_1}
\rp^p
      \rp^{\frac{\m_2}{\m_1}}\rp}
   \frac{
\lp \mE_{{\mathcal U}_1} Z_{i_3}^{\m_1} \rp^{p}    }{ \lp
\sum_{i_3=1}^{l}
\lp \mE_{{\mathcal U}_1} Z_{i_3}^{\m_1} \rp^p \rp   }
      \nonumber \\
      & & \times
\mE_{{\mathcal U}_1} \frac{Z_{i_3}^{\m_1}}{\mE_{{\mathcal U}_1}\lp Z_{i_3}^{\m_1}\rp}
\frac{(C_{i_3}^{(i_1)})^{s}}{Z}  \frac{A_{i_3}^{(i_1,i_2)}}{C_{i_3}^{(i_1)}}\nonumber \\
& = &
 \sum_{i_3=1}^{l}
   \frac{
\lp \mE_{{\mathcal U}_1} Z_{i_3}^{\m_1} \rp^{p}    }{ \lp
\sum_{i_3=1}^{l}
\lp \mE_{{\mathcal U}_1} Z_{i_3}^{\m_1} \rp^p \rp   }
 \mE_{{\mathcal U}_2}   \frac{\lp
 \sum_{i_3=1}^{l} \lp
 \mE_{{\mathcal U}_1}  Z_{i_3}^{\m_1}
\rp^p
      \rp^{\frac{\m_2}{\m_1}}}
      {\mE_{{\mathcal U}_2}\lp
      \lp
      \sum_{i_3=1}^{l}
\lp
 \mE_{{\mathcal U}_1}  Z_{i_3}^{\m_1}
\rp^p
      \rp^{\frac{\m_2}{\m_1}}\rp}
      \mE_{{\mathcal U}_1} \frac{Z_{i_3}^{\m_1}}{\mE_{{\mathcal U}_1}\lp Z_{i_3}^{\m_1}\rp}
\nonumber \\
& & \times
  \sum_{i_1=1}^{l} \frac{(C_{i_3}^{(i_1)})^{s}}{Z_{i_3}}  \sum_{i_2=1}^{l} \frac{A_{i_3}^{(i_1,i_2)}}{C_{i_3}^{(i_1)}}\nonumber\\
 & = &
 \sum_{i_3=1}^{l}
   \frac{
\lp \mE_{{\mathcal U}_1} Z_{i_3}^{\m_1} \rp^{p}    }{ \lp
\sum_{i_3=1}^{l}
\lp \mE_{{\mathcal U}_1} Z_{i_3}^{\m_1} \rp^p \rp   }
 \mE_{{\mathcal U}_2}   \frac{\lp
 \sum_{i_3=1}^{l} \lp
 \mE_{{\mathcal U}_1}  Z_{i_3}^{\m_1}
\rp^p
      \rp^{\frac{\m_2}{\m_1}}}
      {\mE_{{\mathcal U}_2}\lp
      \lp
      \sum_{i_3=1}^{l}
\lp
 \mE_{{\mathcal U}_1}  Z_{i_3}^{\m_1}
\rp^p
      \rp^{\frac{\m_2}{\m_1}}\rp}
      \mE_{{\mathcal U}_1} \frac{Z_{i_3}^{\m_1}}{\mE_{{\mathcal U}_1}\lp Z_{i_3}^{\m_1}\rp}
       \nonumber \\
 & = & 1.
 \end{eqnarray}
As the proofs for other  $\gamma$'s trivially proceed along the above lines we skip them. Following the convention we set earlier, $\langle \cdot \rangle_{a}$ denotes the average with respect to measure $a$. Analogously to (\ref{eq:liftgenAanal19g}), from (\ref{eq:lev2liftgenAanal19d1}) we have
\begin{eqnarray}\label{eq:lev2liftgenAanal19g}
L_1^{(2)} +
L_2^{(2)}
  & = &  \beta^2 \Bigg(\Bigg. \langle \|\x^{(i_1)}\|_2^2\|\y^{(i_2)}\|_2^2\rangle_{\gamma_{01}^{(2)}}   +   (s-1)\langle \|\x^{(i_1)}\|_2^2(\y^{(p_2)})^T\y^{(i_2)}\rangle_{\gamma_{02}^{(2)}} \Bigg. \Bigg).
 \end{eqnarray}
Relying on (\ref{eq:liftgenAanal19c}), we further observe
\begin{eqnarray}\label{eq:lev2liftgenAanal19h}
\sum_{i_1=1}^{l}\sum_{i_2=1}^{l}\sum_{j=1}^{m} \mE_{{\mathcal U}_2} \lp
\sum_{i_3=1}^{l}
\frac{\lp
\sum_{i_3=1}^{l}
\lp
 \mE_{{\mathcal U}_1}  Z_{i_3}^{\m_1}
\rp^p
\rp^{\frac{\m_2}{\m_1}-1} \lp
 \mE_{{\mathcal U}_1}  Z_{i_3}^{\m_1}
\rp^{p-1}   }   {\mE_{{\mathcal U}_2}\lp\lp
\sum_{i_3=1}^{l}
\lp
 \mE_{{\mathcal U}_1}  Z_{i_3}^{\m_1}
\rp^p
\rp^{\frac{\m_2}{\m_1}}\rp}
\frac{\beta_{i_1}\Theta_2}{\sqrt{1-t}}\rp
  =  L^{(2)}_3,
\end{eqnarray}
where
\begin{eqnarray}\label{eq:lev2liftgenAanal19h1}
L^{(2)}_3 & = & -s(1-\m_1)\mE_{{\mathcal U}_2} \Bigg( \Bigg.
\sum_{i_3=1}^{l}
\frac{\lp
\sum_{i_3=1}^{l}
\lp
 \mE_{{\mathcal U}_1}  Z_{i_3}^{\m_1}
\rp^p
\rp^{\frac{\m_2}{\m_1}-1} \lp
 \mE_{{\mathcal U}_1}  Z_{i_3}^{\m_1}
\rp^{p-1}   Z_{i_3}^{\m_1}   }   {\mE_{{\mathcal U}_2}\lp\lp
\sum_{i_3=1}^{l}
\lp
 \mE_{{\mathcal U}_1}  Z_{i_3}^{\m_1}
\rp^p
\rp^{\frac{\m_2}{\m_1}}\rp}
\sum_{i_1=1}^{l}\frac{(C_{i_3}^{(i_1)})^s}{Z_{i_3}}\sum_{i_2=1}^{l}
\frac{A_{i_3}^{(i_1,i_2)}}{C_{i_3}^{(i_1)}} \nonumber \\
& & \times
  \sum_{p_1=1}^{l} \frac{(C_{i_3}^{(p_1)})^s}{Z_{i_3}}\sum_{p_2=1}^{l}\frac{A_{i_3}^{(p_1,p_2)}}{C_{i_3}^{(p_1)}} \beta_{i_1}\beta_{p_1}(\y^{(p_2)})^T\y^{(i_2)} \Bigg.\Bigg)\nonumber \\
& =& -s\beta^2(1-\m_1)\langle \|\x^{(i_1)}\|_2\|\x^{(p_1)}\|_2(\y^{(p_2)})^T\y^{(i_2)} \rangle_{\gamma_{1}^{(2)}}.
\end{eqnarray}
Analogously to (\ref{eq:liftgenAanal19i}), a combination of  (\ref{eq:lev2liftgenAanal19a}), (\ref{eq:lev2liftgenAanal19d}), (\ref{eq:lev2liftgenAanal19g}), (\ref{eq:lev2liftgenAanal19h}), and (\ref{eq:lev2liftgenAanal19h1}) gives
\begin{eqnarray}\label{eq:lev2liftgenAanal19i}
\sum_{i_1=1}^{l}\sum_{i_2=1}^{l}\sum_{j=1}^{m} \beta_{i_1}\frac{T_{1,1,j}}{\sqrt{1-t
}}
& = & (\p_0-\p_1)\beta^2 \Bigg( \Bigg. \mE_{G,{\mathcal U}_3}\langle \|\x^{(i_1)}\|_2^2\|\y^{(i_2)}\|_2^2\rangle_{\gamma_{01}^{(2)}} \nonumber \\
& & +   (s-1)\mE_{G,{\mathcal U}_3}\langle \|\x^{(i_1)}\|_2^2(\y^{(p_2)})^T\y^{(i_2)}\rangle_{\gamma_{02}^{(2)}} \Bigg. \Bigg) \nonumber \\
& & - (\p_0-\p_1)s\beta^2(1-\m_1)\mE_{G,{\mathcal U}_3}\langle \|\x^{(i_1)}\|_2\|\x^{(p_1)}\|_2(\y^{(p_2)})^T\y^{(i_2)} \rangle_{\gamma_{1}^{(2)}}.
\end{eqnarray}

\underline{\textbf{\emph{Determining}} $T_{1,2}$}
\label{sec:lev2hand1T12}

Integrating by parts we first find
\begin{eqnarray}\label{eq:lev2liftgenBanal20}
T_{1,2} & = &   \mE_{G,{\mathcal U}_3,{\mathcal U}_2} \lp
\sum_{i_3=1}^{l}
\frac{\lp
\sum_{i_3=1}^{l}
\lp
 \mE_{{\mathcal U}_1}  Z_{i_3}^{\m_1}
\rp^p
\rp^{\frac{\m_2}{\m_1}-1} \lp
 \mE_{{\mathcal U}_1}  Z_{i_3}^{\m_1}
\rp^{p-1}  }   {\mE_{{\mathcal U}_2}\lp\lp
\sum_{i_3=1}^{l}
\lp
 \mE_{{\mathcal U}_1}  Z_{i_3}^{\m_1}
\rp^p
\rp^{\frac{\m_2}{\m_1}}\rp}
  \mE_{{\mathcal U}_1}\frac{(C_{i_3}^{(i_1)})^{s-1} A_{i_3}^{(i_1,i_2)} \u^{(i_1,3,1)}}{Z_{i_3}^{1-\m_1}} \rp \nonumber \\
& = & (\q_0-\q_1) \mE_{G,{\mathcal U}_3,{\mathcal U}_2} \Bigg ( \Bigg .
\sum_{i_3=1}^{l}
\frac{\lp
\sum_{i_3=1}^{l}
\lp
 \mE_{{\mathcal U}_1}  Z_{i_3}^{\m_1}
\rp^p
\rp^{\frac{\m_2}{\m_1}-1} \lp
 \mE_{{\mathcal U}_1}  Z_{i_3}^{\m_1}
\rp^{p-1}  }   {\mE_{{\mathcal U}_2}\lp\lp
\sum_{i_3=1}^{l}
\lp
 \mE_{{\mathcal U}_1}  Z_{i_3}^{\m_1}
\rp^p
\rp^{\frac{\m_2}{\m_1}}\rp}
\nonumber \\
& &
\times
\mE_{{\mathcal U}_1} \sum_{p_1=1}^{l}\frac{(\x^{(i_1)})^T\x^{(p_1)}}{\|\x^{(i_1)}\|_2\|\x^{(p_1)}\|_2} \frac{d}{d\u^{(p_1,3,1)}}\lp\frac{(C_{i_3}^{(i_1)})^{s-1} A_{(i_3}^{(i_1,i_2)}}{Z_{i_3}^{1-\m_1}}\rp
\Bigg . \Bigg ) , \nonumber \\
\end{eqnarray}
and after following the move from (\ref{eq:liftgenBanal20})  to (\ref{eq:liftgenBanal20b}) , we then write analogously to (\ref{eq:liftgenBanal20b})
 \begin{eqnarray}\label{eq:lev2liftgenBanal20b}
\sum_{i_1=1}^{l}\sum_{i_2=1}^{l} \beta_{i_1}\|\y^{(i_2)}\|_2 \frac{T_{1,2}}{\sqrt{1-t}}
& = & (\q_0-\q_1) \beta^2 \Bigg( \Bigg. \mE_{G,{\mathcal U}_3}\langle \|\x^{(i_1)}\|_2^2\|\y^{(i_2)}\|_2^2\rangle_{\gamma_{01}^{(2)}} \nonumber \\
& & +  (s-1)\mE_{G,{\mathcal U}_3}\langle \|\x^{(i_1)}\|_2^2 \|\y^{(i_2)}\|_2\|\y^{(p_2)}\|_2\rangle_{\gamma_{02}^{(2)}}\Bigg.\Bigg)\nonumber \\
& & - (\q_0-\q_1)s\beta^2(1-\m_1)\mE_{G,{\mathcal U}_3}\langle (\x^{(p_1)})^T\x^{(i_1)}\|\y^{(i_2)}\|_2\|\y^{(p_2)}\|_2 \rangle_{\gamma_{1}^{(2)}}.
\end{eqnarray}

\underline{\textbf{\emph{Determining}} $T_{1,3}$}
\label{sec:lev2hand1T13}

Another integration by parts gives
\begin{eqnarray}\label{eq:lev2liftgenCanal21}
T_{1,3} & = & \mE_{G,{\mathcal U}_3,{\mathcal U}_2} \lp
\sum_{i_3=1}^{l}
\frac{\lp
\sum_{i_3=1}^{l}
\lp
 \mE_{{\mathcal U}_1}  Z_{i_3}^{\m_1}
\rp^p
\rp^{\frac{\m_2}{\m_1}-1} \lp
 \mE_{{\mathcal U}_1}  Z_{i_3}^{\m_1}
\rp^{p-1}  }   {\mE_{{\mathcal U}_2}\lp\lp
\sum_{i_3=1}^{l}
\lp
 \mE_{{\mathcal U}_1}  Z_{i_3}^{\m_1}
\rp^p
\rp^{\frac{\m_2}{\m_1}}\rp} \mE_{{\mathcal U}_1}  \frac{(C^{(i_1)})^{s-1} A^{(i_1,i_2)}u^{(4,1)}}{Z^{1-\m_1}} \rp \nonumber \\
& = & (\p_0\q_0-\p_1\q_1) \mE_{G,{\mathcal U}_3,{\mathcal U}_2} \Bigg ( \Bigg .
\sum_{i_3=1}^{l}
\frac{\lp
\sum_{i_3=1}^{l}
\lp
 \mE_{{\mathcal U}_1}  Z_{i_3}^{\m_1}
\rp^p
\rp^{\frac{\m_2}{\m_1}-1} \lp
 \mE_{{\mathcal U}_1}  Z_{i_3}^{\m_1}
\rp^{p-1}  }   {\mE_{{\mathcal U}_2}\lp\lp
\sum_{i_3=1}^{l}
\lp
 \mE_{{\mathcal U}_1}  Z_{i_3}^{\m_1}
\rp^p
\rp^{\frac{\m_2}{\m_1}}\rp}
\nonumber \\
& &
\times
 \mE_{{\mathcal U}_1} \lp\frac{d}{du^{(4,1)}} \lp\frac{(C^{(i_1)})^{s-1} A^{(i_1,i_2)}u^{(4,1)}}{Z^{1-\m_1}}\rp\rp
 \Bigg . \Bigg ).
\end{eqnarray}
Following the move from (\ref{eq:liftgenCanal21}) to (\ref{eq:liftgenCanal21b}), we then write analogously to (\ref{eq:liftgenCanal21b})
\begin{eqnarray}\label{eq:lev2liftgenCanal21b}
\sum_{i_1=1}^{l}\sum_{i_2=1}^{l} \beta_{i_1}\|\y^{(i_2)}\|_2 \frac{T_{1,3}}{\sqrt{t}}
& = & (\p_0\q_0-\p_1\q_1) \beta^2 \Bigg(\Bigg. \mE_{G,{\mathcal U}_3}\langle \|\x^{(i_1)}\|_2^2\|\y^{(i_2)}\|_2^2\rangle_{\gamma_{01}^{(2)}}\nonumber \\
 & & +  (s-1)\mE_{G,{\mathcal U}_3}\langle \|\x^{(i_1)}\|_2^2 \|\y^{(i_2)}\|_2\|\y^{(p_2)}\|_2\rangle_{\gamma_{02}^{(2)}}\Bigg.\Bigg) \nonumber \\
& & - (\p_0\q_0-\p_1\q_1)s\beta^2(1-\m_1)\mE_{G,{\mathcal U}_3}\langle \|\x^{(i_1)}\|_2\|\x^{(p_`)}\|_2\|\y^{(i_2)}\|_2\|\y^{(p_2)}\|_2 \rangle_{\gamma_{1}^{(2)}}. \nonumber \\
\end{eqnarray}

\subsubsection{$T_2$--group --- second level}
\label{sec:lev2handlT2}

We again handle separately each of the three terms from  $T_2$ group.

\underline{\textbf{\emph{Determining}} $T_{2,1,j}$}
\label{sec:lev2hand1T21}

Gaussian integration by parts first gives
\begin{eqnarray}\label{eq:lev2genDanal19}
T_{2,1,j} & = &   \mE_{G,{\mathcal U}_3,{\mathcal U}_2} \lp
\sum_{i_3=1}^{l}
\frac{\lp
\sum_{i_3=1}^{l}
\lp
 \mE_{{\mathcal U}_1}  Z_{i_3}^{\m_1}
\rp^p
\rp^{\frac{\m_2}{\m_1}-1} \lp
 \mE_{{\mathcal U}_1}  Z_{i_3}^{\m_1}
\rp^{p-1}  }   {\mE_{{\mathcal U}_2}\lp\lp
\sum_{i_3=1}^{l}
\lp
 \mE_{{\mathcal U}_1}  Z_{i_3}^{\m_1}
\rp^p
\rp^{\frac{\m_2}{\m_1}}\rp}
  \mE_{{\mathcal U}_1}\frac{(C_{i_3}^{(i_1)})^{s-1} A_{i_3}^{(i_1,i_2)} \y_j^{(i_2)}\u_j^{(2,2)}}{Z_{i_3}^{1-\m_1}} \rp \nonumber \\
  & = &
\mE_{G,{\mathcal U}_3,{\mathcal U}_1} \Bigg ( \Bigg.
\sum_{i_3=1}^{l}
\frac{ 1 }   {\mE_{{\mathcal U}_2}\lp\lp
\sum_{i_3=1}^{l}
\lp
 \mE_{{\mathcal U}_1}  Z_{i_3}^{\m_1}
\rp^p
\rp^{\frac{\m_2}{\m_1}}\rp}
\nonumber \\
& &
\times
\mE_{{\mathcal U}_2}\lp\mE_{{\mathcal U}_2} (\u_j^{(2,2)}\u_j^{(2,2)})\frac{d}{d\u_j^{(2,2)}}\lp \frac{(C_{i_3}^{(i_1)})^{s-1} A_{i_3}^{(i_1,i_2)}\y_j^{(i_2)}}{Z_{i_3}^{1-\m_1} \lp
\sum_{i_3=1}^{l}
\lp
 \mE_{{\mathcal U}_1}  Z_{i_3}^{\m_1}
\rp^p
\rp^{1-\frac{\m_2}{\m_1}} \lp
 \mE_{{\mathcal U}_1}  Z_{i_3}^{\m_1}
\rp^{1-p}   }
\rp\rp
\Bigg . \Bigg )
 \nonumber \\
& = &
 (\p_1-\p_2) \mE_{G,{\mathcal U}_3,{\mathcal U}_2,{\mathcal U}_1} \Bigg ( \Bigg .
\sum_{i_3=1}^{l}
\frac{\lp
\sum_{i_3=1}^{l}
\lp
 \mE_{{\mathcal U}_1}  Z_{i_3}^{\m_1}
\rp^p
\rp^{\frac{\m_2}{\m_1}-1} \lp
 \mE_{{\mathcal U}_1}  Z_{i_3}^{\m_1}
\rp^{p-1}  }   {\mE_{{\mathcal U}_2}\lp\lp
\sum_{i_3=1}^{l}
\lp
 \mE_{{\mathcal U}_1}  Z_{i_3}^{\m_1}
\rp^p
\rp^{\frac{\m_2}{\m_1}}\rp}
\nonumber \\
& &
\times
 \frac{d}{d\u_j^{(2,2)}}\lp \frac{(C_{i_3}^{(i_1)})^{s-1} A_{i_3}^{(i_1,i_2)}\y_j^{(i_2)}}{Z_{i_3}^{1-\m_1}}\rp\Bigg . \Bigg ) \nonumber \\
&  & +
 (\p_1-\p_2) \mE_{G,{\mathcal U}_3,{\mathcal U}_2,{\mathcal U}_1} \Bigg ( \Bigg .
\sum_{i_3=1}^{l}
\frac{
\lp
 \mE_{{\mathcal U}_1}  Z_{i_3}^{\m_1}
\rp^{p-1}  }   {\mE_{{\mathcal U}_2}\lp\lp
\sum_{i_3=1}^{l}
\lp
 \mE_{{\mathcal U}_1}  Z_{i_3}^{\m_1}
\rp^p
\rp^{\frac{\m_2}{\m_1}}\rp}
\frac{(C_{i_3}^{(i_1)})^{s-1} A_{i_3}^{(i_1,i_2)}\y_j^{(i_2)}}{Z_{i_3}^{1-\m_1}}
\nonumber \\
& &
\times
 \frac{d}{d\u_j^{(2,2)}}\lp    \lp
\sum_{i_3=1}^{l}
\lp
 \mE_{{\mathcal U}_1}  Z_{i_3}^{\m_1}
\rp^p
\rp^{\frac{\m_2}{\m_1} -1 }             \rp\Bigg . \Bigg ) \nonumber \\
&  & +
 (\p_1-\p_2) \mE_{G,{\mathcal U}_3,{\mathcal U}_2,{\mathcal U}_1} \Bigg ( \Bigg .
\sum_{i_3=1}^{l}
\frac{
  \lp
\sum_{i_3=1}^{l}
\lp
 \mE_{{\mathcal U}_1}  Z_{i_3}^{\m_1}
\rp^p
\rp^{\frac{\m_2}{\m_1} -1 }
 }   {\mE_{{\mathcal U}_2}\lp\lp
\sum_{i_3=1}^{l}
\lp
 \mE_{{\mathcal U}_1}  Z_{i_3}^{\m_1}
\rp^p
\rp^{\frac{\m_2}{\m_1}}\rp}
\frac{(C_{i_3}^{(i_1)})^{s-1} A_{i_3}^{(i_1,i_2)}\y_j^{(i_2)}}{Z_{i_3}^{1-\m_1}}
\nonumber \\
& &
\times
 \frac{d}{d\u_j^{(2,2)}}\lp
\lp
 \mE_{{\mathcal U}_1}  Z_{i_3}^{\m_1}
\rp^{p-1}
           \rp\Bigg . \Bigg ) .
 \end{eqnarray}
Similarly to what was done  in (\ref{eq:genDanal19a}), we split $T_{2,1,j}$ into three components
\begin{eqnarray}\label{eq:lev2genDanal19a}
T_{2,1,j}   =   T_{2,1,j}^{c} +  T_{2,1,j}^{d} +  T_{2,1,j}^{e},
\end{eqnarray}
where
\begin{eqnarray}\label{eq:lev2genDanal19b}
T_{2,1,j}^c
& = &
 (\p_1-\p_2) \mE_{G,{\mathcal U}_3,{\mathcal U}_2,{\mathcal U}_1} \Bigg ( \Bigg .
\sum_{i_3=1}^{l}
\frac{\lp
\sum_{i_3=1}^{l}
\lp
 \mE_{{\mathcal U}_1}  Z_{i_3}^{\m_1}
\rp^p
\rp^{\frac{\m_2}{\m_1}-1} \lp
 \mE_{{\mathcal U}_1}  Z_{i_3}^{\m_1}
\rp^{p-1}  }   {\mE_{{\mathcal U}_2}\lp\lp
\sum_{i_3=1}^{l}
\lp
 \mE_{{\mathcal U}_1}  Z_{i_3}^{\m_1}
\rp^p
\rp^{\frac{\m_2}{\m_1}}\rp}
\nonumber \\
& &
\times
 \frac{d}{d\u_j^{(2,2)}}\lp \frac{(C_{i_3}^{(i_1)})^{s-1} A_{i_3}^{(i_1,i_2)}\y_j^{(i_2)}}{Z_{i_3}^{1-\m_1}}\rp\Bigg . \Bigg ) \nonumber \\
T_{2,1,j}^d &  = &
 (\p_1-\p_2) \mE_{G,{\mathcal U}_3,{\mathcal U}_2,{\mathcal U}_1} \Bigg ( \Bigg .
\sum_{i_3=1}^{l}
\frac{
\lp
 \mE_{{\mathcal U}_1}  Z_{i_3}^{\m_1}
\rp^{p-1}  }   {\mE_{{\mathcal U}_2}\lp\lp
\sum_{i_3=1}^{l}
\lp
 \mE_{{\mathcal U}_1}  Z_{i_3}^{\m_1}
\rp^p
\rp^{\frac{\m_2}{\m_1}}\rp}
\frac{(C_{i_3}^{(i_1)})^{s-1} A_{i_3}^{(i_1,i_2)}\y_j^{(i_2)}}{Z_{i_3}^{1-\m_1}}
\nonumber \\
& &
\times
 \frac{d}{d\u_j^{(2,2)}}\lp    \lp
\sum_{i_3=1}^{l}
\lp
 \mE_{{\mathcal U}_1}  Z_{i_3}^{\m_1}
\rp^p
\rp^{\frac{\m_2}{\m_1} -1 }             \rp
\Bigg . \Bigg ) \nonumber \\
T_{2,1,j}^e &  = &
 (\p_1-\p_2) \mE_{G,{\mathcal U}_3,{\mathcal U}_2,{\mathcal U}_1} \Bigg ( \Bigg .
\sum_{i_3=1}^{l}
\frac{
  \lp
\sum_{i_3=1}^{l}
\lp
 \mE_{{\mathcal U}_1}  Z_{i_3}^{\m_1}
\rp^p
\rp^{\frac{\m_2}{\m_1} -1 }
 }   {\mE_{{\mathcal U}_2}\lp\lp
\sum_{i_3=1}^{l}
\lp
 \mE_{{\mathcal U}_1}  Z_{i_3}^{\m_1}
\rp^p
\rp^{\frac{\m_2}{\m_1}}\rp}
\frac{(C_{i_3}^{(i_1)})^{s-1} A_{i_3}^{(i_1,i_2)}\y_j^{(i_2)}}{Z_{i_3}^{1-\m_1}}
\nonumber \\
& &
\times
 \frac{d}{d\u_j^{(2,2)}}\lp
\lp
 \mE_{{\mathcal U}_1}  Z_{i_3}^{\m_1}
\rp^{p-1}
           \rp\Bigg . \Bigg ) .
 \end{eqnarray}
 We observe that $T_{2,1,j}^c$ scaled by $(\p_1-\p_2)$  and  the term considered in the first part of Section \ref{sec:lev2hand1T11} scaled by $(\p_0-\p_1)$ are structurally identical. This then implies
\begin{eqnarray}\label{eq:lev2genDanal19b1}
\sum_{i_1=1}^{l}\sum_{i_2=1}^{l}\sum_{j=1}^{m} \beta_{i_1}\frac{T_{2,1,j}^c}{\sqrt{1-t
}}
& = &  \sum_{i_1=1}^{l}\sum_{i_2=1}^{l}\sum_{j=1}^{m} \beta_{i_1}\frac{T_{1,1,j}}{\sqrt{1-t}}\frac{\p_1-\p_2}{\p_0-\p_1}\nonumber\\
& = & (\p_1-\p_2)\beta^2 \nonumber \\
& & \times \lp \mE_{G,{\mathcal U}_3}\langle \|\x^{(i_1)}\|_2^2\|\y^{(i_2)}\|_2^2\rangle_{\gamma_{01}^{(2)}} +  (s-1)\mE_{G,{\mathcal U}_3}\langle \|\x^{(i_1)}\|_2^2(\y^{(p_2)})^T\y^{(i_2)}\rangle_{\gamma_{02}^{(2)}} \rp \nonumber \\
& & - (\p_1-\p_2)s\beta^2(1-\m_1)\mE_{G,{\mathcal U}_3,{\mathcal U}_2}\langle \|\x^{(i_1)}\|_2\|\x^{(p_1)}\|_2(\y^{(p_2)})^T\y^{(i_2)} \rangle_{\gamma_{1}^{(2)}}.
\end{eqnarray}

We focus on $T_{2,1,j}^d$ next and note the following
\begin{eqnarray}\label{eq:lev2genDanal20}
 \frac{d}{d\u_j^{(2,2)}}\lp
 \lp
\sum_{i_3=1}^{l}
\lp
 \mE_{{\mathcal U}_1}  Z_{i_3}^{\m_1}
\rp^p
\rp^{\frac{\m_2}{\m_1} -1 }
 \rp
=
-p \sum_{p_3=1}^{l}
  \frac{\lp 1-\frac{\m_2}{\m_1} \rp
\lp
 \mE_{{\mathcal U}_1}  Z_{p_3}^{\m_1}
\rp^{p-1}  }{   \lp
\sum_{i_3=1}^{l}
\lp
 \mE_{{\mathcal U}_1}  Z_{i_3}^{\m_1}
\rp^p
\rp^{2 - \frac{\m_2}{\m_1}  }
 }
\mE_{{\mathcal U}_1}\frac{d Z_{p_3}^{\m_1}}{d\u_j^{(2,2)}}.\nonumber \\
\end{eqnarray}
Utilization of(\ref{eq:genDanal21}) allows to write analogously to (\ref{eq:genDanal23})
\begin{eqnarray}\label{eq:lev2genDanal23}
T_{2,1,j}^d &  = & (\p_1-\p_2) \mE_{G,{\mathcal U}_3,{\mathcal U}_2,{\mathcal U}_1}
\Bigg ( \Bigg .
\sum_{i_3=1}^{l}
\frac{
\lp
 \mE_{{\mathcal U}_1}  Z_{i_3}^{\m_1}
\rp^{p-1}  }   {\mE_{{\mathcal U}_2}\lp\lp
\sum_{i_3=1}^{l}
\lp
 \mE_{{\mathcal U}_1}  Z_{i_3}^{\m_1}
\rp^p
\rp^{\frac{\m_2}{\m_1}}\rp}
\frac{(C_{i_3}^{(i_1)})^{s-1} A_{i_3}^{(i_1,i_2)}\y_j^{(i_2)}}{Z_{i_3}^{1-\m_1}}
 \nonumber \\
 & &
 \times
\Bigg . \Bigg (
 -p \sum_{p_3=1}^{l}
  \frac{\lp 1-\frac{\m_2}{\m_1} \rp
\lp
 \mE_{{\mathcal U}_1}  Z_{p_3}^{\m_1}
\rp^{p-1}  }{   \lp
\sum_{i_3=1}^{l}
\lp
 \mE_{{\mathcal U}_1}  Z_{i_3}^{\m_1}
\rp^p
\rp^{2 - \frac{\m_2}{\m_1}  }
 }
\mE_{{\mathcal U}_1}\frac{d Z_{p_3}^{\m_1}}{d\u_j^{(2,2)}}
\Bigg . \Bigg )
\Bigg . \Bigg )
\nonumber \\
&  = & -s\sqrt{1-t}(\p_1-\p_2)(\m_1-\m_2) p \mE_{G,{\mathcal U}_3,{\mathcal U}_2,{\mathcal U}_1}
\Bigg ( \Bigg .
\sum_{i_3=1}^{l}
\frac{
\lp
 \mE_{{\mathcal U}_1}  Z_{i_3}^{\m_1}
\rp^{p-1}  }   {\mE_{{\mathcal U}_2}\lp\lp
\sum_{i_3=1}^{l}
\lp
 \mE_{{\mathcal U}_1}  Z_{i_3}^{\m_1}
\rp^p
\rp^{\frac{\m_2}{\m_1}}\rp}
\nonumber \\
& &
\times
\frac{(C_{i_3}^{(i_1)})^{s-1} A_{i_3}^{(i_1,i_2)}\y_j^{(i_2)}}{Z_{i_3}^{1-\m_1}}
 \nonumber \\
 & &
 \times
\Bigg . \Bigg (
  \sum_{p_3=1}^{l}
  \frac{
\lp
 \mE_{{\mathcal U}_1}  Z_{p_3}^{\m_1}
\rp^{p-1}  }  {   \lp
\sum_{i_3=1}^{l}
\lp
 \mE_{{\mathcal U}_1}  Z_{i_3}^{\m_1}
\rp^p
\rp^{2 - \frac{\m_2}{\m_1}  }
 }
\mE_{{\mathcal U}_1}
  \frac{1}{Z_{p_3}^{1-\m_1}}  \sum_{p_1=1}^{l}  (C_{p_3}^{(p_1)})^{s-1}\sum_{p_2=1}^{l}
\beta_{p_1}A_{p_3}^{(p_1,p_2)}\y_j^{(p_2)}\sqrt{1-t}
\Bigg . \Bigg )
\Bigg . \Bigg )
\nonumber \\
 & = & -s\sqrt{1-t}(\p_1-\p_2)(\m_1-\m_2)\mE_{G,{\mathcal U}_3}\Bigg( \Bigg. \mE_{{\mathcal U}_2}
\frac{ \lp  \sum_{i_3=1}^{l}
\lp
 \mE_{{\mathcal U}_1}  Z_{i_3}^{\m_1}
\rp^p
    \rp^{\frac{\m_2}{\m_1}}  }{\mE_{{\mathcal U}_2}\lp  \lp  \sum_{i_3=1}^{l}
\lp
 \mE_{{\mathcal U}_1}  Z_{i_3}^{\m_1}
\rp^p
    \rp^{\frac{\m_2}{\m_1}}   \rp}
\nonumber \\
& &
\times
  \sum_{i_3=1}^{l}
  \frac{
\lp
 \mE_{{\mathcal U}_1}  Z_{i_3}^{\m_1}
\rp^{p}  }  {   \lp
\sum_{i_3=1}^{l}
\lp
 \mE_{{\mathcal U}_1}  Z_{i_3}^{\m_1}
\rp^p
\rp
 }
\mE_{{\mathcal U}_1}\frac{Z_{i_3}^{\m_1}}{\mE_{{\mathcal U}_1} Z_{i_3}^{\m_1}}
 \frac{(C_{i_3}^{(i_1)})^{s}}{Z_{i_3} }  \frac{A_{i_3}^{(i_1,i_2)}}{C_{i_3}^{(i_1)}}\y_j^{(i_2)} \nonumber \\
& & \times
  \sum_{p_3=1}^{l}
  \frac{
\lp
 \mE_{{\mathcal U}_1}  Z_{p_3}^{\m_1}
\rp^{p}  }  {   \lp
\sum_{p_3=1}^{l}
\lp
 \mE_{{\mathcal U}_1}  Z_{p_3}^{\m_1}
\rp^p
\rp
 }
\lp \mE_{{\mathcal U}_1}  \frac{Z_{p_3}^{\m_1}}{\mE_{{\mathcal U}_1} Z_{p_3}^{\m_1}} \sum_{p_1=1}^{l}  \frac{(C_{p_3}^{(p_1)})^s}{Z_{p_3}}\sum_{p_2=1}^{l}
\frac{A^{(p_1,p_2)}}{C_{p_3}^{(p_1)}}\beta_{p_1}\y_j^{(p_2)}  \rp\Bigg. \Bigg).
\end{eqnarray}
We then also find
\begin{eqnarray}\label{eq:lev2genDanal24}
 \sum_{i_1=1}^{l}  \sum_{i_2=1}^{l} \sum_{j=1}^{m}  \beta_{i_1}\frac{T_{2,1,j}^d}{\sqrt{1-t}}
 & = & -s\beta^2(\p_1-\p_2)(\m_1-\m_2)p
 \mE_{G,{\mathcal U}_3} \langle \|\x^{(i_1)}\|_2\|\x^{(p_1)}\|_2(\y^{(p_2)})^T\y^{(i_2)} \rangle_{\gamma_{21}^{(2)}}.
\end{eqnarray}

Switching the focus to $T_{2,1,j}^e$ we note the following
\begin{eqnarray}\label{eq:lev2genDanal20bb0}
 \frac{d}{d\u_j^{(2,2)}}\lp
 \lp
 \mE_{{\mathcal U}_1}  Z_{i_3}^{\m_1}
\rp^{p-1}
 \rp
=
(p-1)  \lp
 \mE_{{\mathcal U}_1}  Z_{i_3}^{\m_1}
\rp^{p-2}
\mE_{{\mathcal U}_1}\frac{d Z_{i_3}^{\m_1}}{d\u_j^{(2,2)}}.\nonumber \\
\end{eqnarray}
Relying again on (\ref{eq:genDanal21}) we first write
\begin{eqnarray}\label{eq:lev2genDanal23bb1}
T_{2,1,j}^d &  = & (\p_1-\p_2) \mE_{G,{\mathcal U}_3,{\mathcal U}_2,{\mathcal U}_1}
\Bigg ( \Bigg .
\sum_{i_3=1}^{l}
\frac{
 \lp
\sum_{i_3=1}^{l}
\lp
 \mE_{{\mathcal U}_1}  Z_{i_3}^{\m_1}
\rp^p
\rp^{\frac{\m_2}{\m_1}  -1 }  }   {\mE_{{\mathcal U}_2}\lp\lp
\sum_{i_3=1}^{l}
\lp
 \mE_{{\mathcal U}_1}  Z_{i_3}^{\m_1}
\rp^p
\rp^{\frac{\m_2}{\m_1}}\rp  }
\frac{(C_{i_3}^{(i_1)})^{s-1} A_{i_3}^{(i_1,i_2)}\y_j^{(i_2)}}{Z_{i_3}^{1-\m_1}}
 \nonumber \\
 & &
 \times
\Bigg . \Bigg (
(p-1)  \lp
 \mE_{{\mathcal U}_1}  Z_{i_3}^{\m_1}
\rp^{p-2}
\mE_{{\mathcal U}_1}\frac{d Z_{i_3}^{\m_1}}{d\u_j^{(2,2)}}
 \Bigg . \Bigg )
\Bigg . \Bigg )
\nonumber \\
&  = & s\sqrt{1-t}(\p_1-\p_2) \m_1 (p-1)  \mE_{G,{\mathcal U}_3,{\mathcal U}_2,{\mathcal U}_1}
\Bigg ( \Bigg .
\sum_{i_3=1}^{l}
\frac{
 \lp
\sum_{i_3=1}^{l}
\lp
 \mE_{{\mathcal U}_1}  Z_{i_3}^{\m_1}
\rp^p
\rp^{\frac{\m_2}{\m_1}  -1 }  }   {\mE_{{\mathcal U}_2}\lp\lp
\sum_{i_3=1}^{l}
\lp
 \mE_{{\mathcal U}_1}  Z_{i_3}^{\m_1}
\rp^p
\rp^{\frac{\m_2}{\m_1}}\rp  }
\nonumber \\
& &
\times
\frac{(C_{i_3}^{(i_1)})^{s-1} A_{i_3}^{(i_1,i_2)}\y_j^{(i_2)}}{Z_{i_3}^{1-\m_1}}
 \nonumber \\
 & &
 \times
 \Bigg ( \Bigg .
\lp
 \mE_{{\mathcal U}_1}  Z_{i_3}^{\m_1}
\rp^{p-2}
\mE_{{\mathcal U}_1}
  \frac{1}{Z_{i_3}^{1-\m_1}}  \sum_{p_1=1}^{l}  (C_{i_3}^{(p_1)})^{s-1}\sum_{p_2=1}^{l}
\beta_{p_1}A_{i_3}^{(p_1,p_2)}\y_j^{(p_2)}\sqrt{1-t}
\Bigg . \Bigg )
\Bigg . \Bigg )
\nonumber \\
 & = & s\sqrt{1-t}(\p_1-\p_2) \m_1 (p-1)  \mE_{G,{\mathcal U}_3}\Bigg( \Bigg. \mE_{{\mathcal U}_2}
\frac{ \lp  \sum_{i_3=1}^{l}
\lp
 \mE_{{\mathcal U}_1}  Z_{i_3}^{\m_1}
\rp^p
    \rp^{\frac{\m_2}{\m_1}}  }{\mE_{{\mathcal U}_2}\lp  \lp  \sum_{i_3=1}^{l}
\lp
 \mE_{{\mathcal U}_1}  Z_{i_3}^{\m_1}
\rp^p
    \rp^{\frac{\m_2}{\m_1}}   \rp}
\nonumber \\
& &
\times
  \sum_{i_3=1}^{l}
  \frac{
\lp
 \mE_{{\mathcal U}_1}  Z_{i_3}^{\m_1}
\rp^{p}  }  {   \lp
\sum_{i_3=1}^{l}
\lp
 \mE_{{\mathcal U}_1}  Z_{i_3}^{\m_1}
\rp^p
\rp
 }
\mE_{{\mathcal U}_1}\frac{Z_{i_3}^{\m_1}}{\mE_{{\mathcal U}_1} Z_{i_3}^{\m_1}}
 \frac{(C_{i_3}^{(i_1)})^{s}}{Z_{i_3}}  \frac{A_{i_3}^{(i_1,i_2)}}{C_{i_3}^{(i_1)}}\y_j^{(i_2)} \nonumber \\
& & \times
\lp \mE_{{\mathcal U}_1}  \frac{Z_{i_3}^{\m_1}}{\mE_{{\mathcal U}_1} Z_{i_3}^{\m_1}} \sum_{p_1=1}^{l}  \frac{(C_{i_3}^{(p_1)})^s}{Z_{i_3}}\sum_{p_2=1}^{l}
\frac{A_{i_3}^{(p_1,p_2)}}{C_{i_3}^{(p_1)}}\beta_{p_1}\y_j^{(p_2)}  \rp\Bigg. \Bigg),
\end{eqnarray}
and then find
\begin{eqnarray}\label{eq:lev2genDanal24bb2}
 \sum_{i_1=1}^{l}  \sum_{i_2=1}^{l} \sum_{j=1}^{m}  \beta_{i_1}\frac{T_{2,1,j}^e}{\sqrt{1-t}}
 & = & s\beta^2(\p_1-\p_2)\m_1 (p-1)
 \mE_{G,{\mathcal U}_3} \langle \|\x^{(i_1)}\|_2\|\x^{(p_1)}\|_2(\y^{(p_2)})^T\y^{(i_2)} \rangle_{\gamma_{22}^{(2)}}.
\end{eqnarray}
Finally we then also have
\begin{eqnarray}\label{eq:lev2genDanal25}
 \sum_{i_1=1}^{l}  \sum_{i_2=1}^{l} \sum_{j=1}^{m}  \beta_{i_1}\frac{T_{2,1,j}}{\sqrt{1-t}}
 & = & (\p_1-\p_2)\beta^2 \nonumber \\
 & & \times
  \lp \mE_{G,{\mathcal U}_3}\langle \|\x^{(i_1)}\|_2^2\|\y^{(i_2)}\|_2^2\rangle_{\gamma_{01}^{(2)}} +   (s-1)\mE_{G,{\mathcal U}_3}\langle \|\x^{(i_1)}\|_2^2(\y^{(p_2)})^T\y^{(i_2)}\rangle_{\gamma_{02}^{(2)}} \rp\nonumber \\
& & - (\p_1-\p_2)s\beta^2(1-\m_1)\mE_{G,{\mathcal U}_3}\langle \|\x^{(i_1)}\|_2\|\x^{(p_1)}\|_2(\y^{(p_2)})^T\y^{(i_2)} \rangle_{\gamma_{1}^{(2)}}
\nonumber \\
 &   &
  -s\beta^2(\p_1-\p_2)(\m_1-\m_2) p \mE_{G,{\mathcal U}_3} \langle \|\x^{(i_1)}\|_2\|\x^{(p_1)}\|_2(\y^{(p_2)})^T\y^{(i_2)} \rangle_{\gamma_{21}^{(2)}}
\nonumber \\
   &   &
+  s\beta^2(\p_1-\p_2) \m_1 (p-1) \mE_{G,{\mathcal U}_3} \langle \|\x^{(i_1)}\|_2\|\x^{(p_1)}\|_2(\y^{(p_2)})^T\y^{(i_2)} \rangle_{\gamma_{22}^{(2)}}.
\end{eqnarray}

\underline{\textbf{\emph{Determining}} $T_{2,2}$}
\label{sec:lev2hand1T22}

Gaussian integration by parts gives
\begin{eqnarray}\label{eq:lev2liftgenEanal20}
T_{2,2} & = &   \mE_{G,{\mathcal U}_3,{\mathcal U}_2} \lp
\sum_{i_3=1}^{l}
\frac{\lp
\sum_{i_3=1}^{l}
\lp
 \mE_{{\mathcal U}_1}  Z_{i_3}^{\m_1}
\rp^p
\rp^{\frac{\m_2}{\m_1}-1} \lp
 \mE_{{\mathcal U}_1}  Z_{i_3}^{\m_1}
\rp^{p-1}  }   {\mE_{{\mathcal U}_2}\lp\lp
\sum_{i_3=1}^{l}
\lp
 \mE_{{\mathcal U}_1}  Z_{i_3}^{\m_1}
\rp^p
\rp^{\frac{\m_2}{\m_1}}\rp}
  \mE_{{\mathcal U}_1}\frac{(C_{i_3}^{(i_1)})^{s-1} A_{i_3}^{(i_1,i_2)} \u^{(i_1,3,2)}}{Z_{i_3}^{1-\m_1}} \rp
  \nonumber \\
  & = & \mE_{G,{\mathcal U}_3,{\mathcal U}_1} \Bigg( \Bigg.
  \sum_{i_3=1}^{l}
\frac{  1  }   {\mE_{{\mathcal U}_2}\lp\lp
\sum_{i_3=1}^{l}
\lp
 \mE_{{\mathcal U}_1}  Z_{i_3}^{\m_1}
\rp^p
\rp^{\frac{\m_2}{\m_1}}\rp}
  \nonumber \\
 & & \times
\mE_{{\mathcal U}_2} \Bigg. \Bigg(
 \sum_{p_1=1}^{l}\mE_{{\mathcal U}_2}(\u^{(i_1,3,2)}\u^{(p_1,3,2)}) \frac{d}{d\u^{(p_1,3,2)}}
 \Bigg. \Bigg(
 \frac{(C_{i_3}^{(i_1)})^{s-1} A_{i_3}^{(i_1,i_2)}  \lp
 \mE_{{\mathcal U}_1}  Z_{i_3}^{\m_1}
\rp^{p-1}  }  { Z^{1-\m_1}  \lp
\sum_{i_3=1}^{l}
\lp
 \mE_{{\mathcal U}_1}  Z_{i_3}^{\m_1}
\rp^p
\rp^{1 - \frac{\m_2}{\m_1}}  }
\Bigg. \Bigg)
\Bigg. \Bigg)
\Bigg. \Bigg) \nonumber \\
& = &   T_{2,2}^{c} +  T_{2,2}^{d}+  T_{2,2}^{e},
 \end{eqnarray}
where
\begin{eqnarray}\label{eq:lev2genEanal19b}
T_{2,2}^c &  = &
\mE_{G,{\mathcal U}_2,{\mathcal U}_1} \Bigg ( \Bigg .
 \sum_{i_3=1}^{l}
\frac{\lp
\sum_{i_3=1}^{l}
\lp
 \mE_{{\mathcal U}_1}  Z_{i_3}^{\m_1}
\rp^p
\rp^{\frac{\m_2}{\m_1}-1} \lp
 \mE_{{\mathcal U}_1}  Z_{i_3}^{\m_1}
\rp^{p-1}  }   {\mE_{{\mathcal U}_2}\lp\lp
\sum_{i_3=1}^{l}
\lp
 \mE_{{\mathcal U}_1}  Z_{i_3}^{\m_1}
\rp^p
\rp^{\frac{\m_2}{\m_1}}\rp}
 \nonumber \\
 & &
 \times
 \sum_{p_1=1}^{l}\mE_{{\mathcal U}_2}(\u^{(i_1,3,2)}\u^{(p_1,3,2)}) \frac{d}{d\u^{(p_1,3,2)}}\lp\frac{(C_{i_3}^{(i_1)})^{s-1} A_{i_3}^{(i_1,i_2)}}{Z_{i_3}^{1-\m_1}}
 \rp
 \Bigg ) \Bigg .
 \nonumber \\
T_{2,2}^d  &  = &
\mE_{G,{\mathcal U}_2,{\mathcal U}_1} \Bigg ( \Bigg .
 \sum_{i_3=1}^{l}
\frac{ \lp
 \mE_{{\mathcal U}_1}  Z_{i_3}^{\m_1}
\rp^{p-1}  }   {\mE_{{\mathcal U}_2}\lp\lp
\sum_{i_3=1}^{l}
\lp
 \mE_{{\mathcal U}_1}  Z_{i_3}^{\m_1}
\rp^p
\rp^{\frac{\m_2}{\m_1}}\rp}
 \frac{(C_{i_3}^{(i_1)})^{s-1} A_{i_3}^{(i_1,i_2)}}{Z_{i_3}^{1-\m_1}}
 \nonumber \\
 & &
 \times
 \sum_{p_1=1}^{l}\mE_{{\mathcal U}_2}(\u^{(i_1,3,2)}\u^{(p_1,3,2)})
 \frac{d}{d\u^{(p_1,3,2)}}
 \Bigg ( \Bigg .
 \lp
\sum_{i_3=1}^{l}
\lp
 \mE_{{\mathcal U}_1}  Z_{i_3}^{\m_1}
\rp^p
\rp^{\frac{\m_2}{\m_1}-1}
 \Bigg ) \Bigg .
 \Bigg ) \Bigg .
 \nonumber \\
 T_{2,2}^e  &  = &
\mE_{G,{\mathcal U}_2,{\mathcal U}_1} \Bigg ( \Bigg .
 \sum_{i_3=1}^{l}
\frac{
 \lp
\sum_{i_3=1}^{l}
\lp
 \mE_{{\mathcal U}_1}  Z_{i_3}^{\m_1}
\rp^p
\rp^{\frac{\m_2}{\m_1}-1}
 }   {\mE_{{\mathcal U}_2}\lp\lp
\sum_{i_3=1}^{l}
\lp
 \mE_{{\mathcal U}_1}  Z_{i_3}^{\m_1}
\rp^p
\rp^{\frac{\m_2}{\m_1}}\rp}
 \frac{(C_{i_3}^{(i_1)})^{s-1} A_{i_3}^{(i_1,i_2)}}{Z_{i_3}^{1-\m_1}}
 \nonumber \\
 & &
 \times
 \sum_{p_1=1}^{l}\mE_{{\mathcal U}_2}(\u^{(i_1,3,2)}\u^{(p_1,3,2)})
 \frac{d}{d\u^{(p_1,3,2)}}
 \lp
\lp
 \mE_{{\mathcal U}_1}  Z_{i_3}^{\m_1}
\rp^{p-1}
 \rp
 \Bigg . \Bigg ).
\end{eqnarray}
Since
\begin{eqnarray}\label{eq:lev2genEanal19c}
\mE_{{\mathcal U}_2}(\u^{(i_1,3,2)}\u^{(p_1,3,2)}) & = & (\q_1-\q_2)\frac{(\x^{(i_1)})^T\x^{(p_1)}}{\|\x^{(i_1)}\|_2\|\x^{(p_1)}\|_2}
=(\q_1-\q_2)\frac{\beta^2(\x^{(i_1)})^T\x^{(p_1)}}{\beta_{i_1}\beta_{p_1}}, \nonumber \\
\mE_{{\mathcal U}_2}(\u^{(i_1,3,1)}\u^{(p_1,3,1)}) & = & (\q_0-\q_1)\frac{(\x^{(i_1)})^T\x^{(p_1)}}{\|\x^{(i_1)}\|_2\|\x^{(p_1)}\|_2}=(\q_0-\q_1)\frac{\beta^2(\x^{(i_1)})^T\x^{(p_1)}}{\beta_{i_1}\beta_{p_1}},
\end{eqnarray}
we have that $T_{2,2}^c$ scaled by $(\q_1-\q_2)$  and the term considered in the second part of Section \ref{sec:lev2hand1T11} scaled by $(\q_0-\q_1)$ are structurally identical. One can then utilize  (\ref{eq:lev2liftgenBanal20b}) to obtain
\begin{eqnarray}\label{eq:lev2genEanal19c1}
\sum_{i_1=1}^{l}\sum_{i_2=1}^{l} \beta_{i_1}\|\y^{(i_2)}\|_2 \frac{T_{2,2}^c}{\sqrt{1-t}} & = & (\q_1-\q_2)\beta^2\Bigg(\Bigg. \mE_{G,{\mathcal U}_3}\langle \|\x^{(i_1)}\|_2^2\|\y^{(i_2)}\|_2^2\rangle_{\gamma_{01}^{(2)}} \nonumber \\
& & +   (s-1)\mE_{G,{\mathcal U}_3}\langle \|\x^{(i_1)}\|_2^2 \|\y^{(i_2)}\|_2\|\y^{(p_2)}\|_2\rangle_{\gamma_{02}^{(2)}}\Bigg.\Bigg) \nonumber \\
& & - (\q_1-\q_2)s\beta^2(1-\m_1)\mE_{G,{\mathcal U}_3}\langle (\x^{(p_1)})^T\x^{(i_1)}\|\y^{(i_2)}\|_2\|\y^{(p_2)}\|_2 \rangle_{\gamma_{1}^{(2)}}.
\nonumber \\
\end{eqnarray}

To determine  $T_{2,2}^d$, we start with
\begin{eqnarray}\label{eq:lev2genEanal20}
 \frac{d}{d\u^{(p_1,3,2)}}\lp
 \lp
\sum_{i_3=1}^{l}
\lp
 \mE_{{\mathcal U}_1}  Z_{i_3}^{\m_1}
\rp^p
\rp^{\frac{\m_2}{\m_1} -1 }
 \rp
=
-p \sum_{p_3=1}^{l}
  \frac{\lp 1-\frac{\m_2}{\m_1} \rp
\lp
 \mE_{{\mathcal U}_1}  Z_{p_3}^{\m_1}
\rp^{p-1}  }{   \lp
\sum_{i_3=1}^{l}
\lp
 \mE_{{\mathcal U}_1}  Z_{i_3}^{\m_1}
\rp^p
\rp^{2 - \frac{\m_2}{\m_1}  }
 }
\mE_{{\mathcal U}_1}\frac{d Z_{p_3}^{\m_1}}{d\u^{(p_1,3,2)}}.
\end{eqnarray}
From (\ref{eq:genEanal21}) we also have
\begin{eqnarray}\label{eq:lev2genEanal21}
\frac{dZ_{p_3}^{\m_1}}{d\u^{(p_1,3,2)}}  & = & \frac{\m_1}{Z_{p_3}^{1-\m_1}} s  (C_{p_3}^{(p_1)})^{s-1}\sum_{p_2=1}^{l}
\beta_{p_1}A_{p_3}^{(p_1,p_2)}\|\y^{(p_2)}\|_2\sqrt{1-t}.
\end{eqnarray}
A combination of (\ref{eq:lev2genEanal20}) and (\ref{eq:lev2genEanal21}) then gives
\begin{eqnarray}\label{eq:lev2genEanal22}
 \frac{d}{d\u^{(p_1,3,2)}}\lp
 \lp
\sum_{i_3=1}^{l}
\lp
 \mE_{{\mathcal U}_1}  Z_{i_3}^{\m_1}
\rp^p
\rp^{\frac{\m_2}{\m_1} -1 }
 \rp
& = &
-p \sum_{p_3=1}^{l}
  \frac{\lp 1-\frac{\m_2}{\m_1} \rp
\lp
 \mE_{{\mathcal U}_1}  Z_{p_3}^{\m_1}
\rp^{p-1}  }{   \lp
\sum_{i_3=1}^{l}
\lp
 \mE_{{\mathcal U}_1}  Z_{i_3}^{\m_1}
\rp^p
\rp^{2 - \frac{\m_2}{\m_1}  }
 }
 \nonumber \\
 & &
 \times
\mE_{{\mathcal U}_1}
  \frac{\m_1}{Z_{p_3}^{1-\m_1}} s  (C_{p_3}^{(p_1)})^{s-1}\sum_{p_2=1}^{l}
\beta_{p_1}A_{p_3}^{(p_1,p_2)}\|\y^{(p_2)}\|_2\sqrt{1-t}.\nonumber \\
\end{eqnarray}
Combining  further (\ref{eq:lev2genEanal19b}) and (\ref{eq:lev2genEanal22}) we obtain  the following analogue to (\ref{eq:genEanal23})
\begin{eqnarray}\label{eq:lev2genEanal23}
 T_{2,2}^d
  &  = &
\mE_{G,{\mathcal U}_2,{\mathcal U}_1} \Bigg ( \Bigg .
 \sum_{i_3=1}^{l}
\frac{ \lp
 \mE_{{\mathcal U}_1}  Z_{i_3}^{\m_1}
\rp^{p-1}  }   {\mE_{{\mathcal U}_2}\lp\lp
\sum_{i_3=1}^{l}
\lp
 \mE_{{\mathcal U}_1}  Z_{i_3}^{\m_1}
\rp^p
\rp^{\frac{\m_2}{\m_1}}\rp}
 \frac{(C_{i_3}^{(i_1)})^{s-1} A_{i_3}^{(i_1,i_2)}}{Z_{i_3}^{1-\m_1}}
 \nonumber \\
 & &
 \times
 \sum_{p_1=1}^{l}\mE_{{\mathcal U}_2}(\u^{(i_1,3,2)}\u^{(p_1,3,2)})
  \Bigg ( \Bigg .
-p \sum_{p_3=1}^{l}
  \frac{\lp 1-\frac{\m_2}{\m_1} \rp
\lp
 \mE_{{\mathcal U}_1}  Z_{p_3}^{\m_1}
\rp^{p-1}  }{   \lp
\sum_{i_3=1}^{l}
\lp
 \mE_{{\mathcal U}_1}  Z_{i_3}^{\m_1}
\rp^p
\rp^{2 - \frac{\m_2}{\m_1}  }
 }
 \nonumber \\
 & &
 \times
\mE_{{\mathcal U}_1}
  \frac{\m_1}{Z_{p_3}^{1-\m_1}} s  (C_{p_3}^{(p_1)})^{s-1}\sum_{p_2=1}^{l}
\beta_{p_1}A_{p_3}^{(p_1,p_2)}\|\y^{(p_2)}\|_2\sqrt{1-t}
 \Bigg ) \Bigg .
 \Bigg ) \Bigg .
 \nonumber \\
  &  = & -s\sqrt{1-t}\beta^2(\q_1-\q_2)(\m_1-\m_2)\mE_{G,{\mathcal U}_3} \Bigg( \Bigg.
  \mE_{{\mathcal U}_2}
  \frac{   \lp \sum_{i_3=1}^{l} \lp \mE_{{\mathcal U}_1} Z_{i_3}^{\m_1}\rp^p      \rp^{\frac{\m_2}{\m_1}}   }
  {  \mE_{{\mathcal U}_2}   \lp     \lp \sum_{i_3=1}^{l} \lp \mE_{{\mathcal U}_1} Z_{i_3}^{\m_1}\rp^p      \rp^{\frac{\m_2}{\m_1}}      \rp   }
  \nonumber \\
  & &
  \times
\sum_{i_3=1}^{l}
\frac{  \lp \mE_{{\mathcal U}_1} Z_{i_3}^{\m_1}\rp^p      } { \lp \sum_{i_3=1}^{l} \lp \mE_{{\mathcal U}_1} Z_{i_3}^{\m_1}\rp^p      \rp }
    \mE_{{\mathcal U}_1}\frac{Z_{i_3}^{\m_1}}{\mE_{{\mathcal U}_1} Z_{i_3}^{\m_1}}
 \frac{(C_{i_3}^{(i_1)})^{s-1} A_{i_3}^{(i_1,i_2)}}{Z_{i_3}} \nonumber \\
 & & \times
\sum_{p_3=1}^{l}
\frac{  \lp \mE_{{\mathcal U}_1} Z_{p_3}^{\m_1}\rp^p      } { \lp \sum_{i_3=1}^{l} \lp \mE_{{\mathcal U}_1} Z_{p_3}^{\m_1}\rp^p      \rp }
\mE_{{\mathcal U}_1}\frac{Z_{p_3}^{\m_1}}{\mE_{{\mathcal U}_1} Z_{p_3}^{\m_1}}
  \sum_{p_1=1}^{l}
 \frac{(C_{p_3}^{(p_1)})^{s}}{Z_{p_3}}  \sum_{p_2=1}^{l}
\frac{A_{p_3}^{(p_1,p_2)}}{(C_{p_3}^{(p_1)})}\|\y^{(p_2)}\|_2\frac{(\x^{(i_1)})^T\x^{(p_1)}}{\beta_{i_1}} \Bigg. \Bigg). .\nonumber \\
\end{eqnarray}
We then note 
\begin{equation}\label{eq:lev2genEanal24}
\sum_{i_1=1}^{l}\sum_{i_2=1}^{l} \beta_{i_1}\|\y^{(i_2)}\|_2\frac{T_{2,2}^d}{\sqrt{1-t}}
 =
-s\beta^2(\q_1-\q_2)(\m_1-\m_2) p \mE_{G,{\mathcal U}_3} \langle \|\y^{(i_2)}\|_2\|\y^{(p_2)}\|_2(\x^{(i_1)})^T\x^{(p_1)}\rangle_{\gamma_{21}^{(2)}}.
\end{equation}

To determine  $T_{2,2}^e$, we start with
\begin{eqnarray}\label{eq:lev2genEanal20bb0}
 \frac{d}{d\u^{(p_1,3,2)}}\lp
 \lp
 \mE_{{\mathcal U}_1}  Z_{i_3}^{\m_1}
\rp^{p-1}
 \rp
=
(p-1) \lp
 \mE_{{\mathcal U}_1}  Z_{i_3}^{\m_1}
\rp^{p-2}
\mE_{{\mathcal U}_1}\frac{d Z_{i_3}^{\m_1}}{d\u^{(p_1,3,2)}}.
\end{eqnarray}
 A combination of (\ref{eq:lev2genEanal20bb0}) and (\ref{eq:lev2genEanal21}) then gives
\begin{eqnarray}\label{eq:lev2genEanal22bb1}
 \frac{d}{d\u^{(p_1,3,2)}}\lp
 \lp
\sum_{i_3=1}^{l}
\lp
 \mE_{{\mathcal U}_1}  Z_{i_3}^{\m_1}
\rp^p
\rp^{\frac{\m_2}{\m_1} -1 }
 \rp
& = &
(p-1) \lp
 \mE_{{\mathcal U}_1}  Z_{i_3}^{\m_1}
\rp^{p-2}
 \nonumber \\
 & &
 \times
\mE_{{\mathcal U}_1}
  \frac{\m_1}{Z_{i_3}^{1-\m_1}} s  (C_{i_3}^{(p_1)})^{s-1}\sum_{p_2=1}^{l}
\beta_{p_1}A_{i_3}^{(p_1,p_2)}\|\y^{(p_2)}\|_2\sqrt{1-t}.\nonumber \\
\end{eqnarray}
Combining  further (\ref{eq:lev2genEanal19b}) and (\ref{eq:lev2genEanal22bb1}) we first arrive at  the following analogue to (\ref{eq:genEanal23})
\begin{eqnarray}\label{eq:lev2genEanal23bb2}
 T_{2,2}^e
  &  = &
\mE_{G,{\mathcal U}_2,{\mathcal U}_1} \Bigg ( \Bigg .
\mE_{{\mathcal U}_2} \sum_{i_3=1}^{l}
\frac{ \lp\lp
\sum_{i_3=1}^{l}
\lp
 \mE_{{\mathcal U}_1}  Z_{i_3}^{\m_1}
\rp^p
\rp^{\frac{\m_2}{\m_1} -1 }\rp  }   {\mE_{{\mathcal U}_2}\lp\lp
\sum_{i_3=1}^{l}
\lp
 \mE_{{\mathcal U}_1}  Z_{i_3}^{\m_1}
\rp^p
\rp^{\frac{\m_2}{\m_1}}\rp}
 \frac{(C_{i_3}^{(i_1)})^{s-1} A_{i_3}^{(i_1,i_2)}}{Z_{i_3}^{1-\m_1}}
 \nonumber \\
 & &
 \times
 \sum_{p_1=1}^{l}\mE_{{\mathcal U}_2}(\u^{(i_1,3,2)}\u^{(p_1,3,2)})
  \Bigg ( \Bigg .
(p-1) \lp
 \mE_{{\mathcal U}_1}  Z_{i_3}^{\m_1}
\rp^{p-2}
 \nonumber \\
 & &
 \times
\mE_{{\mathcal U}_1}
  \frac{\m_1}{Z_{i_3}^{1-\m_1}} s  (C_{i_3}^{(p_1)})^{s-1}\sum_{p_2=1}^{l}
\beta_{p_1}A_{i_3}^{(p_1,p_2)}\|\y^{(p_2)}\|_2\sqrt{1-t}
 \Bigg ) \Bigg .
 \Bigg ) \Bigg .
 \nonumber \\
  &  = & s\sqrt{1-t}\beta^2(\q_1-\q_2) \m_1  (p-1)  \mE_{G,{\mathcal U}_3} \Bigg( \Bigg.
  \mE_{{\mathcal U}_2}
  \frac{   \lp \sum_{i_3=1}^{l} \lp \mE_{{\mathcal U}_1} Z_{i_3}^{\m_1}\rp^p      \rp^{\frac{\m_2}{\m_1}}   }
  {  \mE_{{\mathcal U}_2}   \lp     \lp \sum_{i_3=1}^{l} \lp \mE_{{\mathcal U}_1} Z_{i_3}^{\m_1}\rp^p      \rp^{\frac{\m_2}{\m_1}}      \rp   }
  \nonumber \\
  & &
  \times
\sum_{i_3=1}^{l}
\frac{  \lp \mE_{{\mathcal U}_1} Z_{i_3}^{\m_1}\rp^p      } { \lp \sum_{i_3=1}^{l} \lp \mE_{{\mathcal U}_1} Z_{i_3}^{\m_1}\rp^p      \rp }
    \mE_{{\mathcal U}_1}\frac{Z_{i_3}^{\m_1}}{\mE_{{\mathcal U}_1} Z_{i_3}^{\m_1}}
 \frac{(C_{i_3}^{(i_1)})^{s-1} A_{i_3}^{(i_1,i_2)}}{Z_{i_3}} \nonumber \\
 & & \times
\mE_{{\mathcal U}_1}\frac{Z_{i_3}^{\m_1}}{\mE_{{\mathcal U}_1} Z_{i_3}^{\m_1}}
  \sum_{p_1=1}^{l}
 \frac{(C_{i_3}^{(p_1)})^{s}}{Z_{i_3}}  \sum_{p_2=1}^{l}
\frac{A_{i_3}^{(p_1,p_2)}}{(C_{i_3}^{(p_1)})}\|\y^{(p_2)}\|_2\frac{(\x^{(i_1)})^T\x^{(p_1)}}{\beta_{i_1}} \Bigg. \Bigg),\nonumber \\
\end{eqnarray}
and then note  
\begin{equation}\label{eq:lev2genEanal24bb3}
\sum_{i_1=1}^{l}\sum_{i_2=1}^{l} \beta_{i_1}\|\y^{(i_2)}\|_2\frac{T_{2,2}^e}{\sqrt{1-t}}
 =
s\beta^2(\q_1-\q_2) \m_1 ( p -1 ) \mE_{G,{\mathcal U}_3} \langle \|\y^{(i_2)}\|_2\|\y^{(p_2)}\|_2(\x^{(i_1)})^T\x^{(p_1)}\rangle_{\gamma_{21}^{(22}}.
\end{equation}
A combination of (\ref{eq:lev2liftgenEanal20}), (\ref{eq:lev2genEanal19c1}),
(\ref{eq:lev2genEanal24}), and (\ref{eq:lev2genEanal24bb3})  gives
\begin{eqnarray}\label{eq:lev2genEanal25}
\sum_{i_1=1}^{l}\sum_{i_2=1}^{l} \beta_{i_1}\|\y^{(i_2)}\|_2\frac{T_{2,2}}{\sqrt{1-t}}
& = & (\q_1-\q_2) \beta^2 \Bigg( \Bigg. \mE_{G,{\mathcal U}_3}\langle \|\x^{(i_1)}\|_2^2\|\y^{(i_2)}\|_2^2\rangle_{\gamma_{01}^{(2)}} \nonumber \\
& & +   (s-1)\mE_{G,{\mathcal U}_3}\langle \|\x^{(i_1)}\|_2^2 \|\y^{(i_2)}\|_2\|\y^{(p_2)}\|_2\rangle_{\gamma_{02}^{(2)}}\Bigg.\Bigg) \nonumber \\
& & - (\q_1-\q_2)s\beta^2(1-\m_1)\mE_{G,{\mathcal U}_3}\langle (\x^{(p_1)})^T\x^{(i_1)}\|\y^{(i_2)}\|_2\|\y^{(p_2)}\|_2 \rangle_{\gamma_{1}^{(2)}} \nonumber \\
&  & -s\beta^2(\q_1-\q_2)(\m_1-\m_2) p  \mE_{G,{\mathcal U}_3} \langle \|\y^{(i_2)}\|_2\|\y^{(p_2)}\|_2(\x^{(i_1)})^T\x^{(p_1)}\rangle_{\gamma_{21}^{(2)}}
\nonumber \\
&  &  + s\beta^2(\q_1-\q_2) \m_1 (p-1) \mE_{G,{\mathcal U}_3} \langle \|\y^{(i_2)}\|_2\|\y^{(p_2)}\|_2(\x^{(i_1)})^T\x^{(p_1)}\rangle_{\gamma_{22}^{(2)}}. \nonumber \\
\end{eqnarray}

\underline{\textbf{\emph{Determining}} $T_{2,3}$}
\label{sec:lev2hand1T23}

After Gaussian integration by parts we obtain
\begin{eqnarray}\label{eq:lev2genFanal21}
T_{2,3} & = &   \mE_{G,{\mathcal U}_3,{\mathcal U}_2} \lp
\sum_{i_3=1}^{l}
\frac{\lp
\sum_{i_3=1}^{l}
\lp
 \mE_{{\mathcal U}_1}  Z_{i_3}^{\m_1}
\rp^p
\rp^{\frac{\m_2}{\m_1}-1} \lp
 \mE_{{\mathcal U}_1}  Z_{i_3}^{\m_1}
\rp^{p-1}  }   {\mE_{{\mathcal U}_2}\lp\lp
\sum_{i_3=1}^{l}
\lp
 \mE_{{\mathcal U}_1}  Z_{i_3}^{\m_1}
\rp^p
\rp^{\frac{\m_2}{\m_1}}\rp}
\mE_{{\mathcal U}_1}\frac{(C_{i_3}^{(i_1)})^{s-1} A_{i_3}^{(i_1,i_2)} u^{(4,2)}}{Z_{i_3}^{1-\m_1}} \rp
\nonumber  \\
& = & \mE_{G,{\mathcal U}_3,{\mathcal U}_1} \Bigg( \Bigg.
\sum_{i_3=1}^{l}
\frac{ 1 }   {\mE_{{\mathcal U}_2}\lp\lp
\sum_{i_3=1}^{l}
\lp
 \mE_{{\mathcal U}_1}  Z_{i_3}^{\m_1}
\rp^p
\rp^{\frac{\m_2}{\m_1}}\rp}
\nonumber \\
& & \times \mE_{{\mathcal U}_2} \lp\mE_{{\mathcal U}_2} (u^{(4,2)}u^{(4,2)})\lp\frac{d}{du^{(4,2)}} \lp\frac{(C_{i_3}^{(i_1)})^{s-1} A_{i_3}^{(i_1,i_2)}}{Z_{i_3}^{1-\m_1}  \lp
\sum_{i_3=1}^{l}
\lp
 \mE_{{\mathcal U}_1}  Z_{i_3}^{\m_1}
\rp^p
\rp^{1-\frac{\m_2}{\m_1}} \lp
 \mE_{{\mathcal U}_1}  Z_{i_3}^{\m_1}
\rp^{1-p}
 }
\rp \rp\rp\Bigg. \Bigg) \nonumber \\
& = & (\p_1\q_1-\p_2\q_2)\mE_{G,{\mathcal U}_3,{\mathcal U}_2,{\mathcal U}_1}  \Bigg ( \Bigg.
\sum_{i_3=1}^{l}
\frac{\lp
\sum_{i_3=1}^{l}
\lp
 \mE_{{\mathcal U}_1}  Z_{i_3}^{\m_1}
\rp^p
\rp^{\frac{\m_2}{\m_1}-1} \lp
 \mE_{{\mathcal U}_1}  Z_{i_3}^{\m_1}
\rp^{p-1}  }   {\mE_{{\mathcal U}_2}\lp\lp
\sum_{i_3=1}^{l}
\lp
 \mE_{{\mathcal U}_1}  Z_{i_3}^{\m_1}
\rp^p
\rp^{\frac{\m_2}{\m_1}}\rp}
\nonumber \\
& &
\times
\lp
\frac{d}{du^{(4,2)}} \lp\frac{(C_{i_3}^{(i_1)})^{s-1} A_{i_3}^{(i_1,i_2)}}{Z_{i_3}^{1-\m_1}}
\rp
\rp
 \Bigg ) \Bigg.
\nonumber \\
& & + (\p_1\q_1-\p_2\q_2)\mE_{G,{\mathcal U}_3,{\mathcal U}_2,{\mathcal U}_1}  \Bigg ( \Bigg.
\sum_{i_3=1}^{l}
\frac{Z_{i_3}^{\m_1-1}(C_{i_3}^{(i_1)})^{s-1} A_{i_3}^{(i_1,i_2)}  \lp
 \mE_{{\mathcal U}_1}  Z_{i_3}^{\m_1}
\rp^{p-1}  }
{  \mE_{{\mathcal U}_2}\lp\lp
\sum_{i_3=1}^{l}
\lp
 \mE_{{\mathcal U}_1}  Z_{i_3}^{\m_1}
\rp^p
\rp^{\frac{\m_2}{\m_1}}\rp   }
\nonumber \\
& &
\times
 \Bigg ( \Bigg.
 \frac{d}{du^{(4,2)}}
  \Bigg ( \Bigg.
\lp
\sum_{i_3=1}^{l}
\lp
 \mE_{{\mathcal U}_1}  Z_{i_3}^{\m_1}
\rp^p
\rp^{\frac{\m_2}{\m_1} -1 }
 \Bigg ) \Bigg.
 \Bigg ) \Bigg.
 \Bigg ) \Bigg.
\nonumber \\
& & + (\p_1\q_1-\p_2\q_2)\mE_{G,{\mathcal U}_3,{\mathcal U}_2,{\mathcal U}_1} \Bigg ( \Bigg.
\sum_{i_3=1}^{l}
\frac{  \lp
\sum_{i_3=1}^{l}
\lp
 \mE_{{\mathcal U}_1}  Z_{i_3}^{\m_1}
\rp^p
\rp^{\frac{\m_2}{\m_1}  -1 }
 Z_{i_3}^{\m_1-1}(C_{i_3}^{(i_1)})^{s-1} A_{i_3}^{(i_1,i_2)}  }
{  \mE_{{\mathcal U}_2}\lp\lp
\sum_{i_3=1}^{l}
\lp
 \mE_{{\mathcal U}_1}  Z_{i_3}^{\m_1}
\rp^p
\rp^{\frac{\m_2}{\m_1}}\rp   }
\nonumber \\
& &
\times
\lp\frac{d}{du^{(4,2)}} \lp
 \lp
 \mE_{{\mathcal U}_1}  Z_{i_3}^{\m_1}
\rp^{p-1}
\rp
\rp
\Bigg .\Bigg ).\nonumber \\
\end{eqnarray}
As usual, it convenient to rewrite the above as
\begin{eqnarray}\label{eq:lev2genFanal22}
T_{2,3} & = & T_{2,3}^c+T_{2,3}^d+T_{2,3}^e,
\end{eqnarray}
where
\begin{eqnarray}\label{eq:lev2genFanal23}
T_{2,3}^c & = & (\p_1\q_1-\p_2\q_2)\mE_{G,{\mathcal U}_3,{\mathcal U}_2,{\mathcal U}_1}  \Bigg ( \Bigg.
\sum_{i_3=1}^{l}
\frac{\lp
\sum_{i_3=1}^{l}
\lp
 \mE_{{\mathcal U}_1}  Z_{i_3}^{\m_1}
\rp^p
\rp^{\frac{\m_2}{\m_1}-1} \lp
 \mE_{{\mathcal U}_1}  Z_{i_3}^{\m_1}
\rp^{p-1}  }   {\mE_{{\mathcal U}_2}\lp\lp
\sum_{i_3=1}^{l}
\lp
 \mE_{{\mathcal U}_1}  Z_{i_3}^{\m_1}
\rp^p
\rp^{\frac{\m_2}{\m_1}}\rp}
\nonumber \\
& &
\times
\lp
\frac{d}{du^{(4,2)}} \lp\frac{(C_{i_3}^{(i_1)})^{s-1} A_{i_3}^{(i_1,i_2)}}{Z_{i_3}^{1-\m_1}}
\rp
\rp
 \Bigg ) \Bigg.
  \nonumber \\
 T_{2,3}^d & = & (\p_1\q_1-\p_2\q_2)\mE_{G,{\mathcal U}_3,{\mathcal U}_2,{\mathcal U}_1}  \Bigg ( \Bigg.
\sum_{i_3=1}^{l}
\frac{Z_{i_3}^{\m_1-1}(C_{i_3}^{(i_1)})^{s-1} A_{i_3}^{(i_1,i_2)}  \lp
 \mE_{{\mathcal U}_1}  Z_{i_3}^{\m_1}
\rp^{p-1}  }
{  \mE_{{\mathcal U}_2}\lp\lp
\sum_{i_3=1}^{l}
\lp
 \mE_{{\mathcal U}_1}  Z_{i_3}^{\m_1}
\rp^p
\rp^{\frac{\m_2}{\m_1}}\rp   }
\nonumber \\
& &
\times
 \Bigg ( \Bigg.
 \frac{d}{du^{(4,2)}}
  \Bigg ( \Bigg.
\lp
\sum_{i_3=1}^{l}
\lp
 \mE_{{\mathcal U}_1}  Z_{i_3}^{\m_1}
\rp^p
\rp^{\frac{\m_2}{\m_1} -1 }
 \Bigg ) \Bigg.
 \Bigg ) \Bigg.
 \Bigg ) \Bigg.
 \nonumber \\
 T_{2,3}^e & = &
 (\p_1\q_1-\p_2\q_2)\mE_{G,{\mathcal U}_3,{\mathcal U}_2,{\mathcal U}_1} \Bigg ( \Bigg.
\sum_{i_3=1}^{l}
\frac{  \lp
\sum_{i_3=1}^{l}
\lp
 \mE_{{\mathcal U}_1}  Z_{i_3}^{\m_1}
\rp^p
\rp^{\frac{\m_2}{\m_1} -1 }
 Z_{i_3}^{\m_1-1}(C_{i_3}^{(i_1)})^{s-1} A_{i_3}^{(i_1,i_2)}  }
{  \mE_{{\mathcal U}_2}\lp\lp
\sum_{i_3=1}^{l}
\lp
 \mE_{{\mathcal U}_1}  Z_{i_3}^{\m_1}
\rp^p
\rp^{\frac{\m_2}{\m_1}  }\rp   }
\nonumber \\
& &
\times
\lp\frac{d}{du^{(4,2)}} \lp
 \lp
 \mE_{{\mathcal U}_1}  Z_{i_3}^{\m_1}
\rp^{p-1}
\rp
\rp
\Bigg .\Bigg ). \nonumber \\
 \end{eqnarray}
Observing $\frac{T_{2,3}^c}{\p_1\q_1-\p_2\q_2}=\frac{T_{1,3}}{\p_0\q_0-\p_1\q_1}$ together with (\ref{eq:lev2liftgenCanal21b}) gives the following analogue to (\ref{eq:genFanal23b})
\begin{eqnarray}\label{eq:lev2genFanal23b}
\sum_{i_1=1}^{l}\sum_{i_2=1}^{l} \beta_{i_1}\|\y^{(i_2)}\|_2 \frac{T_{2,3}^c}{\sqrt{t}} & = &
(\p_1\q_1-\p_2\q_2)\beta^2 \Bigg(\Bigg. \mE_{G,{\mathcal U}_3}\langle \|\x^{(i_1)}\|_2^2\|\y^{(i_2)}\|_2^2\rangle_{\gamma_{01}^{(2)}} \nonumber \\
& & +   (s-1)\mE_{G,{\mathcal U}_3}\langle \|\x^{(i_1)}\|_2^2 \|\y^{(i_2)}\|_2\|\y^{(p_2)}\|_2\rangle_{\gamma_{02}^{(2)}} \Bigg.\Bigg) \nonumber \\
& & - (\p_1\q_1-\p_2\q_2) s\beta^2(1-\m_1)\mE_{G,{\mathcal U}_3}\langle \|\x^{(i_1)}\|_2\|\x^{(p_`)}\|_2\|\y^{(i_2)}\|_2\|\y^{(p_2)}\|_2 \rangle_{\gamma_{1}^{(2)}}. \nonumber \\
\end{eqnarray}

To find $T_{2,3}^d$, we first write
\begin{eqnarray}\label{eq:lev2genFanal24}
 \frac{d}{du^{(4,2)}}\lp
 \lp
\sum_{i_3=1}^{l}
\lp
 \mE_{{\mathcal U}_1}  Z_{i_3}^{\m_1}
\rp^p
\rp^{\frac{\m_2}{\m_1} -1 }
 \rp
=
-p \sum_{p_3=1}^{l}
  \frac{\lp 1-\frac{\m_2}{\m_1} \rp
\lp
 \mE_{{\mathcal U}_1}  Z_{p_3}^{\m_1}
\rp^{p-1}  }{   \lp
\sum_{i_3=1}^{l}
\lp
 \mE_{{\mathcal U}_1}  Z_{i_3}^{\m_1}
\rp^p
\rp^{2 - \frac{\m_2}{\m_1}  }
 }
\mE_{{\mathcal U}_1}   \frac{\m_1}{Z_{p_3}^{1-\m_1}}       \frac{d Z_{p_3} }{du^{(4,2)}}.
\nonumber \\
 \end{eqnarray}
 Recalling on (\ref{eq:genFanal25}) gives
\begin{equation}\label{eq:lev2genFanal25}
\frac{dZ_{p_3}}{du^{(4,2)}}=\frac{d\sum_{p_1=1}^{l}  (C_{p_3}^{(p_1)})^s}{du^{(4,2)}}
=s \sum_{p_1=1}^{l} (C_{p_3}^{(p_1)})^{s-1}\sum_{p_2=1}^{l}\frac{d(A_{p_3}^{(p_1,p_2)})}{du^{(4,2)}}=s \sum_{p_1=1}^{l} (C_{p_3}^{(p_1)})^{s-1}\sum_{p_2=1}^{l}
\beta_{p_1}A_{p_3}^{(p_1,p_2)}\|\y^{(p_2)}\|_2\sqrt{t}.
\end{equation}
Connecting (\ref{eq:lev2genFanal23}), (\ref{eq:lev2genFanal24}) and (\ref{eq:lev2genFanal25}), we first find
\begin{eqnarray}\label{eq:lev2genFanal27}
T_{2,3}^d   & = & -(\p_1\q_1-\p_2\q_2)\mE_{G,{\mathcal U}_3,{\mathcal U}_2,{\mathcal U}_1} \Bigg( \Bigg.
\sum_{i_3=1}^{l}
\frac{Z_{i_3}^{\m_1-1}(C_{i_3}^{(i_1)})^{s-1} A_{i_3}^{(i_1,i_2)}  \lp
 \mE_{{\mathcal U}_1}  Z_{i_3}^{\m_1}
\rp^{p-1}  }
{  \mE_{{\mathcal U}_2}\lp\lp
\sum_{i_3=1}^{l}
\lp
 \mE_{{\mathcal U}_1}  Z_{i_3}^{\m_1}
\rp^p
\rp^{\frac{\m_2}{\m_1}}\rp   }
\nonumber \\
& & \times
p
\sum_{p_3=1}^{l}
  \frac{\lp 1-\frac{\m_2}{\m_1} \rp
\lp
 \mE_{{\mathcal U}_1}  Z_{p_3}^{\m_1}
\rp^{p-1}  }{   \lp
\sum_{i_3=1}^{l}
\lp
 \mE_{{\mathcal U}_1}  Z_{i_3}^{\m_1}
\rp^p
\rp^{2 - \frac{\m_2}{\m_1}  }
 }
\mE_{{\mathcal U}_1} \frac{\m_1}{Z_{p_3}^{1-\m_1}}s \sum_{p_1=1}^{l} (C_{p_3}^{(p_1)})^{s-1}\sum_{p_2=1}^{l}
\beta_{p_1}A_{p_3}^{(p_1,p_2)}\|\y^{(p_2)}\|_2\sqrt{t}\Bigg. \Bigg) \nonumber \\
& = & -s\sqrt{t}(\p_1\q_1-\p_2\q_2)(\m_1-\m_2) p \mE_{G,{\mathcal U}_3} \Bigg( \Bigg.
\sum_{i_3=1}^{l}
\frac{  \lp
\sum_{i_3=1}^{l}
\lp
 \mE_{{\mathcal U}_1}  Z_{i_3}^{\m_1}
\rp^p
\rp^{\frac{\m_2}{\m_1}}
  }
{  \mE_{{\mathcal U}_2}\lp\lp
\sum_{i_3=1}^{l}
\lp
 \mE_{{\mathcal U}_1}  Z_{i_3}^{\m_1}
\rp^p
\rp^{\frac{\m_2}{\m_1}}\rp   }
\nonumber \\
& &
\times
\sum_{i_3=1}^{l}
  \frac{
\lp
 \mE_{{\mathcal U}_1}  Z_{i_3}^{\m_1}
\rp^{p}  }{   \lp
\sum_{i_3=1}^{l}
\lp
 \mE_{{\mathcal U}_1}  Z_{i_3}^{\m_1}
\rp^p
\rp
 }
\mE_{{\mathcal U}_1} \frac{Z_{i_3}^{\m_1}}{\mE_{{\mathcal U}_1} Z_{i_3}^{\m_1}} \frac{(C_{i_3}^{(i_1)})^s}{Z_{i_3}}\frac{A_{i_3}^{(i_1,i_2)}}{(C_{i_3}^{(i_1)})}\nonumber \\
& & \times
\sum_{p_3=1}^{l}
  \frac{
\lp
 \mE_{{\mathcal U}_1}  Z_{p_3}^{\m_1}
\rp^{p}  }{   \lp
\sum_{p_3=1}^{l}
\lp
 \mE_{{\mathcal U}_1}  Z_{p_3}^{\m_1}
\rp^p
\rp
 }
\mE_{{\mathcal U}_1} \frac{Z_{p_3}^{\m_1}}{\mE_{{\mathcal U}_1} Z_{p_3}^{\m_1}} \sum_{p_1=1}^{l} \frac{(C_{p_3}^{(p_1)})^{s}}{Z_{p_3}}\sum_{p_2=1}^{l}
\frac{A_{p_3}^{(p_1,p_2)}}{(C^{(p_1)})}\beta_{p_1}\|\y^{(p_2)}\|_2\Bigg. \Bigg).
\end{eqnarray}
and then
\begin{eqnarray}\label{eq:lev2genFanal28}
\sum_{i_1=1}^{l}\sum_{i_2=1}^{l} \beta_{i_1}\|\y^{(i_2)}\|_2\frac{T_{2,3}^d}{\sqrt{t}}
 & = & -s\beta^2(\p_1\q_1-\p_2\q_2)(\m_1-\m_2) p \mE_{G,{\mathcal U}_3} \langle\|\x^{(i_2)}\|_2\|\x^{(p_2)}\|_2\|\y^{(i_2)}\|_2\|\y^{(p_2)}\rangle_{\gamma_{21}^{(2)}}.\nonumber \\
\end{eqnarray}

To determine $T_{2,3}^e$, we first write
\begin{eqnarray}\label{eq:lev2genFanal24bb0}
 \frac{d}{du^{(4,2)}}\lp
\lp
 \mE_{{\mathcal U}_1}  Z_{i_3}^{\m_1}
\rp^{p-1}
  \rp
=
(p-1) \lp
 \mE_{{\mathcal U}_1}  Z_{i_3}^{\m_1}
\rp^{p-2}
 \mE_{{\mathcal U}_1}   \frac{\m_1}{Z_{i_3}^{1-\m_1}}       \frac{d Z_{i_3} }{du^{(4,2)}}.
\nonumber \\
 \end{eqnarray}
 Connecting (\ref{eq:lev2genFanal23}), (\ref{eq:lev2genFanal24bb0}) and (\ref{eq:lev2genFanal25}), we find
\begin{eqnarray}\label{eq:lev2genFanal27bb2}
T_{2,3}^e   & = & -(\p_1\q_1-\p_2\q_2)\mE_{G,{\mathcal U}_3,{\mathcal U}_2,{\mathcal U}_1} \Bigg( \Bigg.
\sum_{i_3=1}^{l}
\frac{  \lp
\sum_{i_3=1}^{l}
\lp
 \mE_{{\mathcal U}_1}  Z_{i_3}^{\m_1}
\rp^p
\rp^{\frac{\m_2}{\m_1} -1 }
 Z_{i_3}^{\m_1-1}(C_{i_3}^{(i_1)})^{s-1} A_{i_3}^{(i_1,i_2)}  }
{  \mE_{{\mathcal U}_2}\lp\lp
\sum_{i_3=1}^{l}
\lp
 \mE_{{\mathcal U}_1}  Z_{i_3}^{\m_1}
\rp^p
\rp^{\frac{\m_2}{\m_1}}\rp   }
\nonumber \\
& & \times
(p-1) \lp
 \mE_{{\mathcal U}_1}  Z_{i_3}^{\m_1}
\rp^{p-2}
\mE_{{\mathcal U}_1} \frac{\m_1}{Z_{i_3}^{1-\m_1}}s \sum_{p_1=1}^{l} (C_{i_3}^{(p_1)})^{s-1}\sum_{p_2=1}^{l}
\beta_{p_1}A_{i_3}^{(p_1,p_2)}\|\y^{(p_2)}\|_2\sqrt{t}\Bigg. \Bigg) \nonumber \\
& = & s\sqrt{t}(\p_1\q_1-\p_2\q_2)\m_1 (p-1) \mE_{G,{\mathcal U}_3} \Bigg( \Bigg.
\sum_{i_3=1}^{l}
\frac{  \lp
\sum_{i_3=1}^{l}
\lp
 \mE_{{\mathcal U}_1}  Z_{i_3}^{\m_1}
\rp^p
\rp^{\frac{\m_2}{\m_1}}
  }
{  \mE_{{\mathcal U}_2}\lp\lp
\sum_{i_3=1}^{l}
\lp
 \mE_{{\mathcal U}_1}  Z_{i_3}^{\m_1}
\rp^p
\rp^{\frac{\m_2}{\m_1}}\rp   }
\nonumber \\
& &
\times
\sum_{i_3=1}^{l}
  \frac{
\lp
 \mE_{{\mathcal U}_1}  Z_{i_3}^{\m_1}
\rp^{p}  }{   \lp
\sum_{i_3=1}^{l}
\lp
 \mE_{{\mathcal U}_1}  Z_{i_3}^{\m_1}
\rp^p
\rp
 }
\mE_{{\mathcal U}_1} \frac{Z_{i_3}^{\m_1}}{\mE_{{\mathcal U}_1} Z_{i_3}^{\m_1}} \frac{(C_{i_3}^{(i_1)})^s}{Z_{i_3}}\frac{A_{i_3}^{(i_1,i_2)}}{(C_{i_3}^{(i_1)})}\nonumber \\
& & \times
\mE_{{\mathcal U}_1} \frac{Z_{i_3}^{\m_1}}{\mE_{{\mathcal U}_1} Z_{i_3}^{\m_1}} \sum_{p_1=1}^{l} \frac{(C_{i_3}^{(p_1)})^{s}}{Z_{i_3}}\sum_{p_2=1}^{l}
\frac{A_{i_3}^{(p_1,p_2)}}{(C_{i_3}^{(p_1)})}\beta_{p_1}\|\y^{(p_2)}\|_2\Bigg. \Bigg),
\end{eqnarray}
and 
\begin{eqnarray}\label{eq:lev2genFanal28bb3}
\sum_{i_1=1}^{l}\sum_{i_2=1}^{l} \beta_{i_1}\|\y^{(i_2)}\|_2\frac{T_{2,3}^e}{\sqrt{t}}
 & = & s\beta^2(\p_1\q_1-\p_2\q_2) \m_1 (p-1) \mE_{G,{\mathcal U}_3} \langle\|\x^{(i_2)}\|_2\|\x^{(p_2)}\|_2\|\y^{(i_2)}\|_2\|\y^{(p_2)}\rangle_{\gamma_{22}^{(2)}}.\nonumber \\
\end{eqnarray}
A combination of  (\ref{eq:lev2genFanal22}), (\ref{eq:lev2genFanal23b}), (\ref{eq:lev2genFanal28}), and (\ref{eq:lev2genFanal28bb3}) is sufficient to obtain the following analogue to (\ref{eq:genFanal29})
 \begin{eqnarray}\label{eq:lev2genFanal29}
\sum_{i_1=1}^{l}\sum_{i_2=1}^{l} \beta_{i_1}\|\y^{(i_2)}\|_2\frac{T_{2,3}}{\sqrt{t}}
& = &
(\p_1\q_1-\p_2\q_2)\beta^2 \Bigg( \Bigg. \mE_{G,{\mathcal U}_3}\langle \|\x^{(i_1)}\|_2^2\|\y^{(i_2)}\|_2^2\rangle_{\gamma_{01}^{(2)}} \nonumber \\
& & +  (s-1)\mE_{G,{\mathcal U}_3}\langle \|\x^{(i_1)}\|_2^2 \|\y^{(i_2)}\|_2\|\y^{(p_2)}\|_2\rangle_{\gamma_{02}^{(2)}}\Bigg.\Bigg) \nonumber \\
& & - (\p_1\q_1-\p_2\q_2) s\beta^2(1-\m_1)\mE_{G,{\mathcal U}_3}\langle \|\x^{(i_1)}\|_2\|\x^{(p_`)}\|_2\|\y^{(i_2)}\|_2\|\y^{(p_2)}\|_2 \rangle_{\gamma_{1}^{(2)}} \nonumber \\
&  & -s\beta^2(\p_1\q_1-\p_2\q_2)(\m_1-\m_2) p \mE_{G,{\mathcal U}_3} \langle\|\x^{(i_2)}\|_2\|\x^{(p_2)}\|_2\|\y^{(i_2)}\|_2\|\y^{(p_2)}\rangle_{\gamma_{21}^{(2)}}\nonumber \\
&  & +s\beta^2(\p_1\q_1-\p_2\q_2) \m_1 (p-1) \mE_{G,{\mathcal U}_3} \langle\|\x^{(i_2)}\|_2\|\x^{(p_2)}\|_2\|\y^{(i_2)}\|_2\|\y^{(p_2)}\rangle_{\gamma_{22}^{(2)}}.\nonumber \\
\end{eqnarray}

\subsubsection{$T_3$--group --- second level}
\label{sec:lev2x3lev2handlTk2}

As usual, each of the three terms from $T_3$ group is handled separately.

\underline{\textbf{\emph{Determining}} $T_{3,1,j}$}
\label{sec:lev2x3lev2hand1T21}

After Gaussian integration by parts we write
\begin{eqnarray}\label{eq:lev2x3lev2genDanal19}
T_{3,1,j} & = &   \mE_{G,{\mathcal U}_3,{\mathcal U}_2} \lp
\sum_{i_3=1}^{l}
\frac{\lp
\sum_{i_3=1}^{l}
\lp
 \mE_{{\mathcal U}_1}  Z_{i_3}^{\m_1}
\rp^p
\rp^{\frac{\m_2}{\m_1}-1} \lp
 \mE_{{\mathcal U}_1}  Z_{i_3}^{\m_1}
\rp^{p-1}  }   {\mE_{{\mathcal U}_2}\lp\lp
\sum_{i_3=1}^{l}
\lp
 \mE_{{\mathcal U}_1}  Z_{i_3}^{\m_1}
\rp^p
\rp^{\frac{\m_2}{\m_1}}\rp}
  \mE_{{\mathcal U}_1}\frac{(C_{i_3}^{(i_1)})^{s-1} A_{i_3}^{(i_1,i_2)} \y_j^{(i_2)}\u_j^{(2,3)}}{Z_{i_3}^{1-\m_1}} \rp \nonumber \\
& = &  \mE_{G,{\mathcal U}_2,{\mathcal U}_1} \lp \mE_{{\mathcal U}_3}
\sum_{i_3=1}^{l}
\frac{\lp
\sum_{i_3=1}^{l}
\lp
 \mE_{{\mathcal U}_1}  Z_{i_3}^{\m_1}
\rp^p
\rp^{\frac{\m_2}{\m_1}-1} \lp
 \mE_{{\mathcal U}_1}  Z_{i_3}^{\m_1}
\rp^{p-1}  }   {\mE_{{\mathcal U}_2}\lp\lp
\sum_{i_3=1}^{l}
\lp
 \mE_{{\mathcal U}_1}  Z_{i_3}^{\m_1}
\rp^p
\rp^{\frac{\m_2}{\m_1}}\rp }
 \frac{(C_{i_3}^{(i_1)})^{s-1} A_{i_3}^{(i_1,i_2)} \y_j^{(i_2)}\u_j^{(2,3)}}{Z_{i_3}^{1-\m_1}}
 \rp \nonumber \\
& = &  T_{3,1,j}^{c} +  T_{3,1,j}^{d},
 \end{eqnarray}
 where
\begin{eqnarray}\label{eq:lev2x3lev2genDanal19b}
T_{3,1,j}^c &  = &
\mE_{G,{\mathcal U}_3,{\mathcal U}_2,{\mathcal U}_1} \Bigg . \Bigg(
\sum_{i_3=1}^{l}
\frac{ 1 }   {\mE_{{\mathcal U}_2}\lp\lp
\sum_{i_3=1}^{l}
\lp
 \mE_{{\mathcal U}_1}  Z_{i_3}^{\m_1}
\rp^p
\rp^{\frac{\m_2}{\m_1}}\rp}
\nonumber \\
& &
\times
 \Bigg . \Bigg(
\mE_{{\mathcal U}_3} (\u_j^{(2,3)}\u_j^{(2,3)})\frac{d}{d\u_j^{(2,3)}}
 \Bigg . \Bigg(
\frac{(C_{i_3}^{(i_1)})^{s-1} A_{i_3}^{(i_1,i_2)}\y_j^{(i_2)}}{Z_{i_3}^{1-\m_1}  \lp
\sum_{i_3=1}^{l}
\lp
 \mE_{{\mathcal U}_1}  Z_{i_3}^{\m_1}
\rp^p
\rp^{ 1- \frac{\m_2}{\m_1}} \lp
 \mE_{{\mathcal U}_1}  Z_{i_3}^{\m_1}
\rp^{1-p}   }
\Bigg . \Bigg)
\Bigg . \Bigg)
\Bigg . \Bigg)
\nonumber \\
T_{3,1,j}^d &  = & \mE_{G,{\mathcal U}_3,{\mathcal U}_2,{\mathcal U}_1} \Bigg ( \Bigg.
\sum_{i_3=1}^{l}
 \frac{
\lp
\sum_{i_3=1}^{l}
\lp
 \mE_{{\mathcal U}_1}  Z_{i_3}^{\m_1}
\rp^p
\rp^{\frac{\m_2}{\m_1}-1} \lp
 \mE_{{\mathcal U}_1}  Z_{i_3}^{\m_1}
\rp^{p-1}
(C_{i_3}^{(i_1)})^{s-1} A_{i_3}^{(i_1,i_2)}\y_j^{(i_2)}}{Z_{i_3}^{1-\m_1}   }
\nonumber \\
& &
\times
  \Bigg .\Bigg ( \mE_{{\mathcal U}_2} (\u_j^{(2,3)}\u_j^{(2,3)})\frac{d}{d\u_j^{(2,3)}}
 \Bigg .\Bigg ( \frac{1}{   \mE_{{\mathcal U}_2}\lp\lp
\sum_{i_3=1}^{l}
\lp
 \mE_{{\mathcal U}_1}  Z_{i_3}^{\m_1}
\rp^p
\rp^{\frac{\m_2}{\m_1}}\rp  }
 \Bigg .\Bigg ) \Bigg .\Bigg ) \Bigg .\Bigg ).
 \nonumber \\
\end{eqnarray}
Keeping in mind  $\mE_{{\mathcal U}_3} (\u_j^{(2,3)}\u_j^{(2,3)})=\p_2$, we observe that $T_{3,1,j}^c$ scaled by $\p_2$ is structurally identical to $T_{2,1,j}^c$ scaled by $(\p_1-\p_2)$. This then immediately implies
\begin{eqnarray}\label{eq:lev2x3lev2genDanal19b1}
\sum_{i_1=1}^{l}\sum_{i_2=1}^{l}\sum_{j=1}^{m} \beta_{i_1}\frac{T_{3,1,j}^c}{\sqrt{1-t
}}
& = &  \sum_{i_1=1}^{l}\sum_{i_2=1}^{l}\sum_{j=1}^{m} \beta_{i_1}\frac{T_{2,1,j}}{\sqrt{1-t}}\frac{\p_2}{\p_1-\p_2}\nonumber\\
& = & \p_2 \beta^2 \lp \mE_{G,{\mathcal U}_3}\langle \|\x^{(i_1)}\|_2^2\|\y^{(i_2)}\|_2^2\rangle_{\gamma_{01}^{(2)}} +   (s-1)\mE_{G,{\mathcal U}_3}\langle \|\x^{(i_1)}\|_2^2(\y^{(p_2)})^T\y^{(i_2)}\rangle_{\gamma_{02}^{(2)}} \rp\nonumber \\
& & - \p_2s\beta^2(1-\m_1)\mE_{G,{\mathcal U}_3}\langle \|\x^{(i_1)}\|_2\|\x^{(p_1)}\|_2(\y^{(p_2)})^T\y^{(i_2)} \rangle_{\gamma_{1}^{(2)}}\nonumber \\
 &   &
  -s\beta^2\p_2(\m_1-\m_2) p \mE_{G,{\mathcal U}_3} \langle \|\x^{(i_1)}\|_2\|\x^{(p_1)}\|_2(\y^{(p_2)})^T\y^{(i_2)} \rangle_{\gamma_{21}^{(2)}}
  \nonumber \\
   &   &
  +s\beta^2\p_2 \m_1 (p-1) \mE_{G,{\mathcal U}_3} \langle \|\x^{(i_1)}\|_2\|\x^{(p_1)}\|_2(\y^{(p_2)})^T\y^{(i_2)} \rangle_{\gamma_{22}^{(2)}}.
\end{eqnarray}

We focus next on $T_{3,1,j}^d$ and write
\begin{eqnarray}\label{eq:lev2x3lev2genDanal20}
\frac{d}{d\u_j^{(2,3)}} \lp \frac{1}{  \mE_{{\mathcal U}_2}\lp\lp
\sum_{i_3=1}^{l}
\lp
 \mE_{{\mathcal U}_1}  Z_{i_3}^{\m_1}
\rp^p
\rp^{\frac{\m_2}{\m_1}}\rp    }   \rp
& = & -    \mE_{{\mathcal U}_2} \frac{ \frac{\m_2}{\m_1} \lp\lp
\sum_{i_3=1}^{l}
\lp
 \mE_{{\mathcal U}_1}  Z_{i_3}^{\m_1}
\rp^p
\rp^{\frac{\m_2}{\m_1} -1 }\rp    }{\lp    \mE_{{\mathcal U}_2}\lp\lp
\sum_{i_3=1}^{l}
\lp
 \mE_{{\mathcal U}_1}  Z_{i_3}^{\m_1}
\rp^p
\rp^{\frac{\m_2}{\m_1}}\rp      \rp^2}
\nonumber \\
& &
\times
p\sum_{p_3=1}^{l}
\lp
 \mE_{{\mathcal U}_1}  Z_{p_3}^{\m_1}
\rp^{p-1}
 \mE_{{\mathcal U}_1}\frac{d Z_{p_3}^{\m_1}}{d\u_j^{(2,3)}}.\nonumber \\
\end{eqnarray}
Recalling on (\ref{eq:genDanal21}) we first have
\begin{eqnarray}\label{eq:lev2x3lev2genDanal20a}
\frac{dZ_{p_3}^{\m_1}}{d\u_j^{(2,2)}}
& = & \frac{\m_1}{Z_{p_3}^{1-\m_1}}s  \sum_{p_1=1}^{l}  (C_{p_3}^{(p_1)})^{s-1}\sum_{p_2=1}^{l}
\beta_{p_1}A_{p_3}^{(p_1,p_2)}\y_j^{(p_2)}\sqrt{1-t}.
\end{eqnarray}
A combination of  (\ref{eq:lev2x3lev2genDanal19b}), (\ref{eq:lev2x3lev2genDanal20}), and (\ref{eq:lev2x3lev2genDanal20a}) then gives
\begin{eqnarray}\label{eq:lev2x3lev2genDanal21}
T_{3,1,j}^d &  = & -\p_2 \mE_{G,{\mathcal U}_3,{\mathcal U}_2,{\mathcal U}_1} \Bigg( \Bigg.
\sum_{i_3=1}^{l}
 \frac{
\lp
\sum_{i_3=1}^{l}
\lp
 \mE_{{\mathcal U}_1}  Z_{i_3}^{\m_1}
\rp^p
\rp^{\frac{\m_2}{\m_1}-1} \lp
 \mE_{{\mathcal U}_1}  Z_{i_3}^{\m_1}
\rp^{p-1}
(C_{i_3}^{(i_1)})^{s-1} A_{i_3}^{(i_1,i_2)}\y_j^{(i_2)}}{Z_{i_3}^{1-\m_1}   }
\nonumber \\
& &
\times
  \Bigg( \Bigg.
     \mE_{{\mathcal U}_2} \frac{ \frac{\m_2}{\m_1} \lp\lp
\sum_{i_3=1}^{l}
\lp
 \mE_{{\mathcal U}_1}  Z_{i_3}^{\m_1}
\rp^p
\rp^{\frac{\m_2}{\m_1} -1 }\rp    }{\lp    \mE_{{\mathcal U}_2}\lp\lp
\sum_{i_3=1}^{l}
\lp
 \mE_{{\mathcal U}_1}  Z_{i_3}^{\m_1}
\rp^p
\rp^{\frac{\m_2}{\m_1}}\rp      \rp^2}
  \nonumber \\
& & \times
p\sum_{p_3=1}^{l}
\lp
 \mE_{{\mathcal U}_1}  Z_{p_3}^{\m_1}
\rp^{p-1}
\mE_{{\mathcal U}_1} \frac{\m_1}{Z_{p_3}^{1-\m_1}}s  \sum_{p_1=1}^{l}  (C_{p_3}^{(p_1)})^{s-1}\sum_{p_2=1}^{l}
\beta_{p_1}A_{p_3}^{(p_1,p_2)}\y_j^{(p_2)}\sqrt{1-t}   \Bigg. \Bigg) \Bigg.\Bigg).\nonumber \\
&  = & -s\sqrt{1-t} \p_2 \m_2 p
 \nonumber \\
 & & \times
\mE_{G,{\mathcal U}_3,{\mathcal U}_2,{\mathcal U}_1} \Bigg( \Bigg.
\sum_{i_3=1}^{l}
 \frac{
\lp
\sum_{i_3=1}^{l}
\lp
 \mE_{{\mathcal U}_1}  Z_{i_3}^{\m_1}
\rp^p
\rp^{\frac{\m_2}{\m_1}-1} \lp
 \mE_{{\mathcal U}_1}  Z_{i_3}^{\m_1}
\rp^{p-1}
(C_{i_3}^{(i_1)})^{s-1} A_{i_3}^{(i_1,i_2)}\y_j^{(i_2)}}{  \lp    \mE_{{\mathcal U}_2}\lp\lp
\sum_{i_3=1}^{l}
\lp
 \mE_{{\mathcal U}_1}  Z_{i_3}^{\m_1}
\rp^p
\rp^{\frac{\m_2}{\m_1}}\rp      \rp   Z_{i_3}^{1-\m_1}   }
\nonumber \\
& &
\times
  \Bigg( \Bigg.
     \mE_{{\mathcal U}_2} \frac{ \lp\lp
\sum_{i_3=1}^{l}
\lp
 \mE_{{\mathcal U}_1}  Z_{i_3}^{\m_1}
\rp^p
\rp^{\frac{\m_2}{\m_1}  }\rp    }{ \lp    \mE_{{\mathcal U}_2}\lp\lp
\sum_{i_3=1}^{l}
\lp
 \mE_{{\mathcal U}_1}  Z_{i_3}^{\m_1}
\rp^p
\rp^{\frac{\m_2}{\m_1}}\rp      \rp}
  \nonumber \\
& & \times
\sum_{p_3=1}^{l}
\frac{
\lp
 \mE_{{\mathcal U}_1}  Z_{p_3}^{\m_1}
\rp^{p}   }{  \sum_{p_3=1}^{l}
\lp
 \mE_{{\mathcal U}_1}  Z_{p_3}^{\m_1}
\rp^p }
\mE_{{\mathcal U}_1} \frac{Z_{p_3}^{\m_1} }{\mE_{{\mathcal U}_1}  Z_{p_3}^{\m_1}   }   \sum_{p_1=1}^{l}  (C_{p_3}^{(p_1)})^{s-1}\sum_{p_2=1}^{l}
\beta_{p_1}A_{p_3}^{(p_1,p_2)}\y_j^{(p_2)}\sqrt{1-t}   \Bigg. \Bigg) \Bigg.\Bigg).\nonumber \\
 &  = & -s\sqrt{1-t}\p_2 \m_2 p
 \mE_{G,{\mathcal U}_3} \Bigg( \Bigg.
     \mE_{{\mathcal U}_2} \frac{ \lp\lp
\sum_{i_3=1}^{l}
\lp
 \mE_{{\mathcal U}_1}  Z_{i_3}^{\m_1}
\rp^p
\rp^{\frac{\m_2}{\m_1}  }\rp    }{ \lp    \mE_{{\mathcal U}_2}\lp\lp
\sum_{i_3=1}^{l}
\lp
 \mE_{{\mathcal U}_1}  Z_{i_3}^{\m_1}
\rp^p
\rp^{\frac{\m_2}{\m_1}}\rp      \rp}
\sum_{p_3=1}^{l}
\frac{
\lp
 \mE_{{\mathcal U}_1}  Z_{i_3}^{\m_1}
\rp^{p}   }{  \sum_{i_3=1}^{l}
\lp
 \mE_{{\mathcal U}_1}  Z_{i_3}^{\m_1}
\rp^p }
\nonumber \\
& &
\times
\mE_{{\mathcal U}_1} \frac{Z_{i_3}^{\m_1}}{\mE_{{\mathcal U}_1}  Z_{i_3}^{\m_1}  }  \frac{(C_{i_3}^{(i_1)})^{s}}{Z_{i_3}}
\frac{A_{i_3}^{(i_1,i_2)}}{C_{i_3}^{(i_1)}}\y_j^{(i_2)}
 \nonumber \\
& & \times
     \mE_{{\mathcal U}_2} \frac{ \lp\lp
\sum_{i_3=1}^{l}
\lp
 \mE_{{\mathcal U}_1}  Z_{i_3}^{\m_1}
\rp^p
\rp^{\frac{\m_2}{\m_1}  }\rp    }{ \lp    \mE_{{\mathcal U}_2}\lp\lp
\sum_{i_3=1}^{l}
\lp
 \mE_{{\mathcal U}_1}  Z_{i_3}^{\m_1}
\rp^p
\rp^{\frac{\m_2}{\m_1}}\rp      \rp}
\sum_{p_3=1}^{l}
\frac{
\lp
 \mE_{{\mathcal U}_1}  Z_{p_3}^{\m_1}
\rp^{p}   }{  \sum_{p_3=1}^{l}
\lp
 \mE_{{\mathcal U}_1}  Z_{p_3}^{\m_1}
\rp^p }
 \nonumber \\
& & \times
\mE_{{\mathcal U}_1} \frac{Z_{p_3}^{\m_1}}{\mE_{{\mathcal U}_1}  Z_{p_3}^{\m_1}  }  \sum_{p_1=1}^{l}  \frac{(C_{p_3}^{(p_1)})^{s}}{Z_{p_3}}\sum_{p_2=1}^{l}
\frac{A_{p_3}^{(p_1,p_2)}}{C_{p_3}^{(p_1)}}\beta_{p_1}\y_j^{(p_2)}   \Bigg. \Bigg).
\end{eqnarray}
One then observes
\begin{eqnarray}\label{eq:lev2x3lev2genDanal24}
 \sum_{i_1=1}^{l}  \sum_{i_2=1}^{l} \sum_{j=1}^{m}  \beta_{i_1}\frac{T_{3,1,j}^d}{\sqrt{1-t}}
 & = & -s\beta^2\p_2\m_2 p \mE_{G,{\mathcal U}_3} \langle \|\x^{(i_1)}\|_2\|\x^{(p_1)}\|_2(\y^{(p_2)})^T\y^{(i_2)} \rangle_{\gamma_{3}^{(2)}},
\end{eqnarray}
and
\begin{eqnarray}\label{eq:lev2x3lev2genDanal25}
 \sum_{i_1=1}^{l}  \sum_{i_2=1}^{l} \sum_{j=1}^{m}  \beta_{i_1}\frac{T_{3,1,j}}{\sqrt{1-t}}
 & = &  \sum_{i_1=1}^{l}\sum_{i_2=1}^{l}\sum_{j=1}^{m} \beta_{i_1}\frac{T_{2,1,j}}{\sqrt{1-t}}\frac{\p_2}{\p_1-\p_2}\nonumber\\
& = & \p_2 \beta^2 \lp \mE_{G,{\mathcal U}_3}\langle \|\x^{(i_1)}\|_2^2\|\y^{(i_2)}\|_2^2\rangle_{\gamma_{01}^{(2)}} +  (s-1)\mE_{G,{\mathcal U}_3}\langle \|\x^{(i_1)}\|_2^2(\y^{(p_2)})^T\y^{(i_2)}\rangle_{\gamma_{02}^{(2)}} \rp\nonumber \\
& & - \p_2s\beta^2(1-\m_1)\mE_{G,{\mathcal U}_3}\langle \|\x^{(i_1)}\|_2\|\x^{(p_1)}\|_2(\y^{(p_2)})^T\y^{(i_2)} \rangle_{\gamma_{1}^{(2)}}\nonumber \\
 &   &
  -s\beta^2\p_2(\m_1-\m_2) p \mE_{G,{\mathcal U}_3} \langle \|\x^{(i_1)}\|_2\|\x^{(p_1)}\|_2(\y^{(p_2)})^T\y^{(i_2)} \rangle_{\gamma_{21}^{(2)}}\nonumber \\
 &   &
  +s\beta^2\p_2 \m_1 (p-1) p\mE_{G,{\mathcal U}_3} \langle \|\x^{(i_1)}\|_2\|\x^{(p_1)}\|_2(\y^{(p_2)})^T\y^{(i_2)} \rangle_{\gamma_{22}^{(2)}}\nonumber \\
 &  & -s\beta^2\p_2\m_2 p \mE_{G,{\mathcal U}_3} \langle \|\x^{(i_1)}\|_2\|\x^{(p_1)}\|_2(\y^{(p_2)})^T\y^{(i_2)} \rangle_{\gamma_{3}^{(2)}}.
\end{eqnarray}

\underline{\textbf{\emph{Determining}} $T_{3,2}$}
\label{sec:lev2x3lev2hand1T22}

After Gaussian integration by parts we obtain
\begin{eqnarray}\label{eq:lev2x3lev2liftgenEanal20}
T_{3,2} & = &
\mE_{G,{\mathcal U}_3,{\mathcal U}_2} \lp
\sum_{i_3=1}^{l}
\frac{\lp
\sum_{i_3=1}^{l}
\lp
 \mE_{{\mathcal U}_1}  Z_{i_3}^{\m_1}
\rp^p
\rp^{\frac{\m_2}{\m_1}-1} \lp
 \mE_{{\mathcal U}_1}  Z_{i_3}^{\m_1}
\rp^{p-1}  }   {\mE_{{\mathcal U}_2}\lp\lp
\sum_{i_3=1}^{l}
\lp
 \mE_{{\mathcal U}_1}  Z_{i_3}^{\m_1}
\rp^p
\rp^{\frac{\m_2}{\m_1}}\rp}
  \mE_{{\mathcal U}_1}\frac{(C_{i_3}^{(i_1)})^{s-1} A_{i_3}^{(i_1,i_2)} \u^{(i_1,3,3)} }{Z_{i_3}^{1-\m_1}} \rp \nonumber \\
& = &  \mE_{G,{\mathcal U}_2,{\mathcal U}_1} \lp \mE_{{\mathcal U}_3}
\sum_{i_3=1}^{l}
\frac{\lp
\sum_{i_3=1}^{l}
\lp
 \mE_{{\mathcal U}_1}  Z_{i_3}^{\m_1}
\rp^p
\rp^{\frac{\m_2}{\m_1}-1} \lp
 \mE_{{\mathcal U}_1}  Z_{i_3}^{\m_1}
\rp^{p-1}  }   {\mE_{{\mathcal U}_2}\lp\lp
\sum_{i_3=1}^{l}
\lp
 \mE_{{\mathcal U}_1}  Z_{i_3}^{\m_1}
\rp^p
\rp^{\frac{\m_2}{\m_1}}\rp }
 \frac{(C_{i_3}^{(i_1)})^{s-1} A_{i_3}^{(i_1,i_2)}  \u^{(i_1,3,3)}}{Z_{i_3}^{1-\m_1}}
 \rp \nonumber \\
& = &  T_{3,2}^{c} +  T_{3,2}^{d},
 \end{eqnarray}
where
{\small \begin{eqnarray}\label{eq:lev2x3lev2genEanal19b}
T_{3,2}^c &  = &
\mE_{G,{\mathcal U}_3,{\mathcal U}_2,{\mathcal U}_1} \Bigg . \Bigg(
\sum_{i_3=1}^{l}
\frac{ 1 }   {\mE_{{\mathcal U}_2}\lp\lp
\sum_{i_3=1}^{l}
\lp
 \mE_{{\mathcal U}_1}  Z_{i_3}^{\m_1}
\rp^p
\rp^{\frac{\m_2}{\m_1}}\rp}
\nonumber \\
& &
\times
 \Bigg . \Bigg(
\sum_{p_1=1}^{l}\mE_{{\mathcal U}_3} (\u^{(i_1,3,3)}\u^{(p_1,3,3)})\frac{d}{d\u^{(p_1,3,3)}}
 \Bigg . \Bigg(
\frac{(C_{i_3}^{(i_1)})^{s-1} A_{i_3}^{(i_1,i_2)}  }{Z_{i_3}^{1-\m_1}  \lp
\sum_{i_3=1}^{l}
\lp
 \mE_{{\mathcal U}_1}  Z_{i_3}^{\m_1}
\rp^p
\rp^{ 1- \frac{\m_2}{\m_1}} \lp
 \mE_{{\mathcal U}_1}  Z_{i_3}^{\m_1}
\rp^{1-p}   }
\Bigg . \Bigg)
\Bigg . \Bigg)
\Bigg . \Bigg)
\nonumber \\
T_{3,2}^d &  = & \mE_{G,{\mathcal U}_3,{\mathcal U}_2,{\mathcal U}_1} \Bigg ( \Bigg.
\sum_{i_3=1}^{l}
 \frac{
\lp
\sum_{i_3=1}^{l}
\lp
 \mE_{{\mathcal U}_1}  Z_{i_3}^{\m_1}
\rp^p
\rp^{\frac{\m_2}{\m_1}-1} \lp
 \mE_{{\mathcal U}_1}  Z_{i_3}^{\m_1}
\rp^{p-1}
(C_{i_3}^{(i_1)})^{s-1} A_{i_3}^{(i_1,i_2)}\y_j^{(i_2)}}{Z_{i_3}^{1-\m_1}   }
\nonumber \\
& &
\times
  \Bigg .\Bigg (
  \sum_{p_1=1}^{l}
  \mE_{{\mathcal U}_2} (\u^{(i_1,3,3)}\u^{(p_1,3,3)})\frac{d}{d\u^{(p_1,3,3)}}
 \Bigg .\Bigg ( \frac{1}{   \mE_{{\mathcal U}_2}\lp\lp
\sum_{i_3=1}^{l}
\lp
 \mE_{{\mathcal U}_1}  Z_{i_3}^{\m_1}
\rp^p
\rp^{\frac{\m_2}{\m_1}}\rp  }
 \Bigg .\Bigg ) \Bigg .\Bigg ) \Bigg .\Bigg ).
  \end{eqnarray}}

\noindent Since
\begin{eqnarray}\label{eq:lev2x3lev2genEanal19c}
\mE_{{\mathcal U}_3}(\u^{(i_1,3,3)}\u^{(p_1,3,3)}) & = & \q_2\frac{(\x^{(i_1)})^T\x^{(p_1)}}{\|\x^{(i_1)}\|_2\|\x^{(p_1)}\|_2}
= \q_2\frac{\beta^2(\x^{(i_1)})^T\x^{(p_1)}}{\beta_{i_1}\beta_{p_1}},
\end{eqnarray}
we have  that $T_{3,2}^c$ scaled by $\q_2$ is structurally identical to $T_{2,2}$ scaled by $(\q_0-\q_1)$. This then implies
\begin{eqnarray}\label{eq:lev2x3lev2genEanal19c1}
\sum_{i_1=1}^{l}\sum_{i_2=1}^{l} \beta_{i_1}\|\y^{(i_2)}\|_2 \frac{T_{3,2}^c}{\sqrt{1-t}} & = &
\q_2 \beta^2
\Bigg.\Bigg( \mE_{G,{\mathcal U}_3}\langle \|\x^{(i_1)}\|_2^2\|\y^{(i_2)}\|_2^2\rangle_{\gamma_{01}^{(2)}}
 \nonumber \\
 & &
 + (s-1)\mE_{G,{\mathcal U}_3}\langle \|\x^{(i_1)}\|_2^2 \|\y^{(i_2)}\|_2\|\y^{(p_2)}\|_2\rangle_{\gamma_{02}^{(2)}}\Bigg.\Bigg)
\nonumber \\
& & - \q_2s\beta^2(1-\m_1)\mE_{G,{\mathcal U}_3}\langle (\x^{(p_1)})^T\x^{(i_1)}\|\y^{(i_2)}\|_2\|\y^{(p_2)}\|_2 \rangle_{\gamma_{1}^{(2)}} \nonumber \\
&  & -s\beta^2\q_2(\m_1-\m_2) p \mE_{G,{\mathcal U}_3} \langle \|\y^{(i_2)}\|_2\|\y^{(p_2)}\|_2(\x^{(i_1)})^T\x^{(p_1)}\rangle_{\gamma_{21}^{(2)}}
\nonumber \\
&  & +s\beta^2\q_2 \m_1 (p-1) \mE_{G,{\mathcal U}_3} \langle \|\y^{(i_2)}\|_2\|\y^{(p_2)}\|_2(\x^{(i_1)})^T\x^{(p_1)}\rangle_{\gamma_{22}^{(2)}}.
\end{eqnarray}

To determine  $T_{3,2}^d$, we start with
\begin{eqnarray}\label{eq:lev2x3lev2genEanal20}
\frac{d}{d\u^{(p_1,3,3)}} \lp \frac{1}{  \mE_{{\mathcal U}_2}\lp\lp
\sum_{i_3=1}^{l}
\lp
 \mE_{{\mathcal U}_1}  Z_{i_3}^{\m_1}
\rp^p
\rp^{\frac{\m_2}{\m_1}}\rp    }   \rp
& = & -    \mE_{{\mathcal U}_2} \frac{ \frac{\m_2}{\m_1} \lp\lp
\sum_{i_3=1}^{l}
\lp
 \mE_{{\mathcal U}_1}  Z_{i_3}^{\m_1}
\rp^p
\rp^{\frac{\m_2}{\m_1} -1 }\rp    }{\lp    \mE_{{\mathcal U}_2}\lp\lp
\sum_{i_3=1}^{l}
\lp
 \mE_{{\mathcal U}_1}  Z_{i_3}^{\m_1}
\rp^p
\rp^{\frac{\m_2}{\m_1}}\rp      \rp^2}
\nonumber \\
& &
\times
p\sum_{p_3=1}^{l}
\lp
 \mE_{{\mathcal U}_1}  Z_{p_3}^{\m_1}
\rp^{p-1}
 \mE_{{\mathcal U}_1}\frac{d Z_{p_3}^{\m_1}}{d\u^{(p_1,3,3)}}.\nonumber \\
 \end{eqnarray}
 We also note
\begin{eqnarray}\label{eq:lev2x3lev2genEanal21}
\frac{dZ_{p_3}^{\m_1}}{d\u^{(p_1,3,3)}}  & = & \frac{\m_1}{Z_{p_3}^{1-\m_1}} s  (C_{p_3}^{(p_1)})^{s-1}\sum_{p_2=1}^{l}
\beta_{p_1}A_{p_3}^{(p_1,p_2)}\|\y^{(p_2)}\|_2\sqrt{1-t}.
\end{eqnarray}
A combination of  (\ref{eq:lev2x3lev2genEanal19b}), (\ref{eq:lev2x3lev2genEanal20}), and (\ref{eq:lev2x3lev2genEanal21}) gives
{\small \begin{eqnarray}\label{eq:lev2x3lev2genEanal21b}
T_{3,2}^d &  = &
-\q_2 \mE_{G,{\mathcal U}_3,{\mathcal U}_2,{\mathcal U}_1} \Bigg ( \Bigg.
\sum_{i_3=1}^{l}
 \frac{
\lp
\sum_{i_3=1}^{l}
\lp
 \mE_{{\mathcal U}_1}  Z_{i_3}^{\m_1}
\rp^p
\rp^{\frac{\m_2}{\m_1}-1} \lp
 \mE_{{\mathcal U}_1}  Z_{i_3}^{\m_1}
\rp^{p-1}
(C_{i_3}^{(i_1)})^{s-1} A_{i_3}^{(i_1,i_2)}\y_j^{(i_2)}}{Z_{i_3}^{1-\m_1}   }
\nonumber \\
& &
\times
  \Bigg .\Bigg (
  \sum_{p_1=1}^{l}\frac{(\x^{(i_1)})^T\x^{(p_1)}}{\beta_{i_1}}
      \mE_{{\mathcal U}_2} \frac{ \frac{\m_2}{\m_1} \lp\lp
\sum_{i_3=1}^{l}
\lp
 \mE_{{\mathcal U}_1}  Z_{i_3}^{\m_1}
\rp^p
\rp^{\frac{\m_2}{\m_1} -1 }\rp    }{\lp    \mE_{{\mathcal U}_2}\lp\lp
\sum_{i_3=1}^{l}
\lp
 \mE_{{\mathcal U}_1}  Z_{i_3}^{\m_1}
\rp^p
\rp^{\frac{\m_2}{\m_1}}\rp      \rp^2}
 p\sum_{p_3=1}^{l}
\lp
 \mE_{{\mathcal U}_1}  Z_{p_3}^{\m_1}
\rp^{p-1}
\nonumber \\
& & \times
 \mE_{{\mathcal U}_1}
 \frac{\m_1}{Z_{p_3}^{1-\m_1}} s  (C_{p_3}^{(p_1)})^{s-1}\sum_{p_2=1}^{l}
A_{p_3}^{(p_1,p_2)}\|\y^{(p_2)}\|_2\sqrt{1-t}
       \Bigg .\Bigg ) \Bigg .\Bigg ).
\nonumber \\
 &  = & -s\sqrt{1-t}\q_2 \m_2 p \mE_{G,{\mathcal U}_3} \Bigg( \Bigg.
     \mE_{{\mathcal U}_2} \frac{ \lp\lp
\sum_{i_3=1}^{l}
\lp
 \mE_{{\mathcal U}_1}  Z_{i_3}^{\m_1}
\rp^p
\rp^{\frac{\m_2}{\m_1}  }\rp    }{ \lp    \mE_{{\mathcal U}_2}\lp\lp
\sum_{i_3=1}^{l}
\lp
 \mE_{{\mathcal U}_1}  Z_{i_3}^{\m_1}
\rp^p
\rp^{\frac{\m_2}{\m_1}}\rp      \rp}
\sum_{i_3=1}^{l}
\frac{
\lp
 \mE_{{\mathcal U}_1}  Z_{i_3}^{\m_1}
\rp^{p}   }{  \sum_{i_3=1}^{l}
\lp
 \mE_{{\mathcal U}_1}  Z_{i_3}^{\m_1}
\rp^p }
 \nonumber \\
 & & \times
\mE_{{\mathcal U}_1} \frac{Z_{i_3}^{\m_1}}{\mE_{{\mathcal U}_1} \lp Z_{i_3}^{\m_1}\rp}  \frac{(C_{i_3}^{(i_1)})^{s}}{Z_{i_3}}
\frac{A_{i_3}^{(i_1,i_2)}}{C_{i_3}^{(i_1)}}  \nonumber \\
& & \times
     \mE_{{\mathcal U}_2} \frac{ \lp\lp
\sum_{i_3=1}^{l}
\lp
 \mE_{{\mathcal U}_1}  Z_{i_3}^{\m_1}
\rp^p
\rp^{\frac{\m_2}{\m_1}  }\rp    }{ \lp    \mE_{{\mathcal U}_2}\lp\lp
\sum_{i_3=1}^{l}
\lp
 \mE_{{\mathcal U}_1}  Z_{i_3}^{\m_1}
\rp^p
\rp^{\frac{\m_2}{\m_1}}\rp      \rp}
\sum_{p_3=1}^{l}
\frac{
\lp
 \mE_{{\mathcal U}_1}  Z_{p_3}^{\m_1}
\rp^{p}   }{  \sum_{p_3=1}^{l}
\lp
 \mE_{{\mathcal U}_1}  Z_{p_3}^{\m_1}
\rp^p }
\nonumber \\
& & \times
\mE_{{\mathcal U}_1} \frac{Z_{p_3}^{\m_1}}{\mE_{{\mathcal U}_1} \lp Z_{p_3}^{\m_1}\rp}  \sum_{p_1=1}^{l}  \frac{(C_{p_3}^{(p_1)})^{s}}{Z_{p_3}}\sum_{p_2=1}^{l}
\frac{A_{p_3}^{(p_1,p_2)}}{C_{p_3}^{(p_1)}}\frac{(\x^{(i_1)})^T\x^{(p_1)}}{\beta_{i_1}}\|\y^{(p_2)}\|_2   \Bigg. \Bigg).
\end{eqnarray}}

\noindent We then first note
\begin{equation}\label{eq:lev2x3lev2genEanal24}
\sum_{i_1=1}^{l}\sum_{i_2=1}^{l} \beta_{i_1}\|\y^{(i_2)}\|_2\frac{T_{3,2}^d}{\sqrt{1-t}}
 =
-s\beta^2 \q_2 \m_2 p \mE_{G,{\mathcal U}_3} \langle \|\y^{(i_2)}\|_2\|\y^{(p_2)}\|_2(\x^{(i_1)})^T\x^{(p_1)}\rangle_{\gamma_{3}^{(2)}},
\end{equation}
and after a combination of (\ref{eq:lev2x3lev2liftgenEanal20}), (\ref{eq:lev2x3lev2genEanal19c1}), and
(\ref{eq:lev2x3lev2genEanal24}) obtain
\begin{eqnarray}\label{eq:lev2x3lev2genEanal25}
\sum_{i_1=1}^{l}\sum_{i_2=1}^{l} \beta_{i_1}\|\y^{(i_2)}\|_2\frac{T_{3,2}}{\sqrt{1-t}}
& = & \q_2\beta^2
\Bigg(\Bigg. \mE_{G,{\mathcal U}_3}\langle \|\x^{(i_1)}\|_2^2\|\y^{(i_2)}\|_2^2\rangle_{\gamma_{01}^{(2)}}  \nonumber \\
& & +   (s-1)\mE_{G,{\mathcal U}_3}\langle \|\x^{(i_1)}\|_2^2 \|\y^{(i_2)}\|_2\|\y^{(p_2)}\|_2\rangle_{\gamma_{02}^{(2)}}\Bigg.\Bigg)\nonumber \\
& & - \q_2s\beta^2(1-\m_1)\mE_{G,{\mathcal U}_3}\langle (\x^{(p_1)})^T\x^{(i_1)}\|\y^{(i_2)}\|_2\|\y^{(p_2)}\|_2 \rangle_{\gamma_{1}^{(2)}} \nonumber \\
&  & -s\beta^2\q_2(\m_1-\m_2) p \mE_{G,{\mathcal U}_3} \langle \|\y^{(i_2)}\|_2\|\y^{(p_2)}\|_2(\x^{(i_1)})^T\x^{(p_1)}\rangle_{\gamma_{21}^{(2)}} \nonumber \\
&  & +s\beta^2\q_2 \m_1 (p-1) \mE_{G,{\mathcal U}_3} \langle \|\y^{(i_2)}\|_2\|\y^{(p_2)}\|_2(\x^{(i_1)})^T\x^{(p_1)}\rangle_{\gamma_{22}^{(2)}} \nonumber \\
& & -s\beta^2 \q_2 \m_2 p \mE_{G,{\mathcal U}_3} \langle \|\y^{(i_2)}\|_2\|\y^{(p_2)}\|_2(\x^{(i_1)})^T\x^{(p_1)}\rangle_{\gamma_{3}^{(2)}}.
 \end{eqnarray}

\underline{\textbf{\emph{Determining}} $T_{3,3}$}
\label{sec:lev2x3lev2hand1T23}

Gaussian integration by parts also gives
\begin{eqnarray}\label{eq:lev2x3lev2genFanal21}
T_{3,3}& = &
\mE_{G,{\mathcal U}_3,{\mathcal U}_2} \lp
\sum_{i_3=1}^{l}
\frac{\lp
\sum_{i_3=1}^{l}
\lp
 \mE_{{\mathcal U}_1}  Z_{i_3}^{\m_1}
\rp^p
\rp^{\frac{\m_2}{\m_1}-1} \lp
 \mE_{{\mathcal U}_1}  Z_{i_3}^{\m_1}
\rp^{p-1}  }   {\mE_{{\mathcal U}_2}\lp\lp
\sum_{i_3=1}^{l}
\lp
 \mE_{{\mathcal U}_1}  Z_{i_3}^{\m_1}
\rp^p
\rp^{\frac{\m_2}{\m_1}}\rp}
  \mE_{{\mathcal U}_1}\frac{(C_{i_3}^{(i_1)})^{s-1} A_{i_3}^{(i_1,i_2)} u^{(4,3)} }{Z_{i_3}^{1-\m_1}} \rp \nonumber \\
& = &  \mE_{G,{\mathcal U}_2,{\mathcal U}_1} \lp \mE_{{\mathcal U}_3}
\sum_{i_3=1}^{l}
\frac{\lp
\sum_{i_3=1}^{l}
\lp
 \mE_{{\mathcal U}_1}  Z_{i_3}^{\m_1}
\rp^p
\rp^{\frac{\m_2}{\m_1}-1} \lp
 \mE_{{\mathcal U}_1}  Z_{i_3}^{\m_1}
\rp^{p-1}  }   {\mE_{{\mathcal U}_2}\lp\lp
\sum_{i_3=1}^{l}
\lp
 \mE_{{\mathcal U}_1}  Z_{i_3}^{\m_1}
\rp^p
\rp^{\frac{\m_2}{\m_1}}\rp }
 \frac{(C_{i_3}^{(i_1)})^{s-1} A_{i_3}^{(i_1,i_2)}  u^{(4,3)}}{Z_{i_3}^{1-\m_1}}
 \rp,
\end{eqnarray}
or in a more convenient form
\begin{eqnarray}\label{eq:lev2x3lev2genFanal22}
T_{3,3} & = & T_{3,3}^c+T_{3,3}^d,
\end{eqnarray}
where
\begin{eqnarray}\label{eq:lev2x3lev2genFanal23}
T_{3,2}^c &  = &
 \p_2\q_2 \mE_{G,{\mathcal U}_3,{\mathcal U}_2,{\mathcal U}_1} \Bigg . \Bigg(
\sum_{i_3=1}^{l}
\frac{ 1 }   {\mE_{{\mathcal U}_2}\lp\lp
\sum_{i_3=1}^{l}
\lp
 \mE_{{\mathcal U}_1}  Z_{i_3}^{\m_1}
\rp^p
\rp^{\frac{\m_2}{\m_1}}\rp}
\nonumber \\
& &
\times
 \Bigg . \Bigg(
\frac{d}{du^{(4,3)}}
 \Bigg . \Bigg(
\frac{(C_{i_3}^{(i_1)})^{s-1} A_{i_3}^{(i_1,i_2)}  }{Z_{i_3}^{1-\m_1}  \lp
\sum_{i_3=1}^{l}
\lp
 \mE_{{\mathcal U}_1}  Z_{i_3}^{\m_1}
\rp^p
\rp^{ 1- \frac{\m_2}{\m_1}} \lp
 \mE_{{\mathcal U}_1}  Z_{i_3}^{\m_1}
\rp^{1-p}   }
\Bigg . \Bigg)
\Bigg . \Bigg)
\Bigg . \Bigg)
\nonumber \\
T_{3,2}^d &  = &  \p_2\q_2 \mE_{G,{\mathcal U}_3,{\mathcal U}_2,{\mathcal U}_1} \Bigg ( \Bigg.
\sum_{i_3=1}^{l}
 \frac{
\lp
\sum_{i_3=1}^{l}
\lp
 \mE_{{\mathcal U}_1}  Z_{i_3}^{\m_1}
\rp^p
\rp^{\frac{\m_2}{\m_1}-1} \lp
 \mE_{{\mathcal U}_1}  Z_{i_3}^{\m_1}
\rp^{p-1}
(C_{i_3}^{(i_1)})^{s-1} A_{i_3}^{(i_1,i_2)}\y_j^{(i_2)}}{Z_{i_3}^{1-\m_1}   }
\nonumber \\
& &
\times
  \Bigg .\Bigg (
  \frac{d}{d u^{(4,3)}}
 \Bigg .\Bigg ( \frac{1}{   \mE_{{\mathcal U}_2}\lp\lp
\sum_{i_3=1}^{l}
\lp
 \mE_{{\mathcal U}_1}  Z_{i_3}^{\m_1}
\rp^p
\rp^{\frac{\m_2}{\m_1}}\rp  }
 \Bigg .\Bigg ) \Bigg .\Bigg ) \Bigg .\Bigg ).
 \end{eqnarray}
After  observing $\frac{T_{3,3}^c}{\p_2\q_2}=\frac{T_{2,3}}{\p_1\q_1-\p_2\q_2}$, we immediately have
\begin{eqnarray}\label{eq:lev2x3lev2genFanal23b}
\sum_{i_1=1}^{l}\sum_{i_2=1}^{l} \beta_{i_1}\|\y^{(i_2)}\|_2 \frac{T_{3,3}^c}{\sqrt{t}}
& = &
\p_2\q_2\beta^2
 \nonumber \\
 & & \times
 \lp \mE_{G,{\mathcal U}_3}\langle \|\x^{(i_1)}\|_2^2\|\y^{(i_2)}\|_2^2\rangle_{\gamma_{01}^{(2)}} +  (s-1)\mE_{G,{\mathcal U}_3}\langle \|\x^{(i_1)}\|_2^2 \|\y^{(i_2)}\|_2\|\y^{(p_2)}\|_2\rangle_{\gamma_{02}^{(2)}}\rp\nonumber \\
& & -  \p_2\q_2 s\beta^2(1-\m_1)\mE_{G,{\mathcal U}_3}\langle \|\x^{(i_1)}\|_2\|\x^{(p_`)}\|_2\|\y^{(i_2)}\|_2\|\y^{(p_2)}\|_2 \rangle_{\gamma_{1}^{(2)}} \nonumber \\
&  & -s\beta^2 \p_2\q_2(\m_1-\m_2) p \mE_{G,{\mathcal U}_3} \langle\|\x^{(i_2)}\|_2\|\x^{(p_2)}\|_2\|\y^{(i_2)}\|_2\|\y^{(p_2)}\rangle_{\gamma_{21}^{(2)}}
\nonumber \\
&  & +s\beta^2 \p_2\q_2 \m_1 (p-1) \mE_{G,{\mathcal U}_3} \langle\|\x^{(i_2)}\|_2\|\x^{(p_2)}\|_2\|\y^{(i_2)}\|_2\|\y^{(p_2)}\rangle_{\gamma_{22}^{(2)}}.
\end{eqnarray}

To determine $T_{3,3}^d$ we first write
\begin{eqnarray}\label{eq:lev2x3lev2genFanal24}
\frac{d}{d u^{(4,3)}} \lp \frac{1}{  \mE_{{\mathcal U}_2}\lp\lp
\sum_{i_3=1}^{l}
\lp
 \mE_{{\mathcal U}_1}  Z_{i_3}^{\m_1}
\rp^p
\rp^{\frac{\m_2}{\m_1}}\rp    }   \rp
& = & -    \mE_{{\mathcal U}_2} \frac{ \frac{\m_2}{\m_1} \lp\lp
\sum_{i_3=1}^{l}
\lp
 \mE_{{\mathcal U}_1}  Z_{i_3}^{\m_1}
\rp^p
\rp^{\frac{\m_2}{\m_1} -1 }\rp    }{\lp    \mE_{{\mathcal U}_2}\lp\lp
\sum_{i_3=1}^{l}
\lp
 \mE_{{\mathcal U}_1}  Z_{i_3}^{\m_1}
\rp^p
\rp^{\frac{\m_2}{\m_1}}\rp      \rp^2}
\nonumber \\
& &
\times
p\sum_{p_3=1}^{l}
\lp
 \mE_{{\mathcal U}_1}  Z_{p_3}^{\m_1}
\rp^{p-1}
 \mE_{{\mathcal U}_1}\frac{d Z_{p_3}^{\m_1}}{d u^{(4,3)}}.
 \end{eqnarray}
Utilizing
\begin{equation}\label{eq:lev2x3lev2genFanal25}
\frac{dZ_{p_3}^{\m_1}}{du^{(4,3)}} =
\m_1Z_{p_3}^{\m_1-1}s \sum_{p_1=1}^{l} (C_{p_3}^{(p_1)})^{s-1}\sum_{p_2=1}^{l}
\beta_{p_1}A_{p_3}^{(p_1,p_2)}\|\y^{(p_2)}\|_2\sqrt{t},
\end{equation}
we then obtain
\begin{eqnarray}\label{eq:lev2x3lev2genFanal25a}
T_{3,2}^d &  = &  \p_2\q_2 \mE_{G,{\mathcal U}_3,{\mathcal U}_2,{\mathcal U}_1} \Bigg ( \Bigg.
\sum_{i_3=1}^{l}
 \frac{
\lp
\sum_{i_3=1}^{l}
\lp
 \mE_{{\mathcal U}_1}  Z_{i_3}^{\m_1}
\rp^p
\rp^{\frac{\m_2}{\m_1}-1} \lp
 \mE_{{\mathcal U}_1}  Z_{i_3}^{\m_1}
\rp^{p-1}
(C_{i_3}^{(i_1)})^{s-1} A_{i_3}^{(i_1,i_2)}\y_j^{(i_2)}}{Z_{i_3}^{1-\m_1}   }
\nonumber \\
& &
\times
    \mE_{{\mathcal U}_2} \frac{ \frac{\m_2}{\m_1} \lp\lp
\sum_{i_3=1}^{l}
\lp
 \mE_{{\mathcal U}_1}  Z_{i_3}^{\m_1}
\rp^p
\rp^{\frac{\m_2}{\m_1} -1 }\rp    }{\lp    \mE_{{\mathcal U}_2}\lp\lp
\sum_{i_3=1}^{l}
\lp
 \mE_{{\mathcal U}_1}  Z_{i_3}^{\m_1}
\rp^p
\rp^{\frac{\m_2}{\m_1}}\rp      \rp^2}
\nonumber \\
& &
\times
p\sum_{p_3=1}^{l}
\lp
 \mE_{{\mathcal U}_1}  Z_{p_3}^{\m_1}
\rp^{p-1}
 \mE_{{\mathcal U}_1}
 \m_1Z_{p_3}^{\m_1-1}s \sum_{p_1=1}^{l} (C_{p_3}^{(p_1)})^{s-1}\sum_{p_2=1}^{l}
\beta_{p_1}A_{p_3}^{(p_1,p_2)}\|\y^{(p_2)}\|_2\sqrt{t}
   \Bigg .\Bigg ).
\nonumber \\
 &  = & -s\sqrt{t}\p_2\q_2 \m_2 p \mE_{G,{\mathcal U}_3} \Bigg( \Bigg.
      \mE_{{\mathcal U}_2} \frac{ \lp\lp
\sum_{i_3=1}^{l}
\lp
 \mE_{{\mathcal U}_1}  Z_{i_3}^{\m_1}
\rp^p
\rp^{\frac{\m_2}{\m_1}  }\rp    }{ \lp    \mE_{{\mathcal U}_2}\lp\lp
\sum_{i_3=1}^{l}
\lp
 \mE_{{\mathcal U}_1}  Z_{i_3}^{\m_1}
\rp^p
\rp^{\frac{\m_2}{\m_1}}\rp      \rp}
\sum_{i_3=1}^{l}
\frac{
\lp
 \mE_{{\mathcal U}_1}  Z_{i_3}^{\m_1}
\rp^{p}   }{  \sum_{i_3=1}^{l}
\lp
 \mE_{{\mathcal U}_1}  Z_{i_3}^{\m_1}
\rp^p }
\nonumber \\
& & \times
\mE_{{\mathcal U}_1} \frac{Z_{i_3}^{\m_1}}{\mE_{{\mathcal U}_1} \lp Z_{i_3}^{\m_1}\rp}  \frac{(C_{i_3}^{(i_1)})^{s}}{Z_{i_3}}
\frac{A_{i_3}^{(i_1,i_2)}}{C_{i_3}^{(i_1)}}  \nonumber \\
& & \times
     \mE_{{\mathcal U}_2} \frac{ \lp\lp
\sum_{i_3=1}^{l}
\lp
 \mE_{{\mathcal U}_1}  Z_{i_3}^{\m_1}
\rp^p
\rp^{\frac{\m_2}{\m_1}  }\rp    }{ \lp    \mE_{{\mathcal U}_2}\lp\lp
\sum_{i_3=1}^{l}
\lp
 \mE_{{\mathcal U}_1}  Z_{i_3}^{\m_1}
\rp^p
\rp^{\frac{\m_2}{\m_1}}\rp      \rp}
\sum_{p_3=1}^{l}
\frac{
\lp
 \mE_{{\mathcal U}_1}  Z_{p_3}^{\m_1}
\rp^{p}   }{  \sum_{p_3=1}^{l}
\lp
 \mE_{{\mathcal U}_1}  Z_{p_3}^{\m_1}
\rp^p }
\nonumber \\
& & \times
\mE_{{\mathcal U}_1} \frac{Z_{p_3}^{\m_1}}{\mE_{{\mathcal U}_1}   Z_{p_3}^{\m_1}  }  \sum_{p_1=1}^{l}  \frac{(C_{p_3}^{(p_1)})^{s}}{Z_{p_3}}\sum_{p_2=1}^{l}
\frac{A_{p_3}^{(p_1,p_2)}}{C_{p_3}^{(p_1)}} \beta_{p_1}\|\y^{(p_2)}\|_2   \Bigg. \Bigg).
\end{eqnarray}
After observing
\begin{eqnarray}\label{eq:lev2x3lev2genFanal28}
\sum_{i_1=1}^{l}\sum_{i_2=1}^{l} \beta_{i_1}\|\y^{(i_2)}\|_2\frac{T_{3,3}^d}{\sqrt{t}}
 & = & -s\beta^2 \p_2\q_2 \m_2 p \mE_{G,{\mathcal U}_3} \langle\|\x^{(i_2)}\|_2\|\x^{(p_2)}\|_2\|\y^{(i_2)}\|_2\|\y^{(p_2)}\rangle_{\gamma_{3}^{(2)}},\nonumber \\
\end{eqnarray}
and combining it with (\ref{eq:lev2x3lev2genFanal22}) and (\ref{eq:lev2x3lev2genFanal23b}) we arrive at the following
\begin{eqnarray}\label{eq:lev2x3lev2genFanal29}
\sum_{i_1=1}^{l}\sum_{i_2=1}^{l} \beta_{i_1}\|\y^{(i_2)}\|_2\frac{T_{3,3}}{\sqrt{t}}
& = &
\p_2\q_2 \beta^2
 \nonumber \\
 & & \times
 \lp \mE_{G,{\mathcal U}_3}\langle \|\x^{(i_1)}\|_2^2\|\y^{(i_2)}\|_2^2\rangle_{\gamma_{01}^{(2)}} +   (s-1)\mE_{G,{\mathcal U}_3}\langle \|\x^{(i_1)}\|_2^2 \|\y^{(i_2)}\|_2\|\y^{(p_2)}\|_2\rangle_{\gamma_{02}^{(2)}}\rp\nonumber \\
& & -  \p_2\q_2 s\beta^2(1-\m_1)\mE_{G,{\mathcal U}_3}\langle \|\x^{(i_1)}\|_2\|\x^{(p_`)}\|_2\|\y^{(i_2)}\|_2\|\y^{(p_2)}\|_2 \rangle_{\gamma_{1}^{(2)}} \nonumber \\
&  & -s\beta^2 \p_2\q_2(\m_1-\m_2) p \mE_{G,{\mathcal U}_3} \langle\|\x^{(i_2)}\|_2\|\x^{(p_2)}\|_2\|\y^{(i_2)}\|_2\|\y^{(p_2)}\rangle_{\gamma_{21}^{(2)}} \nonumber \\
&  & +s\beta^2 \p_2\q_2 \m_1 (p-1)  \mE_{G,{\mathcal U}_3} \langle\|\x^{(i_2)}\|_2\|\x^{(p_2)}\|_2\|\y^{(i_2)}\|_2\|\y^{(p_2)}\rangle_{\gamma_{22}^{(2)}} \nonumber \\
 &  & -s\beta^2 \p_2\q_2 \m_2\mE_{G,{\mathcal U}_3} \langle\|\x^{(i_2)}\|_2\|\x^{(p_2)}\|_2\|\y^{(i_2)}\|_2\|\y^{(p_2)}\rangle_{\gamma_{3}^{(2)}}.
\end{eqnarray}

\subsubsection{$T_G$--group -- second level}
\label{sec:lev2handlTG}

We first recall on
\begin{eqnarray}\label{eq:lev2genGanal1}
T_{G,j} & = &  \mE_{G,{\mathcal U}_3,{\mathcal U}_2} \lp
\sum_{i_3=1}^{l}
\frac{\lp
\sum_{i_3=1}^{l}
\lp
 \mE_{{\mathcal U}_1}  Z_{i_3}^{\m_1}
\rp^p
\rp^{\frac{\m_2}{\m_1}-1} \lp
 \mE_{{\mathcal U}_1}  Z_{i_3}^{\m_1}
\rp^{p-1}  }   {\mE_{{\mathcal U}_2}\lp\lp
\sum_{i_3=1}^{l}
\lp
 \mE_{{\mathcal U}_1}  Z_{i_3}^{\m_1}
\rp^p
\rp^{\frac{\m_2}{\m_1}}\rp}
  \mE_{{\mathcal U}_1}\frac{(C_{i_3}^{(i_1)})^{s-1} A_{i_3}^{(i_1,i_2)} \y_j^{(i_2)}\bar{\u}_j^{(i_1,1)}}{Z_{i_3}^{1-\m_1}} \rp \nonumber \\
 & = &  \mE_{{\mathcal U}_3,{\mathcal U}_2,{\mathcal U}_1   } \lp
\mE_{G}  \sum_{i_3=1}^{l}
\frac{\lp
\sum_{i_3=1}^{l}
\lp
 \mE_{{\mathcal U}_1}  Z_{i_3}^{\m_1}
\rp^p
\rp^{\frac{\m_2}{\m_1}-1} \lp
 \mE_{{\mathcal U}_1}  Z_{i_3}^{\m_1}
\rp^{p-1}  }   {\mE_{{\mathcal U}_2}\lp\lp
\sum_{i_3=1}^{l}
\lp
 \mE_{{\mathcal U}_1}  Z_{i_3}^{\m_1}
\rp^p
\rp^{\frac{\m_2}{\m_1}}\rp}
\frac{(C_{i_3}^{(i_1)})^{s-1} A_{i_3}^{(i_1,i_2)} \y_j^{(i_2)}\bar{\u}_j^{(i_1,1)}}{Z_{i_3}^{1-\m_1}} \rp \nonumber \\
 & = &  \mE_{{\mathcal U}_3,{\mathcal U}_2,{\mathcal U}_1} \Bigg(\Bigg.
   \mE_{G}  \sum_{p_1=1}^{l} \mE_G (\u_j^{(i_1,1)}\u_j^{(p_1,1)}) \nonumber \\
    & & \times \frac{d}{d\u_j^{(p_1,1)}} \Bigg(\Bigg.
  \sum_{i_3=1}^{l}
\frac{\lp
\sum_{i_3=1}^{l}
\lp
 \mE_{{\mathcal U}_1}  Z_{i_3}^{\m_1}
\rp^p
\rp^{\frac{\m_2}{\m_1}-1} \lp
 \mE_{{\mathcal U}_1}  Z_{i_3}^{\m_1}
\rp^{p-1}  }   {\mE_{{\mathcal U}_2}\lp\lp
\sum_{i_3=1}^{l}
\lp
 \mE_{{\mathcal U}_1}  Z_{i_3}^{\m_1}
\rp^p
\rp^{\frac{\m_2}{\m_1}}\rp}
\frac{(C_{i_3}^{(i_1)})^{s-1} A_{i_3}^{(i_1,i_2)} \y_j^{(i_2)}\bar{\u}_j^{(i_1,1)}}{Z_{i_3}^{1-\m_1}}     \Bigg.\Bigg) \Bigg.\Bigg). \nonumber \\
 \end{eqnarray}
We set
\begin{eqnarray}\label{eq:lev2genGanal3}
\Theta_{G,1}^{(2)} & = & \sum_{p_1=1}^{l} \mE_G (\u_j^{(i_1,1)}\u_j^{(p_1,1)})\frac{d}{d\u_j^{(p_1,1)}}\lp \frac{(C^{(i_1)})^{s-1} A^{(i_1,i_2)}\y_j^{(i_2)}}{Z^{1-\m_1}}\rp \nonumber \\
\Theta_{G,21}^{(2)} & = &   \Bigg ( \Bigg .
   \sum_{i_3=1}^{l}
\frac{
 \lp
 \mE_{{\mathcal U}_1}  Z_{i_3}^{\m_1}
\rp^{p-1}
 }   {\mE_{{\mathcal U}_2}\lp\lp
\sum_{i_3=1}^{l}
\lp
 \mE_{{\mathcal U}_1}  Z_{i_3}^{\m_1}
\rp^p
\rp^{\frac{\m_2}{\m_1}}\rp}
 \nonumber \\
 & & \times
 \frac{(C^{(i_1)})^{s-1} A^{(i_1,i_2)}\y_j^{(i_2)}}{Z^{1-\m_1}      }          \sum_{p_1=1}^{l} \mE_G (\u_j^{(i_1,1)}\u_j^{(p_1,1)})\frac{d}{d\u_j^{(p_1,1)}}\Bigg ( \Bigg .
  \lp
\sum_{i_3=1}^{l}
\lp
 \mE_{{\mathcal U}_1}  Z_{i_3}^{\m_1}
\rp^p
\rp^{\frac{\m_2}{\m_1}-1}
   \Bigg .\Bigg )\Bigg .\Bigg ) \nonumber \\
\Theta_{G,22}^{(2)} & = &   \Bigg ( \Bigg .
   \sum_{i_3=1}^{l}
\frac{\lp
\sum_{i_3=1}^{l}
\lp
 \mE_{{\mathcal U}_1}  Z_{i_3}^{\m_1}
\rp^p
\rp^{\frac{\m_2}{\m_1}-1} }   {\mE_{{\mathcal U}_2}\lp\lp
\sum_{i_3=1}^{l}
\lp
 \mE_{{\mathcal U}_1}  Z_{i_3}^{\m_1}
\rp^p
\rp^{\frac{\m_2}{\m_1}}\rp}
 \nonumber \\
 & & \times
 \frac{(C^{(i_1)})^{s-1} A^{(i_1,i_2)}\y_j^{(i_2)}}{Z^{1-\m_1}     }          \sum_{p_1=1}^{l} \mE_G (\u_j^{(i_1,1)}\u_j^{(p_1,1)})\frac{d}{d\u_j^{(p_1,1)}}\lp   \lp
 \mE_{{\mathcal U}_1}  Z_{i_3}^{\m_1}
\rp^{p-1}  \rp \Bigg .\Bigg ) \nonumber \\
\Theta_{G,3}^{(2)} & = &   \Bigg ( \Bigg .
    \sum_{i_3=1}^{l}
\lp
\sum_{i_3=1}^{l}
\lp
 \mE_{{\mathcal U}_1}  Z_{i_3}^{\m_1}
\rp^p
\rp^{\frac{\m_2}{\m_1}-1} \lp
 \mE_{{\mathcal U}_1}  Z_{i_3}^{\m_1}
\rp^{p-1}
\nonumber \\
& & \times
  \frac{(C^{(i_1)})^{s-1} A^{(i_1,i_2)}\y_j^{(i_2)}}{Z^{1-\m_1}    }          \sum_{p_1=1}^{l} \mE_G (\u_j^{(i_1,1)}\u_j^{(p_1,1)})\frac{d}{d\u_j^{(p_1,1)}}
   \Bigg . \Bigg (
  \frac{1}{  \mE_{{\mathcal U}_2}\lp\lp
\sum_{i_3=1}^{l}
\lp
 \mE_{{\mathcal U}_1}  Z_{i_3}^{\m_1}
\rp^p
\rp^{\frac{\m_2}{\m_1}}\rp   }
   \Bigg . \Bigg )
 \Bigg . \Bigg ) \nonumber \\
T_{G,j}^c & = &  \mE_{G,{\mathcal U}_3} \lp
 \mE_{{\mathcal U}_2}\mE_{{\mathcal U}_1}   \sum_{i_3=1}^{l}
\frac{\lp
\sum_{i_3=1}^{l}
\lp
 \mE_{{\mathcal U}_1}  Z_{i_3}^{\m_1}
\rp^p
\rp^{\frac{\m_2}{\m_1}-1} \lp
 \mE_{{\mathcal U}_1}  Z_{i_3}^{\m_1}
\rp^{p-1}  }   {\mE_{{\mathcal U}_2}\lp\lp
\sum_{i_3=1}^{l}
\lp
 \mE_{{\mathcal U}_1}  Z_{i_3}^{\m_1}
\rp^p
\rp^{\frac{\m_2}{\m_1}}\rp}
 \Theta_{G,1}^{(2)}  \rp \nonumber \\
T_{G,j}^{d_{1}} & = &  \mE_{G,{\mathcal U}_3}\lp \mE_{{\mathcal U}_2}\mE_{{\mathcal U}_1}  \Theta_{G,21}^{(2)} \rp \nonumber \\
T_{G,j}^{e_{1}} & = &  \mE_{G,{\mathcal U}_3}\lp \mE_{{\mathcal U}_2}\mE_{{\mathcal U}_1}  \Theta_{G,22}^{(2)} \rp \nonumber \\
T_{G,j}^{d_2} & = &  \mE_{G,{\mathcal U}_3}\lp \mE_{{\mathcal U}_2}\mE_{{\mathcal U}_1}  \Theta_{G,3}^{(2)} \rp,
 \end{eqnarray}
and note that
 \begin{eqnarray}\label{eq:lev2genGanal4}
T_{G,j}  & = &  T_{G,j}^c+T_{G,j}^{d_1}+T_{G,j}^{e_1}+T_{G,j}^{d_2}.
 \end{eqnarray}
Following (\ref{eq:genGanal5}) we have
\begin{eqnarray}\label{eq:lev2genGanal5}
\Theta_{G,1}^{(2)} & = & \lp \frac{\y_j^{(i_2)}}{Z^{1-\m_1}}\lp(C^{(i_1)})^{s-1}\beta_{i_1}A^{(i_1,i_2)}\y_j^{(i_2)}\sqrt{t}+(s-1)(C^{(i_1)})^{s-2}\beta_{i_1}\sum_{p_2=1}^{l}A^{(i_1,p_2)}\y_j^{(p_2)}\sqrt{t}\rp \rp \nonumber \\
& &  -(1-\m_1)
\mE \lp\sum_{p_1=1}^{l} \frac{(\x^{(i_1)})^T\x^{(p_1)}}{\|\x^{(i_1)}\|_2\|\x^{(p_1)}\|_2}
\frac{(C^{(i_1)})^{s-1} A^{(i_1,i_2)}\y_j^{(i_2)}}{Z^{2-\m_1}}
s  (C^{(p_1)})^{s-1}\sum_{p_2=1}^{l}\beta_{p_1}A^{(p_1,p_2)}\y_j^{(p_2)}\sqrt{t}\rp,\nonumber \\
\end{eqnarray}
and then analogously to (\ref{eq:genGanal6})
\begin{eqnarray}\label{eq:lev2genGanal6}
 \sum_{i_1=1}^{l} \sum_{i_2=1}^{l}\sum_{j=1}^{m}\beta_{i_1} \frac{T_{G,j}^c}{\sqrt{t}}
 & = & \sum_{i_1=1}^{l} \sum_{i_2=1}^{l}\sum_{j=1}^{m}\beta_{i_1}
 \frac{ \mE_{G,{\mathcal U}_3} \lp
   \mE_{{\mathcal U}_2}\mE_{{\mathcal U}_1}  \sum_{i_3=1}^{l}
\frac{\lp
\sum_{i_3=1}^{l}
\lp
 \mE_{{\mathcal U}_1}  Z_{i_3}^{\m_1}
\rp^p
\rp^{\frac{\m_2}{\m_1}-1} \lp
 \mE_{{\mathcal U}_1}  Z_{i_3}^{\m_1}
\rp^{p-1}  }   {\mE_{{\mathcal U}_2}\lp\lp
\sum_{i_3=1}^{l}
\lp
 \mE_{{\mathcal U}_1}  Z_{i_3}^{\m_1}
\rp^p
\rp^{\frac{\m_2}{\m_1}}\rp}
 \Theta_{G,1}^{(2)}  \rp }
 {\sqrt{t}}
 \nonumber\\
 & = & \sum_{i_1=1}^{l} \sum_{i_2=1}^{l}\sum_{j=1}^{m}\beta_{i_1}
 \frac{ \mE_{G,{\mathcal U}_3} \lp
 \Phi_{{\mathcal U}_2} \sum_{i_3=1}^{l}\gamma_{00}(i_3)
 \frac{1}{\mE Z_{i_3}^{\m_1}}
 \Theta_{G,1}^{(2)}  \rp }
 {\sqrt{t}}
 \nonumber\\
 & = & \beta^2 \lp \mE_{G,{\mathcal U}_3} \langle \|\x^{(i_1)}\|_2^2\|\y^{(i_2)}\|_2^2\rangle_{\gamma_{01}^{(2)}} +   (s-1)\mE_{G,{\mathcal U}_3}\langle \|\x^{(i_1)}\|_2^2(\y^{(p_2)})^T\y^{(i_2)}\rangle_{\gamma_{02}^{(2)}}   \rp    \nonumber \\
 &  & -s\beta^2(1-\m_1) \mE_{G,{\mathcal U}_3}\langle (\x^{(p_1)})^T\x^{(i_1)}(\y^{(p_2)})^T\y^{(i_2)}\rangle_{\gamma_{1}^{(2)}}.
 \end{eqnarray}

To find $T_{G,j}^{d_{1}}$, we start with
\begin{eqnarray}\label{eq:lev2genGanal7}
 \frac{d}{d\bar{\u}_j^{(p_1,1)}}\lp
 \lp
\sum_{i_3=1}^{l}
\lp
 \mE_{{\mathcal U}_1}  Z_{i_3}^{\m_1}
\rp^p
\rp^{\frac{\m_2}{\m_1} -1 }
 \rp
=
-p \sum_{p_3=1}^{l}
  \frac{\lp 1-\frac{\m_2}{\m_1} \rp
\lp
 \mE_{{\mathcal U}_1}  Z_{p_3}^{\m_1}
\rp^{p-1}  }{   \lp
\sum_{i_3=1}^{l}
\lp
 \mE_{{\mathcal U}_1}  Z_{i_3}^{\m_1}
\rp^p
\rp^{2 - \frac{\m_2}{\m_1}  }
 }
\mE_{{\mathcal U}_1}\frac{d Z_{p_3}^{\m_1}}{d\bar{\u}_j^{(p_1,1)}}.
 \end{eqnarray}
From \cite{Stojnicnflgscompyx23}'s (174) we also have
\begin{eqnarray}\label{eq:lev2genGanal7a}
  \frac{dZ_{p_3}^{\m_1}}{d\bar{\u}_j^{(p_1,1)}}
 & = &   \m_1 \frac{Z_{p_3}^{\m_1}}{Z_{p_3}} s (C_{p_3}^{(p_1)})^{s-1}\sum_{p_2=1}^{l}A_{p_3}^{(p_1,p_2)}\beta_{p_1}\y_j^{(p_2)}\sqrt{t}.
\end{eqnarray}
A combination of (\ref{eq:lev2genGanal7}) and (\ref{eq:lev2genGanal7a}) gives
\begin{eqnarray}\label{eq:lev2genGanal7b}
 \frac{d}{d\bar{\u}_j^{(p_1,1)}}\lp
 \lp
\sum_{i_3=1}^{l}
\lp
 \mE_{{\mathcal U}_1}  Z_{i_3}^{\m_1}
\rp^p
\rp^{\frac{\m_2}{\m_1} -1 }
 \rp
&  = &
-p \sum_{p_3=1}^{l}
  \frac{\lp \m_1- \m_2 \rp
\lp
 \mE_{{\mathcal U}_1}  Z_{p_3}^{\m_1}
\rp^{p}  }{   \lp
\sum_{i_3=1}^{l}
\lp
 \mE_{{\mathcal U}_1}  Z_{i_3}^{\m_1}
\rp^p
\rp^{2 - \frac{\m_2}{\m_1}  }
 }
\mE_{{\mathcal U}_1}\frac{ Z_{p_3}^{\m_1} } { \mE_{{\mathcal U}_1} Z_{p_3}^{\m_1}  }
\nonumber \\
& & \times
  \frac{1}{Z_{p_3}} s (C_{p_3}^{(p_1)})^{s-1}\sum_{p_2=1}^{l}A_{p_3}^{(p_1,p_2)}\beta_{p_1}\y_j^{(p_2)}\sqrt{t}.
\end{eqnarray}
Combining further  the expression for $\Theta_{G,21}^{(2)}$ from (\ref{eq:lev2genGanal3}) with (\ref{eq:lev2genGanal7b}), one obtains
\begin{eqnarray}\label{eq:lev2genGanal8}
 \mE_{{\mathcal U}_2,{\mathcal U}_1} \Theta_{G,21}^{(2)} \hspace{-.1in} & = &
 -  \mE_{{\mathcal U}_2}\mE_{{\mathcal U}_1}
 \Bigg ( \Bigg .
   \sum_{i_3=1}^{l}
\frac{
 \lp
 \mE_{{\mathcal U}_1}  Z_{i_3}^{\m_1}
\rp^{p-1}
 }   {\mE_{{\mathcal U}_2}\lp\lp
\sum_{i_3=1}^{l}
\lp
 \mE_{{\mathcal U}_1}  Z_{i_3}^{\m_1}
\rp^p
\rp^{\frac{\m_2}{\m_1}}\rp}
 \nonumber \\
 & & \times
 \frac{(C_{i_3}^{(i_1)})^{s-1} A_{i_3}^{(i_1,i_2)}\y_j^{(i_2)}}{Z_{i_3}^{1-\m_1}      }          \sum_{p_1=1}^{l} \mE_G (\u_j^{(i_1,1)}\u_j^{(p_1,1)})
 p \sum_{p_3=1}^{l}
  \frac{\lp \m_1- \m_2 \rp
\lp
 \mE_{{\mathcal U}_1}  Z_{p_3}^{\m_1}
\rp^{p}  }{   \lp
\sum_{i_3=1}^{l}
\lp
 \mE_{{\mathcal U}_1}  Z_{i_3}^{\m_1}
\rp^p
\rp^{2 - \frac{\m_2}{\m_1}  }
 }
\mE_{{\mathcal U}_1}\frac{ Z_{p_3}^{\m_1} } { \mE_{{\mathcal U}_1} Z_{p_3}^{\m_1}  }
\nonumber \\
& & \times
  \frac{1}{Z_{p_3}} s (C_{p_3}^{(p_1)})^{s-1}\sum_{p_2=1}^{l}A_{p_3}^{(p_1,p_2)}\beta_{p_1}\y_j^{(p_2)}\sqrt{t}
   \Bigg .\Bigg )
   \nonumber \\
 & = &
 -  \mE_{{\mathcal U}_2}
 \Bigg ( \Bigg .
\frac{
\lp
\sum_{i_3=1}^{l}
\lp
 \mE_{{\mathcal U}_1}  Z_{i_3}^{\m_1}
\rp^p
\rp^{\frac{\m_2}{\m_1}}
 }   {\mE_{{\mathcal U}_2}\lp\lp
\sum_{i_3=1}^{l}
\lp
 \mE_{{\mathcal U}_1}  Z_{i_3}^{\m_1}
\rp^p
\rp^{\frac{\m_2}{\m_1}}\rp  }
\mE_{{\mathcal U}_1}
   \sum_{i_3=1}^{l}
   \lp
 \mE_{{\mathcal U}_1}  Z_{i_3}^{\m_1}
\rp^{p-1}
\frac{Z_{i_3}^{\m_1} }  {  \lp
\sum_{i_3=1}^{l}
 \lp
 \mE_{{\mathcal U}_1}  Z_{i_3}^{\m_1}
\rp^p
\rp }
 \nonumber \\
 & & \times
 \frac{(C_{i_3}^{(i_1)})^{s-1} A_{i_3}^{(i_1,i_2)}\y_j^{(i_2)}}{Z_{i_3}      }          \sum_{p_1=1}^{l} \mE_G (\u_j^{(i_1,1)}\u_j^{(p_1,1)})
 p \sum_{p_3=1}^{l}
  \frac{\lp \m_1- \m_2 \rp
\lp
 \mE_{{\mathcal U}_1}  Z_{p_3}^{\m_1}
\rp^{p}  }{   \lp
\sum_{i_3=1}^{l}
\lp
 \mE_{{\mathcal U}_1}  Z_{i_3}^{\m_1}
\rp^p
\rp
 }
\mE_{{\mathcal U}_1}\frac{ Z_{p_3}^{\m_1} } { \mE_{{\mathcal U}_1} Z_{p_3}^{\m_1}  }
\nonumber \\
& & \times
  \frac{1}{Z_{p_3}} s (C_{p_3}^{(p_1)})^{s-1}\sum_{p_2=1}^{l}A_{p_3}^{(p_1,p_2)}\beta_{p_1}\y_j^{(p_2)}\sqrt{t}
   \Bigg .\Bigg )
\nonumber \\
 & = & \Phi_{{\mathcal U}_2}\frac{\m_1-\m_2}{\m_1} \mE_{{\mathcal U}_1}
   \sum_{i_3=1}^{l}
   \lp
 \mE_{{\mathcal U}_1}  Z_{i_3}^{\m_1}
\rp^{p-1} \Theta_{G,2}.
  \end{eqnarray}
Connecting $T_{G,j}^{d}$ from (\ref{eq:genGanal3}), (\ref{eq:genGanal9}), $T_{G,j}^{d_1}$ from (\ref{eq:lev2genGanal3}), and (\ref{eq:lev2genGanal8}), we also obtain
\begin{eqnarray}\label{eq:lev2genGanal9}
 \sum_{i_1=1}^{l} \sum_{i_2=1}^{l}\sum_{j=1}^{m}\beta_{i_1} \frac{T_{G,j}^{d_1}}{\sqrt{t}}
 & = & \sum_{i_1=1}^{l} \sum_{i_2=1}^{l}\sum_{j=1}^{m}\beta_{i_1} \frac{\mE_{G,{\mathcal U}_3}\lp \mE_{{\mathcal U}_2} \mE_{{\mathcal U}_1}  \Theta_{G,2}^{(2)}\rp}{\sqrt{t}} \nonumber\\
 & = &  \mE_{G,{\mathcal U}_3}\lp \mE_{{\mathcal U}_2} \mE_{{\mathcal U}_1} \sum_{i_1=1}^{l} \sum_{i_2=1}^{l}\sum_{j=1}^{m}\beta_{i_1} \frac{\Theta_{G,2}^{(2)}}{\sqrt{t}}\rp \nonumber\\
 & = &  \mE_{G,{\mathcal U}_3}\lp \Phi_{{\mathcal U}_2}\frac{\m_1-\m_2}{\m_1} \mE_{{\mathcal U}_1} \sum_{i_1=1}^{l} \sum_{i_2=1}^{l}\sum_{j=1}^{m}\beta_{i_1} \frac{
   \sum_{i_3=1}^{l}
   \lp
 \mE_{{\mathcal U}_1}  Z_{i_3}^{\m_1}
\rp^{p-1} \Theta_{G,2}}{\sqrt{t}}\rp \nonumber\\
 & = &  -\mE_{G,{\mathcal U}_3}\lp \Phi_{{\mathcal U}_2}\frac{\m_1-\m_2}{\m_1} s\beta^2\m_1 p \langle(\x^{(p_1)})^T\x^{(i_1)} (\y^{(p_2)})^T\y^{(i_2)} \rangle_{\gamma_{21}^{(1)}}\rp \nonumber\\
  & = & -s\beta^2(\m_1-\m_2) p \mE_{G,{\mathcal U}_3} \langle(\x^{(p_1)})^T\x^{(i_1)} (\y^{(p_2)})^T\y^{(i_2)} \rangle_{\gamma_{22}^{(2)}}.
 \end{eqnarray}

To find $T_{G,j}^{e_{1}}$, we start with
\begin{eqnarray}\label{eq:lev2genGanal7bb0}
 \frac{d}{d\bar{\u}_j^{(p_1,1)}}\lp
\sum_{i_3=1}^{l}
\lp
 \mE_{{\mathcal U}_1}  Z_{i_3}^{\m_1}
\rp^{p-1}
 \rp
=
(p-1)
\lp
 \mE_{{\mathcal U}_1}  Z_{i_3}^{\m_1}
\rp^{p-2}
\mE_{{\mathcal U}_1}\frac{d Z_{i_3}^{\m_1}}{d\bar{\u}_j^{(p_1,1)}}.
 \end{eqnarray}
A combination of (\ref{eq:lev2genGanal7bb0}) and (\ref{eq:lev2genGanal7a}) gives
\begin{eqnarray}\label{eq:lev2genGanal7bbb0}
 \frac{d}{d\bar{\u}_j^{(p_1,1)}}\lp
\sum_{i_3=1}^{l}
\lp
 \mE_{{\mathcal U}_1}  Z_{i_3}^{\m_1}
\rp^{p-1}
 \rp
&  = &
\m_1 (p-1) \lp
 \mE_{{\mathcal U}_1}  Z_{i_3}^{\m_1}
\rp^{p-1}
\mE_{{\mathcal U}_1}\frac{ Z_{i_3}^{\m_1} } { \mE_{{\mathcal U}_1} Z_{i_3}^{\m_1}  }
\nonumber \\
& & \times
  \frac{1}{Z_{i_3}} s (C_{i_3}^{(p_1)})^{s-1}\sum_{p_2=1}^{l}A_{i_3}^{(p_1,p_2)}\beta_{p_1}\y_j^{(p_2)}\sqrt{t}.
\end{eqnarray}
Combining further  the expression for $\Theta_{G,22}^{(2)}$ from (\ref{eq:lev2genGanal3}) with (\ref{eq:lev2genGanal7bbb0}) while recalling on (\ref{eq:genGanal8bb2}), one obtains
\begin{eqnarray}\label{eq:lev2genGanal8bb0}
\mE_{{\mathcal U}_2}\mE_{{\mathcal U}_1}  \Theta_{G,22}^{(2)} & = &
\mE_{{\mathcal U}_2}\mE_{{\mathcal U}_1}
 \Bigg ( \Bigg .
   \sum_{i_3=1}^{l}
\frac{\lp
\sum_{i_3=1}^{l}
\lp
 \mE_{{\mathcal U}_1}  Z_{i_3}^{\m_1}
\rp^p
\rp^{\frac{\m_2}{\m_1}-1} }   {\mE_{{\mathcal U}_2}\lp\lp
\sum_{i_3=1}^{l}
\lp
 \mE_{{\mathcal U}_1}  Z_{i_3}^{\m_1}
\rp^p
\rp^{\frac{\m_2}{\m_1}}\rp}
 \nonumber \\
 & & \times
 \frac{(C_{i_3}^{(i_1)})^{s-1} A_{i_3}^{(i_1,i_2)}\y_j^{(i_2)}}{Z_{i_3}^{1-\m_1}     }          \sum_{p_1=1}^{l} \mE_G (\u_j^{(i_1,1)}\u_j^{(p_1,1)}) \m_1 (p-1)
\lp
 \mE_{{\mathcal U}_1}  Z_{i_3}^{\m_1}
\rp^{p-1}
\mE_{{\mathcal U}_1}\frac{ Z_{i_3}^{\m_1} } { \mE_{{\mathcal U}_1} Z_{i_3}^{\m_1}  }
\nonumber \\
& & \times
  \frac{1}{Z_{i_3}} s (C_{i_3}^{(p_1)})^{s-1}\sum_{p_2=1}^{l}A_{i_3}^{(p_1,p_2)}\beta_{p_1}\y_j^{(p_2)}\sqrt{t}
   \Bigg .\Bigg )
   \nonumber \\
   & = &
\mE_{{\mathcal U}_2}\mE_{{\mathcal U}_1}
 \Bigg ( \Bigg .
   \sum_{i_3=1}^{l}
\frac{\lp
\sum_{i_3=1}^{l}
\lp
 \mE_{{\mathcal U}_1}  Z_{i_3}^{\m_1}
\rp^p
\rp^{\frac{\m_2}{\m_1}} }   {\mE_{{\mathcal U}_2}\lp\lp
\sum_{i_3=1}^{l}
\lp
 \mE_{{\mathcal U}_1}  Z_{i_3}^{\m_1}
\rp^p
\rp^{\frac{\m_2}{\m_1}}\rp}
\frac{ \lp
 \mE_{{\mathcal U}_1}  Z_{i_3}^{\m_1}
\rp^{p} }
{  \lp
\sum_{i_3=1}^{l}
\lp
 \mE_{{\mathcal U}_1}  Z_{i_3}^{\m_1}
\rp^p
\rp  }
\mE_{{\mathcal U}_1}\frac{ Z_{i_3}^{\m_1} } { \mE_{{\mathcal U}_1} Z_{i_3}^{\m_1}  }
 \nonumber \\
 & & \times
 \frac{(C_{i_3}^{(i_1)})^{s-1} A_{i_3}^{(i_1,i_2)}\y_j^{(i_2)}}{Z_{i_3}    }          \sum_{p_1=1}^{l} \mE_G (\u_j^{(i_1,1)}\u_j^{(p_1,1)})  \m_1  (p-1)
\mE_{{\mathcal U}_1}\frac{ Z_{i_3}^{\m_1} } { \mE_{{\mathcal U}_1} Z_{i_3}^{\m_1}  }
\nonumber \\
& & \times
  \frac{1}{Z_{i_3}} s (C_{i_3}^{(p_1)})^{s-1}\sum_{p_2=1}^{l}A_{i_3}^{(p_1,p_2)}\beta_{p_1}\y_j^{(p_2)}\sqrt{t}
   \Bigg .\Bigg ) \nonumber \\
  & = & \Phi_{{\mathcal U}_2}  \mE_{{\mathcal U}_1}
 \Theta_{G,3}.
  \end{eqnarray}
Connecting  $T_{G,j}^{e}$ from (\ref{eq:genGanal3}),  (\ref{eq:genGanal9bb3}), $T_{G,j}^{e_1}$ from (\ref{eq:lev2genGanal3}), and (\ref{eq:lev2genGanal8bb0}), we find
\begin{eqnarray}\label{eq:lev2genGanal9bb0}
 \sum_{i_1=1}^{l} \sum_{i_2=1}^{l}\sum_{j=1}^{m}\beta_{i_1} \frac{T_{G,j}^{e_1}}{\sqrt{t}}
 & = & \sum_{i_1=1}^{l} \sum_{i_2=1}^{l}\sum_{j=1}^{m}\beta_{i_1} \frac{\mE_{G,{\mathcal U}_3}\lp \mE_{{\mathcal U}_2} \mE_{{\mathcal U}_1}  \Theta_{G,22}^{(2)}\rp}{\sqrt{t}} \nonumber\\
 & = &  \mE_{G,{\mathcal U}_3}\lp \mE_{{\mathcal U}_2} \mE_{{\mathcal U}_1} \sum_{i_1=1}^{l} \sum_{i_2=1}^{l}\sum_{j=1}^{m}\beta_{i_1} \frac{\Theta_{G,22}^{(22)}}{\sqrt{t}}\rp \nonumber\\
 & = &  \mE_{G,{\mathcal U}_3}\lp \Phi_{{\mathcal U}_2} \mE_{{\mathcal U}_1} \sum_{i_1=1}^{l} \sum_{i_2=1}^{l}\sum_{j=1}^{m}\beta_{i_1} \frac{
 \Theta_{G,3}}{\sqrt{t}}\rp \nonumber\\
 & = &  \mE_{G,{\mathcal U}_3}\lp \Phi_{{\mathcal U}_2} s\beta^2\m_1 (p-1) \langle(\x^{(p_1)})^T\x^{(i_1)} (\y^{(p_2)})^T\y^{(i_2)} \rangle_{\gamma_{22}^{(1)}}\rp \nonumber\\
  & = & s\beta^2 \m_1 (p-1) \mE_{G,{\mathcal U}_3} \langle(\x^{(p_1)})^T\x^{(i_1)} (\y^{(p_2)})^T\y^{(i_2)} \rangle_{\gamma_{22}^{(2)}}.
 \end{eqnarray}

To find $T_{G,j}^{d_2}$, we start with
\begin{eqnarray}\label{eq:lev2genGanal10}
\frac{d}{d \bar{\u}_j^{(p_1,1)}} \lp \frac{1}{  \mE_{{\mathcal U}_2}\lp\lp
\sum_{i_3=1}^{l}
\lp
 \mE_{{\mathcal U}_1}  Z_{i_3}^{\m_1}
\rp^p
\rp^{\frac{\m_2}{\m_1}}\rp    }   \rp
& = & -    \mE_{{\mathcal U}_2} \frac{ \frac{\m_2}{\m_1} \lp\lp
\sum_{i_3=1}^{l}
\lp
 \mE_{{\mathcal U}_1}  Z_{i_3}^{\m_1}
\rp^p
\rp^{\frac{\m_2}{\m_1} -1 }\rp    }{\lp    \mE_{{\mathcal U}_2}\lp\lp
\sum_{i_3=1}^{l}
\lp
 \mE_{{\mathcal U}_1}  Z_{i_3}^{\m_1}
\rp^p
\rp^{\frac{\m_2}{\m_1}}\rp      \rp^2}
\nonumber \\
& &
\times
p\sum_{p_3=1}^{l}
\lp
 \mE_{{\mathcal U}_1}  Z_{p_3}^{\m_1}
\rp^{p-1}
 \mE_{{\mathcal U}_1}\frac{d Z_{p_3}^{\m_1}}{d \bar{\u}_j^{(p_1,1)}}
\nonumber \\
& = & -    \mE_{{\mathcal U}_2} \frac{ \frac{\m_2}{\m_1} \lp\lp
\sum_{i_3=1}^{l}
\lp
 \mE_{{\mathcal U}_1}  Z_{i_3}^{\m_1}
\rp^p
\rp^{\frac{\m_2}{\m_1} -1 }\rp    }{\lp    \mE_{{\mathcal U}_2}\lp\lp
\sum_{i_3=1}^{l}
\lp
 \mE_{{\mathcal U}_1}  Z_{i_3}^{\m_1}
\rp^p
\rp^{\frac{\m_2}{\m_1}}\rp      \rp^2}
\nonumber \\
& &
\times
p\sum_{p_3=1}^{l}
\lp
 \mE_{{\mathcal U}_1}  Z_{p_3}^{\m_1}
\rp^{p-1}
  \nonumber\\
& &
\times \mE_{{\mathcal U}_1}   \m_1   \frac{Z_{p_3}^{\m_1}}{Z_{p_3}} s (C_{p_3}^{(p_1)})^{s-1}\sum_{p_2=1}^{l}A_{p_3}^{(p_1,p_2)}\beta_{p_1}\y_j^{(p_2)}\sqrt{t}.
\nonumber \\
\end{eqnarray}
Combining  the expression for $\Theta_{G,3}^{(2)}$ from (\ref{eq:lev2genGanal3}) with (\ref{eq:lev2genGanal10}), we obtain
\begin{eqnarray}\label{eq:lev2genGanal11}
 \mE_{{\mathcal U}_2,{\mathcal U}_1} \Theta_{G,3}^{(2)} & = &
  \mE_{{\mathcal U}_2,{\mathcal U}_1}  \Bigg ( \Bigg .
    \sum_{i_3=1}^{l}
\lp
\sum_{i_3=1}^{l}
\lp
 \mE_{{\mathcal U}_1}  Z_{i_3}^{\m_1}
\rp^p
\rp^{\frac{\m_2}{\m_1}-1} \lp
 \mE_{{\mathcal U}_1}  Z_{i_3}^{\m_1}
\rp^{p-1}
\nonumber \\
& & \times
  \frac{(C_{i_3}^{(i_1)})^{s-1} A_{i_3}^{(i_1,i_2)}\y_j^{(i_2)}}{Z_{i_3}^{1-\m_1}    }          \sum_{p_1=1}^{l} \mE_G (\u_j^{(i_1,1)}\u_j^{(p_1,1)})
    \mE_{{\mathcal U}_2} \frac{ \frac{\m_2}{\m_1} \lp\lp
\sum_{i_3=1}^{l}
\lp
 \mE_{{\mathcal U}_1}  Z_{i_3}^{\m_1}
\rp^p
\rp^{\frac{\m_2}{\m_1} -1 }\rp    }{\lp    \mE_{{\mathcal U}_2}\lp\lp
\sum_{i_3=1}^{l}
\lp
 \mE_{{\mathcal U}_1}  Z_{i_3}^{\m_1}
\rp^p
\rp^{\frac{\m_2}{\m_1}}\rp      \rp^2}
\nonumber \\
& &
\times
p\sum_{p_3=1}^{l}
\lp
 \mE_{{\mathcal U}_1}  Z_{p_3}^{\m_1}
\rp^{p-1}
   \mE_{{\mathcal U}_1}   \m_1   \frac{Z_{p_3}^{\m_1}}{Z_{p_3}} s (C_{p_3}^{(p_1)})^{s-1}\sum_{p_2=1}^{l}A_{p_3}^{(p_1,p_2)}\beta_{p_1}\y_j^{(p_2)}\sqrt{t}
 \Bigg . \Bigg ) \nonumber \\
  & = &
s\beta^2\m_2 p
\Bigg( \Bigg.
    \mE_{{\mathcal U}_2} \frac{ \lp
\sum_{i_3=1}^{l}
\lp
 \mE_{{\mathcal U}_1}  Z_{i_3}^{\m_1}
\rp^p
\rp^{\frac{\m_2}{\m_1} }     }   {   \mE_{{\mathcal U}_2}\lp\lp
\sum_{i_3=1}^{l}
\lp
 \mE_{{\mathcal U}_1}  Z_{i_3}^{\m_1}
\rp^p
\rp^{\frac{\m_2}{\m_1}}\rp     }
\sum_{i_3=1}^{l}
\frac{ \lp
 \mE_{{\mathcal U}_1}  Z_{i_3}^{\m_1}
\rp^p }{  \sum_{i_3=1}^{l}
\lp
 \mE_{{\mathcal U}_1}  Z_{i_3}^{\m_1}
\rp^p    }
  \mE_{{\mathcal U}_1}
\frac{ Z_{i_3}^{\m_1}  }{   \mE_{{\mathcal U}_1}  Z_{i_3}^{\m_1}   }
\nonumber \\
& & \times
  \frac{(C_{i_3}^{(i_1)})^s}{Z_{i_3}}
  \frac{ A_{i_3}^{(i_1,i_2)}}{C_{i_3}^{(i_1)}}\y_j^{(i_2)}
  \sum_{p_1=1}^{l} \frac{(\x^{(p_1)})^T\x^{(i_1)}}{\beta_{i_1}}
 \nonumber \\
& & \times
    \mE_{{\mathcal U}_2} \frac{ \lp
\sum_{i_3=1}^{l}
\lp
 \mE_{{\mathcal U}_1}  Z_{p_3}^{\m_1}
\rp^p
\rp^{\frac{\m_2}{\m_1} }     }   {   \mE_{{\mathcal U}_2}\lp\lp
\sum_{p_3=1}^{l}
\lp
 \mE_{{\mathcal U}_1}  Z_{p_3}^{\m_1}
\rp^p
\rp^{\frac{\m_2}{\m_1}}\rp     }
\sum_{p_3=1}^{l}
\frac{ \lp
 \mE_{{\mathcal U}_1}  Z_{p_3}^{\m_1}
\rp^p }{  \sum_{p_3=1}^{l}
\lp
 \mE_{{\mathcal U}_1}  Z_{p_3}^{\m_1}
\rp^p    }
  \mE_{{\mathcal U}_1}
\frac{ Z_{p_3}^{\m_1}  }{   \mE_{{\mathcal U}_1}  Z_{p_3}^{\m_1}   }
\nonumber \\
& & \times
 \frac{1}{Z_{p_3}}  (C_{p_3}^{(p_1)})^{s-1}\sum_{p_2=1}^{l}A_{p_3}^{(p_1,p_2)}\y_j^{(p_2)}\sqrt{t} \Bigg. \Bigg).
  \end{eqnarray}
After a further combination of $T_{G,j}^{d_2}$ from (\ref{eq:lev2genGanal3}) and (\ref{eq:lev2genGanal11}) we find
\begin{eqnarray}\label{eq:lev2genGanal12}
 \sum_{i_1=1}^{l} \sum_{i_2=1}^{l}\sum_{j=1}^{m}\beta_{i_1} \frac{T_{G,j}^{d_2}}{\sqrt{t}}
 & = & \sum_{i_1=1}^{l} \sum_{i_2=1}^{l}\sum_{j=1}^{m}\beta_{i_1} \frac{\mE_{G,{\mathcal U}_3}\lp \mE_{{\mathcal U}_2} \mE_{{\mathcal U}_1}  \Theta_{G,3}^{(2)}\rp}{\sqrt{t}} \nonumber\\
   & = & -s\beta^2\m_2 p \mE_{G,{\mathcal U}_3} \langle(\x^{(p_1)})^T\x^{(i_1)} (\y^{(p_2)})^T\y^{(i_2)} \rangle_{\gamma_{3}^{(2)}}.
 \end{eqnarray}
Finally, an additional combination of (\ref{eq:lev2genGanal4}), (\ref{eq:lev2genGanal6}), (\ref{eq:lev2genGanal9}), (\ref{eq:lev2genGanal9bb0}), and (\ref{eq:lev2genGanal12}) gives
\begin{eqnarray}\label{eq:lev2genGanal13}
 \sum_{i_1=1}^{l} \sum_{i_2=1}^{l}\sum_{j=1}^{m}\beta_{i_1} \frac{T_{G,j}}{\sqrt{t}}
 & = &  \sum_{i_1=1}^{l} \sum_{i_2=1}^{l}\sum_{j=1}^{m}\beta_{i_1} \frac{T_{G,j}^c+T_{G,j}^{d_1}+T_{G,j}^{e_1}+T_{G,j}^{d_2}}{\sqrt{t}} \nonumber \\
& = & \beta^2  \lp \mE_{G,{\mathcal U}_3} \langle \|\x^{(i_1)}\|_2^2\|\y^{(i_2)}\|_2^2\rangle_{\gamma_{01}^{(2)}} +   (s-1) \mE_{G,{\mathcal U}_3}\langle \|\x^{(i_1)}\|_2^2(\y^{(p_2)})^T\y^{(i_2)}\rangle_{\gamma_{02}^{(2)}}   \rp    \nonumber \\
 &  & -s\beta^2(1-\m_1) \mE_{G,{\mathcal U}_3}\langle (\x^{(p_1)})^T\x^{(i_1)}(\y^{(p_2)})^T\y^{(i_2)}\rangle_{\gamma_{1}^{(2)}} \nonumber \\
 &  & -s\beta^2(\m_1-\m_2) p \mE_{G,{\mathcal U}_3} \langle(\x^{(p_1)})^T\x^{(i_1)} (\y^{(p_2)})^T\y^{(i_2)} \rangle_{\gamma_{21}^{(2)}} \nonumber \\
 &  & +s\beta^2 \m_1 (p-1) \mE_{G,{\mathcal U}_3} \langle(\x^{(p_1)})^T\x^{(i_1)} (\y^{(p_2)})^T\y^{(i_2)} \rangle_{\gamma_{22}^{(2)}} \nonumber \\
  &  & -s\beta^2\m_2 p \mE_{G,{\mathcal U}_3} \langle(\x^{(p_1)})^T\x^{(i_1)} (\y^{(p_2)})^T\y^{(i_2)} \rangle_{\gamma_{3}^{(2)}}.
 \end{eqnarray}

\subsubsection{Connecting everything together -- second level}
\label{sec:lev2conalt}

 Similarly to what was done in Section \ref{sec:gencon}, we are now in position to connect together all the results obtained above. To that end we first utilize (\ref{eq:lev2genanal10e}) and (\ref{eq:lev2genanal10f}) to write
\begin{eqnarray}\label{eq:lev2ctp1}
\frac{d\psi(t)}{dt}  & = &       \frac{\mbox{sign}(s)}{2\beta\sqrt{n}} \lp \Omega_G+\Omega_1+\Omega_2+\Omega_3\rp,
\end{eqnarray}
with
\begin{eqnarray}\label{eq:lev2ctp2}
\Omega_G & = & \sum_{i_1=1}^{l}  \sum_{i_2=1}^{l}\sum_{j=1}^{m} \beta_{i_1}\frac{T_{G,j}}{\sqrt{t}}  \nonumber\\
\Omega_1 & = &
-\sum_{i_1=1}^{l}  \sum_{i_2=1}^{l} \sum_{j=1}^{m}\beta_{i_1}\frac{T_{3,1,j}}{\sqrt{1-t}}-\sum_{i_1=1}^{l}  \sum_{i_2=1}^{l} \sum_{j=1}^{m}\beta_{i_1}\frac{T_{2,1,j}}{\sqrt{1-t}}
-\sum_{i_1=1}^{l}  \sum_{i_2=1}^{l} \sum_{j=1}^{m}\beta_{i_1}\frac{T_{1,1,j}}{\sqrt{1-t}} \nonumber\\
\Omega_2 & = &
-\sum_{i_1=1}^{l}  \sum_{i_2=1}^{l}\beta_{i_1}\|\y^{(i_2)}\|_2\frac{T_{3,2}}{\sqrt{1-t}}
-\sum_{i_1=1}^{l}  \sum_{i_2=1}^{l}\beta_{i_1}\|\y^{(i_2)}\|_2\frac{T_{2,2}}{\sqrt{1-t}}
-\sum_{i_1=1}^{l}  \sum_{i_2=1}^{l}\beta_{i_1}\|\y^{(i_2)}\|_2\frac{T_{1,2}}{\sqrt{1-t}} \nonumber\\
\Omega_3 & = &
\sum_{i_1=1}^{l}  \sum_{i_2=1}^{l}\beta_{i_1}\|\y^{(i_2)}\|_2\frac{T_{3,3}}{\sqrt{t}}
\sum_{i_1=1}^{l}  \sum_{i_2=1}^{l}\beta_{i_1}\|\y^{(i_2)}\|_2\frac{T_{2,3}}{\sqrt{t}}
+ \sum_{i_1=1}^{l}  \sum_{i_2=1}^{l}\beta_{i_1}\|\y^{(i_2)}\|_2\frac{T_{1,3}}{\sqrt{t}}.
\end{eqnarray}
From (\ref{eq:lev2genGanal13}) we have
\begin{eqnarray}\label{eq:lev2cpt3}
\Omega_G & = & \beta^2 \lp \mE_{G,{\mathcal U}_3} \langle \|\x^{(i_1)}\|_2^2\|\y^{(i_2)}\|_2^2\rangle_{\gamma_{01}^{(2)}} +  (s-1) \mE_{G,{\mathcal U}_3}\langle \|\x^{(i_1)}\|_2^2(\y^{(p_2)})^T\y^{(i_2)}\rangle_{\gamma_{02}^{(2)}}   \rp    \nonumber \\
 &  & -s\beta^2(1-\m_1) \mE_{G,{\mathcal U}_3}\langle (\x^{(p_1)})^T\x^{(i_1)}(\y^{(p_2)})^T\y^{(i_2)}\rangle_{\gamma_{1}^{(2)}} \nonumber \\
 &  & -s\beta^2(\m_1-\m_2) \mE_{G,{\mathcal U}_3} \langle(\x^{(p_1)})^T\x^{(i_1)} (\y^{(p_2)})^T\y^{(i_2)} \rangle_{\gamma_{21}^{(2)}} \nonumber \\
 &  & +s\beta^2 \m_1 (p-1) \mE_{G,{\mathcal U}_3} \langle(\x^{(p_1)})^T\x^{(i_1)} (\y^{(p_2)})^T\y^{(i_2)} \rangle_{\gamma_{22}^{(2)}} \nonumber \\
  &  & -s\beta^2\m_2\mE_{G,{\mathcal U}_3} \langle(\x^{(p_1)})^T\x^{(i_1)} (\y^{(p_2)})^T\y^{(i_2)} \rangle_{\gamma_{3}^{(2)}}.
 \end{eqnarray}
From (\ref{eq:lev2liftgenAanal19i}), (\ref{eq:lev2genDanal25}), and (\ref{eq:lev2x3lev2genDanal25}) we have
\begin{eqnarray}\label{eq:lev2cpt4}
-\Omega_1
& = & (\p_0-\p_1)\beta^2 \lp \mE_{G,{\mathcal U}_3}\langle \|\x^{(i_1)}\|_2^2\|\y^{(i_2)}\|_2^2\rangle_{\gamma_{01}^{(2)}} +  (s-1)\mE_{G,{\mathcal U}_3}\langle \|\x^{(i_1)}\|_2^2(\y^{(p_2)})^T\y^{(i_2)}\rangle_{\gamma_{02}^{(2)}} \rp \nonumber \\
& & - (\p_0-\p_1)s\beta^2(1-\m_1)\mE_{G,{\mathcal U}_3}\langle \|\x^{(i_1)}\|_2\|\x^{(p_1)}\|_2(\y^{(p_2)})^T\y^{(i_2)} \rangle_{\gamma_{1}^{(2)}} \nonumber \\
&  & +(\p_1-\p_2)\beta^2\lp \mE_{G,{\mathcal U}_3}\langle \|\x^{(i_1)}\|_2^2\|\y^{(i_2)}\|_2^2\rangle_{\gamma_{01}^{(2)}} +  (s-1)\mE_{G,{\mathcal U}_3}\langle \|\x^{(i_1)}\|_2^2(\y^{(p_2)})^T\y^{(i_2)}\rangle_{\gamma_{02}^{(2)}} \rp\nonumber \\
& & - (\p_1-\p_2)s\beta^2(1-\m_1)\mE_{G,{\mathcal U}_3}\langle \|\x^{(i_1)}\|_2\|\x^{(p_1)}\|_2(\y^{(p_2)})^T\y^{(i_2)} \rangle_{\gamma_{1}^{(2)}}
\nonumber \\
 &   &
  -s\beta^2(\p_1-\p_2)(\m_1-\m_2) p \mE_{G,{\mathcal U}_3} \langle \|\x^{(i_1)}\|_2\|\x^{(p_1)}\|_2(\y^{(p_2)})^T\y^{(i_2)} \rangle_{\gamma_{21}^{(2)}} \nonumber \\
 &   &
  +s\beta^2(\p_1-\p_2) \m_1 (p-1) \mE_{G,{\mathcal U}_3} \langle \|\x^{(i_1)}\|_2\|\x^{(p_1)}\|_2(\y^{(p_2)})^T\y^{(i_2)} \rangle_{\gamma_{22}^{(2)}}   \nonumber \\
&  & +\p_2 \beta^2\lp \mE_{G,{\mathcal U}_3}\langle \|\x^{(i_1)}\|_2^2\|\y^{(i_2)}\|_2^2\rangle_{\gamma_{01}^{(2)}} +  (s-1)\mE_{G,{\mathcal U}_3}\langle \|\x^{(i_1)}\|_2^2(\y^{(p_2)})^T\y^{(i_2)}\rangle_{\gamma_{02}^{(2)}} \rp\nonumber \\
& & - \p_2s\beta^2(1-\m_1)\mE_{G,{\mathcal U}_3}\langle \|\x^{(i_1)}\|_2\|\x^{(p_1)}\|_2(\y^{(p_2)})^T\y^{(i_2)} \rangle_{\gamma_{1}^{(2)}}\nonumber \\
 &   &
  -s\beta^2\p_2(\m_1-\m_2) p \mE_{G,{\mathcal U}_3} \langle \|\x^{(i_1)}\|_2\|\x^{(p_1)}\|_2(\y^{(p_2)})^T\y^{(i_2)} \rangle_{\gamma_{21}^{(2)}}\nonumber \\
 &   &
  +s\beta^2\p_2 \m_1  (p-1) \mE_{G,{\mathcal U}_3} \langle \|\x^{(i_1)}\|_2\|\x^{(p_1)}\|_2(\y^{(p_2)})^T\y^{(i_2)} \rangle_{\gamma_{22}^{(2)}}\nonumber \\
 &  & -s\beta^2\p_2\m_2 p \mE_{G,{\mathcal U}_3} \langle \|\x^{(i_1)}\|_2\|\x^{(p_1)}\|_2(\y^{(p_2)})^T\y^{(i_2)} \rangle_{\gamma_{3}^{(2)}}
 \nonumber \\
 & = & \p_0 \beta^2\lp\mE_{G,{\mathcal U}_3}\langle \|\x^{(i_1)}\|_2^2\|\y^{(i_2)}\|_2^2\rangle_{\gamma_{01}^{(2)}} +  (s-1)\mE_{G,{\mathcal U}_3}\langle \|\x^{(i_1)}\|_2^2(\y^{(p_2)})^T\y^{(i_2)}\rangle_{\gamma_{02}^{(2)}} \rp\nonumber \\
& & - \p_0s\beta^2(1-\m_1)\mE_{G,{\mathcal U}_3}\langle \|\x^{(i_1)}\|_2\|\x^{(p_1)}\|_2(\y^{(p_2)})^T\y^{(i_2)} \rangle_{\gamma_{1}^{(2)}}\nonumber \\
 &   &
  -s\beta^2\p_1(\m_1-\m_2) p \mE_{G,{\mathcal U}_3} \langle \|\x^{(i_1)}\|_2\|\x^{(p_1)}\|_2(\y^{(p_2)})^T\y^{(i_2)} \rangle_{\gamma_{21}^{(2)}}\nonumber \\
 &   &
  +s\beta^2\p_1 \m_1 (p-1) \mE_{G,{\mathcal U}_3} \langle \|\x^{(i_1)}\|_2\|\x^{(p_1)}\|_2(\y^{(p_2)})^T\y^{(i_2)} \rangle_{\gamma_{22}^{(2)}}\nonumber \\
 &  & -s\beta^2\p_2\m_2p\mE_{G,{\mathcal U}_3} \langle \|\x^{(i_1)}\|_2\|\x^{(p_1)}\|_2(\y^{(p_2)})^T\y^{(i_2)} \rangle_{\gamma_{3}^{(2)}}.
   \end{eqnarray}
From (\ref{eq:lev2liftgenBanal20b}), (\ref{eq:lev2genEanal25}), and (\ref{eq:lev2x3lev2genEanal25}), we find
\begin{eqnarray}\label{eq:lev2cpt5}
-\Omega_2 & = & (\q_0-\q_1)\beta^2\lp\mE_{G,{\mathcal U}_3}\langle \|\x^{(i_1)}\|_2^2\|\y^{(i_2)}\|_2^2\rangle_{\gamma_{01}^{(2)}} +  (s-1)\mE_{G,{\mathcal U}_3}\langle \|\x^{(i_1)}\|_2^2 \|\y^{(i_2)}\|_2\|\y^{(p_2)}\|_2\rangle_{\gamma_{02}^{(2)}}\rp\nonumber \\
& & - (\q_0-\q_1)s\beta^2(1-\m_1)\mE_{G,{\mathcal U}_3}\langle (\x^{(p_1)})^T\x^{(i_1)}\|\y^{(i_2)}\|_2\|\y^{(p_2)}\|_2 \rangle_{\gamma_{1}^{(2)}}\nonumber \\
&  & +(\q_1-\q_2)\beta^2 \lp \mE_{G,{\mathcal U}_3}\langle \|\x^{(i_1)}\|_2^2\|\y^{(i_2)}\|_2^2\rangle_{\gamma_{01}^{(2)}} +   (s-1)\mE_{G,{\mathcal U}_3}\langle \|\x^{(i_1)}\|_2^2 \|\y^{(i_2)}\|_2\|\y^{(p_2)}\|_2\rangle_{\gamma_{02}^{(2)}}\rp
\nonumber \\
& & - (\q_1-\q_2)s\beta^2(1-\m_1)\mE_{G,{\mathcal U}_3}\langle (\x^{(p_1)})^T\x^{(i_1)}\|\y^{(i_2)}\|_2\|\y^{(p_2)}\|_2 \rangle_{\gamma_{1}^{(2)}} \nonumber \\
&  & -s\beta^2(\q_1-\q_2)(\m_1-\m_2) p \mE_{G,{\mathcal U}_2} \langle \|\y^{(i_2)}\|_2\|\y^{(p_2)}\|_2(\x^{(i_1)})^T\x^{(p_1)}\rangle_{\gamma_{21}^{(2)}} \nonumber \\
&  & +s\beta^2(\q_1-\q_2) \m_1  (p-1)  \mE_{G,{\mathcal U}_2} \langle \|\y^{(i_2)}\|_2\|\y^{(p_2)}\|_2(\x^{(i_1)})^T\x^{(p_1)}\rangle_{\gamma_{22}^{(2)}} \nonumber \\
&  & +\q_2\beta^2\lp\mE_{G,{\mathcal U}_3}\langle \|\x^{(i_1)}\|_2^2\|\y^{(i_2)}\|_2^2\rangle_{\gamma_{01}^{(2)}} +  (s-1)\mE_{G,{\mathcal U}_3}\langle \|\x^{(i_1)}\|_2^2 \|\y^{(i_2)}\|_2\|\y^{(p_2)}\|_2\rangle_{\gamma_{02}^{(2)}}\rp\nonumber \\
& & - \q_2s\beta^2(1-\m_1)\mE_{G,{\mathcal U}_3}\langle (\x^{(p_1)})^T\x^{(i_1)}\|\y^{(i_2)}\|_2\|\y^{(p_2)}\|_2 \rangle_{\gamma_{1}^{(2)}} \nonumber \\
&  & -s\beta^2\q_2(\m_1-\m_2) p \mE_{G,{\mathcal U}_3} \langle \|\y^{(i_2)}\|_2\|\y^{(p_2)}\|_2(\x^{(i_1)})^T\x^{(p_1)}\rangle_{\gamma_{21}^{(2)}} \nonumber \\
&  & +s\beta^2\q_2 \m_1 (p-1) \mE_{G,{\mathcal U}_3} \langle \|\y^{(i_2)}\|_2\|\y^{(p_2)}\|_2(\x^{(i_1)})^T\x^{(p_1)}\rangle_{\gamma_{22}^{(2)}} \nonumber \\
& & -s\beta^2 \q_2 \m_2 p \mE_{G,{\mathcal U}_3} \langle \|\y^{(i_2)}\|_2\|\y^{(p_2)}\|_2(\x^{(i_1)})^T\x^{(p_1)}\rangle_{\gamma_{3}^{(2)}} \nonumber \\
& = & \q_0 \beta^2 \lp \mE_{G,{\mathcal U}_3}\langle \|\x^{(i_1)}\|_2^2\|\y^{(i_2)}\|_2^2\rangle_{\gamma_{01}^{(2)}} +   (s-1)\mE_{G,{\mathcal U}_3}\langle \|\x^{(i_1)}\|_2^2 \|\y^{(i_2)}\|_2\|\y^{(p_2)}\|_2\rangle_{\gamma_{02}^{(2)}}\rp\nonumber \\
& & - \q_0s\beta^2(1-\m_1)\mE_{G,{\mathcal U}_3}\langle (\x^{(p_1)})^T\x^{(i_1)}\|\y^{(i_2)}\|_2\|\y^{(p_2)}\|_2 \rangle_{\gamma_{1}^{(2)}} \nonumber \\
&  & -s\beta^2\q_1(\m_1-\m_2) p\mE_{G,{\mathcal U}_3} \langle \|\y^{(i_2)}\|_2\|\y^{(p_2)}\|_2(\x^{(i_1)})^T\x^{(p_1)}\rangle_{\gamma_{21}^{(2)}} \nonumber \\
&  & +s\beta^2\q_1 \m_1 ( p -1 ) \mE_{G,{\mathcal U}_3} \langle \|\y^{(i_2)}\|_2\|\y^{(p_2)}\|_2(\x^{(i_1)})^T\x^{(p_1)}\rangle_{\gamma_{22}^{(2)}} \nonumber \\
& & -s\beta^2 \q_2 \m_2 p \mE_{G,{\mathcal U}_3} \langle \|\y^{(i_2)}\|_2\|\y^{(p_2)}\|_2(\x^{(i_1)})^T\x^{(p_1)}\rangle_{\gamma_{3}^{(2)}}.
\end{eqnarray}

From (\ref{eq:lev2liftgenCanal21b}) and (\ref{eq:lev2genFanal29}), we also have
  \begin{eqnarray}\label{eq:lev2cpt6}
\Omega_3
& = & (\p_0\q_0-\p_1\q_1)\beta^2 \lp \mE_{G,{\mathcal U}_3}\langle \|\x^{(i_1)}\|_2^2\|\y^{(i_2)}\|_2^2\rangle_{\gamma_{01}^{(2)}} +   (s-1)\mE_{G,{\mathcal U}_3}\langle \|\x^{(i_1)}\|_2^2 \|\y^{(i_2)}\|_2\|\y^{(p_2)}\|_2\rangle_{\gamma_{02}^{(2)}}\rp\nonumber \\
& & - (\p_0\q_0-\p_1\q_1)s\beta^2(1-\m_1)\mE_{G,{\mathcal U}_3}\langle \|\x^{(i_1)}\|_2\|\x^{(p_`)}\|_2\|\y^{(i_2)}\|_2\|\y^{(p_2)}\|_2 \rangle_{\gamma_{1}^{(2)}} \nonumber \\
&  &
+(\p_1\q_1-\p_2\q_2)\beta^2\lp\mE_{G,{\mathcal U}_3}\langle \|\x^{(i_1)}\|_2^2\|\y^{(i_2)}\|_2^2\rangle_{\gamma_{01}^{(2)}} +   (s-1)\mE_{G,{\mathcal U}_3}\langle \|\x^{(i_1)}\|_2^2 \|\y^{(i_2)}\|_2\|\y^{(p_2)}\|_2\rangle_{\gamma_{02}^{(2)}}\rp\nonumber \\
& & - (\p_1\q_1-\p_2\q_2) s\beta^2(1-\m_1)\mE_{G,{\mathcal U}_3}\langle \|\x^{(i_1)}\|_2\|\x^{(p_`)}\|_2\|\y^{(i_2)}\|_2\|\y^{(p_2)}\|_2 \rangle_{\gamma_{1}^{(2)}} \nonumber \\
&  & -s\beta^2(\p_1\q_1-\p_2\q_2)(\m_1-\m_2) p \mE_{G,{\mathcal U}_3} \langle\|\x^{(i_2)}\|_2\|\x^{(p_2)}\|_2\|\y^{(i_2)}\|_2\|\y^{(p_2)}\rangle_{\gamma_{21}^{(2)}}
\nonumber \\
&  & +s\beta^2(\p_1\q_1-\p_2\q_2) \m_1 (p-1) \mE_{G,{\mathcal U}_3} \langle\|\x^{(i_2)}\|_2\|\x^{(p_2)}\|_2\|\y^{(i_2)}\|_2\|\y^{(p_2)}\rangle_{\gamma_{22}^{(2)}}
\nonumber \\
&  &
+\p_2\q_2\beta^2\lp\mE_{G,{\mathcal U}_3}\langle \|\x^{(i_1)}\|_2^2\|\y^{(i_2)}\|_2^2\rangle_{\gamma_{01}^{(2)}} +  (s-1)\mE_{G,{\mathcal U}_3}\langle \|\x^{(i_1)}\|_2^2 \|\y^{(i_2)}\|_2\|\y^{(p_2)}\|_2\rangle_{\gamma_{02}^{(2)}}\rp\nonumber \\
& & -  \p_2\q_2 s\beta^2(1-\m_1)\mE_{G,{\mathcal U}_3}\langle \|\x^{(i_1)}\|_2\|\x^{(p_`)}\|_2\|\y^{(i_2)}\|_2\|\y^{(p_2)}\|_2 \rangle_{\gamma_{1}^{(2)}} \nonumber \\
&  & -s\beta^2 \p_2\q_2(\m_1-\m_2) p \mE_{G,{\mathcal U}_3} \langle\|\x^{(i_2)}\|_2\|\x^{(p_2)}\|_2\|\y^{(i_2)}\|_2\|\y^{(p_2)}\rangle_{\gamma_{21}^{(2)}} \nonumber \\
&  & +s\beta^2 \p_2\q_2 \m_1 (p-1) \mE_{G,{\mathcal U}_3} \langle\|\x^{(i_2)}\|_2\|\x^{(p_2)}\|_2\|\y^{(i_2)}\|_2\|\y^{(p_2)}\rangle_{\gamma_{22}^{(2)}} \nonumber \\
 &  & -s\beta^2 \p_2\q_2 \m_2 p \mE_{G,{\mathcal U}_3} \langle\|\x^{(i_2)}\|_2\|\x^{(p_2)}\|_2\|\y^{(i_2)}\|_2\|\y^{(p_2)}\rangle_{\gamma_{3}^{(2)}}
 \nonumber\\
 & = &
\p_0\q_0\beta^2\lp\mE_{G,{\mathcal U}_3}\langle \|\x^{(i_1)}\|_2^2\|\y^{(i_2)}\|_2^2\rangle_{\gamma_{01}^{(2)}} +   (s-1)\mE_{G,{\mathcal U}_3}\langle \|\x^{(i_1)}\|_2^2 \|\y^{(i_2)}\|_2\|\y^{(p_2)}\|_2\rangle_{\gamma_{02}^{(2)}}\rp\nonumber \\
& & -  \p_0\q_0 s\beta^2(1-\m_1)\mE_{G,{\mathcal U}_3}\langle \|\x^{(i_1)}\|_2\|\x^{(p_`)}\|_2\|\y^{(i_2)}\|_2\|\y^{(p_2)}\|_2 \rangle_{\gamma_{1}^{(2)}} \nonumber \\
&  & -s\beta^2 \p_1\q_1(\m_1-\m_2) p \mE_{G,{\mathcal U}_3} \langle\|\x^{(i_2)}\|_2\|\x^{(p_2)}\|_2\|\y^{(i_2)}\|_2\|\y^{(p_2)}\rangle_{\gamma_{21}^{(2)}} \nonumber \\
&  & +s\beta^2 \p_1\q_1 \m_1 ( p -1) \mE_{G,{\mathcal U}_3} \langle\|\x^{(i_2)}\|_2\|\x^{(p_2)}\|_2\|\y^{(i_2)}\|_2\|\y^{(p_2)}\rangle_{\gamma_{21}^{(2)}} \nonumber \\
 &  & -s\beta^2 \p_2\q_2 \m_2 p \mE_{G,{\mathcal U}_3} \langle\|\x^{(i_2)}\|_2\|\x^{(p_2)}\|_2\|\y^{(i_2)}\|_2\|\y^{(p_2)}\rangle_{\gamma_{3}^{(2)}}.
\end{eqnarray}
Finally, a combination of (\ref{eq:lev2ctp1}) and (\ref{eq:lev2cpt3})-(\ref{eq:lev2cpt6}) gives
\begin{eqnarray}\label{eq:lev2cpt7}
\frac{d\psi(\calX,\calY,\q,\m,\beta,s,t)}{dt}  & = &       \frac{\mbox{sign}(s)\beta}{2\sqrt{n}} \lp \phi_1^{(2)}+\phi_{21}^{(2)}+\phi_{22}^{(2)}+\phi_3^{(2)} +\phi_{01}^{(2)} +\phi_{02}^{(2)} \rp ,
 \end{eqnarray}
where
\begin{eqnarray}\label{eq:lev2cpt8}
\phi_1^{(2)} & = &
-s(1-\m_1)\mE_{G,{\mathcal U}_3} \langle (\p_0\|\x^{(i_1)}\|_2\|\x^{(p_1)}\|_2 -(\x^{(p_1)})^T\x^{(i_1)})(\q_0\|\y^{(i_2)}\|_2\|\y^{(p_2)}\|_2 -(\y^{(p_2)})^T\y^{(i_2)})\rangle_{\gamma_{1}^{(2)}} \nonumber \\
\phi_{21}^{(2)} & = &
-s(\m_1-\m_2) p \mE_{G,{\mathcal U}_3} \langle (\p_1\|\x^{(i_1)}\|_2\|\x^{(p_1)}\|_2 -(\x^{(p_1)})^T\x^{(i_1)})(\q_1\|\y^{(i_2)}\|_2\|\y^{(p_2)}\|_2 -(\y^{(p_2)})^T\y^{(i_2)})\rangle_{\gamma_{21}^{(2)}} \nonumber \\
\phi_{22}^{(2)} & = &
s \m_1  (p-1) \mE_{G,{\mathcal U}_3} \langle (\p_1\|\x^{(i_1)}\|_2\|\x^{(p_1)}\|_2 -(\x^{(p_1)})^T\x^{(i_1)})(\q_1\|\y^{(i_2)}\|_2\|\y^{(p_2)}\|_2 -(\y^{(p_2)})^T\y^{(i_2)})\rangle_{\gamma_{22}^{(2)}} \nonumber \\
\phi_3^{(2)} & = &
-s\m_2 p \mE_{G,{\mathcal U}_3} \langle (\p_2\|\x^{(i_1)}\|_2\|\x^{(p_1)}\|_2 -(\x^{(p_1)})^T\x^{(i_1)})(\q_2\|\y^{(i_2)}\|_2\|\y^{(p_2)}\|_2 -(\y^{(p_2)})^T\y^{(i_2)})\rangle_{\gamma_{3}^{(2)}} \nonumber \\
\phi_{01}^{(2)} & = & (1-\p_0)(1-\q_0)\mE_{G,{\mathcal U}_3}\langle \|\x^{(i_1)}\|_2^2\|\y^{(i_2)}\|_2^2\rangle_{\gamma_{01}^{(2)}} \nonumber\\
\phi_{02}^{(2)} & = & (1-\p_0)\mE_{G,{\mathcal U}_3}\left\langle \|\x^{(i_1)}\|_2^2 \lp\q_0\|\y^{(i_2)}\|_2\|\y^{(p_2)}\|_2-(\y^{(p_2)})^T\y^{(i_2)}\rp\right\rangle_{\gamma_{02}^{(2)}}. \end{eqnarray}

The above results are summarized in the following proposition.
\begin{proposition}
\label{thm:thm2}
 For $k=\{1,2,3\}$ let components of $G\in\mR^{m \times n},u^{(4,k)}\in\mR^1,\u^{(2,k)}\in\mR^{m\times 1}$, and $\h^{(k)}\in\mR^{n\times 1}$ be independent  zero-mean Gaussians. For vectors $\m=[\m_1,\m_2]$, $\p=[\p_0,\p_1,\p_2,\p_3]$ with $\p_0\geq \p_1\geq \p_2\geq \p_3= 0$ and $\q=[\q_0,\q_1,\q_2,\q_3]$ with $\q_0\geq \q_1\geq \q_2\geq \q_3= 0$, let the variances of the elements of $G$, $u^{(4,k)}$, $\u^{(2,k)}$, and $\h^{(k)}$ be $1$, $\p_{k-1}\q_{k-1}-\p_k\q_k$, $\p_{k-1}-\p_{k}$, $\q_{k-1}-\q_{k}$, respectively.
Let sets ${\mathcal X}, \bar{{\mathcal X}}, {\mathcal Y}$, scalars $\beta$, $p$, $s$, and function $f_{\bar{\x}^{(i_3)}}(\cdot)$ be as in Proposition
\ref{thm:thm1}. Set ${\mathcal U}_k=[u^{(4,k)},\u^{(2,k)},\h^{(2k)}]$
 and consider the following
\begin{equation}\label{eq:thm2eq1}
\psi(t)  =  \mE_{G,{\mathcal U}_3} \frac{1}{p|s|\sqrt{n}\m_2} \log \lp \mE_{{\mathcal U}_2}\lp\lp  \sum_{i_3=1}^{l}\lp \mE_{{\mathcal U}_1}  Z_{i_3}^{\m_1}\rp^p\rp^{\frac{\m_2}{\m_1}}\rp\rp,
\end{equation}
where
\begin{eqnarray}\label{eq:thm2eq2}
Z_{i_3} & \triangleq & \sum_{i_1=1}^{l}\lp\sum_{i_2=1}^{l}e^{\beta D_0^{(i_1,i_2,i_3)}} \rp^{s} \nonumber \\
 D_0^{(i_1,i_2,i_3)} & \triangleq & \sqrt{t}(\y^{(i_2)})^T
 G\x^{(i_1)}+\sqrt{1-t}\|\x^{(i_1)}\|_2 (\y^{(i_2)})^T(\u^{(2,1)}+\u^{(2,2)}+\u^{(2,3)})\nonumber \\
 & & +\sqrt{t}\|\x^{(i_1)}\|_2\|\y^{(i_2)}\|_2(u^{(4,1)}+u^{(4,2)}+u^{(4,3)}) +\sqrt{1-t}\|\y^{(i_2)}\|_2(\h^{(1)}+\h^{(2)}+\h^{(3)})^T\x^{(i_1)} \nonumber \\
 & & + f_{\bar{\x}^{(i_3)}} (\x^{(i_1)}).
 \end{eqnarray}
Then
\begin{eqnarray}\label{eq:prop1eq3}
\frac{d\psi(t)}{dt}  & = &        \frac{\mbox{sign}(s)\beta}{2\sqrt{n}} \lp \phi_1^{(2)}+\phi_{21}^{(2)}+\phi_{22}^{(2)}+\phi_3^{(2)} +\phi_{01}^{(2)} +\phi_{02}^{(2)} \rp,
 \end{eqnarray}
where $\phi$'s are as in (\ref{eq:lev2cpt8}) and $\gamma$'s are as defined in (\ref{eq:lev2genAanal19d2}) and (\ref{eq:lev2genAanal19e}).
\end{proposition}
\begin{proof}
  Follows from the discussion presented above.
\end{proof}

\section{Move to the $r$-th ($r\geq 3$) level of full lifting}
\label{sec:rthlev}

After introducing key conceptual ingredients in Section \ref{sec:gencon}, we  in Section \ref{sec:seclev} formalized them into a sequence of mathematical steps needed to proceed to second and higher levels of lifting. In this section we present the final results obtained via applying the procedure from Section \ref{sec:seclev} to any level of lifting $r\in\mN$. The presentation of  \cite{Stojnicnflgscompyx23}'s Section 4
 turns out to be of great help in that regard. In fact, as will soon be clear, many of the results obtained therein apply here either in an unaltered fashion or with fairly minimal adjustments.

As in \cite{Stojnicnflgscompyx23}'s Section 4, we now take $r\geq 3$ and  consider vectors $\m=[\m_1,\m_2,...,\m_r]$,
$\p=[\p_0,\p_1,...,\p_r,\p_{r+1}]$ with $\p_0\geq \p_1\geq \p_2\geq \dots \geq \p_r\geq\p_{r+1} = 0$ and $\q=[\q_0,\q_1,\q_2,\dots,\q_r,\q_{r+1}]$ with $\q_0\geq \q_1\geq \q_2\geq \dots \geq \q_r\geq \q_{r+1} = 0$. Also, we assume that the elements of $G\in\mR^{m \times n}$ and  $u^{(4,k)}\in\mR^1,\u^{(2,k)}\in\mR^{m\times 1}$, and $\h^{(k)}\in\mR^{n\times 1}$ (for any $k\in\{1,2,\dots,r+1\}$)  are independent  zero-mean normals. Moreover, he variances of the components of $G$, $u^{(4,k)}$, $\u^{(2,k)}$, and $\h^{(k)}$ are assumed to be $1$, $\p_{k-1}\q_{k-1}-\p_k\q_k$, $\p_{k-1}-\p_{k}$, $\q_{k-1}-\q_{k}$, respectively. Similarly to what was done earlier, we also set ${\mathcal U}_k\triangleq [u^{(4,k)},\u^{(2,k)},\h^{(2k)}]$ and assume that sets ${\mathcal X}$, $\bar{{\mathcal X}}$, ${\mathcal Y}$, scalars  $\beta$, $p$, $s$, and function $f_{\bar{\x}^{(i_3)}}(\cdot)$ are as in Proposition \ref{thm:thm1}. We then focus on the following function
\begin{equation}\label{eq:rthlev2genanal3}
\psi(t)  =  \mE_{G,{\mathcal U}_{r+1}} \frac{1}{\beta|s|\sqrt{n}\m_r} \log
\lp \mE_{{\mathcal U}_{r}} \lp \dots \lp \mE_{{\mathcal U}_2}\lp\lp \sum_{i_3=1}^{l} \lp \mE_{{\mathcal U}_1}  Z_{i_3}^{\m_1}\rp^p \rp^{\frac{\m_2}{\m_1}}\rp\rp^{\frac{\m_3}{\m_2}} \dots \rp^{\frac{\m_{r}}{\m_{r-1}}}\rp,
\end{equation}
where
\begin{eqnarray}\label{eq:rthlev2genanal3a}
Z_{i_3} & \triangleq & \sum_{i_1=1}^{l}\lp\sum_{i_2=1}^{l}e^{\beta D_0^{(i_1,i_2,i_3)}} \rp^{s} \nonumber \\
 D_0^{(i_1,i_2,i_3)} & \triangleq & \sqrt{t}(\y^{(i_2)})^T
 G\x^{(i_1)}+\sqrt{1-t}\|\x^{(i_1)}\|_2 (\y^{(i_2)})^T\lp\sum_{k=1}^{r+1}\u^{(2,k)}\rp\nonumber \\
 & & +\sqrt{t}\|\x^{(i_1)}\|_2\|\y^{(i_2)}\|_2\lp\sum_{k=1}^{r+1}u^{(4,k)}\rp +\sqrt{1-t}\|\y^{(i_2)}\|_2\lp\sum_{k=1}^{r+1}\h^{(k)}\rp^T\x^{(i_1)}
  + f_{\bar{\x}^{(i_3)}}(\x^{(i_3)}). \nonumber \\
 \end{eqnarray}
Analogously to (\ref{eq:genanal4}), we find it convenient to set
\begin{eqnarray}\label{eq:rthlev2genanal4}
\bar{\u}^{(i_1,1)} & =  & \frac{G\x^{(i_1)}}{\|\x^{(i_1)}\|_2} \nonumber \\
\bar{\u}^{(i_1,3,k)} & =  & \frac{(\h^{(k)})^T\x^{(i_1)}}{\|\x^{(i_1)}\|_2},
\end{eqnarray}
and recall that
\begin{eqnarray}\label{eq:rthlev2genanal5}
\bar{\u}_j^{(i_1,1)} & =  & \frac{G_{j,1:n}\x^{(i_1)}}{\|\x^{(i_1)}\|_2},1\leq j\leq m.
\end{eqnarray}
We observe that the elements of $\bar{\u}^{(i_1,1)}$ are i.i.d. standard normals, whereas the elements of $\u^{(2,k)}$  are zero-mean independent Gaussians with variances $\p_{k-1}-\p_k$ and the elements of $\u^{(i_1,3,k)}$ are zero-mean  independent  Gaussians with variances $\q_{k-1}-\q_k$. One can then rewrite (\ref{eq:rthlev2genanal3}) in the following way
\begin{equation}\label{eq:rthlev2genanal6}
\psi(t)  =  \mE_{G,{\mathcal U}_{r+1}} \frac{1}{p|s|\sqrt{n}\m_r} \log
\lp \mE_{{\mathcal U}_{r}} \lp \dots \lp \mE_{{\mathcal U}_2}\lp\lp  \sum_{i_3=1}^{l} \lp \mE_{{\mathcal U}_1} Z_{i_3}^{\m_1}\rp^p \rp^{\frac{\m_2}{\m_1}}\rp\rp^{\frac{\m_3}{\m_2}} \dots \rp^{\frac{\m_{r}}{\m_{r-1}}}\rp,
\end{equation}
where $\beta_{i_1}=\beta\|\x^{(i_1)}\|_2$ and now
\begin{eqnarray}\label{eq:rthlev2genanal7}
B^{(i_1,i_2)} & \triangleq &  \sqrt{t}(\y^{(i_2)})^T\bar{\u}^{(i_1,1)}+\sqrt{1-t} (\y^{(i_2)})^T\lp \sum_{k=1}^{r+1}\u^{(2,k)} \rp \nonumber \\
D^{(i_1,i_2,i_3)} & \triangleq &  B^{(i_1,i_2)}+\sqrt{t}\|\y^{(i_2)}\|_2 \lp \sum_{k=1}^{r+1}u^{(4,k)}\rp+\sqrt{1-t}\|\y^{(i_2)}\|_2 \lp \sum_{k=1}^{r+1}\u^{(i_1,3,k)}  \rp   + f_{\bar{\x}^{(i_3)}} (\x^{(i_1)})  \nonumber \\
A_{i_3}^{(i_1,i_2)} & \triangleq &  e^{\beta_{i_1}D^{(i_1,i_2,i_3)}}\nonumber \\
C_{i_3}^{(i_1)} & \triangleq &  \sum_{i_2=1}^{l}A_{i_3}^{(i_1,i_2)}\nonumber \\
Z_{i_3} & \triangleq & \sum_{i_1=1}^{l} \lp \sum_{i_2=1}^{l} A_{i_3}^{(i_1,i_2)}\rp^s =\sum_{i_1=1}^{l}  (C_{i_3}^{(i_1)})^s.
\end{eqnarray}
We set $\m_0=1$ and  analogously to \cite{Stojnicnflgscompyx23}'s (199)
\begin{eqnarray}\label{eq:rthlev2genanal7a}
\zeta_r\triangleq \mE_{{\mathcal U}_{r}} \lp \dots \lp \mE_{{\mathcal U}_2}\lp\lp  \sum_{i_3=1}^{l} \lp \mE_{{\mathcal U}_1}  Z_{i_3}^{\frac{\m_1}{\m_0}}\rp^p \rp^{\frac{\m_2}{\m_1}}\rp\rp^{\frac{\m_3}{\m_2}} \dots \rp^{\frac{\m_{r}}{\m_{r-1}}}, r\geq 2.
\end{eqnarray}
Then one has
\begin{eqnarray}\label{eq:rthlev2genanal7b}
\zeta_k = \mE_{{\mathcal U}_{k}} \lp  \zeta_{k-1} \rp^{\frac{\m_{k}}{\m_{k-1}}}, k\geq 2,\quad \mbox{and} \quad
\zeta_1=  \sum_{i_3=1}^{l} \lp \mE_{{\mathcal U}_1}  Z_{i_3}^{\frac{\m_1}{\m_0}}\rp^p.
\end{eqnarray}
Utilizing  \cite{Stojnicnflgscompyx23}'s (201) we write
\begin{eqnarray}\label{eq:rthlev2genanal9}
\frac{d\psi(t)}{dt}
 & = &  \mE_{G,{\mathcal U}_{r+1}} \frac{1}{p|s|\sqrt{n}\m_1\zeta_r}
\prod_{k=r}^{2}\mE_{{\mathcal U}_{k}} \zeta_{k-1}^{\frac{\m_k}{\m_{k-1}}-1}
 \frac{d\zeta_{1}}{dt} \nonumber \\
& = &   \mE_{G,{\mathcal U}_{r+1}} \frac{1}{|s|\sqrt{n}\m_1\zeta_r}
\prod_{k=r}^{2}\mE_{{\mathcal U}_{k}} \zeta_{k-1}^{\frac{\m_k}{\m_{k-1}}-1}
\sum_{i_3=1}^{l} \lp \mE_{{\mathcal U}_1} Z_{i_3}^{\m_1}  \rp^{p-1}
 \frac{d \mE_{{\mathcal U}_1} Z_{i_3}^{\m_1} }{dt} \nonumber \\
 & = &
\mE_{G,{\mathcal U}_{r+1}} \frac{1}{|s|\sqrt{n}\zeta_r}
\prod_{k=r}^{2}\mE_{{\mathcal U}_{k}} \zeta_{k-1}^{\frac{\m_k}{\m_{k-1}}-1}
\sum_{i_3=1}^{l} \lp \mE_{{\mathcal U}_1} Z_{i_3}^{\m_1}  \rp^{p-1}
 \mE_{{\mathcal U}_1} \frac{1}{Z_{i_3}^{1-\m_1}}  \sum_{i=1}^{l} (C_{i_3}^{(i_1)})^{s-1} \nonumber \\
& & \times \sum_{i_2=1}^{l}\beta_{i_1}A_{i_3}^{(i_1,i_2)}\frac{dD^{(i_1,i_2,i_3)}}{dt},
\end{eqnarray}
with product running in an \emph{index decreasing order} and
\begin{eqnarray}\label{eq:rthlev2genanal9a}
\frac{dD^{(i_1,i_2,i_3)}}{dt}= \lp \frac{dB^{(i_1,i_2)}}{dt}+\frac{\|\y^{(i_2)}\|_2 (\sum_{k=1}^{r+1} u^{(4,k)})}{2\sqrt{t}}-\frac{\|\y^{(i_2)}\|_2 (\sum_{k=1}^{r+1}\u^{(i_1,3,k)})}{2\sqrt{1-t}}\rp.
\end{eqnarray}
As shown in \cite{Stojnicnflgscompyx23}'s (203), we have
\begin{eqnarray}\label{eq:rthlev2genanal10}
\frac{dB^{(i_1,i_2)}}{dt}
 & = &
\sum_{j=1}^{m}\lp \frac{\y_j^{(i_2)}\u_j^{(i_1,1)}}{2\sqrt{t}}-\frac{\y_j^{(i_2)} \sum_{k=1}^{r+1} \u_j^{(2,k)}}{2\sqrt{1-t}}\rp.
\end{eqnarray}
and then  analogously to (\ref{eq:genanal10e}) (and \cite{Stojnicnflgscompyx23}'s (203)-(205))
\begin{equation}\label{eq:rthlev2genanal10e}
\frac{d\psi(t)}{dt}  =       \frac{\mbox{sign}(s)}{2\beta\sqrt{n}} \sum_{i_1=1}^{l}  \sum_{i_2=1}^{l}
\beta_{i_1}\lp T_G + \sum_{k=1}^{r+1}T_k\rp,
\end{equation}
where
\begin{eqnarray}\label{eq:rthlev2genanal10f}
T_G & = & \sum_{j=1}^{m}\frac{T_{G,j}}{\sqrt{t}}  \nonumber\\
T_k & = & -\sum_{j=1}^{m}\frac{T_{k,1,j}}{\sqrt{1-t}}-\|\y^{(i_2)}\|_2\frac{T_{k,2}}{\sqrt{1-t}}+\|\y^{(i_2)}\|_2\frac{T_{k,3}}{\sqrt{t}}, k\in\{1,2,\dots,r+1\}.
 \end{eqnarray}
 \begin{eqnarray}\label{eq:rthlev2genanal10g}
T_{G,j} & = &  \mE_{G,{\mathcal U}_{r+1}} \lp
\zeta_r^{-1}\prod_{v=r}^{2}\mE_{{\mathcal U}_{v}} \zeta_{v-1}^{\frac{\m_v}{\m_{v-1}}-1}
\sum_{i_3=1}^{l} \lp \mE_{{\mathcal U}_1} Z_{i_3}^{\m_1}  \rp^{p-1}
  \mE_{{\mathcal U}_1}\frac{(C_{i_3}^{(i_1)})^{s-1} A_{i_3}^{(i_1,i_2)} \y_j^{(i_2)}\u_j^{(i_1,1)}}{Z_{i_3}^{1-\m_1}} \rp \nonumber \\
T_{k,1,j} & = &   \mE_{G,{\mathcal U}_{r+1}} \lp
\zeta_r^{-1}\prod_{v=r}^{2}\mE_{{\mathcal U}_{v}} \zeta_{v-1}^{\frac{\m_v}{\m_{v-1}}-1}
\sum_{i_3=1}^{l} \lp \mE_{{\mathcal U}_1} Z_{i_3}^{\m_1}  \rp^{p-1}
  \mE_{{\mathcal U}_1}\frac{(C_{i_3}^{(i_1)})^{s-1} A_{i_3}^{(i_1,i_2)} \y_j^{(i_2)}\u_j^{(2,k)}}{Z_{i_3}^{1-\m_1}} \rp \nonumber \\
T_{k,2} & = &   \mE_{G,{\mathcal U}_{r+1}} \lp
\zeta_r^{-1}\prod_{v=r}^{2}\mE_{{\mathcal U}_{v}} \zeta_{v-1}^{\frac{\m_v}{\m_{v-1}}-1}
\sum_{i_3=1}^{l} \lp \mE_{{\mathcal U}_1} Z_{i_3}^{\m_1}  \rp^{p-1}
  \mE_{{\mathcal U}_1}\frac{(C_{i_3}^{(i_1)})^{s-1} A_{i_3}^{(i_1,i_2)} \u^{(i_1,3,k)}}{Z_{i_3}^{1-\m_1}} \rp \nonumber \\
T_{k,3} & = &  \mE_{G,{\mathcal U}_{r+1}} \lp
\zeta_r^{-1}\prod_{v=r}^{2}\mE_{{\mathcal U}_{v}} \zeta_{v-1}^{\frac{\m_v}{\m_{v-1}}-1}
\sum_{i_3=1}^{l} \lp \mE_{{\mathcal U}_1} Z_{i_3}^{\m_1}  \rp^{p-1}
  \mE_{{\mathcal U}_1}\frac{(C_{i_3}^{(i_1)})^{s-1} A_{i_3}^{(i_1,i_2)} u^{(4,k)}}{Z_{i_3}^{1-\m_1}} \rp.
\end{eqnarray}

We first discuss the three sequences $\lp T_{k,1,j}\rp_{k=1:r+1}$, $\lp T_{k,2}\rp_{k=1:r+1}$, and $\lp T_{k,3}\rp_{k=1:r+1}$  and then $T_{G,j}$. As Section \ref{sec:seclev} showed, the internal relations among the components in any of these sequences are for all that we need identical. It will therefore be sufficient to take one of them, say $\lp T_{k,1,j}\rp_{k=1:r+1}$, and show how the results of Section \ref{sec:seclev} immediately extend from $r=2$ to any integer $r>2$. For the remaining two, $\lp T_{k,2}\rp_{k=1:r+1}$, and $\lp T_{k,3}\rp_{k=1:r+1}$, we avoid unnecessarily repeating the same strategy and instead  quickly deduce the final results.

\subsection{Successive scaling and canceling out}
\label{sec:scaledcanout}

Following the strategy of \cite{Stojnicnflgscompyx23}, we consider sequence $\lp T_{k,1,j}\rp_{k=1:r+1}$, take one component, say $T_{k_1,1,j}$, $k_1\geq 2$, and show how it relates to the very next one, $T_{k_1+1,1,j}$. Two aspects of such a recursive relation play a key role: 1) \emph{successive scaling}; and 2) appearance of an appropriately \emph{reweightedly averaged (over a $\gamma$ measure) new term}.

Utilizing \cite{Stojnicnflgscompyx23}'s (207)-(209), we have
 \begin{eqnarray}\label{eq:rthlev2genanal11}
 T_{k_1,1,j}   \hspace{-.06in}  & = &  \hspace{-.06in}  (\p_{k_1-1}-\p_{k_1}) \mE_{G,{\mathcal U}_{r+1}} \Bigg(\Bigg.
\zeta_r^{-1}\prod_{v=r}^{k_1+1}\mE_{{\mathcal U}_{v}} \zeta_{v-1}^{\frac{\m_v}{\m_{v-1}}-1}
\mE_{{\mathcal U}_{k_1}}  \nonumber \\
& &  \hspace{-.06in} \times
\frac{d}{d\u_j^{(2,k_1)}} \lp \zeta_{k_1-1}^{\frac{\m_{k_1}}{\m_{k_1-1}}-1}\prod_{v=k_1-1}^{2}\mE_{{\mathcal U}_{v}} \zeta_{v-1}^{\frac{\m_v}{\m_{v-1}}-1}
\sum_{i_3=1}^{l} \lp \mE_{{\mathcal U}_1} Z_{i_3}^{\m_1}  \rp^{p-1}
  \mE_{{\mathcal U}_1}\frac{(C_{i_3}^{(i_1)})^{s-1} A_{i_3}^{(i_1,i_2)} \y_j^{(i_2)}\u_j^{(2,k_1)}}{Z_{i_3}^{1-\m_1}} \rp \Bigg.\Bigg)  , \nonumber \\
\end{eqnarray}
and also
 \begin{eqnarray}\label{eq:rthlev2genanal12}
 T_{k_1+1,1,j}
       & = & (\p_{k_1} -\p_{k_1+1})
  \mE_{G,{\mathcal U}_{r+1}} \Bigg( \Bigg.
\zeta_r^{-1}\prod_{v=r}^{k_1+2}\mE_{{\mathcal U}_{v}} \zeta_{v-1}^{\frac{\m_v}{\m_{v-1}}-1}
\mE_{{\mathcal U}_{k_1+1}}  \nonumber \\
& & \times  \frac{d}{d\u_j^{(2,k_1+1)}}  \lp \zeta_{k_1}^{\frac{\m_{k_1+1}}{\m_{k_1}}-1}
\prod_{v=k_1}^{2}\mE_{{\mathcal U}_{v}} \zeta_{v-1}^{\frac{\m_v}{\m_{v-1}}-1}
\sum_{i_3=1}^{l} \lp \mE_{{\mathcal U}_1} Z_{i_3}^{\m_1}  \rp^{p-1}
  \mE_{{\mathcal U}_1}\frac{(C_{i_3}^{(i_1)})^{s-1} A_{i_3}^{(i_1,i_2)} \y_j^{(i_2)}}{Z_{i_3}^{1-\m_1}} \rp   \Bigg. \Bigg)   \nonumber \\
& = &  \Pi_{k_1+1,1,j}^{(1)}
+  \Pi_{k_1+1,1,j}^{(2)}, \nonumber \\
 \end{eqnarray}
where
 \begin{eqnarray}\label{eq:rthlev2genanal13}
 \Pi_{k_1+1,1,j}^{(1)}
      & = & (\p_{k_1}-\p_{k_1+1})
  \mE_{G,{\mathcal U}_{r+1}} \Bigg( \Bigg.
\zeta_r^{-1}\prod_{v=r}^{k_1+2}\mE_{{\mathcal U}_{v}} \zeta_{v-1}^{\frac{\m_v}{\m_{v-1}}-1}
\mE_{{\mathcal U}_{k_1+1}}  \nonumber \\
& & \times \zeta_{k_1}^{\frac{\m_{k_1+1}}{\m_{k_1}}-1}
 \frac{d}{d\u_j^{(2,k_1+1)}}  \lp
\prod_{v=k_1}^{2}\mE_{{\mathcal U}_{v}} \zeta_{v-1}^{\frac{\m_v}{\m_{v-1}}-1}
\sum_{i_3=1}^{l} \lp \mE_{{\mathcal U}_1} Z_{i_3}^{\m_1}  \rp^{p-1}
  \mE_{{\mathcal U}_1}\frac{(C_{i_3}^{(i_1)})^{s-1} A_{i_3}^{(i_1,i_2)} \y_j^{(i_2)}}{Z_{i_3}^{1-\m_1}} \rp   \Bigg. \Bigg)   \nonumber \\
\Pi_{k_1+1,1,j}^{(2)}
& = &(\p_{k_1}-\p_{k_1+1})
  \mE_{G,{\mathcal U}_{r+1}} \Bigg( \Bigg.
\zeta_r^{-1}\prod_{v=r}^{k_1+2}\mE_{{\mathcal U}_{v}} \zeta_{v-1}^{\frac{\m_v}{\m_{v-1}}-1}
\mE_{{\mathcal U}_{k_1+1}}  \nonumber \\
& & \times \prod_{v=k_1}^{2}\mE_{{\mathcal U}_{v}} \zeta_{v-1}^{\frac{\m_v}{\m_{v-1}}-1}
\sum_{i_3=1}^{l} \lp \mE_{{\mathcal U}_1} Z_{i_3}^{\m_1}  \rp^{p-1}
  \mE_{{\mathcal U}_1}\frac{(C_{i_3}^{(i_1)})^{s-1} A_{i_3}^{(i_1,i_2)} \y_j^{(i_2)}}{Z_{i_3}^{1-\m_1}}
   \frac{d}{d\u_j^{(2,k_1+1)}}  \lp \zeta_{k_1}^{\frac{\m_{k_1+1}}{\m_{k_1}}-1}
 \rp   \Bigg. \Bigg). \nonumber \\
\end{eqnarray}
Analogously to \cite{Stojnicnflgscompyx23}'s (210), we further have  for $\Pi_{k_1+1,1,j}^{(1)}$
 \begin{eqnarray}\label{eq:rthlev2genanal14}
 \Pi_{k_1+1,1,j}^{(1)}
\hspace{-.01in}       & = & (\p_{k_1}-\p_{k_1+1})
  \mE_{G,{\mathcal U}_{r+1}} \Bigg( \Bigg.
\zeta_r^{-1}\prod_{v=r}^{k_1+1}\mE_{{\mathcal U}_{v}} \zeta_{v-1}^{\frac{\m_v}{\m_{v-1}}-1}
\mE_{{\mathcal U}_{k_1}}
\nonumber \\
& & \times
 \frac{d}{d\u_j^{(2,k_1+1)}}  \lp
\zeta_{k_1-1}^{\frac{\m_{k_1}}{\m_{k_1-1}}-1}
\prod_{v=k_1-1}^{2}\mE_{{\mathcal U}_{v}} \zeta_{v-1}^{\frac{\m_v}{\m_{v-1}}-1}
\sum_{i_3=1}^{l} \lp \mE_{{\mathcal U}_1} Z_{i_3}^{\m_1}  \rp^{p-1}
  \mE_{{\mathcal U}_1}\frac{(C_{i_3}^{(i_1)})^{s-1} A_{i_3}^{(i_1,i_2)} \y_j^{(i_2)}}{Z_{i_3}^{1-\m_1}} \rp   \Bigg. \Bigg).   \nonumber \\
 \end{eqnarray}
 Following \cite{Stojnicnflgscompyx23}'s (211)-(214), we set
  \begin{eqnarray}\label{eq:rthlev2genanal15}
 f_c(B^{(i_1,i_2)}) = \lp
\zeta_{k_1-1}^{\frac{\m_{k_1}}{\m_{k_1-1}}-1}
\prod_{v=k_1-1}^{2}\mE_{{\mathcal U}_{v}} \zeta_{v-1}^{\frac{\m_v}{\m_{v-1}}-1}
  \mE_{{\mathcal U}_1}\frac{(C_{i_3}^{(i_1)})^{s-1} A_{i_3}^{(i_1,i_2)} \y_j^{(i_2)}}{Z_{i_3}^{1-\m_1}} \rp,
 \end{eqnarray}
 and observe
 \begin{eqnarray}\label{eq:rthlev2genanal18}
   \frac{d}{d\u_j^{(2,k_1)}}  \lp f_c(B^{(i_1,i_2)}) \rp=
   \frac{d}{d\u_j^{(2,k_1+1)}}  \lp f_c(B^{(i_1,i_2)}) \rp.
 \end{eqnarray}
After combining  (\ref{eq:rthlev2genanal11}), (\ref{eq:rthlev2genanal14}), and (\ref{eq:rthlev2genanal18}) we arrive at the \emph{successive scaling} relation (the key component  of the \emph{canceling out} mechanism)
 \begin{eqnarray}\label{eq:rthlev2genanal19}
 \Pi_{k_1+1,1,j}^{(1)}=\frac{(\p_{k_1} -\p_{k_1+1})}{(\p_{k_1-1}-\p_{k_1})}T_{k_1,1,j}.
 \end{eqnarray}

For $\Pi_{k_1+1,1,j}^{(2)}$, we have
\begin{eqnarray}\label{eq:rthlev2genanal20}
 \Pi_{k_1+1,1,j}^{(2)}
& = &(\p_{k_1}-\p_{k_1+1})
  \mE_{G,{\mathcal U}_{r+1}} \Bigg( \Bigg.
\zeta_r^{-1}\prod_{v=r}^{k_1+2}\mE_{{\mathcal U}_{v}} \zeta_{v-1}^{\frac{\m_v}{\m_{v-1}}-1}
\mE_{{\mathcal U}_{k_1+1}}  \nonumber \\
& & \times \prod_{v=k_1}^{2}\mE_{{\mathcal U}_{v}} \zeta_{v-1}^{\frac{\m_v}{\m_{v-1}}-1}
\sum_{i_3=1}^{l} \lp \mE_{{\mathcal U}_1} Z_{i_3}^{\m_1}  \rp^{p-1}
  \mE_{{\mathcal U}_1}\frac{(C_{i_3}^{(i_1)})^{s-1} A_{i_3}^{(i_1,i_2)} \y_j^{(i_2)}}{Z_{i_3}^{1-\m_1}}
   \frac{d}{d\u_j^{(2,k_1+1)}}  \lp \zeta_{k_1}^{\frac{\m_{k_1+1}}{\m_{k_1}}-1}
 \rp   \Bigg. \Bigg), \nonumber \\
\end{eqnarray}
where analogously to \cite{Stojnicnflgscompyx23}'s (229)
\begin{eqnarray}\label{eq:rthlev2genanal23}
    \frac{d}{d\u_j^{(2,k_1+1)}}  \lp \zeta_{k_1}^{\frac{\m_{k_1+1}}{\m_{k_1}}-1}
 \rp   \Bigg. \Bigg)
  & =  & s \lp \m_{k_1+1} - \m_{k_1} p \rp \zeta_{k_1}^{\frac{\m_{k_1+1}}{\m_{k_1}}-2}
\prod_{k=k_1}^{2}\mE_{{\mathcal U}_{k}} \zeta_{k-1}^{\frac{\m_k}{\m_{k-1}}-1}
\sum_{p_3=1}^{l} \lp \mE_{{\mathcal U}_1} Z_{p_3}^{\m_1}  \rp^{p-1}
\nonumber  \\
& & \times
 \mE_{{\mathcal U}_1} \frac{1}{Z_{p_3}^{1-\m_1}}  \sum_{p_1=1}^{l} (C_{p_3}^{(p_1)})^{s-1}
 \sum_{p_2=1}^{l}\beta_{p_1}A_{p_3}^{(p_1,p_2)}\sqrt{1-t} \y_j^{(p_2)}.
\end{eqnarray}
Combining (\ref{eq:rthlev2genanal20}) and (\ref{eq:rthlev2genanal23}) one obtains the following analogue of \cite{Stojnicnflgscompyx23}'s (220)
\begin{eqnarray}\label{eq:rthlev2genanal24}
 \Pi_{k_1+1,1,j}^{(2)}
 & = & s(\p_{k_1}-\p_{k_1+1}) \lp \m_{k_1+1} - \m_{k_1}  p \rp
  \mE_{G,{\mathcal U}_{r+1}} \Bigg( \Bigg.
\prod_{v=r}^{k_1+1} \mE_{{\mathcal U}_{v}} \frac{\zeta_{v-1}^{\frac{\m_v}{\m_{v-1}}}}{\zeta_v} \nonumber \\
& & \times \prod_{v=k_1}^{2}\mE_{{\mathcal U}_{v}} \frac{\zeta_{v-1}^{\frac{\m_v}{\m_{v-1}}}}{\zeta_v}
\sum_{i_3=1}^{l} \lp \mE_{{\mathcal U}_1} Z_{i_3}^{\m_1}  \rp^{p-1}
   \mE_{{\mathcal U}_1}\frac{Z_{i_3}^{\m_1}}{\zeta_1}   \frac{(C_{i_3}^{(i_1)})^{s} }{Z_{i_3}}
\frac{ A_{i_3}^{(i_1,i_2)} }{C_{i_3}^{(i_1)}}\y_j^{(i_2)}
\nonumber \\
& & \times \prod_{k=k_1}^{2}  \mE_{{\mathcal U}_{k}} \frac{\zeta_{k-1}^{\frac{\m_k}{\m_{k-1}}}}{\zeta_k}
\sum_{p_3=1}^{l} \lp \mE_{{\mathcal U}_1} Z_{p_3}^{\m_1}  \rp^{p-1}
 \mE_{{\mathcal U}_1}\frac{Z_{p_3}^{\m_1}}{\zeta_1}  \sum_{p_1=1}^{l} \frac{(C_{p_3}^{(p_1)})^s}{Z_{p_3}}
 \sum_{p_2=1}^{l} \frac{A_{p_3}^{(p_1,p_2)}}{C_{p_3}^{(p_1)}}\sqrt{1-t} \beta_{p_1} \y_j^{(p_2)}
   \Bigg. \Bigg). \nonumber \\
\end{eqnarray}
We then set
\begin{eqnarray}\label{eq:rthlev2genanal25}
 \Phi_{{\mathcal U}_1}^{(i_3)} & \triangleq &  \mE_{{\mathcal U}_{1}} \frac{  Z_{i_3}^{\m_1}  } {  \mE_{{\mathcal U}_{1}} Z_{i_3}^{\m_1}  } \nonumber \\
 \Phi_{{\mathcal U}_k} & \triangleq &  \mE_{{\mathcal U}_{k}} \frac{\zeta_{k-1}^{\frac{\m_k}{\m_{k-1}}}}{\zeta_k} ,k\geq 2 \nonumber \\
 \gamma_{00}(i_3) & = &
\frac{  \lp  \mE_{{\mathcal U}_{1}} Z_{i_3}^{\m_1} \rp^p   }  { \sum_{i_3=1}^{l} \lp  \mE_{{\mathcal U}_{1}} Z_{i_3}^{\m_1} \rp^p   } \nonumber \\
 \gamma_0(i_1,i_2;i_3) & = &
\frac{(C_{i_3}^{(i_1)})^{s}}{Z_{i_3}}  \frac{A_{i_3}^{(i_1,i_2)}}{C_{i_3}^{(i_1)}} \nonumber \\
\gamma_{01}^{(r)}  & = & \prod_{k=r}^{2}\Phi_{{\mathcal U}_k} \gamma_{00}(i_3) \Phi_{{\mathcal U}_1}^{(i_3)} (\gamma_0(i_1,i_2;i_3))  \nonumber \\
\gamma_{02}^{(r)}  & = & \prod_{k=r}^{2} \Phi_{{\mathcal U}_k}   \gamma_{00}(i_3)    \Phi_{{\mathcal U}_1}^{(i_3)}   (\gamma_0(i_1,i_2;i_3)\times \gamma_0(i_1,p_2;i_3)) \nonumber \\
\gamma_{1}^{(r)}  & = & \prod_{k=r}^{2} \Phi_{{\mathcal U}_k}   \gamma_{00}(i_3)    \Phi_{{\mathcal U}_1}^{(i_3)}   (\gamma_0(i_1,i_2;i_3)\times \gamma_0(p_1,p_2;i_3)) \nonumber \\
\gamma_{21}^{(r)}  & = & \prod_{k=r}^{2}\Phi_{{\mathcal U}_k} \lp
  \gamma_{00}(i_3)    \Phi_{{\mathcal U}_1}^{(i_3)}   \gamma_0(i_1,i_2;i_3)\times
    \gamma_{00}(p_3)    \Phi_{{\mathcal U}_1}^{(p_3)}
  \gamma_0(p_1,p_2;p_3) \rp  \nonumber \\
\gamma_{22}^{(r)}  & = & \prod_{k=r}^{2}\Phi_{{\mathcal U}_k} \lp
  \gamma_{00}(i_3)  \lp  \Phi_{{\mathcal U}_1}^{(i_3)}   \gamma_0(i_1,i_2;i_3)\times
     \Phi_{{\mathcal U}_1}^{(i_3)}
  \gamma_0(p_1,p_2;i_3) \rp \rp \nonumber \\
\gamma_{k_1+1}^{(r)}  & = & \prod_{k=r}^{k_1+1}\Phi_{{\mathcal U}_k} \lp \prod_{k=k_1}^{2}\Phi_{{\mathcal U}_k}
  \gamma_{00}(i_3)    \Phi_{{\mathcal U}_1}^{(i_3)}   \gamma_0(i_1,i_2;i_3)\times \prod_{k=k_1}^{2} \Phi_{{\mathcal U}_k}
    \gamma_{00}(p_3)    \Phi_{{\mathcal U}_1}^{(p_3)}
  \gamma_0(p_1,p_2;p_3) \rp, k_1\geq 2, \nonumber \\
 \end{eqnarray}
and from (\ref{eq:rthlev2genanal24}) find
\begin{eqnarray}\label{eq:rthlev2genanal26}
 \sum_{i_1=1}^{l}\sum_{i_2=1}^{l} \sum_{j=1}^{m}
\beta_{i_1} \frac{\Pi_{k_1+1,1,j}^{(2)}}{\sqrt{1-t}}
 & = & -s(\p_{k_1}-\p_{k_1+1}) \lp \m_{k_1}  - \m_{k_1+1} \rp p
  \mE_{G,{\mathcal U}_{r+1}} \Bigg( \Bigg.
\prod_{v=r}^{k_1+1} \Phi_{{\mathcal U}_v} \nonumber \\
& & \times \prod_{v=k_1}^{2} \Phi_{{\mathcal U}_v}
\sum_{i_3=1}^{l}\gamma_{00}(i_3) \Phi_{ {\mathcal U}_1 }^{(i_3)}
\sum_{i_1=1}^{l}\sum_{i_2=1}^{l} \frac{(C^{(i_1)})^{s} }{Z}
\frac{ A^{(i_1,i_2)} }{C^{(i_1)}}\y_j^{(i_2)}
\nonumber \\
& & \times \prod_{k=k_1}^{2} \Phi_{{\mathcal U}_k}
\sum_{p_3=1}^{p}\gamma_{00}(p_3) \Phi_{ {\mathcal U}_1 }^{(p_3)}
 \sum_{p_1=1}^{l} \frac{(C^{(p_1)})^s}{Z}
 \sum_{p_2=1}^{l} \frac{A^{(p_1,p_2)}}{C^{(p_1)}} \beta_{i_1} \beta_{p_1} (\y^{(p_2)})^T \y^{(i_2)}
   \Bigg. \Bigg) \nonumber \\
 & = & -s\beta^2(\p_{k_1}-\p_{k_1+1}) \lp \m_{k_1} - \m_{k_1+1} \rp p \nonumber\\
 & & \times
  \mE_{G,{\mathcal U}_{r+1}} \left \langle \|\x^{(i_1)}\|_2\|\x^{(p_1)}\|_2 (\y^{(p_2)})^T \y^{(i_2)} \right \rangle_{\gamma_{k_1+1}^{(r)}}.
\end{eqnarray}
A combination of  (\ref{eq:rthlev2genanal12}), (\ref{eq:rthlev2genanal19}), and (\ref{eq:rthlev2genanal26}) then gives
 \begin{eqnarray}\label{eq:rthlev2genanal27}
 \sum_{i_1=1}^{l}\sum_{i_2=1}^{l} \sum_{j=1}^{m}
\beta_{i_1} \frac{T_{k_1+1,1,j}}{\sqrt{1-t}}
& = &  \sum_{i_1=1}^{l}\sum_{i_2=1}^{l} \sum_{j=1}^{m}
\beta_{i_1} \frac{\Pi_{k_1+1,1,j}^{(1)}}{\sqrt{1-t}} +  \sum_{i_1=1}^{l}\sum_{i_2=1}^{l} \sum_{j=1}^{m}
\beta_{i_1} \frac{\Pi_{k_1+1,1,j}^{(2)}}{\sqrt{1-t}} \nonumber \\
& = &  \frac{\p_{k_1}-\p_{k_1+1}}{\p_{k_1-1}-\p_{k_1}}\sum_{i_1=1}^{l}\sum_{i_2=1}^{l} \sum_{j=1}^{m}
\beta_{i_1} \frac{T_{k_1,1,j}}{\sqrt{1-t}} \nonumber \\
 &  & - \Bigg( \Bigg. s\beta^2(\p_{k_1}-\p_{k_1+1}) \lp \m_{k_1} -  \m_{k_1+1} \rp p \nonumber \\
 & & \times
  \mE_{G,{\mathcal U}_{r+1}} \left \langle \|\x^{(i_1)}\|_2\|\x^{(p_1)}\|_2 (\y^{(p_2)})^T \y^{(i_2)} \right \rangle_{\gamma_{k_1+1}^{(r)} } \Bigg. \Bigg).
\end{eqnarray}
Keeping in mind  first and second level results obtained earlier, one can then immediately write analogously for the other two sequences $\lp T_{k,2}\rp_{k=1:r+1}$ and $\lp T_{k,3}\rp_{k=1:r+1}$
\begin{eqnarray}\label{eq:rthlev2genanal28}
 \sum_{i_1=1}^{l}\sum_{i_2=1}^{l}
\beta_{i_1} \|\y^{(i_2)}\|_2 \frac{T_{k_1+1,2}}{\sqrt{1-t}}
 & = &  \frac{\q_{k_1}-\q_{k_1+1}}{\q_{k_1-1}-\q_{k_1}}\sum_{i_1=1}^{l}\sum_{i_2=1}^{l}
\beta_{i_1} \|\y^{(i_2)}\|_2 \frac{T_{k_1,2}}{\sqrt{1-t}} \nonumber \\
 &  & - \Bigg( \Bigg. s\beta^2(\q_{k_1}-\q_{k_1+1}) \lp \m_{k_1} - \m_{k_1+1} \rp p \nonumber \\
 & &
\times  \mE_{G,{\mathcal U}_{r+1}} \left \langle   (\x^{(p_1)})^T \x^{(i_1)} \|\y^{(i_2)}\|_2\|\y^{(p_2)}\|_2  \right \rangle_{\gamma_{k_1+1}^{(r)} } \Bigg.\Bigg), \nonumber \\
\end{eqnarray}
and
 \begin{eqnarray}\label{eq:rthlev2genanal29}
 \sum_{i_1=1}^{l}\sum_{i_2=1}^{l}
\beta_{i_1} \|\y^{(i_2)}\|_2 \frac{T_{k_1+1,3}}{\sqrt{1-t}}
 & = &  \frac{\p_{k_1}\q_{k_1}-\p_{k_1+1}\q_{k_1+1}}{\p_{k_1-1}\q_{k_1-1}-\p_{k_1}\q_{k_1}}\sum_{i_1=1}^{l}\sum_{i_2=1}^{l}
\beta_{i_1} \|\y^{(i_2)}\|_2 \frac{T_{k_1,3}}{\sqrt{1-t}} \nonumber \\
 &  & - \Bigg( \Bigg. s\beta^2(\p_{k_1}\q_{k_1}-\p_{k_1+1}\q_{k_1+1}) \lp  \m_{k_1} - \m_{k_1+1} \rp p \nonumber \\
& & \times   \mE_{G,{\mathcal U}_{r+1}} \left \langle   \|\x^{(i_1)}\|_2\|\x^{(p_1)}\|_2 \|\y^{(i_2)}\|_2\|\y^{(p_2)}\|_2  \right \rangle_{\gamma_{k_1+1}^{(r)} } \Bigg. \Bigg).
\end{eqnarray}
As observed in \cite{Stojnicnflgscompyx23},  adding the elements in each of the sequences results in canceling out parts of the first and second summands.

\subsection{$T_{G,j}$ --- $r$-th level}
\label{sec:rthTG}

After introducing $\m_{r+1}=0$ and utilizing (\ref{eq:rthlev2genanal10g}), we write the following analogue of \cite{Stojnicnflgscompyx23}'s (226)
 \begin{eqnarray}\label{eq:rthlev2genanal30}
T_{G,j}
 & = &  \mE_{{\mathcal U}_{r+1},{{\mathcal U}_{r}},\dots,{{\mathcal U}_{1}}} \lp
\mE_{G} \prod_{v=r+1}^{2} \zeta_{v-1}^{\frac{\m_v}{\m_{v-1}}-1}
\sum_{i_3=1}^{l} \lp  \mE_{{\mathcal U}_{1}} Z_{i_3}^{ \m_1 } \rp^{p-1}
    \frac{(C_{i_3}^{(i_1)})^{s-1} A_{i_3}^{(i_1,i_2)} \y_j^{(i_2)}\bar{\u}_j^{(i_1,1)}}{Z_{i_3}^{1-\m_1}} \rp
    \nonumber \\
  & = &  \beta^2 \mE_{G,{\mathcal U}_{r+1},{{\mathcal U}_{r}},\dots,{{\mathcal U}_{1}}}  \Bigg . \Bigg (
 \sum_{p_1=1}^{l} (\x^{(p_1)})^T\x^{(i_1)} \lp \beta_{i_1}\beta_{p_1} \rp^{-1}
\prod_{v=r+1}^{2} \zeta_{v-1}^{\frac{\m_v}{\m_{v-1}}-1}
\sum_{i_3=1}^{l} \lp  \mE_{{\mathcal U}_{1}} Z_{i_3}^{ \m_1 } \rp^{p-1}
\nonumber \\
& & \times
\frac{d}{d\bar{\u}_j^{(p_1,1)}} \lp
    \frac{(C_{i_3}^{(i_1)})^{s-1} A_{i_3}^{(i_1,i_2)} \y_j^{(i_2)}}{Z_{i_3}^{1-\m_1}} \rp \Bigg . \Bigg ) \nonumber \\
 &  & + \sum_{k_1=2}^{2} \beta^2 \mE_{G,{\mathcal U}_{r+1},{{\mathcal U}_{r}},\dots,{{\mathcal U}_{1}}} \Bigg(\Bigg.
 \sum_{p_1=1}^{l} (\x^{(p_1)})^T\x^{(i_1)}    \lp \beta_{i_1}\beta_{p_1} \rp^{-1}
\prod_{v=r+1,v\neq k_1}^{2} \zeta_{v-1}^{\frac{\m_v}{\m_{v-1}}-1}
\sum_{i_3=1}^{l} \lp  \mE_{{\mathcal U}_{1}} Z_{i_3}^{ \m_1 } \rp^{p-1}
\nonumber \\
& &  \times
    \frac{(C_{i_3}^{(i_1)})^{s-1} A_{i_3}^{(i_1,i_2)} \y_j^{(i_2)}}{Z_{i_3}^{1-\m_1}}    \frac{d}{d\bar{\u}_j^{(p_1,1)}} \lp
 \zeta_{k_1-1}^{\frac{\m_{k_1}}{\m_{k_1-1}}-1}
  \rp \Bigg.\Bigg) \nonumber \\
 &  & +  \beta^2 \mE_{G,{\mathcal U}_{r+1},{{\mathcal U}_{r}},\dots,{{\mathcal U}_{1}}} \Bigg(\Bigg.
 \sum_{p_1=1}^{l} (\x^{(p_1)})^T\x^{(i_1)}   \lp \beta_{i_1}\beta_{p_1} \rp^{-1}
\prod_{v=r+1}^{2} \zeta_{v-1}^{\frac{\m_v}{\m_{v-1}}-1}
\sum_{i_3=1}^{l} \nonumber \\
& &  \times
    \frac{(C_{i_3}^{(i_1)})^{s-1} A_{i_3}^{(i_1,i_2)} \y_j^{(i_2)}}{Z_{i_3}^{1-\m_1}}    \frac{d}{d\bar{\u}_j^{(p_1,1)}} \lp
\lp  \mE_{{\mathcal U}_{1}} Z_{i_3}^{ \m_1 } \rp^{p-1}
  \rp \Bigg.\Bigg) \nonumber \\
 &  & + \sum_{k_1=3}^{r+1} \beta^2 \mE_{G,{\mathcal U}_{r+1},{{\mathcal U}_{r}},\dots,{{\mathcal U}_{1}}} \Bigg(\Bigg.
 \sum_{p_1=1}^{l} (\x^{(p_1)})^T\x^{(i_1)}   \lp \beta_{i_1}\beta_{p_1} \rp^{-1}
\prod_{v=r+1,v\neq k_1}^{2} \zeta_{v-1}^{\frac{\m_v}{\m_{v-1}}-1}
\sum_{i_3=1}^{l} \lp  \mE_{{\mathcal U}_{1}} Z_{i_3}^{ \m_1 } \rp^{p-1}
\nonumber \\
& &  \times
    \frac{(C_{i_3}^{(i_1)})^{s-1} A_{i_3}^{(i_1,i_2)} \y_j^{(i_2)}}{Z_{i_3}^{1-\m_1}}    \frac{d}{d\bar{\u}_j^{(p_1,1)}} \lp
 \zeta_{k_1-1}^{\frac{\m_{k_1}}{\m_{k_1-1}}-1}
  \rp \Bigg.\Bigg) \nonumber \\
& = & T_{G,j}^c+ T_{G,j}^{d_{1}} + T_{G,j}^{e_{1}} +\sum_{k_1=3}^{r+1} T_{G,j}^{d_{k_1-1}},
 \end{eqnarray}
 where
 \begin{eqnarray}\label{eq:rthlev2genanal31}
T_{G,j}^c   & = &
   \mE_{G,{\mathcal U}_{r+1},{{\mathcal U}_{r}},\dots,{{\mathcal U}_{1}}} \lp
\prod_{v=r+1}^{2} \zeta_{v-1}^{\frac{\m_v}{\m_{v-1}}-1}
\Theta_{G,1}^{(2)} \rp,
\end{eqnarray}
and
 \begin{eqnarray}\label{eq:rthlev2genanal32}
T_{G,j}^{d_{1}} & = &  \sum_{k_1=2}^{2}   \beta^2 \mE_{G,{\mathcal U}_{r+1},{{\mathcal U}_{r}},\dots,{{\mathcal U}_{1}}} \Bigg(\Bigg.
 \sum_{p_1=1}^{l} (\x^{(p_1)})^T\x^{(i_1)}  \lp \beta_{i_1}\beta_{p_1} \rp^{-1}
\prod_{v=r+1,v\neq k_1}^{2} \zeta_{v-1}^{\frac{\m_v}{\m_{v-1}}-1}
\sum_{i_3=1}^{l} \lp  \mE_{{\mathcal U}_{1}} Z_{i_3}^{ \m_1 } \rp^{p-1}
    \nonumber \\
& &
\times
\frac{(C_{i_3}^{(i_1)})^{s-1} A_{i_3}^{(i_1,i_2)} \y_j^{(i_2)}}{Z_{i_3}^{1-\m_1}}
     \frac{d}{d\bar{\u}_j^{(p_1,1)}} \lp
 \zeta_{k_1-1}^{\frac{\m_{k_1}}{\m_{k_1-1}}-1}
  \rp \Bigg.\Bigg)
\nonumber \\
T_{G,j}^{e_{1}} & = &   \beta^2 \mE_{G,{\mathcal U}_{r+1},{{\mathcal U}_{r}},\dots,{{\mathcal U}_{1}}} \Bigg(\Bigg.
 \sum_{p_1=1}^{l} (\x^{(p_1)})^T\x^{(i_1)}  \lp \beta_{i_1}\beta_{p_1} \rp^{-1}
\prod_{v=r+1}^{2} \zeta_{v-1}^{\frac{\m_v}{\m_{v-1}}-1}
\sum_{i_3=1}^{l}
    \nonumber \\
& &
\times
\frac{(C_{i_3}^{(i_1)})^{s-1} A_{i_3}^{(i_1,i_2)} \y_j^{(i_2)}}{Z_{i_3}^{1-\m_1}}
     \frac{d}{d\bar{\u}_j^{(p_1,1)}} \lp
 \lp  \mE_{{\mathcal U}_{1}} Z_{i_3}^{ \m_1 } \rp^{p-1}
  \rp \Bigg.\Bigg)
\nonumber \\
 T_{G,j}^{d_{k_1-1}} & = &  \beta^2 \mE_{G,{\mathcal U}_{r+1},{{\mathcal U}_{r}},\dots,{{\mathcal U}_{1}}} \Bigg(\Bigg.
 \sum_{p_1=1}^{l} (\x^{(p_1)})^T\x^{(i_1)}  \lp \beta_{i_1}\beta_{p_1} \rp^{-1}
\prod_{v=r+1,v\neq k_1}^{2} \zeta_{v-1}^{\frac{\m_v}{\m_{v-1}}-1}
\sum_{i_3=1}^{l} \lp  \mE_{{\mathcal U}_{1}} Z_{i_3}^{ \m_1 } \rp^{p-1}
    \nonumber \\
& &
\times
\frac{(C_{i_3}^{(i_1)})^{s-1} A_{i_3}^{(i_1,i_2)} \y_j^{(i_2)}}{Z_{i_3}^{1-\m_1}}
     \frac{d}{d\bar{\u}_j^{(p_1,1)}} \lp
 \zeta_{k_1-1}^{\frac{\m_{k_1}}{\m_{k_1-1}}-1}
  \rp \Bigg.\Bigg), k_1\geq 3.
 \end{eqnarray}

A combination of (\ref{eq:lev2genGanal5}), (\ref{eq:rthlev2genanal31}), and \cite{Stojnicnflgscompyx23}'s (229) allows to write
\begin{eqnarray}\label{eq:rthlev2genanal33}
T_{G,j}^c
  & = &
   \mE_{G,{\mathcal U}_{r+1}} \Bigg( \Bigg.
\prod_{v=r}^{2} \mE_{{\mathcal U}_{v}} \frac{\zeta_{v-1}^{\frac{\m_v}{\m_{v-1}}}}{\zeta_v}
\sum_{i_3=1}^{l} \lp \mE_{{\mathcal U}_{1}}  Z_{i_3}^{\m_1} \rp^{p-1}
 \nonumber \\
& &  \times
  \Bigg. \Bigg (
\mE_{{\mathcal U}_{1}}\frac{Z_{i_3}^{\m_1}}{\zeta_1}
 \frac{\y_j^{(i_2)}}{Z_{i_3}}\lp(C_{i_3}^{(i_1)})^{s-1}\beta_{i_1}A_{i_3}^{(i_1,i_2)}\y_j^{(i_2)}\sqrt{t}+(s-1)(C_{i_3}^{(i_1)})^{s-2}\beta_{i_1}\sum_{p_2=1}^{l}A_{i_3}^{(i_1,p_2)}\y_j^{(p_2)}\sqrt{t}\rp \nonumber \\
& &  -(1-\m_1)
\mE_{{\mathcal U}_{1}}\frac{Z_{i_3}^{\m_1}}{\zeta_1}
\nonumber \\
& & \times
  \lp\sum_{p_1=1}^{l} \frac{(\x^{(i_1)})^T\x^{(p_1)}}{\|\x^{(i_1)}\|_2\|\x^{(p_1)}\|_2}
\frac{(C_{i_3}^{(i_1)})^{s-1} A_{i_3}^{(i_1,i_2)}\y_j^{(i_2)}}{Z_{i_3}^{2}}
s  (C_{i_3}^{(p_1)})^{s-1}\sum_{p_2=1}^{l}\beta_{p_1}A_{i_3}^{(p_1,p_2)}\y_j^{(p_2)}\sqrt{t}\rp \Bigg. \Bigg)  \Bigg. \Bigg) \nonumber \\
 & = &
  \mE_{G,{\mathcal U}_{r+1}} \Bigg( \Bigg.
\prod_{v=r}^{2} \Phi_{{\mathcal U}_{v}}
\sum_{i_3=1}^{l}\gamma_{00} ((i_3)
   \nonumber \\
& &  \times
  \Bigg. \Bigg(
\Phi_{{\mathcal U}_{1}}
 \frac{\y_j^{(i_2)}}{Z_{i_3}}\lp(C_{i_3}^{(i_1)})^{s-1}\beta_{i_1}A_{i_3}^{(i_1,i_2)}\y_j^{(i_2)}\sqrt{t}+(s-1)(C_{i_3}^{(i_1)})^{s-2}\beta_{i_1}\sum_{p_2=1}^{l}A_{i_3}^{(i_1,p_2)}\y_j^{(p_2)}\sqrt{t}\rp \nonumber \\
& &  -(1-\m_1)
\Phi_{{\mathcal U}_{1}}
\nonumber \\
& & \times
  \lp\sum_{p_1=1}^{l} \frac{(\x^{(i_1)})^T\x^{(p_1)}}{\|\x^{(i_1)}\|_2\|\x^{(p_1)}\|_2}
\frac{(C_{i_3}^{(i_1)})^{s-1} A_{i_3}^{(i_1,i_2)}\y_j^{(i_2)}}{Z_{i_3}^{2}}
s  (C_{i_3}^{(p_1)})^{s-1}\sum_{p_2=1}^{l}\beta_{p_1}A_{i_3}^{(p_1,p_2)}\y_j^{(p_2)}\sqrt{t}\rp \Bigg. \Bigg) \Bigg. \Bigg),
\end{eqnarray}
and
\begin{eqnarray}\label{eq:rthlev2genanal34}
\sum_{i_1=1}^{l}\sum_{i_2=1}^{l}\sum_{j=1}^{m} \beta_{i_1} \frac{T_{G,j}^c}{\sqrt{t}}   & = &  \beta^2 \lp  \mE_{G,{\mathcal U}_{r+1}} \langle \|\x^{(i_1)}\|_2^2\|\y^{(i_2)}\|_2^2\rangle_{\gamma_{01}^{(r)}} +   (s-1)  \mE_{G,{\mathcal U}_{r+1}}\langle \|\x^{(i_1)}\|_2^2(\y^{(p_2)})^T\y^{(i_2)}\rangle_{\gamma_{02}^{(r)}} \rp      \nonumber \\
 &  & -s\beta^2(1-\m_1) \mE_{G,{\mathcal U}_{r+1}}\langle (\x^{(p_1)})^T\x^{(i_1)}(\y^{(p_2)})^T\y^{(i_2)}\rangle_{\gamma_1^{(r)}}.
\end{eqnarray}

To find $T_{G,j}^{d_{k_1-1}}$ we first note
 \begin{eqnarray}\label{eq:rthlev2genanal35}
T_{G,j}^{d_{k_1-1}} & = &  \beta^2 \mE_{G,{\mathcal U}_{r+1},{{\mathcal U}_{r}},\dots,{{\mathcal U}_{1}}} \Bigg(\Bigg.
 \sum_{p_1=1}^{l} (\x^{(p_1)})^T\x^{(i_1)} \lp  \beta_{i_1}\beta_{p_1} \rp^{-1}
\prod_{v=r+1,v\neq k_1}^{2} \zeta_{v-1}^{\frac{\m_v}{\m_{v-1}}-1}
  \sum_{i_3=1}^{l} \lp \mE_{{\mathcal U}_1} Z_{i_3}^{\m_1} \rp^{p-1}
     \nonumber \\
& &
\times
\frac{(C_{i_3}^{(i_1)})^{s-1} A_{i_3}^{(i_1,i_2)} \y_j^{(i_2)}}{Z_{i_3}^{1-\m_1}}
     \frac{d}{d\bar{\u}_j^{(p_1,1)}} \lp
 \zeta_{k_1-1}^{\frac{\m_{k_1}}{\m_{k_1-1}}-1}
  \rp \Bigg.\Bigg) \nonumber \\
& = &  \beta^2 \mE_{G,{\mathcal U}_{r+1},{{\mathcal U}_{r}},\dots,{{\mathcal U}_{1}}} \Bigg(\Bigg.
 \sum_{p_1=1}^{l} (\x^{(p_1)})^T\x^{(i_1)}
\prod_{v=r+1,v\neq k_1}^{2} \zeta_{v-1}^{\frac{\m_v}{\m_{v-1}}-1}
  \sum_{i_3=1}^{l} \lp \mE_{{\mathcal U}_1} Z_{i_3}^{\m_1} \rp^{p-1}
     \nonumber \\
& &
\times
\frac{(C_{i_3}^{(i_1)})^{s-1} A_{i_3}^{(i_1,i_2)} \y_j^{(i_2)}}{Z_{i_3}^{1-\m_1}}
       s \lp \m_{k_1} - \m_{k_1-1}\rp
       \zeta_{k_1-1}^{\frac{\m_{k_1}}{\m_{k_1-1}}-2}
\prod_{k=k_1-1}^{2}\mE_{{\mathcal U}_{k}} \zeta_{k-1}^{\frac{\m_k}{\m_{k-1}}-1}
p  \sum_{p_3=1}^{l} \lp \mE_{{\mathcal U}_1} Z_{p_3}^{\m_1} \rp^{p-1}
 \nonumber \\
& & \times
 \mE_{{\mathcal U}_1} \frac{1}{Z_{p_3}^{1-\m_1}}  (C_{p_3}^{(p_1)})^{s-1}
  \sum_{p_2=1}^{l}   \frac{1} { \beta_{i_1} }      A_{p_3}^{(p_1,p_2)}\frac{dD^{(p_1,p_2)}}{d\u_j^{(p_1,1)}}
 \Bigg.\Bigg).
 \end{eqnarray}
 Since
\begin{eqnarray}\label{eq:rthlev2genanal36}
 \frac{dD^{(p_1,p_2)}}{d\u_j^{(p_1,1)}}=
 \frac{dB^{(p_1,p_2)}}{d\u_j^{(p_1,1)}}=\sqrt{t} \y_j^{(p_2)},
\end{eqnarray}
after combining  (\ref{eq:rthlev2genanal35}) and \cite{Stojnicnflgscompyx23}'s (233) we further find
\begin{eqnarray}\label{eq:rthlev2genanal37}
T_{G,j}^{d_{k_1-1}}
  & = &  \beta^2 \mE_{G,{\mathcal U}_{r+1}} \Bigg(\Bigg.
\prod_{k=r}^{k_1+1}\mE_{{\mathcal U}_{k}} \frac{\zeta_{k-1}^{\frac{\m_k}{\m_{k-1}}}}{\zeta_k}
 \prod_{k=k_1}^{2}\mE_{{\mathcal U}_{k}} \frac{\zeta_{k-1}^{\frac{\m_k}{\m_{k-1}}}}{\zeta_k}
 \sum_{i_3=1}^{l} \lp \mE_{{\mathcal U}_1} Z_{i_3}^{\m_1} \rp^{p-1}
 \mE_{{\mathcal U}_1} \frac{Z_{i_3}^{\m_1}}{\zeta_1}
    \frac{(C_{i_3}^{(i_1)})^{s} }{Z_{i_3}}     \frac{ A_{i_3}^{(i_1,i_2)} }{C_{i_3}^{(i_1)}}  \nonumber \\
& &
\times   s \lp \m_{k_1} - \m_{k_1-1}\rp p
\prod_{k=k_1}^{2}\mE_{{\mathcal U}_{k}} \frac{\zeta_{k-1}^{\frac{\m_k}{\m_{k-1}}}}{\zeta_k}
 \sum_{p_3=1}^{l} \lp \mE_{{\mathcal U}_1} Z_{p_3}^{\m_1} \rp^{p-1}
 \mE_{{\mathcal U}_1} \frac{Z_{p_3}^{\m_1}}{\zeta_1}   \sum_{p_1=1}^{l}\frac{(C_{p_3}^{(p_1)})^{s} }{Z_{p_3}}  \nonumber \\
& & \times \sum_{p_2=1}^{l} \frac{A_{p_3}^{(p_1,p_2)}}{C_{p_3}^{(p_1)}}\sqrt{t}  \frac{ 1} { \beta_{i_1} } (\x^{(p_1)})^T\x^{(i_1)}  \y_j^{(p_2)}\y_j^{(i_2)}
 \Bigg.\Bigg) \nonumber \\
  & = &  s\beta^2 \lp \m_{k_1} - \m_{k_1-1}\rp p \mE_{G,{\mathcal U}_{r+1}} \Bigg(\Bigg.
\prod_{k=r}^{k_1+1}\Phi_{{\mathcal U}_{k}}
 \prod_{k=k_1}^{2}\Phi_{{\mathcal U}_{k}}
\sum_{i_3=1}^{l}
 \gamma_{00}(i_3)
 \Phi_{{\mathcal U}_{1}}^{(i_3)}
    \frac{(C_{i_3}^{s} }{Z_{i_3}}     \frac{ A_{i_3}^{(i_1,i_2)} }{C_{i_3}^{(i_1)}}  \nonumber \\
& &
\times
\sum_{p_3=1}^{l}
 \gamma_{00}(p_3)
 \Phi_{{\mathcal U}_{1}}^{(p_3)}
\prod_{k=k_1}^{1}\Phi_{{\mathcal U}_{k}}    \sum_{p_1=1}^{l}\frac{(C_{p_3}^{(p_1)})^{s} }{Z_{p_3}}  \sum_{p_2=1}^{l} \frac{A_{p_3}^{(p_1,p_2)}}{C_{p_3}^{(p_1)}}\sqrt{t}  \frac{ 1 } {\beta_{i_1} }   (\x^{(p_1)})^T\x^{(i_1)}  \y_j^{(p_2)}\y_j^{(i_2)}
  \Bigg.\Bigg), k_1\geq 2. \nonumber \\
\end{eqnarray}

To find $T_{G,j}^{e_{1}}$ we first write
 \begin{eqnarray}\label{eq:rthlev2genanal35bb0}
T_{G,j}^{e_{1}} & = &   \beta^2 \mE_{G,{\mathcal U}_{r+1},{{\mathcal U}_{r}},\dots,{{\mathcal U}_{1}}} \Bigg(\Bigg.
 \sum_{p_1=1}^{l} (\x^{(p_1)})^T\x^{(i_1)}  \lp  \beta_{i_1}\beta_{p_1} \rp^{-1}
\prod_{v=r+1}^{2} \zeta_{v-1}^{\frac{\m_v}{\m_{v-1}}-1}
\sum_{i_3=1}^{l}
    \nonumber \\
& &
\times
\frac{(C_{i_3}^{(i_1)})^{s-1} A_{i_3}^{(i_1,i_2)} \y_j^{(i_2)}}{Z_{i_3}^{1-\m_1}}
     \frac{d}{d\bar{\u}_j^{(p_1,1)}} \lp
 \lp  \mE_{{\mathcal U}_{1}} Z_{i_3}^{ \m_1 } \rp^{p-1}
  \rp \Bigg.\Bigg).
\end{eqnarray}
Utilizing (\ref{eq:lev2genGanal7bbb0}) we also have
 \begin{eqnarray}\label{eq:rthlev2genanal35bb1}
T_{G,j}^{e_{1}}
& = &   \beta^2 \mE_{G,{\mathcal U}_{r+1},{{\mathcal U}_{r}},\dots,{{\mathcal U}_{1}}} \Bigg(\Bigg.
 \sum_{p_1=1}^{l} (\x^{(p_1)})^T\x^{(i_1)}
\prod_{v=r+1}^{2} \zeta_{v-1}^{\frac{\m_v}{\m_{v-1}}-1}
\sum_{i_3=1}^{l}
\mE_{{\mathcal U}_1}\frac{  Z_{i_3}^{\m_1} } { \mE_{{\mathcal U}_1}  Z_{i_3}^{\m_1}  }
\frac{(C_{i_3}^{(i_1)})^{s-1} A_{i_3}^{(i_1,i_2)} \y_j^{(i_2)}}{Z_{i_3} }
    \nonumber \\
& &
\times
     \m_1 (p-1) \lp
 \mE_{{\mathcal U}_1}  Z_{i_3}^{\m_1}
\rp^{p}
 \mE_{{\mathcal U}_1}\frac{ Z_{i_3}^{\m_1} } { \mE_{{\mathcal U}_1} Z_{i_3}^{\m_1}  }
  \frac{1}{Z_{i_3}} s (C_{i_3}^{(p_1)})^{s-1}\sum_{p_2=1}^{l}A_{i_3}^{(p_1,p_2)}
 \frac{1} { \beta_{i_1} }   \y_j^{(p_2)}\sqrt{t}
     \Bigg.\Bigg)
\nonumber \\
& = &   \beta^2 \mE_{G,{\mathcal U}_{r+1}} \Bigg(\Bigg.
 \prod_{v=r}^{2}
  \mE_{{\mathcal U}_r}
 \frac {  \zeta_{v-1}^{\frac{\m_v}{\m_{v-1}}}  }  {  \zeta_{v} }
\sum_{i_3=1}^{l}
\mE_{{\mathcal U}_1}\frac{  Z_{i_3}^{\m_1} } { \mE_{{\mathcal U}_1}  Z_{i_3}^{\m_1}  }
\frac{(C_{i_3}^{(i_1)})^{s-1} A_{i_3}^{(i_1,i_2)} \y_j^{(i_2)}}{Z_{i_3}}
    \nonumber \\
& &
\times
   s  \m_1 (p-1)
   \frac{ \lp
 \mE_{{\mathcal U}_1}  Z_{i_3}^{\m_1}
\rp^{p} }  {  \zeta_1 }
\mE_{{\mathcal U}_1}\frac{ Z_{i_3}^{\m_1} } { \mE_{{\mathcal U}_1} Z_{i_3}^{\m_1}  }
 \sum_{p_1=1}^{l} (\x^{(p_1)})^T\x^{(i_1)}
  \frac{1}{Z_{i_3}}  (C_{i_3}^{(p_1)})^{s-1}\sum_{p_2=1}^{l}A_{i_3}^{(p_1,p_2)}
  \frac{ 1 } { \beta_{i_1} }  \y_j^{(p_2)}\sqrt{t}
     \Bigg.\Bigg) \nonumber \\
& = &     s
 \beta^2  \m_1 (p-1) \mE_{G,{\mathcal U}_{r+1}} \Bigg(\Bigg.
 \prod_{v=r}^{2}
\Phi_{ {\mathcal U}_v  }
\sum_{i_3=1}^{l}
\Phi_{ {\mathcal U}_1  } ^ { (i_3 )}
\frac{(C_{i_3}^{(i_1)})^{s-1} A_{i_3}^{(i_1,i_2)} \y_j^{(i_2)}}{Z_{i_3} }
    \nonumber \\
& &
\times
 \gamma_{00}(i_3)  \Phi_{ {\mathcal U}_1 } ^{(i_3)}
  \sum_{p_1=1}^{l} (\x^{(p_1)})^T\x^{(i_1)}
  \frac{1}{Z_{i_3}}  (C_{i_3}^{(p_1)})^{s-1}\sum_{p_2=1}^{l}A_{i_3}^{(p_1,p_2)}
  \frac{ 1 } { \beta_{i_1} }  \y_j^{(p_2)}\sqrt{t}
     \Bigg.\Bigg).
\end{eqnarray}

A combination of  (\ref{eq:rthlev2genanal30}), (\ref{eq:rthlev2genanal34}), (\ref{eq:rthlev2genanal37}), and (\ref{eq:rthlev2genanal35bb1})  then gives
\begin{eqnarray}\label{eq:rthlev2genanal38}
\sum_{i_1=1}^{l}\sum_{i_2=1}^{l}\sum_{j=1}^{m} \beta_{i_1 }\frac{T_{G,j}}{\sqrt{t}}
& = & \sum_{i_1=1}^{l}\sum_{i_2=1}^{l}\sum_{j=1}^{m} \beta_{i_1 }\frac{(T_{G,j}^c+T_{G,j}^{d_1}+T_{G,j}^{e_1}+\sum_{k_1=3}^{r+1} T_{G,j}^{d_{k_1-1}})}{\sqrt{t}} \nonumber \\
& = & \beta^2 \lp \mE_{G,{\mathcal U}_{r+1}} \langle \|\x^{(i_1)}\|_2^2\|\y^{(i_2)}\|_2^2\rangle_{\gamma_{01}^{(r)}} +   (s-1)\mE_{G,{\mathcal U}_{r+1}}\langle \|\x^{(i_1)}\|_2^2(\y^{(p_2)})^T\y^{(i_2)}\rangle_{\gamma_{02}^{(r)}}   \rp    \nonumber \\
 &  & -s\beta^2(1-\m_1) \mE_{G,{\mathcal U}_{r+1}}\langle (\x^{(p_1)})^T\x^{(i_1)}(\y^{(p_2)})^T\y^{(i_2)}\rangle_{\gamma_1^{(r)}}
  \nonumber \\
  &  & -s\beta^2 \sum_{k_1=2}^{2}  (\m_{k_1-1}-\m_{k_1}) p \mE_{G,{\mathcal U}_{r+1}}\langle (\x^{(p_1)})^T\x^{(i_1)}(\y^{(p_2)})^T\y^{(i_2)}\rangle_{\gamma_{21}^{(r)}}  \nonumber \\
  &  & + s\beta^2 \m_1 (p-1) \mE_{G,{\mathcal U}_{r+1}}\langle (\x^{(p_1)})^T\x^{(i_1)}(\y^{(p_2)})^T\y^{(i_2)}\rangle_{\gamma_{22}^{(r)}}
    \nonumber \\
  &  & -s\beta^2 \sum_{k_1=2}^{r+1}  (\m_{k_1-1}-\m_{k_1}) p \mE_{G,{\mathcal U}_{r+1}}\langle (\x^{(p_1)})^T\x^{(i_1)}(\y^{(p_2)})^T\y^{(i_2)}\rangle_{\gamma_{k_1}^{(r)}}.
\end{eqnarray}
With $\m_0=1$, $\gamma_{2}^{(r)}\triangleq \gamma_{21}^{(r)}$, and
\begin{eqnarray}\label{eq:rthlev2genanal38aa0}
\omega(x;p) \triangleq \begin{cases}
                         1, & \mbox{if } x=1 \\
                         p, & \mbox{otherwise}.
                       \end{cases},
\end{eqnarray}
 the above can also be  rewritten in a bit more compact form
\begin{eqnarray}\label{eq:rthlev2genanal39}
\sum_{i_1=1}^{l}\sum_{i_2=1}^{l}\sum_{j=1}^{m} \beta_{i_1 }\frac{T_{G,j}}{\sqrt{t}}
 & = & \beta^2  \lp \mE_{G,{\mathcal U}_{r+1}} \langle \|\x^{(i_1)}\|_2^2\|\y^{(i_2)}\|_2^2\rangle_{\gamma_{01}^{(r)}} +   (s-1)\mE_{G,{\mathcal U}_{r+1}}\langle \|\x^{(i_1)}\|_2^2(\y^{(p_2)})^T\y^{(i_2)}\rangle_{\gamma_{02}^{(r)}}   \rp    \nonumber \\
   &  & -s\beta^2 \sum_{k_1=1}^{r+1}  (\m_{k_1-1}-\m_{k_1}) \omega(k_1;p)  \mE_{G,{\mathcal U}_{r+1}}\langle (\x^{(p_1)})^T\x^{(i_1)}(\y^{(p_2)})^T\y^{(i_2)}\rangle_{\gamma_{k_1}^{(r)}}
 \nonumber  \\
  &  & + s\beta^2 \m_1 (p-1) \mE_{G,{\mathcal U}_{r+1}}\langle (\x^{(p_1)})^T\x^{(i_1)}(\y^{(p_2)})^T\y^{(i_2)}\rangle_{\gamma_{22}^{(r)}} .
\end{eqnarray}

\subsection{Connecting everything together}
\label{sec:connect}

We summarize the above in the following theorem.
\begin{theorem}
\label{thm:thm3}
For any integer $r\geq 2$ let $k\in\{1,2,\dots,r+1\}$. Consider three vectors $\m=[\m_0,\m_1,\m_2,...,\m_r,\m_{r+1}]$, $\p=[\p_0,\p_1,...,\p_r,\p_{r+1}]$, and $\q=[\q_0,\q_1,\q_2,\dots,\q_r,\q_{r+1}]$  such that  $\m_0=1$, $\m_{r+1}=0$, $1\geq\p_0\geq \p_1\geq \p_2\geq \dots \geq \p_r\geq \p_{r+1} =0$, and $1\geq\q_0\geq \q_1\geq \q_2\geq \dots \geq \q_r\geq \q_{r+1} = 0$. Let the variances of the independent zero-mean  normal elements of $G\in\mR^{m\times n}$, $u^{(4,k)}\in\mR$, $\u^{(2,k)}\in\mR^m$, and $\h^{(k)}\in\mR^n$ be $1$, $\p_{k-1}\q_{k-1}-\p_k\q_k$, $\p_{k-1}-\p_{k}$, $\q_{k-1}-\q_{k}$, respectively.   Also, let sets ${\mathcal X}$, $\bar{{\mathcal X}}$, ${\mathcal Y}$, scalars $\beta$, $p$, $s$, and function $f_{\bar{\x}^{ ( i_3 ) } }   (\cdot) $  be as in Proposition \ref{thm:thm1}. For  ${\mathcal U}_k\triangleq [u^{(4,k)},\u^{(2,k)},\h^{(2k)}]$ we consider the following
\begin{equation}\label{eq:thm3eq1}
\psi(t)  =  \mE_{G,{\mathcal U}_{r+1}} \frac{1}{p|s|\sqrt{n}\m_r} \log
\lp \mE_{{\mathcal U}_{r}} \lp \dots \lp \mE_{{\mathcal U}_2}\lp\lp    \sum_{i_3=1}^{l} \lp \mE_{{\mathcal U}_1}  Z_{i_3}^{\m_1}\rp^p\rp^{\frac{\m_2}{\m_1}}\rp\rp^{\frac{\m_3}{\m_2}} \dots \rp^{\frac{\m_{r}}{\m_{r-1}}}\rp,
\end{equation}
where
\begin{eqnarray}\label{eq:thm3eq2}
Z_{i_3} & \triangleq & \sum_{i_1=1}^{l}\lp\sum_{i_2=1}^{l}e^{\beta D_0^{(i_1,i_2,i_3)}} \rp^{s} \nonumber \\
 D_0^{(i_1,i_2,i_3)} & \triangleq & \sqrt{t}(\y^{(i_2)})^T
 G\x^{(i_1)}+\sqrt{1-t}\|\x^{(i_1)}\|_2 (\y^{(i_2)})^T\lp\sum_{k=1}^{r+1}\u^{(2,k)}\rp\nonumber \\
 & & +\sqrt{t}\|\x^{(i_1)}\|_2\|\y^{(i_2)}\|_2\lp\sum_{k=1}^{r+1}u^{(4,k)}\rp +\sqrt{1-t}\|\y^{(i_2)}\|_2\lp\sum_{k=1}^{r+1}\h^{(k)}\rp^T\x^{(i_1)}
 + f_{\bar{\x}^{(i_3)}  } ( \x^{ (i_1) }  ). \nonumber \\
 \end{eqnarray}
For  $\zeta_k$ defined in (\ref{eq:rthlev2genanal7a}) and (\ref{eq:rthlev2genanal7b})   set
\begin{eqnarray}\label{eq:thm3eq4}
 \Phi_{{\mathcal U}_1}^{(i_3)} & \triangleq &  \mE_{{\mathcal U}_{1}} \frac{  Z_{i_3}^{\m_1}  } {  \mE_{{\mathcal U}_{1}} Z_{i_3}^{\m_1}  } \nonumber \\
 \Phi_{{\mathcal U}_k} & \triangleq &  \mE_{{\mathcal U}_{k}} \frac{\zeta_{k-1}^{\frac{\m_k}{\m_{k-1}}}}{\zeta_k} ,k\geq 2 \nonumber \\
 \gamma_{00}(i_3) & = &
\frac{  \lp  \mE_{{\mathcal U}_{1}} Z_{i_3}^{\m_1} \rp^p   }  { \sum_{i_3=1}^{l} \lp  \mE_{{\mathcal U}_{1}} Z_{i_3}^{\m_1} \rp^p   } \nonumber \\
 \gamma_0(i_1,i_2;i_3) & = &
\frac{(C_{i_3}^{(i_1)})^{s}}{Z_{i_3}}  \frac{A_{i_3}^{(i_1,i_2)}}{C_{i_3}^{(i_1)}} \nonumber \\
\gamma_{01}^{(r)}  & = & \prod_{k=r}^{2}\Phi_{{\mathcal U}_k} \gamma_{00}(i_3) \Phi_{{\mathcal U}_1}^{(i_3)} (\gamma_0(i_1,i_2;i_3))  \nonumber \\
\gamma_{02}^{(r)}  & = & \prod_{k=r}^{2} \Phi_{{\mathcal U}_k}   \gamma_{00}(i_3)    \Phi_{{\mathcal U}_1}^{(i_3)}   (\gamma_0(i_1,i_2;i_3)\times \gamma_0(i_1,p_2;i_3)) \nonumber \\
 \gamma_{1}^{(r)}  & = & \prod_{k=r}^{2} \Phi_{{\mathcal U}_k}   \gamma_{00}(i_3)    \Phi_{{\mathcal U}_1}^{(i_3)}   (\gamma_0(i_1,i_2;i_3)\times \gamma_0(p_1,p_2;i_3)) \nonumber \\
\gamma_{2}^{(r)} \triangleq \gamma_{21}^{(r)}  & = & \prod_{k=r}^{2}\Phi_{{\mathcal U}_k} \lp
  \gamma_{00}(i_3)    \Phi_{{\mathcal U}_1}^{(i_3)}   \gamma_0(i_1,i_2;i_3)\times
    \gamma_{00}(p_3)    \Phi_{{\mathcal U}_1}^{(p_3)}
  \gamma_0(p_1,p_2;p_3) \rp  \nonumber \\
\gamma_{22}^{(r)}  & = & \prod_{k=r}^{2}\Phi_{{\mathcal U}_k} \lp
  \gamma_{00}(i_3)  \lp  \Phi_{{\mathcal U}_1}^{(i_3)}   \gamma_0(i_1,i_2;i_3)\times
     \Phi_{{\mathcal U}_1}^{(i_3)}
  \gamma_0(p_1,p_2;i_3) \rp \rp \nonumber \\
\gamma_{k_1+1}^{(r)}  & = & \prod_{k=r}^{k_1+1}\Phi_{{\mathcal U}_k} \lp \prod_{k=k_1}^{2}\Phi_{{\mathcal U}_k}
  \gamma_{00}(i_3)    \Phi_{{\mathcal U}_1}^{(i_3)}   \gamma_0(i_1,i_2;i_3)\times \prod_{k=k_1}^{2} \Phi_{{\mathcal U}_k}
    \gamma_{00}(p_3)    \Phi_{{\mathcal U}_1}^{(p_3)}
  \gamma_0(p_1,p_2;p_3) \rp, k_1\geq 2. \nonumber \\
  \end{eqnarray}
For $k_1\geq 1$ and $\omega(x;p)$ as in (\ref{eq:rthlev2genanal38aa0}), let
\begin{eqnarray}\label{eq:thm3eq5}
 \phi_{k_1}^{(r)} & = &
-s(\m_{k_1-1}-\m_{k_1}) \omega(k_1;p) \nonumber \\
&  & \times
\mE_{G,{\mathcal U}_{r+1}} \langle (\p_{k_1-1}\|\x^{(i_1)}\|_2\|\x^{(p_1)}\|_2 -(\x^{(p_1)})^T\x^{(i_1)})(\q_{k_1-1}\|\y^{(i_2)}\|_2\|\y^{(p_2)}\|_2 -(\y^{(p_2)})^T\y^{(i_2)})\rangle_{\gamma_{k_1}^{(r)}}
\nonumber \\
 \phi_{22}^{(r)} & = &
s \m_1(p-1) \nonumber \\
&  & \times
\mE_{G,{\mathcal U}_{r+1}} \langle (\p_{k_1-1}\|\x^{(i_1)}\|_2\|\x^{(p_1)}\|_2 -(\x^{(p_1)})^T\x^{(i_1)})(\q_{k_1-1}\|\y^{(i_2)}\|_2\|\y^{(p_2)}\|_2 -(\y^{(p_2)})^T\y^{(i_2)})\rangle_{\gamma_{22}^{(r)}} \nonumber \\
 \phi_{01}^{(r)} & = & (1-\p_0)(1-\q_0)\mE_{G,{\mathcal U}_{r+1}}\langle \|\x^{(i_1)}\|_2^2\|\y^{(i_2)}\|_2^2\rangle_{\gamma_{01}^{(r)}} \nonumber\\
\phi_{02}^{(r)} & = & (s-1)(1-\p_0)\mE_{G,{\mathcal U}_{r+1}}\left\langle \|\x^{(i_1)}\|_2^2 \lp\q_0\|\y^{(i_2)}\|_2\|\y^{(p_2)}\|_2-(\y^{(p_2)})^T\y^{(i_2)}\rp\right\rangle_{\gamma_{02}^{(r)}}. \end{eqnarray}

\noindent Then
\begin{eqnarray}\label{eq:thm3eq6}
\frac{d\psi(t)}{dt}  & = &       \frac{\mbox{sign}(s)\beta^2}{2\sqrt{n}} \lp  \lp\sum_{k_1=1}^{r+1} \phi_{k_1}^{(r)}\rp
+ \phi_{22}^{(r)}
+\phi_{01}^{(r)}+\phi_{02}^{(r)}\rp.
 \end{eqnarray}
For a particular choice $\p_0=\q_0=1$ one additionally has
\begin{eqnarray}\label{eq:rthlev2genanal43}
\frac{d\psi(t)}{dt}  & = &       \frac{\mbox{sign}(s)\beta^2}{2\sqrt{n}}
\lp  \sum_{k_1=1}^{r+1} \phi_{k_1}^{(r)}  + \phi_{22}^{(r)} \rp.
 \end{eqnarray}
 \end{theorem}
\begin{proof}
The $r=2$ scenario follows by results presented in  Proposition \ref{thm:thm2}. The $r\geq 3$ scenario
follows via a combination of (\ref{eq:rthlev2genanal10e})-(\ref{eq:rthlev2genanal10g}), (\ref{eq:rthlev2genanal27})-(\ref{eq:rthlev2genanal29}), (\ref{eq:rthlev2genanal39}), after one performs summing and canceling scaled terms while additionally  noting that, what was shown in (\ref{eq:lev2liftgenAanal19i}),  (\ref{eq:lev2liftgenBanal20b}), and (\ref{eq:lev2liftgenCanal21b}) for the  second level in an analogous manner holds for the $r$-th level as well. In particular, one has the following analogue of   (\ref{eq:lev2liftgenAanal19i})
\begin{eqnarray}\label{eq:rthlev2genanal40}
\sum_{i_1=1}^{l}\sum_{i_2=1}^{l}\sum_{j=1}^{m} \beta_{i_1}\frac{T_{1,1,j}}{\sqrt{1-t
}}
& = & \beta^2(\p_0-\p_1) \Bigg( \Bigg. \mE_{G,{\mathcal U}_{r+1}}\langle \|\x^{(i_1)}\|_2^2\|\y^{(i_2)}\|_2^2\rangle_{\gamma_{01}^{(r)}} \nonumber \\
& & +  (s-1)\mE_{G,{\mathcal U}_{r+1}}\langle \|\x^{(i_1)}\|_2^2(\y^{(p_2)})^T\y^{(i_2)}\rangle_{\gamma_{02}^{(r)}} \Bigg.\Bigg)  \nonumber \\
& & - (\p_0-\p_1)s\beta^2(1-\m_1)\mE_{G,{\mathcal U}_{r+1}}\langle \|\x^{(i_1)}\|_2\|\x^{(p_1)}\|_2(\y^{(p_2)})^T\y^{(i_2)} \rangle_{\gamma_1^{(r)}},
\end{eqnarray}
the following analogue of  (\ref{eq:lev2liftgenBanal20b})
\begin{eqnarray}\label{eq:rthlev2genanal41}
\sum_{i_1=1}^{l}\sum_{i_2=1}^{l} \beta_{i_1}\|\y^{(i_2)}\|_2 \frac{T_{1,2}}{\sqrt{1-t}}
& = & \beta^2(\q_0-\q_1)\Bigg( \Bigg.  \mE_{G,{\mathcal U}_{r+1}}\langle \|\x^{(i_1)}\|_2^2\|\y^{(i_2)}\|_2^2\rangle_{\gamma_{01}^{(r)}} \nonumber \\
& &
+   (s-1)\mE_{G,{\mathcal U}_{r+1}}\langle \|\x^{(i_1)}\|_2^2 \|\y^{(i_2)}\|_2\|\y^{(p_2)}\|_2\rangle_{\gamma_{02}^{(r)}}\Bigg.\Bigg) \nonumber \\
& & - (\q_0-\q_1)s\beta^2(1-\m_1)\mE_{G,{\mathcal U}_{r+1}}\langle (\x^{(p_1)})^T\x^{(i_1)}\|\y^{(i_2)}\|_2\|\y^{(p_2)}\|_2 \rangle_{\gamma_1^{(r)}},\nonumber \\
\end{eqnarray}
and the following analogue of (\ref{eq:lev2liftgenCanal21b})
  \begin{eqnarray}\label{eq:rthlev2genanal42}
\sum_{i_1=1}^{l}\sum_{i_2=1}^{l} \beta_{i_1}\|\y^{(i_2)}\|_2 \frac{T_{1,3}}{\sqrt{t}}
& = & \beta^2(\p_0\q_0-\p_1\q_1)\Bigg( \Bigg. \mE_{G,{\mathcal U}_{r+1}}\langle \|\x^{(i_1)}\|_2^2\|\y^{(i_2)}\|_2^2\rangle_{\gamma_{01}^{(r)}} \nonumber \\
& & +   (s-1)\mE_{G,{\mathcal U}_{r+1}}\langle \|\x^{(i_1)}\|_2^2 \|\y^{(i_2)}\|_2\|\y^{(p_2)}\|_2\rangle_{\gamma_{02}^{(r)}}\Bigg.\Bigg)    \nonumber \\
& & - (\p_0\q_0-\p_1\q_1)s\beta^2(1-\m_1)\mE_{G,{\mathcal U}_{r+1}}\langle \|\x^{(i_1)}\|_2\|\x^{(p_`)}\|_2\|\y^{(i_2)}\|_2\|\y^{(p_2)}\|_2 \rangle_{\gamma_1^{(r)}}.\nonumber \\
\end{eqnarray}
The summing  and canceling out procedure works out in the same way for all three sequences, $\lp T_{k,1,j}\rp_{k=1:r+1}$, $\lp T_{k,2}\rp_{k=1:r+1}$, and $\lp T_{k,3}\rp_{k=1:r+1}$. To avoid showing the same mechanism repetitively, we present in details how it works for the first sequence. By trivial
adjustments of the results from the first part of Section \ref{sec:lev2handlT2} we have the following $r$ level analogue of (\ref{eq:lev2genDanal25})
\begin{eqnarray}\label{eq:rthlev2genanal43}
 \sum_{i_1=1}^{l}  \sum_{i_2=1}^{l} \sum_{j=1}^{m}  \beta_{i_1}\frac{T_{2,1,j}}{\sqrt{1-t}}
 & = & (\p_1-\p_2)\beta^2 \nonumber \\
 & & \times
  \lp \mE_{G,{\mathcal U}_{r+1}  }\langle \|\x^{(i_1)}\|_2^2\|\y^{(i_2)}\|_2^2\rangle_{\gamma_{01}^{(r)}} +   (s-1)\mE_{G,{\mathcal U}_{r+1} }\langle \|\x^{(i_1)}\|_2^2(\y^{(p_2)})^T\y^{(i_2)}\rangle_{\gamma_{02}^{(r)}} \rp\nonumber \\
& & - (\p_1-\p_2)s\beta^2(1-\m_1)\mE_{G,{\mathcal U}_{r+1} }\langle \|\x^{(i_1)}\|_2\|\x^{(p_1)}\|_2(\y^{(p_2)})^T\y^{(i_2)} \rangle_{\gamma_{1}^{(r)}}
\nonumber \\
 &   &
  -s\beta^2(\p_1-\p_2)(\m_1-\m_2) p \mE_{G,{\mathcal U}_{r+1}   } \langle \|\x^{(i_1)}\|_2\|\x^{(p_1)}\|_2(\y^{(p_2)})^T\y^{(i_2)} \rangle_{\gamma_{21}^{(r)}}
\nonumber \\
   &   &
+  s\beta^2(\p_1-\p_2) \m_1 (p-1) \mE_{G,{\mathcal U}_{r+1}  } \langle \|\x^{(i_1)}\|_2\|\x^{(p_1)}\|_2(\y^{(p_2)})^T\y^{(i_2)} \rangle_{\gamma_{22}^{(r)}}
\nonumber \\
 & = & (\p_1-\p_2)\beta^2 \nonumber \\
 & & \times \Bigg .\Bigg (
  \lp \mE_{G,{\mathcal U}_{r+1}}\langle \|\x^{(i_1)}\|_2^2\|\y^{(i_2)}\|_2^2\rangle_{\gamma_{01}^{(r)}} +   (s-1)\mE_{G,{\mathcal U}_{r+1}}\langle \|\x^{(i_1)}\|_2^2(\y^{(p_2)})^T\y^{(i_2)}\rangle_{\gamma_{02}^{(r)}} \rp\nonumber \\
  &   &
  -s \sum_{v=1}^{2}  (\m_{v-1}-\m_v) \omega(v;p) \mE_{G,{\mathcal U}_{r+1}  } \langle \|\x^{(i_1)}\|_2\|\x^{(p_1)}\|_2(\y^{(p_2)})^T\y^{(i_2)} \rangle_{\gamma_{v}^{(r)}}
\nonumber \\
   &   &
+  s  \m_1 (p-1) \mE_{G,{\mathcal U}_{r+1}   } \langle \|\x^{(i_1)}\|_2\|\x^{(p_1)}\|_2(\y^{(p_2)})^T\y^{(i_2)} \rangle_{\gamma_{22}^{(r)}}
\Bigg .\Bigg ).
\end{eqnarray}
To generalize, we assume for $k_1\geq 2$
 \begin{eqnarray}\label{eq:rthlev2genanal44}
 \sum_{i_1=1}^{l}\sum_{i_2=1}^{l} \sum_{j=1}^{m}
\beta_{i_1} \frac{T_{k_1,1,j}}{\sqrt{1-t}}
 & = & \beta^2 (\p_{k_1-1}-\p_{k_1}) \Bigg( \Bigg. \mE_{G,{\mathcal U}_{r+1}}\langle \|\x^{(i_1)}\|_2^2\|\y^{(i_2)}\|_2^2\rangle_{\gamma_{01}^{(r)}} \nonumber \\
 & & +  (s-1)\mE_{G,{\mathcal U}_{r+1}}\langle \|\x^{(i_1)}\|_2^2(\y^{(p_2)})^T\y^{(i_2)}\rangle_{\gamma_{02}^{(r)}}   \nonumber \\
 &  & - s \sum_{v=1}^{k_1} \lp \m_{v-1} -  \m_{v} \rp \omega(v;p)
  \mE_{G,{\mathcal U}_{r+1}} \left \langle \|\x^{(i_1)}\|_2\|\x^{(p_1)}\|_2 (\y^{(p_2)})^T \y^{(i_2)} \right \rangle_{\gamma_{v}^{(r)} }
  \nonumber \\
   &   &
+  s  \m_1 (p-1) \mE_{G,{\mathcal U}_{r+1}   } \langle \|\x^{(i_1)}\|_2\|\x^{(p_1)}\|_2(\y^{(p_2)})^T\y^{(i_2)} \rangle_{\gamma_{22}^{(r)}}
  \Bigg. \Bigg),\nonumber \\
\end{eqnarray}
and utilizing  (\ref{eq:rthlev2genanal27}) find
 \begin{eqnarray}\label{eq:rthlev2genanal45}
 \sum_{i_1=1}^{l}\sum_{i_2=1}^{l} \sum_{j=1}^{m}
\beta_{i_1} \frac{T_{k_1+1,1,j}}{\sqrt{1-t}}
 & = &  \frac{\p_{k_1}-\p_{k_1+1}}{\p_{k_1-1}-\p_{k_1}}\sum_{i_1=1}^{l}\sum_{i_2=1}^{l} \sum_{j=1}^{m}
\beta_{i_1} \frac{T_{k_1,1,j}}{\sqrt{1-t}} \nonumber \\
 &  & - \Bigg( \Bigg. s\beta^2(\p_{k_1}-\p_{k_1+1}) \lp \m_{k_1} -  \m_{k_1+1}
\rp  \omega(k_1;p)
  \nonumber \\
  & & \times
  \mE_{G,{\mathcal U}_{r+1}} \left \langle \|\x^{(i_1)}\|_2\|\x^{(p_1)}\|_2 (\y^{(p_2)})^T \y^{(i_2)} \right \rangle_{\gamma_{k_1+1}^{(r)} } \Bigg. \Bigg) \nonumber \\
 & = &  \frac{\p_{k_1}-\p_{k_1+1}}{\p_{k_1-1}-\p_{k_1}}\beta^2 (\p_{k_1-1}-\p_{k_1}) \Bigg( \Bigg. \mE_{G,{\mathcal U}_{r+1}}\langle \|\x^{(i_1)}\|_2^2\|\y^{(i_2)}\|_2^2\rangle_{\gamma_{01}^{(r)}} \nonumber \\
 & & +  (s-1)\mE_{G,{\mathcal U}_{r+1}}\langle \|\x^{(i_1)}\|_2^2(\y^{(p_2)})^T\y^{(i_2)}\rangle_{\gamma_{02}^{(r)}}   \nonumber \\
 &  & - s \sum_{v=1}^{k_1} \lp \m_{v-1} -  \m_{v} \rp \omega(v;p)
  \nonumber \\
  & & \times
  \mE_{G,{\mathcal U}_{r+1}} \left \langle \|\x^{(i_1)}\|_2\|\x^{(p_1)}\|_2 (\y^{(p_2)})^T \y^{(i_2)} \right \rangle_{\gamma_{v}^{(r)} }
        \nonumber \\
   &   &
+  s  \m_1 (p-1) \mE_{G,{\mathcal U}_{r+1}   } \langle \|\x^{(i_1)}\|_2\|\x^{(p_1)}\|_2(\y^{(p_2)})^T\y^{(i_2)} \rangle_{\gamma_{22}^{(r)}}
    \Bigg. \Bigg) \nonumber \\
   &  & - \Bigg( \Bigg. s \beta^2(\p_{k_1}-\p_{k_1+1}) \lp \m_{k_1} -  \m_{k_1+1} \rp \omega(k_1;p)
  \nonumber \\
  & & \times
  \mE_{G,{\mathcal U}_{r+1}} \left \langle \|\x^{(i_1)}\|_2\|\x^{(p_1)}\|_2 (\y^{(p_2)})^T \y^{(i_2)} \right \rangle_{\gamma_{k_1+1}^{(r)} }
 \Bigg. \Bigg)
   \nonumber \\
 & = &  \beta^2 (\p_{k_1}-\p_{k_1+1}) \Bigg( \Bigg. \mE_{G,{\mathcal U}_{r+1}}\langle \|\x^{(i_1)}\|_2^2\|\y^{(i_2)}\|_2^2\rangle_{\gamma_{01}^{(r)}} \nonumber \\
 & & +  (s-1)\mE_{G,{\mathcal U}_{r+1}}\langle \|\x^{(i_1)}\|_2^2(\y^{(p_2)})^T\y^{(i_2)}\rangle_{\gamma_{02}^{(r)}}   \nonumber \\
 &  & - s \sum_{v=1}^{k_1+1} \lp \m_{v-1} -  \m_{v} \rp \omega(v;p)
  \mE_{G,{\mathcal U}_{r+1}} \left \langle \|\x^{(i_1)}\|_2\|\x^{(p_1)}\|_2 (\y^{(p_2)})^T \y^{(i_2)} \right \rangle_{\gamma_{v}^{(r)} }
    \nonumber \\
   &   &
+  s  \m_1 (p-1) \mE_{G,{\mathcal U}_{r+1}   } \langle \|\x^{(i_1)}\|_2\|\x^{(p_1)}\|_2(\y^{(p_2)})^T\y^{(i_2)} \rangle_{\gamma_{22}^{(r)}}
    \Bigg. \Bigg)
. \nonumber \\
    \end{eqnarray}
Since (\ref{eq:rthlev2genanal44}) holds for $k_1=2$, (\ref{eq:rthlev2genanal45}) ensures that it actually holds for any $k_1\in\{1,2,\dots,r+1\}$. Finally summing over $k_1$ and utilizing \cite{Stojnicnflgscompyx23}'s (249) we obtain
 \begin{eqnarray}\label{eq:rthlev2genanal46}
 \sum_{k_1=1}^{r+1}\sum_{i_1=1}^{l}\sum_{i_2=1}^{l} \sum_{j=1}^{m}
\beta_{i_1} \frac{T_{k_1,1,j}}{\sqrt{1-t}}
 & = & \beta^2  \sum_{k_1=1}^{r+1}  (\p_{k_1-1}-\p_{k_1}) \Bigg( \Bigg. \mE_{G,{\mathcal U}_{r+1}}\langle \|\x^{(i_1)}\|_2^2\|\y^{(i_2)}\|_2^2\rangle_{\gamma_{01}^{(r)}} \nonumber \\
 & & +  (s-1)\mE_{G,{\mathcal U}_{r+1}}\langle \|\x^{(i_1)}\|_2^2(\y^{(p_2)})^T\y^{(i_2)}\rangle_{\gamma_{02}^{(r)}}   \nonumber \\
 &  & - s \sum_{v=1}^{k_1} \lp \m_{v-1} -  \m_{v} \rp \omega(v;p)
  \mE_{G,{\mathcal U}_{r+1}} \left \langle \|\x^{(i_1)}\|_2\|\x^{(p_1)}\|_2 (\y^{(p_2)})^T \y^{(i_2)} \right \rangle_{\gamma_{k_1}^{(r)} }
\Bigg. \Bigg)
       \nonumber \\
   &   &
+  s  \sum_{k_1=2}^{r+1}  (\p_{k_1-1}-\p_{k_1})  \m_1 (p-1) \mE_{G,{\mathcal U}_{r+1}   } \langle \|\x^{(i_1)}\|_2\|\x^{(p_1)}\|_2(\y^{(p_2)})^T\y^{(i_2)} \rangle_{\gamma_{22}^{(r)}}
\nonumber \\
  & = & \beta^2 (\p_{0}-\p_{r+1}) \Bigg( \Bigg. \mE_{G,{\mathcal U}_{r+1}}\langle \|\x^{(i_1)}\|_2^2\|\y^{(i_2)}\|_2^2\rangle_{\gamma_{01}^{(r)}} \nonumber \\
 & & +  (s-1)\mE_{G,{\mathcal U}_{r+1}}\langle \|\x^{(i_1)}\|_2^2(\y^{(p_2)})^T\y^{(i_2)}\rangle_{\gamma_{02}^{(r)}}  \Bigg. \Bigg) \nonumber \\
 &  & - s  \beta^2\sum_{v=1}^{r+1} \p_{v-1}\lp \m_{v-1} -  \m_{v} \rp \omega(v;p) \mE_{G,{\mathcal U}_{r+1}} \left \langle \|\x^{(i_1)}\|_2\|\x^{(p_1)}\|_2 (\y^{(p_2)})^T \y^{(i_2)} \right \rangle_{\gamma_{v}^{(r)} }
        \nonumber \\
   &   &
+  s  \beta^2 \p_{1}  \m_1 (p-1) \mE_{G,{\mathcal U}_{r+1}   } \langle \|\x^{(i_1)}\|_2\|\x^{(p_1)}\|_2(\y^{(p_2)})^T\y^{(i_2)} \rangle_{\gamma_{22}^{(r)}}
 .\nonumber \\
\end{eqnarray}
We then analogously find for the other two sequences
 \begin{eqnarray}\label{eq:rthlev2genanal47}
 \sum_{k_1=1}^{r+1}\sum_{i_1=1}^{l}\sum_{i_2=1}^{l}
\beta_{i_1}\|\y^{(i_2)}\|_2 \frac{T_{k_1,2}}{\sqrt{1-t}}
   & = & \beta^2 (\q_{0}-\q_{r+1}) \Bigg( \Bigg. \mE_{G,{\mathcal U}_{r+1}}\langle \|\x^{(i_1)}\|_2^2\|\y^{(i_2)}\|_2^2\rangle_{\gamma_{01}^{(r)}} \nonumber \\
 & & +  (s-1)\mE_{G,{\mathcal U}_{r+1}}\langle \|\x^{(i_1)}\|_2^2 \|\y^{(i_2)}\|_2\|\y^{(p_2)}\|_2\rangle_{\gamma_{02}^{(r)}} \Bigg. \Bigg)  \nonumber \\
 &  & - s \beta^2 \sum_{v=1}^{r+1} \q_{v-1}\lp \m_{v-1} -  \m_{v} \rp   \omega(v;p)
  \nonumber \\
  & & \times
  \mE_{G,{\mathcal U}_{r+1}} \left \langle \|\y^{(i_2)}\|_2\|\y^{(p_2)}\|_2 (\x^{(p_1)})^T \x^{(i_1)} \right \rangle_{\gamma_{v}^{(r)} }
          \nonumber \\
   &   &
+  s  \beta^2 \p_{1}  \m_1 (p-1) \mE_{G,{\mathcal U}_{r+1}   } \langle \|\x^{(i_1)}\|_2\|\x^{(p_1)}\|_2(\y^{(p_2)})^T\y^{(i_2)} \rangle_{\gamma_{22}^{(r)}},
\nonumber \\
\end{eqnarray}
and
 \begin{eqnarray}\label{eq:rthlev2genanal48}
 \sum_{k_1=1}^{r+1}\sum_{i_1=1}^{l}\sum_{i_2=1}^{l}
\beta_{i_1}\|\y^{(i_2)}\|_2 \frac{T_{k_1,3}}{\sqrt{t}}
   & = & \beta^2 (\p_{0}\q_{0}-\p_{r+1}\q_{r+1}) \Bigg( \Bigg. \mE_{G,{\mathcal U}_{r+1}}\langle \|\x^{(i_1)}\|_2^2\|\y^{(i_2)}\|_2^2\rangle_{\gamma_{01}^{(r)}} \nonumber \\
 & & +  (s-1)\mE_{G,{\mathcal U}_{r+1}}\langle \|\x^{(i_1)}\|_2^2\|\y^{(i_2)}\|_2\|\y^{(p_2)}\|_2\rangle_{\gamma_{02}^{(r)}} \Bigg. \Bigg)  \nonumber \\
 &  & - s  \sum_{v=1}^{r+1} \p_{v-1}\q_{v-1}\lp \m_{v-1} -  \m_{v} \rp   \omega(v;p)
 \nonumber \\
 & & \times
 \mE_{G,{\mathcal U}_{r+1}} \left \langle \|\y^{(i_2)}\|_2\|\y^{(p_2)}\|_2 \|\x^{(i_1)}\|_2\|\x^{(p_1)}\|_2 \right \rangle_{\gamma_{v}^{(r)} }
         \nonumber \\
   &   &
+  s  \beta^2 \p_{1}  \m_1 (p-1) \mE_{G,{\mathcal U}_{r+1}   } \langle \|\x^{(i_1)}\|_2\|\x^{(p_1)}\|_2(\y^{(p_2)})^T\y^{(i_2)} \rangle_{\gamma_{22}^{(r)}}
, \nonumber \\
\end{eqnarray}
Combining (\ref{eq:rthlev2genanal10e})-(\ref{eq:rthlev2genanal10g}), (\ref{eq:rthlev2genanal39}), and (\ref{eq:rthlev2genanal46})-(\ref{eq:rthlev2genanal48}) one obtains that (\ref{eq:thm3eq5}) and (\ref{eq:thm3eq6}) indeed hold.
\end{proof}

\section{Practical relevance}
\label{sec:examples}

As noted in \cite{Stojnicnflgscompyx23}, considered models and associated  bilinearly indexed random processes relate to a host of well known random structures and accompanying classical optimization problems. These include both feasibility and standard optimal objective seeking ones. The key features associated with such problems (optimal objective, critical dimensionality, etc.) are usually associated with the so-called  \emph{typical} behavior and the concepts from \cite{Stojnicnflgscompyx23} are powerful enough to handle them. On the other hand, in recent years features associated with \emph{atypical} behavior gained a lot of interest as they are presumed to be tightly connected to the appearance of the so called \emph{computational gap} -- a phenomenon where known practically feasible algorithms fail to achieve the predicted theoretical limit. The machinery introduced here presents an  upgrade on \cite{Stojnicnflgscompyx23} that
is powerful enough to allow for studying a much wider class of random structures features including algorithmically the most sought after atypical ones.

To see how this connection materializes, we  first recall that the key object of interest from Theorem \ref{thm:thm3} is the following function
\begin{equation}\label{eq:exampleseq1}
\psi( t)  =  \mE_{G,{\mathcal U}_{r+1}} \frac{1}{ p|s|\sqrt{n}\m_r} \log
\lp \mE_{{\mathcal U}_{r}} \lp \dots \lp \mE_{{\mathcal U}_2}\lp\lp   \sum_{i_3=1}^{l}\lp \mE_{{\mathcal U}_1} Z^{\m_1}\rp^p \rp^{\frac{\m_2}{\m_1}}\rp\rp^{\frac{\m_3}{\m_2}} \dots \rp^{\frac{\m_{r}}{\m_{r-1}}}\rp,
\end{equation}
where
\begin{eqnarray}\label{eq:exampleseq2}
Z_{i_3} & \triangleq & \sum_{i_1=1}^{l}\lp\sum_{i_2=1}^{l}e^{\beta D_0^{(i_1,i_2,i_3)}} \rp^{s} \nonumber \\
 D_0^{(i_1,i_2,i_3)} & \triangleq & \sqrt{t}(\y^{(i_2)})^T
 G\x^{(i_1)}+\sqrt{1-t}\|\x^{(i_1)}\|_2 (\y^{(i_2)})^T\lp\sum_{k=1}^{r+1}\u^{(2,k)}\rp\nonumber \\
 & & +\sqrt{t}\|\x^{(i_1)}\|_2\|\y^{(i_2)}\|_2\lp\sum_{k=1}^{r+1}u^{(4,k)}\rp +\sqrt{1-t}\|\y^{(i_2)}\|_2\lp\sum_{k=1}^{r+1}\h^{(k)}\rp^T\x^{(i_1)}
 + f_{\x^{(i_3)} }  (\x^{(i_1)}  ), \nonumber \\
 \end{eqnarray}
and
${\mathcal X}=\{\x^{(1)},\x^{(2)},\dots,\x^{(l)}\}$ with $\x^{(i)}\in \mR^{n},1\leq i\leq l$ and ${\mathcal Y}=\{\y^{(1)},\y^{(2)},\dots,\y^{(l)}\}$ with $\y^{(i)}\in \mR^{m},1\leq i\leq l$. The so-called linear regime, $\frac{m}{n}=\alpha$, where $\alpha$ remains constant as $n\rightarrow\infty$ is of prevalent interest.

As \cite{Stojnicnflgscompyx23} demonstrated, classical neural network perceptron concepts can be cast as feasibility problems and fit the model studied here. For example, for the so-called spherical perceptron one takes $s=-1$, ${\mathcal X}=\mS^n$, $\bar{{\mathcal X}}=\emptyset$, ${\mathcal Y}=\mS_+^n$  (where $\mS_+^m$ is the unit sphere positive orthant portion), and $p=1$ and then observes that $\psi(\cdot)$ in (\ref{eq:exampleseq1}) is properly adjusted associated free energy (see, e.g., \cite{StojnicGardGen13,StojnicGardSphErr13,StojnicGardSphNeg13,GarDer88,Gar88,SchTir02,SchTir03}).
As is the case in general random optimizations, two values of this function are often of key relevance (see, e.g., \cite{FPSUZ17,FraHwaUrb19,FraPar16,FraSclUrb19,FraSclUrb20,AlaSel20,StojnicGardGen13,StojnicGardSphErr13,StojnicGardSphNeg13,GarDer88,Gar88,Schlafli,Cover65,Winder,Winder61,Wendel62,Cameron60,Joseph60,BalVen87,Ven86,SchTir02,SchTir03}). In particular, we have
 \begin{equation}\label{eq:exampleseq7}
  \lim_{n,\beta\rightarrow\infty} \frac{\psi\lp 1\rp }{\beta}=  \lim_{n\rightarrow\infty}
  \mE_G
   \frac{\min_{\x\in\mS^m}\max_{\y\in\mS_+^m} \y^TG\x}{\sqrt{n}}
\end{equation}
 and  $\lim_{n,\beta\rightarrow\infty}\frac{\psi\lp 0\rp}{\beta}$ which are precisely the associated ground state energy and its intended (presumably analytically  simpler) decoupled counterpart in the thermodynamic limit.

The famous binary perceptron (BP) (see, e.g., \cite{StojnicGardGen13,GarDer88,Gar88,StojnicDiscPercp13,KraMez89,GutSte90,KimRoc98,TalBook11a,NakSun23,BoltNakSunXu22,PerkXu21,CXu21,DingSun19,Huang24,Stojnicbinperflrdt23,LiSZ24}) fits the considered model in exactly the same way as the spherical does with the only difference being that now ${\mathcal X}=\{-\frac{1}{\sqrt{n}},\frac{1}{\sqrt{n}}\}^n$. One then analogously has that
\begin{equation}\label{eq:exampleseq8}
  \lim_{n,\beta\rightarrow\infty} \frac{\psi\lp 1\rp }{\beta}
   =  \lim_{n\rightarrow\infty} \mE_G \frac{\min_{\x\in\left \{-\frac{1}{\sqrt{m}},\frac{1}{\sqrt{m}}\right \}^n}\max_{\y\in\mS_+^m} \y^TG\x}{\sqrt{n}}
\end{equation}
 and  $\lim_{n,\beta\rightarrow\infty} \frac{ \psi\lp  0\rp }{\beta}$ are the  thermodynamic limit values of the associated ground state energy and its decoupled counterpart (the concentrations actually ensure that $\mE_G$ can be removed in  (\ref{eq:exampleseq7}) and  (\ref{eq:exampleseq8})). Various other perceptron variants including closely related symmetric binary perceptrons (SBP) from \cite{AubPerZde19,AbbLiSly21a,AbbLiSly21b,Bald20,GamKizPerXu22,PerkXu21,ElAlGam24,SahSaw23,Barb24,djalt22,BarbAKZ23} (or their discrepancy minimization equivalents from, e.g., \cite{KaKLO86,Spen85,LovMek15,GamKizPerXu23,Roth17,AlwLiuSaw21})  can be fit into the model in an analogous fashion as well.

 The machinery of \cite{Stojnicnflgscompyx23}  can be used to determine the critical $\alpha$ (the so-called perceptron capacity, $\alpha_c$) for which $\psi(1)$ transitions from zero to a nonzero value which corresponds to BP transitioning from successful to unsuccessful associative memory. Moreover, \cite{Stojnicnflgscompyx23} can be used to characterize various other features associated with the so-called typical behavior. For example, when one operates in regimes below the capacity ($\alpha< \alpha_c$), there will be a large number of feasibility problem solutions (BP memories) and they are ordered in a typical hierarchical way. The key features of these solutions, their number (entropy) as well as how they relate to each other can be precisely determined through the utilization of \cite{Stojnicnflgscompyx23}. On the other hand, when one tries to algorithmically find these solutions, an NP problem is faced. Since NP-ness is a \emph{worst case} concept it rarely can give a proper explanation about typical solvability of underlying optimization problems. Things are practically actually even worse. For example, a host of excellent BP algorithms exist that work fairly well for a large portion of $\alpha< \alpha_c$ region (see, e.g., \cite{BrZech06,BaldassiBBZ07,Hubara16,KimRoc98}). In concrete terms, the zero threshold BP (precisely the one characterized via (\ref{eq:exampleseq8})) has the capacity $\alpha_c\approx 0.833$ and very good algorithms can be constructed in the $\alpha$-regions up to  $\alpha_{alg}\approx 0.75$ (sometimes this can be even pushed to $0.77$). It is an extraordinary challenge to determine precisely the critical algorithmic value $\alpha_{alg}$ below which the efficient algorithms exist and to characterize the structural phenomenology behind their existence. The fact that the current $\alpha_{alg}$ is below the $\alpha_c$ constitutes the presence of the so-called \emph{computational gap} -- a key feature present in a plethora of (random) optimization problems (see, e.g., \cite{MMZ05,GamarSud14,GamarSud17,GamarSud17a,AchlioptasR06,AchlioptasCR11,GamMZ22}).

 An enormous effort has been put forth in studying these concepts over the last 10-15 years in a range of optimization problems that go way beyond the BP.  Many great results have been achieved but the generic conclusions and overall resolution remain unreachable. Among the most attractive approaches are studies of clustered solutions. The so-called \emph{overlap gap property} (OGP) approach \cite{Gamar21,GamarSud14,GamarSud17,GamarSud17a,AchlioptasCR11,MMZ05} connects algorithmic efficacy to the existence of pairs (in general $m$-tuples) of solutions at arbitrary Hamming distances. For example, for an analytically presumably more convenient SBP alternative, it is known (see, e.g., \cite{GamKizPerXu22,Bald20} or \cite{GamKizPerXu23} for the corresponding discrepancy minimization) that OGP is absent for $\alpha$'s well below the critical capacity, indicating potential existence of computational gap (provided that algorithmic tractability is indeed impacted by OGP; for a recent shortest path problem example that disproves OGP generic hardness impact see, e.g., \cite{LiSch24} (earlier disproving examples included simple algebraically solvable $k$-XOR ones); it should be noted  however, that this does not disprove OGP's relevance when it comes to particular algorithms or potentially its generic hardness value for different optimization problems). While the OGP's generic role with respect to hardness remains to be determined, its absence precludes many particular classes of algorithms (see, e.g.,  \cite{GamKizPerXu22}). In fact, the OGP's relevance  is further strengthened for many problems by the existence of practically  feasible algorithms in a similar or even the exact $\alpha$ range (see, e.g., \cite{RahVir17,GamarSud14,GamarSud17,GamAW24,Wein22}). Its direct impact on polynomial solvability of famous SK model \cite{SheKir72} might be even more valuable \cite{Montanari19}.

 An alternative view to the above OGP was put forth in \cite{Huang13,Huang14} where the entropy of typical solutions was considered. A stronger entropic refinement was then presented in \cite{Bald15,Bald16,Bald20}, where the existence of atypical well connected clusters is considered and the concepts related to the so-called \emph{local entropy} (LE) are proposed and discussed. The idea is that while typical solutions might be well apart (basically disconnected from each other and unreachable via simple bit flipping type of local roaming \cite{Huang13,Huang14,PerkXu21,AbbLiSly21b}), there still exist well connected but rather rare (subdominant) clusters. It is then postulated (and in some cases fully justified by running concrete algorithms) that somewhat miraculously efficient algorithms manage to find precisely these rare clusters (for a sampling type of justification in this direction for SBP, see, e.g., \cite{ElAlGam24}). This consequently suggested that properties  of such clusters might be responsible for generic computational hardness. In particular, behavior of their local entropy (its monotonicity, negativity, or even non-existence) was speculated in \cite{Bald15,Bald16,Bald20} to be directly connected to the algorithmic hardness. We should also add that for the SBP, \cite{AbbLiSly21a} showed that for sufficiently low $\alpha$ rare well-connected clusters of maximal diameter indeed exist. Moreover, existence of such clusters, albeit of linear diameter, was also shown in \cite{AbbLiSly21a} even for any $\alpha<\alpha_c$. Reconnecting back to the OGP, \cite{BarbAKZ23} showed that LE results for small $\alpha$ SBP at least scaling wise to a large degree match the OGP predictions of \cite{GamKizPerXu22} (and modulo a log term the algorithmic performance of \cite{BanSpen20}).  While this establishes a nice OGP -- LE correspondence,
 providing a definite answer as to whether or not any of these properties indeed critically impacts \emph{generic} algorithmic hardness (and if so then precisely how and why) remains extraordinary challenge. On the other hand, irrespective to their algorithmic hardness relevance, all of these properties  provide rather deep insights into the intrinsic organization of random structures and studying them warrants an independent interest as well.

Clearly, all of the above concepts are logically more foundational and analytically much harder than the original studying of the capacity and its phase-transitioning behavior. Consequently, with a few special cases exceptions the analytical progress is often restricted to statistical physics replica methods. The machinery that we introduce here provides a large deviation view of the underlying phenomena and  as such can be used for studying atypical features. To see this practically, let us again look at the binary perceptron and for a given configurational overlap  $\bar{\delta}$, take  $\bar{{\mathcal X}}=\{-\frac{1}{\sqrt{n}},\frac{1}{\sqrt{n}}\}^n$, ${\mathcal X} ={\mathcal X}^{(i_3)} = \{\x| \x\in  \bar{{\mathcal X}}, \lp \bar{\x}^{(i_3)}\rp^T \x = \bar{\delta}  \} $,  and $f_{\bar{\x}^{(i_3)} }=0$. For $h(\cdot)$ denoting the standard Heaviside function, one then observes that
 \begin{equation}\label{eq:exampleseq8aa0}
  \lim_{n,\beta,p\rightarrow\infty} \frac{\psi\lp 1\rp }{\sqrt{n}}
   =  \lim_{n\rightarrow\infty} \frac{
    \log  \lp \max_{\bar{\x}^{(i_3)} } \sum_{\x^{(i_1)}\in {\mathcal X}, \lp \bar{\x}^{(i_3)}\rp^T \x^{(i_1)} = \bar{\delta}  } \lp 1-h\lp \max_{\y\in\mS_+^m} \y^TG\x^{(i_1)}    \rp \rp\rp
    }
    {n} \triangleq \sigma(\bar{\delta})
\end{equation}
 and  $\lim_{n,\beta\rightarrow\infty} \frac{ \psi\lp  0\rp }{\beta}$ are the  thermodynamic limit values of the local entropy (LE), $\sigma(\bar{\delta})$, and its decoupled counterpart (these are the LEs of a general so-called reference configuration; various other highly relevant types of LEs can be considered as well, see, e.g., \cite{Bald15,Bald21,BaldMPZ23,BMPZ23} for BP or e.g., \cite{BarbAKZ23} for SBP). It is not that difficult to see that  $\sigma(\bar{\delta})$ is effectively the $n$-scaled exponent of the number of solutions in the densest cluster at Hamming distance $\frac{1-\bar{\delta}}{2}$ from the cluster's center. We also note that the constraint free variant of (\ref{eq:exampleseq8aa0}) can be conveniently obtained by taking $f_{\bar{\x}^{(i_3)} }=\nu \lp  \lp \bar{\x}^{(i_3)}\rp^T \x - \bar{\delta}  \rp $, where scalar $\nu$ is an optimizing variable (this avenue is pursued in companion paper \cite{Stojnicabple25} in full detail).

The above example is only for illustrative purposes and is selected due to perceptron's enormous popularity over the last decade primarily motivated by the machine learning and neureal networks applications. However, a rather large  collection of closely related relevant examples exists as well. In addition to the feasibility problems (such as perceptrons) it  also includes many standard optimal objective seeking optimizations as well. For example, in famous Hopfield models, one is often interested in behavior of the free energy. Usually a single optimal configuration is associated with the ground state energy values (in a similar way how a single typical solution is associated with the $\alpha_c$ capacity in perceptrons). However, away from the ground state, many other energetic band levels can be achieved by a large set of configurations. Their typical and atypical behavior can then be studied via the concepts introduced here. While the key mathematical engines that each of these studies relies on are already shown here,  additional  problem specific technical adjustments are often needed as well so that one can obtain practically usable concrete results. For several widely popular problems we discuss these adjustments in separate companion  papers.

\section{Conclusion}
\label{sec:lev2x3lev2liftconc}

In \cite{Stojnicnflgscompyx23} a powerful \emph{fully lifted} (fl)  probabilistic interpolating  mechanism of bilinearly indexed random processes (blirps) was introduced. We here present a large deviation view upgrade. Differently from\cite{Stojnicnflgscompyx23}, which is intended for studying \emph{typical} features of random structures, the large deviations nature of the machinery introduced here allows to substantially extend the range of applicability. In particular, studying analytically much harder \emph{atypical} features is now possible as well.

To give a bit of a flavor  regarding practical applications, we discussed how the presented concepts and underlying models directly relate to studying the so-called \emph{computational gaps} in NP problems.  The existence of atypical dense solutions clusters is predicated to play a key role in justifying the success of efficient algorithms. Studying such atypicalities  is therefore certainly of great interest for many classes of algorithms but on its own or together with other features might be even universally relevant driving force behind the computational gaps. Due to their enormous popularity over the last decade, perceptrons were chosen as  convenient illustrative examples to highlight the connection between the presented mathematical concepts and the relevance of studying atypical random structures' features. However, the presented machinery is a very powerful generic tool that allows for various further extensions and a range of applications way beyond classical perceptrons. As such extensions/applications are problem specific, we discuss them together with the relevance of the results that they provide in separate papers.

\begin{singlespace}
\bibliographystyle{plain}
\bibliography{nflgscompyxRefs}
\end{singlespace}

\end{document}